\documentclass[12pt]{report}

\usepackage[francais]{babel}
\usepackage[latin1]{inputenc}
\usepackage[T1]{fontenc}
\usepackage{ae,aecompl,aeguill}
\usepackage[dvips,final]{graphicx}
\usepackage{amssymb,amsfonts,amsthm,amsmath}
\usepackage{graphicx}
\usepackage{xcolor}
\usepackage{geometry}
\usepackage{multirow}
\usepackage{stmaryrd}
\usepackage[french]{minitoc}

\usepackage{babelbib}
\theoremstyle{plain}
\newtheorem{thm}{Théorème}[chapter]
\newtheorem{lem}[thm]{Lemme}
\newtheorem{prop}[thm]{Proposition}
\newtheorem{coro}[thm]{Corollaire}

\theoremstyle{definition}
\newtheorem{defi}[thm]{Définition}

\theoremstyle{remark}
\newtheorem{rem}[thm]{Remarque}
\newtheorem{ex}[thm]{Exemple}

\newcommand\mat[4]{{\left(\begin{smallmatrix} #1 & #2\\#3 & #4\end{smallmatrix}\right)}}

\def\Z{{\mathbb{Z}}}
\def\C{{\mathbb{C}}}
\def\R{{\mathbb{R}}}
\def\Q{{\mathbb{Q}}}
\def\pte{{\mathbb{P}^1(\mathbb{Q})}}
\def\H{{\mathcal{H}}}
\def\G{{SL_2(\Z)}}
\def\S{{\mathfrak{S}}}
\renewcommand{\d}{\, {d}}
\renewcommand{\Re}{\mathrm{Re}}
\renewcommand{\dim}{\mathrm{dim}}
\renewcommand{\Im}{\mathrm{Im}}

\def\Ker{{\mathrm{Ker}}}
\def\Per{{\mathrm{Per}}}
\def\Ser{{\mathrm{Ser}}}
\def\holo{{\mathrm{holo}}}
\def\Res{{\mathrm{Res}}}

\title{
Valeurs multiples de fonctions $L$ de formes modulaires
\\ \large{Sous la direction de Loïc Merel}
}
\author{Nicolas Provost}

\date{9\text{ décembre }2014}
\usepackage[refpage]{nomencl}
\makenomenclature

\begin{document}


\thispagestyle{empty}

\begin{center}
  \begin{tabular}{ll}
	\includegraphics[height = 3cm,width = 1cm]{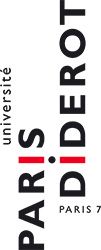}
	& {\large Université Paris VII - Denis Diderot}
  \end{tabular}
\end{center}

\begin{center}
\vspace{\stretch{1}}
{\Large \textbf{École Doctorale de Science Mathématiques de Paris Centre}}

\vspace{\stretch{2}}

{\Huge \textsc{Thèse de doctorat}}

\vspace{\stretch{1}}

{\LARGE Discipline : Mathématiques}

\vspace{\stretch{3}}

{\large présentée par}
\vspace{\stretch{1}}

\textbf{{\LARGE Nicolas \textsc{Provost}}}

\vspace{\stretch{2}}
\hrule
\vspace{\stretch{1}}
{\LARGE \textbf{
Valeurs multiples de fonctions $L$ de formes modulaires
}}
\vspace{\stretch{1}}
\hrule
\vspace{\stretch{2}}

{\Large 
dirigée par Loïc \textsc{Merel}}

\vspace{\stretch{3}}

{\Large Soutenue le 12 décembre 2014 devant le jury composé de :}

\vspace{\stretch{1}}
{\Large
\begin{tabular}{ll}
M. Pierre \textsc{Cartier} & Institut des Hautes \'Etudes Scientifiques\\
M. Loïc \textsc{Merel} & Université Denis Diderot Paris\\
M. Joseph \textsc{Oesterlé} & Université Pierre et Marie Curie Paris\\
M. Tanguy \textsc{Rivoal} & Institut Fourier Grenoble\\
M. Don \textsc{Zagier} & Max-Planck Institute Bonn
\end{tabular}
}
\vspace{\stretch{3}}

{\Large Rapportée par :}

\vspace{\stretch{1}}
{\Large
\begin{tabular}{ll}
M. Mladen \textsc{Dimitrov} & Université Lille 1\\
M. Don \textsc{Zagier} & Max-Planck Institute Bonn
\end{tabular}
}

\end{center}


\newpage

\vspace*{\fill}

\noindent
\begin{flushright}

\begin{minipage}[0]{9cm}
Institut de Mathématiques de Jussieu\\
Paris Rive Gauche\\
58-56, avenue de France \\
75 013 Paris\\
\\
UPMC - Paris Diderot\\
 Ecole Doctorale de Sciences\\
 Mathématiques de Paris Centre\\
 4 place Jussieu\\
 75252 Paris Cedex 05
 \end{minipage}
\end{flushright}

\newpage

\hrule
\begin{center}
\huge \textbf{Remerciements}
\end{center}
\hrule
\vspace{2.8cm}
{\large
\par
 
Je suis heureux d'exprimer ici ma reconnaissance à mon directeur de thèse, Loïc Merel, qui m'a proposé ce sujet d'études particulièrement propice à la recherche et à l'inventivité. Ses orientations et nos échanges ont toujours été d'une aide précieuse. Son soutien durant cette période et ses conseils toujours pertinents ont été très utiles à l'avancement de ce travail.\par

Mes remerciements vont également à Daniel Bertrand qui, en tant que tuteur, a su être à la fois à l'écoute et force de proposition pour rendre mon travail le plus accessible, me permettant d'avoir un nouveau point de vue de ce travail.\par

Don Zagier et Mladen Dimitrov m'ont fait l'honneur de bien vouloir officier comme rapporteurs de cette thèse. Je les en remercie vivement. Leurs remarques ouvrent de nombreuses pistes de progrès pour poursuivre sur ce sujet de recherche.\par

Je remercie également chaleureusement, Pierre Cartier, Joseph Oesterlé et Tanguy Rivoal pour avoir accepté de faire partie de mon jury.\par

Je remercie mes professeurs de Master $2$, et en particulier Marc Hindry, Christophe Cornut et David Harari, qui m'ont permis de découvrir un large panorama des Mathématiques que j'apprécie. Je remercie tout particulièrement Pierre Charollois qui en acceptant de m'encadrer pour l'élaboration de mon mémoire m'as permis de découvrir de belles propriétés des séries d'Eisenstein.\par

Je souhaite remercier mes professeurs de terminale et de classes préparatoires, Mr Pot, Mr Mélin et Mr Yebbou, qui m'ont donné à la fois le goût et les notions essentielles permettant de construire nos objets d'études mathématiques. Ma formation à l'ENS m'a alors permis d'entrevoir l'étendue et la richesse des mathématiques modernes.\par

Enfin, des remerciements plus personnels pour ma famille et mes proches qui m'ont encouragé et soutenu tout au long de ce travail de thèse.\par
}

\dominitoc


\tableofcontents

\chapter*{Introduction}
\addcontentsline{toc}{chapter}{Introduction}

\section*{0. Polyzêtas, périodes d'une forme modulaire et leurs généralisations selon Manin}

L'objet de cette étude est la rencontre de deux mondes celui des formes modulaires d'une part et d'autre part celui des polyzêtas.\par

Euler démontre que certaines valeurs aux entiers de la fonction $\zeta$ de Riemann sont algébriquement liées. On a par exemple $\zeta(2)^2=\frac{5}{2}\zeta(4)$. Ces relations prennent un sens parmi les relations de dépendance vérifiées par \textit{les valeurs multiples de zêta}(polyzêta selon la terminologie de Cartier \cite{Car}) formalisés par Hoffman \cite{Ho92} et Zagier \cite{Za94}:
\begin{equation}
\zeta(m_1,...,m_n)=\sum_{l_1,...,l_n\in\Z_{>0}}(l_1+...+l_n)^{-m_1}(l_2+...+l_n)^{-m_2}...l_n^{-m_n}.
\end{equation}
L'entier $n$ est appelé la \textit{longueur} et la somme $m_1+...+m_n$ le \textit{poids} de la valeur multiple de zêta.
On retrouve les relations connues ainsi que des nouvelles, comme $\zeta(3,2)=\frac{9}{2}\zeta(5)-2\zeta(2)\zeta(3)$.\par
L'étude de l'algèbre engendrée sur $\Q$ par les valeurs multiples de zêta et graduée par le poids a permis à Ecalle \cite{Eca} et Racinet \cite{Rac} de mettre en évidence des relations algébriques entre ces valeurs multiples de zêta conjecturées exhaustives:
\begin{equation}
\mathcal{Z}=\bigoplus_{k\geq 0} \mathcal{Z}_k,\quad\text{ où }\mathcal{Z}_k=\langle\zeta(m_1,...,m_n)\text{ pour }m_1+...+m_n=k\rangle_{\Q}.
\end{equation}
Les travaux de Kontsevitch et Zagier leurs ont permis de formuler la conjecture:
\begin{equation}
\dim_{\Q}(\mathcal{Z}_k)=d_k\text{ où on a posé }\sum_{k\geq 0}d_kt^k=(1-t^2-t^3)^{-1}.
\end{equation}
Grâce à une étude réalisée par Brown \cite{Br12}, on sait désormais que:
\begin{equation}
\mathcal{Z}_k=\langle\zeta(m_1,...,m_n)\text{ pour }m_1+...+m_n=k\text{ et }m_1,...,m_n\in\{2,3\}\rangle_{\Q}.
\end{equation}\par

Manin introduit des objets analogues pour la théorie des formes modulaires. Soit $k\geq 2$ un entier. Notons $S_k$ l'espace des formes modulaires paraboliques de poids $k$ et de niveau $1$. Toute forme modulaire $f(z)=\sum_{n>0} a_n(f)\exp(2i\pi n z)$ est caractérisée par sa fonction $L(f,s)=\sum_{n>0} a_n(f)n^{-s}$ associée. Plus précisément encore, Eichler et Shimura \cite{Sh59} démontrent l'injectivité de l'application:
\begin{equation}
S_k\to\C^{k-1},\quad f\mapsto \left(L(f,m)\right)_{ 1\leq m\leq k-1}.
\end{equation}
De plus, si $f$ est une forme primitive de Hecke alors il existe des nombres $\Omega_f^+,\Omega_f^-\in\C$ tel que:
\begin{equation}\label{per6}
L(f,m)\in \pi^{m}\Omega_f^{(-1)^m}\Q_f^*\text{ pour tout entier critique }1\leq m\leq k-1,
\end{equation}
où $\Q_f=\langle a_n(f);n>0\rangle_{\Q}$ est le corps engendré par les coefficients du développement en $q$-série de $f$.\par

Dans \cite{Ma1}, Manin combine ces deux constructions et définit une série génératrice en des variables non commutatives $A_1,...,A_n$ donnée par l'intégrale itérée
\begin{equation}
\sum_{N\geq 0}\sum_{1\leq i_1,...,i_n\leq N}\int_{0<t_1<...<t_N} \omega_{i_1}(it_1)...\omega_{i_N}(it_N)A_{i_1}...A_{i_N},.
\end{equation}
où $\omega_1,...,\omega_n\in\Omega^1(\H)$ sont des $1$-formes différentielles sur $\H$, le demi-plan de Poincaré, liées à une famille $(f_1,...,f_n)$ de formes modulaires.
Les coefficients de cette série peuvent être vus comme des valeurs multiples de fonctions $L$ de formes modulaires $L(f_1,...,f_n;m_1,...,m_n)$, où $f_j\in S_{k_j}$ et $1\leq m_j\leq k_j-1$ sont des entiers critiques, de la façon suivante.

\section*{1. Fonctions $L$ à plusieurs variables}

Soit $n\geq 0$ un entier. Soit $k=(k_1,...,k_n)$ une famille d'entiers naturels pairs. 
Soit $f_1,...,f_n$ une famille de formes modulaires paraboliques de niveau $1$ et de poids respectifs $k_1,...,k_n$. Pour chacune d'elles, on a un développement en $q$-série: $f(z)=\sum_{l>0}a_l(f)\exp(2i\pi lz)$ qui permet à Manin de définir la série de Dirichlet multiple:
\begin{equation}\label{defL}
L(f_1,...,f_n;s_1,...,s_n)=\sum_{l_1,...,l_n\in\Z_{>0}} \frac{a_{l_1}(f_1)...a_{l_n}(f_n)}{(l_1+...+l_n)^{s_1}(l_2+...+l_n)^{s_2}...l_n^{s_n}}.\\
\end{equation}
Considérons par ailleurs la transformée de Mellin itérée:
\begin{equation}\label{defLambda}
\Lambda(f_1,...,f_n;s_1,...,s_n)=\int_{0<t_1<...<t_n}f_1(it_1)t_1^{s_1-1}...f_n(it_n)t_n^{s_n-1}\d t_1...\d t_n,
\end{equation}
pour $(s_1,...,s_n)$ vérifiant $\Re(s_j)>k_j\text{ pour }1\leq j\leq n$.  On prolonge alors $\Lambda$ en une fonction holomorphe sur $\C^n$. On a une équation fonctionnelle:
\begin{equation}\label{equafonctio}
\Lambda(f_1,...,f_n;s_1,...,s_n)=i^{\sum_{j=1}^n k_j}\Lambda(f_n,...,f_1;k_n-s_n,...,k_1-s_1)
\end{equation}

Pour tout $\sigma\in\mathfrak{S}_n$, notons $\Lambda^{\sigma}(f_1,...,f_n;s_1,...,s_n)=\Lambda(f_{\sigma(1)},...,f_{\sigma(n)};s_{\sigma(1)},...,s_{\sigma(n)})$. On dispose alors de la formule, pour tout $(s_1,...,s_n)\in\C^n$:
\begin{equation}
\sum_{\sigma\in\mathfrak{S}_n}\Lambda^{\sigma}(f_1,...,f_n;s_1,...,s_n)=\Lambda(f_1,s_1)...\Lambda(f_n,s_n).
\end{equation}

\section*{2. Multipériodes des formes de niveau $1$}

On peut relier les formules (\ref{defL}) et (\ref{defLambda}), chapitre $2$ section $1$:
\begin{prop}
Soient $m_1,...,m_n\geq 1$ des entiers. On a: 
\begin{align}
\frac{L(f_1,...,f_n;m_1,...,m_n)}{(2\pi)^{m_1+...+m_n}}&=\sum_{\substack{M_1,...,M_n\geq 1\\M_1+..+M_n=m_1+...+m_n}}A_{m_1,...,m_n}(M_1,...,M_n)\frac{\Lambda(f_1,...,f_n;M_1,...,M_n)}{(M_1-1)!...(M_n-1)!},\\
\frac{\Lambda(f_1,...,f_n;m_1,...,m_n)}{(m_1-1)!...(m_n-1)!}&=\sum_{\substack{M_1,...,M_n\geq 1\\ M_1+..+M_n=m_1+...+m_n}}B_{m_1,...,m_n}(M_1,...,M_n)\frac{L(f_1,...,f_n;M_1,...,M_n)}{(2\pi)^{M_1+...+M_n}}.
\end{align}
On a dans $\Z$:
\begin{align}
A_{m_1,...,m_n}(M_1,...,M_n)&=(-1)^{\sum_{j=1}^n (n+1-j)(M_j-m_j)}\prod_{j=1}^{n-1}\binom{m_{j+1}-1}{\sum_{a=1}^j (M_a-m_a)},\\
B_{m_1,...,m_n}(M_1,...,M_n)&=\sum_{\substack{\alpha_{a,b}\geq 0\\ \sum_{a=1}^j\alpha_{a,j}=m_j-1\text{ et } \sum_{b=j}^n\alpha_{j,b}=M_j-1}}\frac{(m_1-1)!...(m_n-1)!}{\prod_{1\leq a\leq b\leq n} \alpha_{a,b}!}.
\end{align}
\end{prop}

Ces formules ont été obtenues indépendamment par Sreekantan \cite{Sr}.
On obtient par exemple dans le cas $n=2$:
\begin{equation}
\frac{\Lambda(f_1,f_2;m_1,m_2)}{(m_1-1)!(m_2-1)!}=\sum_{\alpha=0}^{m_2-1}\binom{m_1-1+\alpha}{\alpha}\frac{L(f_1,f_2;m_1+\alpha,m_2-\alpha)}{(2\pi)^{m_1+m_2}}.
\end{equation}\par

La fonction $\Lambda$ nous paraît plus naturelle que la fonction $L$. En effet, elle vérifie notamment l'équation fonctionnelle (\ref{equafonctio}) et est donc plus maniable pour nos démonstrations.\par

Notons $B_k=\{(m_1,...,m_n)\in\Z^n\text{ tel que }1\leq m_j\leq k_j-1\text{ pour tout }1\leq j\leq n\}$ l'ensemble des \textit{multi-entiers critiques} pour la famille $(f_1,...,f_n)$.\par
Dans la théorie d'Eichler-Shimura-Manin-Zagier(voir par exemple Kohnen-Zagier \cite{KZ}), les nombres $\Lambda(f,m)$ (pour $1\leq m\leq k-1$) sont les \textit{périodes} de $f$. Par analogie, nous dirons que les nombres suivants sont les \textit{multipériodes de longueur} $n$ de la famille $(f_1,...,f_n)$:
\begin{equation}
\Lambda^{\sigma}(f_{1},...,f_{n};m_{1},...,m_{n})\text{ pour }(m_1,...,m_n)\in B_k\text{ et }\sigma\in\mathfrak{S}_n.
\end{equation}

Il résulte simplement (chapitre $4$ section $1$) des formules (5) et (11) que la famille des multipériodes caractérise, à proportionnalité près, la famille des formes $(f_1,...,f_n)$ (i.e. l'élément $f_1\otimes...\otimes f_n$).

\begin{prop}\label{ninj}
L'application $\C$-linéaire suivante est injective:
\begin{equation}
\otimes_{j=1}^n S_{k_j}(SL_2(\Z))\to \C^{\mathfrak{S}_n}\otimes \C^{B_k},
\quad f_1\otimes...\otimes f_n\mapsto\left(\Lambda_{f_1,...,f_n}^{\sigma}(m_1,...,m_n)\right)_{\substack{\sigma\in\mathfrak{S}_n\\(m_1,...,m_n)\in B_k}}.
\end{equation}
\end{prop}
Lorsque $n=1$, Eichler-Shimura-Manin-Zagier ont déterminé l'image de cette application. La généralisation à $n>1$ est l'objet de notre étude.

\section*{3. Détermination des multipériodes par les périodes}

Certaines formes modulaires se prêtent à des calculs explicites. Considérons la famille des formes $F_m^k\in S_k$ pour les entiers critiques $1\leq m\leq k-1$, voir par exemple Cohen \cite{Cohen81}, définie par:
$$F_m^k(z)=C_m\sum_{\mat{a}{b}{c}{d}\in\G}(az+b)^{-m}(cz+d)^{m-k}\text{ avec }C_m=-\frac{(2i)^{k-2}i^m}{\pi\binom{k-2}{m-1}}.$$
C'est une famille génératrice de $S_k$. On peut encore caractériser $F_m^k$ comme l'unique forme de $S_k$ telle que:
$$L(f,m)=\langle f, F_m^k\rangle,\text{ pour tout }f\in S_k,$$
où $\langle , \rangle$ est le produit scalaire de Petersson.
Kohnen et Zagier \cite{KZ} ont donnés une formule explicite pour $\Lambda(F_m,n)$ pour des entiers critiques $m$ et $n$ de parités opposées. Nous la rappelons chapitre $1$ section $6$. Dans le même esprit, nous obtenons la formule suivante (chapitre $2$ section $7$):

\begin{thm}\label{thmchap2}
Supposons que les entiers $m_1$, $n_1$, $m_2$ et $n_2$ vérifient les propriétés suivantes:
\begin{itemize}
\item Les entiers $m_1$ et $n_1$ sont de parités différentes.
\item Les entiers $m_2$ et $n_2$ sont de parités différentes.
\item Soit $n_1<m_1$, $n_1<\tilde{m_1}$, $n_2<m_2$ et $n_2<\tilde{m_2}$,\\
soit $m_1<n_1$, $\tilde{m_1}<n_1$, $m_2<n_2$ et $\tilde{m_2}<n_2$.
\end{itemize}
Alors on a:
\begin{align*}
\Lambda(F_{m_1}^{k_1},&F_{m_2}^{k_2};n_1,n_2)=\frac{i^{k_1+k_2}2^{k_1+k_2-2}}{(k_1-2)!(k_2-2)!}(Q(m_1,n_1,m_2,n_2)+i^{k_1}Q(\tilde{m_1},n_1,m_2,n_2)\\
&+i^{k_2}Q(m_1,n_1,\tilde{m_2},n_2)+i^{k_1+k_2}Q(\tilde{m_1},n_1,\tilde{m_2},n_2)+i^{k_1+k_2}Q(m_2,\tilde{n_2},m_1,\tilde{n_1})\\
&+i^{k_1}Q(\tilde{m_2},\tilde{n_2},m_1,\tilde{n_1})+i^{k_2}Q(m_2,\tilde{n_2},\tilde{m_1},\tilde{n_1})+Q(\tilde{m_2},\tilde{n_2},\tilde{m_1},\tilde{n_1})+R_2),
\end{align*}
où: $$\frac{Q(m_1,n_1,m_2,n_2)}{(\tilde{m_1}-1)!(n_1-1)!(\tilde{m_2}-1)!(n_2-1)!}$$
$$=\begin{cases}
\sum_{a=0}^{n_1-1}\sum_{b=0}^{m_2-1}(-1)^b\binom{m_2-1}{b}\binom{n_2-1+a}{a}\frac{\zeta(n_2-m_2+1+a+b,n_1-m_1+1-a-b)}{(2\pi)^{n_1+n_2-m_1-m_2+2}} &\text{si } n_1+n_2>m_1+m_2\\
C_{m_1}C_{m_2}J(m_2,m_1)&\text{si } n_1+n_2=m_1+m_2-2\\
 0&\text{sinon,}
 \end{cases}$$
avec :
$$\zeta(A,B)=\sum_{0<l_2<l_1} l_1^{-A}l_2^{-B},$$
$$J(\alpha,\beta)=\int_0^1 \d u\int_{-\infty}^{\infty}(1+iux)^{-\alpha}\d x\int_0^u\d v \int_{-\infty}^{\infty}(1+ivy)^{-\beta}\d y\in\frac{\pi^2}{2(\alpha-1)(\beta-1)}+\Q,$$
et enfin le nombre rationnel:
\begin{align*}
R_2&=1/2(\delta_{(n_1+n_2=2)}+i^{k_1}i^{k_2}\delta_{(\tilde{n_1}+\tilde{n_2}=2)})\frac{i^{m_1}\zeta(m_1)\zeta(\tilde{m_1})}{(k_1-1)\zeta(k_1)}\frac{i^{m_2}\zeta(m_2)\zeta(\tilde{m_2})}{(k_2-1)\zeta(k_2)}\\
&+L(F_{m_1},n_1)\delta_{(n_2=1)}\frac{i^{m_2}\zeta(m_2)\zeta(\tilde{m_2})}{(k_2-1)\zeta(k_2)}\\
&+L(F_{m_2},n_2)\delta_{(\tilde{n_1}=1)}\frac{i^{\tilde{m_1}}\zeta(m_1)\zeta(\tilde{m_1})}{(k_1-1)\zeta(k_1)}\\
&+\delta_{(\tilde{n_1}=1)}\frac{i^{\tilde{m_1}}\zeta(m_1)\zeta(\tilde{m_1})}{(k_1-1)\zeta(k_1)}\delta_{(n_2=1)}\frac{\zeta(m_2)\zeta(\tilde{m_2})}{(k_2-1)\zeta(k_2)}.
\end{align*}
\end{thm}

Les nombres $J(\alpha,\beta)$ définissent la série génératrice:
$$H(X,Y)=\sum_{\alpha,\beta\geq 2} (\alpha-1)(\beta-1)J(\alpha,\beta)X^{\alpha-2}Y^{\beta-2}.$$
Elle est entièrement déterminée par les formules:
$$\frac{d}{dY}\left((1-Y)H(X,Y)\right)=\frac{4}{(2-X-Y)(Y-X)}\log\left(\frac{1-X}{1-Y}\right)$$
et
$$H(X,Y)+H(Y,X)=\frac{\pi^2}{(1-X)(1-Y)}.$$

Nous en déduisons le résultat suivant pour la forme modulaire de Ramanujan de poids $12$, donnée par la formule:
$$\Delta(z)=q\prod_{n\geq 1}(1-q^n)^{24}=\sum_{n\geq 1}\tau(n)q^n\text{ avec }q=\exp(2i\pi z).$$

\begin{thm}\label{thmdeltadelta}
Soit $(n_1,n_2)\in \llbracket 1,6\rrbracket^2\cup \llbracket 6,11\rrbracket^2$. Alors on a:
$$\frac{\Lambda(\Delta,\Delta;n_1,n_2)}{\Lambda(\Delta,n_1)\Lambda(\Delta,n_2)}\in \sum_{\substack{2<a+b\text{ pair}\\1<a}}\Q\frac{\zeta(a,b)}{\pi^{a+b}}.$$
\end{thm}

Nous avons fait le calcul explicite dans un cas particulier et on obtient:
\begin{align*}
\Lambda(\Delta,\Delta&;2,3)=2^73^35^27 \Lambda(\Delta;2)\Lambda(\Delta;3)\Big[30\frac{\zeta(6,4)}{(2\pi)^{10}}+180\frac{\zeta(7,3)}{(2\pi)^{10}}+630\frac{\zeta(8,2)}{(2\pi)^{10}}+1680\frac{\zeta(9,1)}{(2\pi)^{10}}\\
&+\frac{\zeta(4,4)}{(2\pi)^{8}}+4\frac{\zeta(5,3)}{(2\pi)^{8}}+10\frac{\zeta(6,2)}{(2\pi)^{8}}+20\frac{\zeta(7,1)}{(2\pi)^{8}}-630\frac{\zeta(10)}{(2\pi)^{10}}\\
&-\frac{1105}{126}\frac{\zeta(7)}{(2\pi)^{8}}+\frac{12155}{63}\frac{\zeta(13)}{(2\pi)^{8}}-\frac{12155}{6}\frac{\zeta(15)}{(2\pi)^{8}}+\frac{5525}{3}\frac{\zeta(17)}{(2\pi)^{8}}\\
&+\frac{7456}{3}\frac{\zeta(9)}{(2\pi)^{10}}-330\frac{\zeta(11)}{(2\pi)^{10}}+143\frac{\zeta(13)}{(2\pi)^{10}}+\frac{67925}{3}\frac{\zeta(15)}{(2\pi)^{10}}-24310\frac{\zeta(17)}{(2\pi)^{10}}\Big].
\end{align*}

\section*{4. Polynôme des multipériodes}

Pour tout sous-anneau $A$ de $\C$, notons $V^{A}_{k}=\{P\in A[X_1,...,X_n];deg_{X_j}(P)\leq k_j-2\}$. Le $\C$-espace vectoriel $V_k=V_k^{\C}$ est munie d'une $\R$-structure donnée par $V_k^{\R}$ et donc d'une conjugaison complexe.

Le groupe $\S_n$ agit sur $P\G^n$ par:
\begin{equation}
(\gamma_1,...,\gamma_n)^{\rho}=(\gamma_{\rho(1)},...,\gamma_{\rho(n)}),\text{ pour }\gamma_1,...,\gamma_n\in P\G\text{ et }\rho\in\S_n.
\end{equation}

On a ainsi un produit semi-direct $P\G^n\rtimes \S_n$ donné par:
\begin{equation}
[\gamma,\rho][\gamma',\rho']=[\gamma^{\rho'}\gamma',\rho\rho'].
\end{equation}

Le groupe $P\G^n\rtimes \S_n$ opère sur $\left(V_k\right)^{\S_n}$ par:
\begin{equation}
\left(\sigma\mapsto P^{\sigma}(X_1,...,X_n)\right)|_{[\gamma,\rho]}=\left(\sigma\mapsto P^{\rho\sigma}(X_1|_{\gamma_{\sigma(1)}},...,X_n|_{\gamma_{\sigma(n)}})\right),
\end{equation}
où $[\gamma,\rho]\in P\G^n\rtimes \S_n$ et où l'action de $\mat{a}{b}{c}{d}\in PSL_2(\Z)$ sur un polynôme $P\in\C_{k_j-2}[X_j]$ est donnée par:
\begin{equation}
P\left(X_j|_{\mat{a}{b}{c}{d}}\right)=P\left(\frac{aX_j+b}{cX_j+d}\right)(cX_j+d)^{k_j-2}.
\end{equation}
Par $\Z$-linéarité, on obtient une action de $\Z[P\G^n\rtimes\S_n]$.

Définissons alors \textit{le polynôme des multipériodes}, en les indéterminées $X_1,...,X_n$, par:
\begin{equation}
P_{f_1,...,f_n}(X_1,...,X_n)=\int_{0<t_1<...<t_n} f_1(it_1)(X_1-it_1)^{k_1-2}...f_n(it_n)(X_n-it_n)^{k_n-2}\d t_1...\d t_n.
\end{equation}
C'est un élément de l'espace $V_k$. Considérons le polynôme des multipériodes de la famille permuté par $\sigma\in\S_n$ dans $V_k$ de $(f_{\sigma(1)},...,f_{\sigma(n)})$, ce polynôme donnés par:
\begin{equation}
P^{\sigma}_{f_1,...,f_n}(X_1,...,X_n)=P_{f_{\sigma(1)},...,f_{\sigma(n)}}(X_{\sigma(1)},...,X_{\sigma(n)}).
\end{equation}

Ceci permet de considérer l'application $\C$-linéaire et injective par la Proposition \ref{ninj} :
\begin{equation}
R_k:\otimes_{j=1}^n S_{k_j}(SL_2(\Z))\to \left(V_k\right)^{\S_n},\quad f_1\otimes...\otimes f_n\mapsto \left(\sigma\mapsto P_{f_1,...,f_n}^{\sigma}(X_1,...,X_n)\right).
\end{equation}
Notons $MP(k)$ son image.

Soit $a,b\geq 0$ des entiers tel que $a+b=n$. On définit le sous-ensemble de permutations:
\begin{equation}
\mathfrak{S}_{a,b}=\{\sigma\in\mathfrak{S}_n\text{ telle que }\sigma(1)<...<\sigma(a)\text{ et }\sigma(a+1)<...<\sigma(a+b)\}.
\end{equation}
Manin \cite{Ma1} met en évidence des propriétés liées que nous interprétons comme des relations de mélange (chapitre 2 section 5):

\begin{prop}
Les polynômes des périodes vérifient les relations de mélange:
\begin{equation}
\sum_{\sigma\in\mathfrak{S}_{a,b}}P_{f_1,...,f_n}^{\sigma}(X_1,...,X_n)=P_{f_1,...,f_a}(X_1,...,X_a)P_{f_{a+1},...,f_{a+b}}(X_{a+1},...,X_{a+b}).
\end{equation}
\end{prop}

Notons $S=\mat{0}{-1}{1}{\phantom{-}0}$ et $U=\mat{\phantom{-}0}{1}{-1}{1}$.\par 
Pour $n=1$, rappelons que $P_{f_1}$ vérifie les relations de Manin:
\begin{equation}
P_{f_1}(X_1|_{1+S})=0\text{ et }P_{f_1}(X_1|_{1+U+U^2})=0.
\end{equation}
Elles engendrent essentiellement les relations linéaires entre périodes de $f_1$. En effet, on a:
\begin{equation}\label{decomp29}
W_{k_1}=\{P\in \C_{k_1-2}[X_1]\text{ tel que }P|_{1+S}=P|_{1+U+U^2}=0\}=MP({k_1})\oplus \overline{MP({k_1})}\oplus \langle 1-X_1^{k_1-2}\rangle_{\C}.
\end{equation}
Il résulte de \cite{Ma1} la généralisation suivante de ces relations:
\begin{prop}
On a:
\begin{align}
\sum_{\substack{a,b\geq 0\\ a+b=n}}P_{f_1,...,f_a}|_{(S,...,S)}&\otimes P_{f_{a+1},...,f_{a+b}}=0\\
\sum_{\substack{a,b,c\geq 0\\ a+b+c=n}}P_{f_1,...,f_a}|_{(U^2,...,U^2)}&\otimes P_{f_{a+1},...,f_{a+b}}|_{(U,...,U)}\otimes P_{f_{a+b+1},...,f_{a+b+c}}=0,
\end{align}
où nous posons $P_{F_{\emptyset}}=1\in\C$ pour $F_{\emptyset}$ la famille vide afin de simplifier les notations.
\end{prop}

\section*{5. Délimitation de $MP(k)$ par un $\Z[P\G^n\rtimes\S_n]-$module}

Nous définissons un analogue pour $n>1$ de $W_{k_1}$, un espace optimal $W_k$ délimitant $MP(k)$.
Pour tout idéal $A\subset\Z[\G^n\rtimes\S_n]$, posons:
$$\left(V_k^{\Q}\right)^{\S_n}[A]=\{P\in \left(V_k^{\Q}\right)^{\S_n}\text{ tel que }P|_a=0\text{ pour tout }a\in A\}.$$
Nous donnerons la construction par récurrence sur $n$ d'idéaux $\mathcal{A}_n$ et $\mathcal{A}_{n-1}[j,\sigma]$ pour tout $0\leq j\leq n$ et $\sigma\in\S_n$ vérifiant:
$$\mathcal{A}_n=\bigcap_{j=0}^n\bigcap_{\sigma\in\S_n} \mathcal{A}_{n-1}[j,\sigma].$$
De plus, si on pose $W_k^{\Q}=\left(V_k^{\Q}\right)^{\S_n}[\mathcal{A}_n]$ et
$ME_k^{\Q}=\sum_{j=0}^n\sum_{\sigma\in\S_n}\left(V_k^{\Q}\right)^{\S_n}[\mathcal{A}_{n-1}[j,\sigma]],$
alors on a (chapitre $4$ section $3$):

\begin{thm}\label{thmchap4}
Le $\Q$-espace vectoriel $W_k^{\Q}$ est le plus petit sous-$\Q$-espace de $\left(V_k^{\Q}\right)^{\S_n}$ contenant $ME_k^{\Q}$ dont l'extension des scalaires à $\C$ contient $MP(k)$:
\begin{equation}
ME_k^{\Q}\subset W_k^{\Q}\text{ et }MP(k)\subset W_k^{\Q}\otimes_{\Q}\C.
\end{equation}
\end{thm}

Voici la construction de $\mathcal{A}_n$. Posons $\mathcal{I}_1=(1+S,1+U+U^2)$. Définissons les applications, pour les indices $0< j< n$:
\begin{equation}
\phi_j^n:P\G^{n-1}\to P\G^n, (\gamma_1,...,\gamma_{n-1})\mapsto (\gamma_1,...,\gamma_j,\gamma_j,...,\gamma_{n-1}).
\end{equation}
On définit alors par récurrence sur $n$:
\begin{equation}
\mathcal{I}_{n-1}[j]=
\begin{cases}
\phi_j^n(\mathcal{I}_{n-1})&\text{ si }0<j<n\\
(1-T,1,...,1)\Z[P\G^n]+(1,\mathcal{I}_{n-1})&\text{ si }j=0\\
(1,...,1,1-US)\Z[P\G^n]+(\mathcal{I}_{n-1},1)&\text{ si }j=n.
\end{cases}
\text{ et }\mathcal{I}_n=\bigcap_{j=0}^n\mathcal{I}_{n-1}[j].
\end{equation}

On pose alors :
\begin{equation}
\mathcal{A}_n=\sum_{a+b=n}[(\widetilde{\mathcal{I}_a},1),\sum_{\sigma\in\S_{a,b}}(\sigma)]+([(S,...,S),id]-[(id,...,id),(n,...,1)])\Z[P\G^n\rtimes\S_n],
\end{equation}
où on utilise l'antiautomorphisme de $\Z[P\G^n]$ donné par $\widetilde{(\gamma_1,...,\gamma_a)}=(\gamma_1^{-1},...,\gamma_a^{-1})$.

Cette méthode donne notamment un algorithme de construction par récurrence sur $n$ des idéaux $\mathcal{A}_n$. Ces idéaux sont de type fini et nous avons déterminé une famille de générateurs lorsque $n=2$.

\section*{6. Description explicite pour $n=2$}

Ceci est l'objet du chapitre $3$. Posons $n=2$ et $k=(k_1,k_2)$.
\begin{thm}\label{thmchap3}
1) On dispose d'un système de générateurs de $\mathcal{A}_2$. En effet, on a: 
$$\mathcal{A}_2=\left( [(1,1),id]-[(S,S),(2,1)],[\widetilde{\mathcal{I}_2},id]\right),$$
 où $\widetilde{\mathcal{I}_2}$ est l'idéal à droite de $\Z[P\G^2]$ engendré par:
\begin{multline}
(1+S,1+S),\quad[(1,1)+(S,S)](1+U+U^2,1),\quad[(1,1)+(S,S)](1,1+U+U^2)\\
(1+U+U^2,1+U+U^2)\text{ et }(S,S)+(S,SU^2)+(SU^2,SU^2)+(1,U^2)-(U,U).
\end{multline}
2) On obtient la décomposition explicite, l'analogue de (\ref{decomp29}):
\begin{equation}
W_{k_1,k_2}=MP(k_1,k_2)\oplus\overline{MP(k_1,k_2)}\oplus \left[\left(W_{k_1}\otimes 1\right)+ V_{k_1,k_2}[I_D] +\left(X_1^{k_1-2}\otimes W_{k_2}\right)\right],
\end{equation}
où $V_{k_1,k_2}[I_D]=\{P\in V_{k_1,k_2};P|_{(1,1)+(S,S)}=P_{(1,1)+(U,U)+(U^2,U^2)}=0\}$. Ce dernier est calculable au sens où:
\begin{equation}
V_{k_1,k_2}[I_D]\cong\bigoplus_{l=0}^{min(k_1,k_2)} W_{k_1+k_2-2l}.(X_1-X_2)^l.
\end{equation}
\end{thm}

\section*{7. Séries d'Eisenstein}

Comme Zagier \cite{Za91} l'a fait pour $n=1$, on définit un polynôme des multipériodes pour $f_1,...,f_n$ forme modulaires holomorphes pour $\G$ de poids $k_1,...,k_n$.
En particulier, nous donnons une formule pour les périodes doubles d'un couple de série d'Eisenstein (chapitre 2 section 8).

\section*{8. Diverses Généralisations}
Dans le chapitre $5$, on étudie certaines généralisations. Soient $\rho_j: GL_2(\Z)\to GL(V_j)$ une famille de représentations de dimension finie de $GL_2(\Z)$ sur $\C$ pour $1\leq j\leq n$. On définit $\Omega_{\rho_1,...,\rho_n}$ l'ensemble des $n$-formes différentielles de $\Omega^n(\H^n,\otimes_{j=1}^nV_j)$ invariantes par $GL_2(\Z)^n$. Ceci permet de définir \textit{l'espace des multipériodes}:
\begin{equation}
MP({\rho_1,...,\rho_n})=\left\{\left(\sigma\mapsto \int_{\tau_n^{\sigma}} \omega\right)\text{ pour }\omega\in\Omega_{\rho_1,...,\rho_n}\right\},
\end{equation}
où $(\tau_n^{\sigma})_{\sigma\in\mathfrak{S}_n}$ est une famille explicite de $n$-cycles de $\H^n$.\par

Les théorèmes \ref{thmchap3} et \ref{thmchap4} se généralisent alors de la manière suivante:
\begin{thm}\label{thmchap5}
1) L'idéal $\mathcal{A}_n$ donne un contrôle à nouveau de l'espace des multipériodes:
\begin{equation}
MP({\rho_1,...,\rho_n})\subset W_{\rho_1,...,\rho_n}=\{v\in\left(\otimes_{j=1}^nV_j\right)^{\mathfrak{S}_n};v|_{a}=0\text{ pour tout }a\in\mathcal{A}_n\}.
\end{equation}
2) Pour $n=2$, on a:
\begin{equation}\label{decomp2}
W_{\rho_1,\rho_2}=MP(\rho_1,\rho_2)\oplus\overline{MP(\rho_1,\rho_2)}\oplus \left[\left(W_{\rho_1}\otimes V_2^{T}\right) +\left(V_1^{US}\otimes W_{\rho_2}\right)+ W_{\rho_1\otimes \rho_2}\right].
\end{equation}
\end{thm}

On applique ceci au cas d'un niveau quelconque et du poids $2$. Soient $\Gamma_1,...,\Gamma_n$ des sous-groupes de congruence de $\G$ normalisés par $\mat{-1}{0}{\phantom{-}0}{1}$. Pour $1\leq j\leq n$, posons $V_{\Gamma_j}=\C^{\Gamma_j\backslash SL_2(\Z)}$. 
Dans ce cadre le rôle joué par le polynôme des bipériodes dans la section 4 de l'introduction est maintenant joué par:
\begin{equation}
\int_{0<t_1<t_2} f_1|_{g_1}(it_1)f_2|_{g_2}(it_2)\d t_1\d t_2\text{ pour toutes classes }\Gamma_j g_j\in \Gamma_j\backslash SL_2(\Z),j=1,2.
\end{equation}

\section*{9. Organisation du texte}
Le chapitre 1 ne contient pas de résultat nouveau. Nous donnons toutes les définitions dont nous avons besoin, y compris au niveau le plus élémentaire. En particulier, nous rappelons les résultats de Kohnen et Zagier sur les périodes de formes modulaires.\par
Dans le chapitre 2, nous rappelons la théorie des intégrales itérées de Manin et son application aux formes modulaires. Cela nous amène à la construction des polynômes des multipériodes. Nous démontrons les théorèmes \ref{thmchap2} et \ref{thmdeltadelta}. Nous étendons le polynôme des multipériodes aux formes non nécessairement paraboliques. Nous calculons certaines bipériodes d'un couple de série d'Eisenstein.\par
Les chapitres 3 et 4 sont consacrés aux cas $n=2$ et $n>2$ respectivement. On y démontre les théorèmes \ref{thmchap3} et \ref{thmchap4} respectivement.\par
Les généralisations aux formes de niveaux plus grand que $1$ sont contenues dans le chapitre 5. Nous démontrons le théorème \ref{thmchap5}.

\chapter{Rappels sur les formes modulaires et les périodes simples}
\setcounter{mtc}{2}
\minitoc
\bigskip
Ce chapitre ne contient aucun résultat nouveau. Il rassemble les éléments sur lesquels nous fondons notre travail.

L'étude des courbes elliptiques nous amène à étudier l'action de $SL_2(\Z)$\nomenclature{$\G$}{Groupe spéciale linéaire de matrice}
 sur le demi-plan de Poincaré $\mathfrak{H}=\{z\in\C; Im(z)>0\}$\nomenclature{$\mathfrak{H}$}{Demi-plan de Poincaré}. 
Une matrice $\mat{a}{b}{c}{d}\in\G$ agit à gauche sur $z\in\mathfrak{H}$ par:
\begin{equation}
\mat{a}{b}{c}{d}.z=\frac{az+b}{cz+d}.
\end{equation}
On peut étendre cette définition au bord rationnel de $\mathfrak{H}$ que l'on définit comme $\pte$:\nomenclature{$\pte$}{Droite projective des rationnels}
\begin{equation}
\mat{a}{b}{c}{d}.[r_0,r_1]=[ar_0+br_1,cr_0+dr_1].
\end{equation}
On notera ainsi $\H=\mathfrak{H}\sqcup \pte$\nomenclature{$\H$}{Demi-plan de Poincaré complété}
 que l'on munie de la topologie usuelle. Dans ce contexte, on notera $i\infty=[1,0]$ et $r=[r:1]$.\par
L'étude des surfaces de Riemann compactes $\Gamma\backslash\H$ ainsi obtenues, se fait alors via l'étude d'applications dites modulaires.

\section{Définitions et exemples}
\begin{defi}
Soient $k\geq 2$ un entier pair le poids et $\Gamma$ un groupe de congruence c'est à dire un sous-groupe d'indice fini de $SL_2(\Z)$ contenant 
$\Gamma_{\infty}=\langle T^n;n\in\Z\rangle$ \nomenclature{$\Gamma_{\infty}$}{Sous-groupe des translations}
où $T=\mat{1}{1}{0}{1}$.\\
Une fonction $f:\mathfrak{H}\to\C$ est \textit{modulaire} si:
\begin{itemize}
\item $f$ est holomorphe sur $\mathfrak{H}$
\item Pour tout $\mat{a}{b}{c}{d}\in\Gamma, f|_k\mat{a}{b}{c}{d}(z)=(cz+d)^{-k} f\left(\frac{az+b}{cz+d}\right) =f(z),\text{ pour tout }z\in\mathfrak{H}$.
\item $f$ est holomorphe en toute pointe $p\in\Gamma\backslash \pte$  de $\Gamma\backslash\H$.
\end{itemize}
\end{defi}

Toute pointe peut s'écrire $p=[g]i\infty$ avec $[g]\in \Gamma\backslash SL_2(\Z)$. Or on sait que pour tout $\gamma\in\Gamma, f|_k\gamma g=f|_k g$. Donc l'application $z\mapsto f|_k [g](z)$ est bien définie.\\
De plus, comme la multiplication à gauche par $T$ est d'ordre fini dans $\Gamma\backslash SL_2(\Z)$ alors il existe un entier $h>0$ tel que $f|_k [g](z+h)=f|_k [g](z)$. Cette périodicité permet d'écrire un développement en série de Fourier: 
\begin{equation}
f|_k [g](z)=\sum_{n\geq 0} a_n^g(f) \exp(2i\pi nz/h).
\end{equation}
On dit alors que $f$ est \textit{holomorphe en la pointe} $[g]i\infty$ si : $\lim_{z\to i\infty} f|_k [g](z)=a_0^g(f)$.\par

On note $M_k(\Gamma)$\nomenclature{$M_k(\Gamma)$}{Espace des formes modulaires}
 l'ensemble des applications modulaires pour $\Gamma$ de poids $k$.\par
De plus, si pour toute pointe $[g]i\infty, a_0^g(f)=0$, on dit alors que $f$ est \textit{parabolique} et on note $S_k(\Gamma)$\nomenclature{$S_k(\Gamma)$}{Espace des formes modulaires paraboliques}
 l'ensemble de ces applications.

\begin{ex}[Série d'Eisenstein]
$SL_2(\Z)$ est le groupe de congruence trivial et pour tout poids $k\geq 4$, on peut définir:
\begin{equation}
E_k(z)=\sum_{(m,n)\neq (0,0)} (mz+n)^{-k}\in M_k(SL_2(\Z)).\nomenclature{$E_k$}{Série d'Eisenstein de poids $k$ de niveau $1$}
\end{equation}
La convergence normale de la série de fonctions est assurée par la condition $k\geq 4$ et donne en particulier l'holomorphie sur $\mathfrak{H}$ ainsi que:
$$\lim_{z\to i\infty} E_k(z)=\sum_{n\neq 0} n^{-k}=2\zeta(k).$$
D'autre part, pour $\mat{a}{b}{c}{d}\in SL_2(\Z)$ et $z\in\mathfrak{H}$, on a:
\begin{align*}
E_k|_k\mat{a}{b}{c}{d}(z)&=(cz+d)^{-k}\sum_{(m,n)\neq (0,0)} \left(m \frac{az+b}{cz+d} +n\right)^{-k}\\
&=\sum_{(m,n)\neq (0,0)} \left((am+cn)z+(bm+dn)\right)^{-k}=E_k(z).
\end{align*}
Car toute matrice de $SL_2(\Z)$ induit un automorphisme de $\Z^2 \setminus \{(0,0)\}$.\par
On peut alors calculer la série de Fourier associée à $E_k$. Elle est donnée pour $z\in\mathfrak{H}$ par :
$$E_k(z)=\frac{2(2i\pi)^k}{(k-1)!}G_k(z),$$
\begin{equation}
 \text{où }G_k(z)=\frac{(k-1)!\zeta(k)}{(2i\pi)^k}+\sum_{n>0} \sigma_{k-1}(n)q_z^n\in\mathbb{Q}[[q_z]]\text{ et }q_z=\exp(2i\pi z).
\end{equation}
\end{ex}

Pour démontrer ce résultat on utilise le lemme suivant (voir par exemple Lang \cite{Lang}):

\begin{lem}[Formule d'Hurwitz]
Pour tout entier $m>1$ et $z\in\mathfrak{H}$, on a:
\begin{equation}
\sum_{N\in\Z} (z+N)^{-m}=\frac{(-2i\pi)^m}{(m-1)!}\sum_{\alpha>0}\alpha^{m-1}\exp(2i\pi \alpha z).
\end{equation}
\end{lem}

\begin{proof}
Ce lemme provient de le dérivation récursive de la formule:
$$\sum_{N\in\Z}(z+N)^{-1}=\frac{\pi}{\tan(\pi z)},$$
valide pour tout $z\in\mathfrak{H}$ du demi-plan de Poincaré.
\end{proof}

La série de Fourier de $E_k$ découle alors du calcul:
\begin{align*}
E_k(z)&=\sum_{(m,n)\neq (0,0)} (mz+n)^{-k}\\
&=2\zeta(k)+2\sum_{m>0}\sum_{n\in\Z}(mz+n)^{-k}\\
&=2\zeta(k)+2\sum_{m>0}\frac{(-2i\pi)^k}{(k-1)!}\sum_{\alpha>0}\alpha^{k-1}\exp(2i\pi \alpha(mz))\\
&=2\zeta(k)+\frac{2(2i\pi)^k}{(k-1)!}\sum_{n>0}\sigma_{k-1}(n)q_z^n.
\end{align*}
Ceci nous donne de plus l'holomorphie de $E_k$ en $i\infty$ l'unique pointe.

\begin{ex}[Groupe modulaire de niveau $N$]
Pour tout entier $N\geq 1$, on dispose des groupes de congruence:
\begin{equation}
\Gamma_0(N)=\{\mat{a}{b}{c}{d}\equiv\mat{*}{*}{0}{*}\mod N\}\text{ et }\Gamma_1(N)=\{\mat{a}{b}{c}{d}\equiv\mat{1}{*}{0}{1}\mod N\}.
\end{equation}
Pour tout caractère non nul $\chi:\Z\to\C$ multiplicatif et $N$-périodique, on peut définir:
\begin{equation}
E_{k,\chi}(z)=\sum_{(m,n)\neq (0,0)} \frac{\chi(m)}{(mz+n)^k}\in M_k(\Gamma_1(N)).
\end{equation}
La convergence normale de la série de fonctions nous donne à nouveau l'holomorphie sur $\mathfrak{H}$ et $\lim_{i\infty} E_{k,\chi}=2\chi(0)\zeta(k)$.
On calcul alors pour tout $\mat{a}{b}{c}{d}\in SL_2(\Z)$:
\begin{align*}
E_{k,\chi}|_k\mat{a}{b}{c}{d}(z)&=\sum_{(m,n)\neq (0,0)} \frac{\chi(m)}{\left((am+cn)z+(bm+dn)\right)^k}\\
&=\sum_{(m,n)\neq (0,0)}\frac{\chi(dm-cn)}{(mz+n)^k}.
\end{align*}
Ainsi on obtient bien la condition de modularité car $\chi(dm-cn)=\chi(m)$ lorsque $\mat{a}{b}{c}{d}\in\Gamma_1(N)$.\par
On s'intéresse alors à l'holomorphie aux pointes, le calcul donne: 
$$\lim_{i\infty} E_{k,\chi}|_k\mat{a}{b}{c}{d} = \sum_{n\neq 0} \chi(-cn) n^{-k} =\chi(c)(1+\chi(-1))L_{\chi}(k).$$
Et la formule d'Hurwitz permet de développer $\sum_{m\neq 0} \sum_{n\in \Z} \frac{\chi(dm-cn)}{(mz+n)^k}$ de la même manière que précédemment en série de Fourier avec un terme constant nul. Ceci nous permet d'obtenir l'holomorphie aux pointes.\par
De plus, si le caractère $\chi$ est impair (i.e. $\chi(-1)=-1$), on remarque que cette application est parabolique.
\end{ex}

\begin{ex}[Groupe de congruence quelconque]
Soient $\Gamma$ un groupe de congruence et \\
$\chi:\Gamma_{\infty}\backslash SL_2(\Z) / \Gamma\to \C$ un caractère multiplicatif. On définit :
\begin{equation}
E_{k,\chi}(z)=\sum_{g=\mat{*}{*}{m}{n}\in \Gamma_{\infty}\backslash SL_2(\Z)} \frac{\chi(g)}{(mz+n)^k}.
\end{equation}
Un calcul direct montre que $E_{k,\chi}(\gamma z)=\bar{\chi}(\gamma)(cz+d)^kE_{k,\chi}(z)$, pour tout $z\in\mathfrak{H}$ et $\gamma=\mat{a}{b}{c}{d}\in\G$.
L'holomorphie de la fonction sur $\mathfrak{H}$ est dû à la convergence normale et donne une valeur des limites aux pointes.
 L'holomorphie aux pointes est à nouveau due aux développements en série de Fourier donnés par le résultat d'Hurwitz.\par
Ceci nous permet d'en déduire que $E_{k,\chi}\in M_k(\Gamma)$.
\end{ex}

Pour étudier la structure de ces espaces d'applications l'introduction d'un produit scalaire par Petersson donne un outil puissant, \cite{Za82}:

\begin{defi}[Produit de Petersson]
On définit le \textit{produit scalaire de Petersson} sur $S_k(\Gamma)\times M_k(\Gamma)$ par:
\begin{equation}
\langle f , g \rangle=\frac{i}{2}\int_{\Gamma\backslash \H} f(z)\overline{g(z)} Im(z)^{k-2} \d z\d\bar{z}.
\end{equation}
Pour des problèmes de convergence, on ne peut étendre ce produit à $M_k(\Gamma)^2$ mais la même définition convient dès lors qu'une des formes est parabolique. Ceci permet de considérer l'orthogonal de $S_k(\Gamma)$ dans $M_k(\Gamma)$. Ses éléments sont appelés \textit{séries d'Eisenstein}.
\end{defi}

\begin{proof}
Tout d'abord, on démontre que le domaine d'intégration a bien un sens :\\
On remarque que $Im(z)^{-2}\frac{i}{2}\d z\d\bar{z}=\frac{dxdy}{y^2}$ est une forme invariante par $SL_2(\Z)$, car :
\begin{equation}
Im\left(\frac{az+b}{cz+d}\right)=\frac{Im(z)}{|cz+d|^2}\text{ et }\frac{d}{dz}\left(\frac{az+b}{cz+d}\right)=\frac{1}{(cz+d)^2}.
\end{equation}
Ces deux relations peuvent être déduites de la suivante :
\begin{equation}
\left(\frac{az_1+b}{cz_1+d}\right) - \left(\frac{az_2+b}{cz_2+d}\right)=\frac{z_1-z_2}{(cz_1+d)(cz_2+d)}.
\end{equation}
Ainsi l'action de $\gamma=\mat{a}{b}{c}{d}\in\Gamma$ sur l'intégrande est triviale étant donnée que :
\begin{align*}
f(\gamma z)\overline{g(\gamma z)} Im(\gamma z)^{k}
&=f|_k\gamma(z)(cz+d)^k\overline{g|_k\gamma(z)(cz+d)^k} \left(\frac{Im(z)}{|cz+d|^2}\right)^{k}\\
&=f(z)\overline{g(z)}Im(z)^k.
\end{align*}
De plus comme une des deux formes est parabolique alors l'intégrande tend vers $0$ aux voisinages des pointes. Donc il existe une borne de l'intégrande sur $\Gamma\backslash \H$ qui par ailleurs est de mesure fini sous la forme $\frac{\d x\d y}{y^2}$.\\
En effet $\int_{SL_2(\Z)\backslash \H}\frac{\d x\d y}{y^2}=\frac{\pi}{6}$ et $\Gamma$ est d'indice fini dans $SL_2(\Z)$. Donc :
\begin{equation}
\int_{\Gamma\backslash \H}\frac{\d x\d y}{y^2}=\frac{\pi}{6}[\Gamma:SL_2(\Z)].
\end{equation}
\end{proof}

\begin{ex}
L'ensemble des séries d'Eisenstein de $M_k(SL_2(\Z))$ est $\C E_k$.\par
En effet, pour tout $f\in S_k(SL_2(\Z))$:
\begin{align*}
\langle f , E_k \rangle &= \int_{SL_2(\Z)\backslash\H} \sum_{(m,n)\neq (0,0)} \frac{f(z)}{(m\bar{z}+n)^k} Im(z)^{k} \frac{\d x\d y}{y^2}\\
&=\int_{SL_2(\Z)\backslash\H} \sum_{\alpha>0}\sum_{c\wedge d=1} \frac{f(z)}{(\alpha c\bar{z}+\alpha d)^k} Im(z)^{k} \frac{\d x\d y}{y^2}\\
&=\int_{SL_2(\Z)\backslash\H} \zeta(k)\sum_{\mat{a}{b}{c}{d}\in \Gamma_{\infty}\backslash SL_2(\Z)} \frac{f|k\mat{a}{b}{c}{d}(z)}{(c\bar{z}+d)^k} Im(z)^{k} \frac{\d x\d y}{y^2}\\
&=\zeta(k)\int_{SL_2(\Z)\backslash\H} \sum_{\gamma\in \Gamma_{\infty}\backslash SL_2(\Z)} f(\gamma z) Im(\gamma z)^{k} \frac{\d x\d y}{y^2}\\
&=\zeta(k)\int_{\Gamma_{\infty}\backslash\H} f(z) Im(z)^{k} \frac{\d x\d y}{y^2}\\
&=\zeta(k)\int_0^1\d x\int_0^{\infty}\d yf(x+iy)y^{k-2}=0.
\end{align*}
Or $S_k(SL_2(\Z))$ est un hyperplan de $M_k(SL_2(\Z))$ car il est défini comme le noyau de la forme linéaire non nulle $f\mapsto a_0(f)$. Donc l'ensemble des séries d'Eisenstein est bien la droite supplémentaire $\C E_k$:
\begin{equation}
M_k(SL_2(\Z))=S_k(SL_2(\Z))\oplus\C E_k.
\end{equation}
\end{ex}

\section{Fonction $\Delta$ et dimension de $M_k(SL_2(\Z))$}

Le calcul des formes modulaires de petit poids révèlent une forme parabolique remarquable de poids $12$ introduisant la fonction $\tau$ de Ramanujan. Nous allons voir qu'elle joue un rôle central dans la structure des espaces des formes modulaires.

\begin{prop}
On définit la fonction $\Delta(z)=q_z\prod_{n>0}(1-q_z^n)^{24}=\sum_{n>0}\tau(n)q_z^n$.\\
$\Delta$ est une forme modulaire parabolique de poids $12$ pour $SL_2(\Z)$.\nomenclature{$\Delta$}{La forme modulaire parabolique normalisée de poids $12$}
\end{prop}

\begin{proof}
Par définition, $\Delta$ est périodique de période $1$, holomorphe sur $\mathfrak{H}$ et en $i\infty$. Ainsi il reste à démontrer $\Delta|_{kS}=\Delta$ car $S$ et $T$ engendre le groupe $SL_2(\Z)$. Or :
\begin{align*}
\frac{\d}{\d z} \log(\Delta(z))&=2i\pi+24\sum_{n>0}\frac{-2i\pi nq_z^{n}}{1-q_z^n}\\
&=2i\pi-48i\pi\sum_{n,m>0}nq_z^{mn}=2i\pi G_2(z),
\end{align*}
où $G_2(z)=1-24\sum_{n>0} \sigma_1(n)q_z^n$ vérifie: $G_2(-1/z)z^{-2}=G_2(z)+\frac{12}{2i\pi z}$. Ainsi:
$$\frac{\d}{\d z} \log(\Delta(-1/z)z^{-12})=\frac{2i\pi}{z^2}G_2(-1/z)-\frac{12}{z}=2i\pi G_2(z)=\frac{\d}{\d z}\log\Delta(z).$$
Ce qui donne le résultat à une constante multiplicative près. Or $i$ est fixé par $S$ donc la constante est triviale et $\Delta$ est modulaire et parabolique.
\end{proof}

\begin{thm}[Dimension de $M_k(SL_2(\Z))$]
On dispose d'un algorithme de construction des espaces $M_k$ donné par :
\begin{itemize}
\item $S_2=M_2=\{0\}$ et $S_{12}=\C \Delta$,
\item Pour tout $k\in\{4,6,8,10,14\}, S_k=\{0\}$ et $M_k=\C E_k$,
\item Pour tout $k\geq 4, M_k=S_k\oplus\C E_k,$
\item Pour tout $k\geq 16, S_k=\Delta M_{k-12}$.
\end{itemize} 
On obtient ainsi les dimensions $\dim(M_k)=\left\lfloor k/12 \right\rfloor+1-\delta_{(k=2\mod 12)}$.
\end{thm}

\begin{proof}
Une démonstration donnée par Serre \cite{Se70} peut se faire par l'étude des valuations des pôles des fonctions. On a en effet une formule qui dérive du Théorème de Riemann-Roch qui lie leurs valuations aux poids des formes et qui permet une récurrence sur les poids:
\begin{equation}
\frac{k}{12}=val_{i\infty}(f)+\frac{1}{2}val_{i}(f)+\frac{1}{3}val_{\rho}(f)+\sum_{\substack{z\in SL_2(\Z)\backslash\H\\z\neq i\infty,i,\rho}} val_z(f).
\end{equation}
\end{proof}

\section{Fonction $L$ d'une forme modulaire}
On se place dans le cadre où $\Gamma=SL_2(\Z)$.\par
On peut associer à toute forme modulaire une série de Dirichlet $L$ qui vérifie une équation fonctionnelle. La donnée de cette fonction $L$ et donc leur étude est équivalente à celle des formes en elles-mêmes. Dans la pratique, nous introduirons aussi la transformée de Mellin $\Lambda$ de la forme qui s'avère être une simple renormalisation de $L$.

\begin{defi}
Pour une application modulaire $f(z)=\sum_{n\geq 0} a_n(f) q_z^n\in M_k(SL_2(\mathbb{Z}))$, on définit les fonctions:
\begin{align}
L(f,s)&=\sum_{n>0} \frac{a_n(f)}{n^s} \text{ pour }Re(s)>k/2+1,\\
\text{et }\Lambda(f,s)&=\int_0^{\infty} \left(f(it)-a_0(f)\right)t^{s-1}\d t\text{ pour }Re(s)>k.
\end{align}
Il existe des prolongements méromorphes à $\C$ de ces fonctions. Elles vérifient alors pour tout $s\in\C$:
\begin{equation}
\Lambda(f,s)=\frac{\Gamma(s)}{(2\pi)^s}L(f,s).
\end{equation}
\end{defi}

\begin{proof}
La convergence de la série de Dirichlet provient d'une estimation des coefficients due à Hecke $|a_n(f)|\leq C n^{k/2}$. En effet :
$$|a_n(f)|\exp(-2\pi ny)=\left|\int_0^1 f(x+iy)\exp(-2i\pi nx)\d x\right|\leq \left(\int_0^1 |f(z)Im(z)^{k/2}|\d x\right)y^{-k/2}.$$
On obtient le résultat en prenant $y=1/n$ et en remarquant que $|f(z)Im(z)^{k/2}|$ est invariant par $\Gamma$ et est donc bornée car $f$ holomorphe à l'infini.\par
On remarquera qu'il existe un résultat plus fort de Ramanujan bornant les coefficients:
\begin{equation}
\forall\varepsilon>0, |a_l(f)|=O_{l\to\infty}(l^{\frac{k-1}{2}+\varepsilon}).
\end{equation}
Pourtant ici le résultat de Hecke est suffisant à la définition de la fonction $L$.\par
La convergence de l'intégrale provient de l'annulation à l'infini de $f(it)-a_0(f)$ et de la modularité par $S=\mat{\phantom{-}0}{1}{-1}{0}$. On peut alors effectuer le calcul pour $Re(s)>k$:
\begin{align*}
\Lambda(f,s)&=\int_0^{\infty}\sum_{n>0} a_n(f) \exp(-2\pi nt)t^{s-1}\d t\\
&=\sum_{n>0}a_n(f)(2\pi n)^{-s}\int_0^{\infty}\exp(-t)t^{s-1}\d t=\frac{\Gamma(s)}{(2\pi)^s}L(f,s),
\end{align*}
qui s'étend à $\C$ par prolongement méromorphe.
\end{proof}

\begin{prop}
Pour tout $f\in M_k(SL_2(\Z))$, on dispose de l'équation fonctionnelle:
\begin{equation}
\Lambda(f,s)=i^k\Lambda(f,k-s).
\end{equation}
De plus, on  dispose d'une écriture généralisée de $\Lambda$, pour tout $s\in\C$ et $t_0\in\R_+^*$:
\begin{equation}
\Lambda(f,s)=\int_0^{t_0} \left(f(it)-\frac{a_0(f)}{(it)^k}\right)t^{s-1}\d t+\int_{t_0}^{\infty}\left(f(it)-a_0(f)\right)t^{s-1}\d t
-a_0(f)\left(\frac{t_0^s}{s}+\frac{i^k t_0^{s-k}}{k-s}\right).
\end{equation}
\end{prop}

\begin{proof}
Pour démontrer l'écriture généralisée, on commence par démontrer que l'écriture ne dépend pas du choix de $t_0$ par une simple dérivation. Puis il suffit de remarquer qu'on dispose bien d'une fonction méromorphe qui coïncide à la définition pour tout $Re(s)>k$ lorsque $t_0\to 0$.\par
Enfin l'équation fonctionnelle va ainsi intervertir $t_0$ en $1/t_0$, en effet:
$$\int_0^{t_0} \left(f(it)-\frac{a_0(f)}{(it)^k}\right)t^{s}\frac{\d t}{t}=\int_{1/t_0}^{\infty} \left(f\left(\frac{-1}{it}\right)-i^ka_0(f)t^k\right)\frac{t^{-s}\d t}{t}=i^k\int_{1/t_0}^{\infty}\left(f(it)-a_0(f)\right)t^{k-s-1}\d t,$$
échangeant les deux premiers termes de l'écriture généralisée. L'échange des deux derniers par l'équation fonctionnelle étant trivial.
\end{proof}

\begin{rem}
On peut faire disparaitre le paramètre $t_0$ lorsque $0<Re(s)<k$ en remarquant que:
$$\frac{t_0^s}{s}=\int_0^{t_0} t^{s-1}\d t\text{ et }\frac{i^k t_0^{s-k}}{k-s}=\int_{t_0}^{\infty} \frac{t^{s-1}\d t}{(it)^k}.$$
On peut ainsi construire l'intégrale convergente suivante par la formule :
\begin{equation}
\Lambda(f,s)=\int_0^{\infty} \left[f(it)-a_0(f)\left(1+(it)^{-k}\right)\right]t^{s-1}\d t\quad, 0<Re(s)<k.
\end{equation}
On notera $f^*(z)=f(z)-a_0(f)\left(1+z^{-k}\right)$ la \textit{forme modulaire épointée}.
\end{rem}

\begin{ex}[Périodes des séries d'Eisenstein]
Pour tout complexe $s\in\mathbb{C}$, on a:
\begin{equation}
L(G_k,s)=\zeta(s)\zeta(s-k+1).
\end{equation}
Pour tout entier critique $m$ (i.e. $1\leq m\leq k-1$), on a:
\begin{equation}
\Lambda(E_k,m)=\begin{cases}4 i^m\frac{(m-1)!(k-m-1)!}{(k-1)!} \zeta(m) \zeta(k-m)\text{ si $m$ pair}\\
0\text{ sinon.}\end{cases}
\end{equation}
Pour cela on effectue d'une part le calcul qui est valide pour $Re(s)>k$: 
$$L(G_k,s)=\sum_{l>0} \frac{\sigma_{k-1}(l)}{l^s}=\sum_{m,n>0} \frac{m^{k-1}}{(mn)^s}=\zeta(s)\zeta(s-k+1),$$
Puis on étend ce résultat à tout le plan complexe par méromorphie des deux fonctions.\par
D'autre part, comme: $E_k^*(z)=E_k(z)-2\zeta(k)\left(1+z^{-k}\right)=\sum_{ab\neq 0}(az+b)^{-k}$, on obtient:
\begin{align*}
\Lambda(E_k,m)&=\int_0^{\infty}\sum_{ab\neq 0} (ait+b)^{-k} t^{m-1}\d t\\
&=2\sum_{a>0}\frac{1}{(ia)^k}\sum_{b\neq 0}\int_0^{\infty}\frac{t^{m-1}\d t}{(t-ib/a)^k}\\
&=2i^k\sum_{a>0}a^{-k}\sum_{b\neq 0}\left[-\sum_{l=1}^m \frac{t^{m-l}}{(t-ib/a)^{k-l}}\frac{(m-1)!}{(m-l)!}\frac{(k-l-1)!}{(k-1)!}\right]_{t=0}^{\infty}\\
&=2i^k\sum_{a>0}a^{-k}\sum_{b\neq 0} \left(\frac{ia}{b}\right)^{k-m}\frac{(m-1)!(k-m-1)!}{(k-1)!}\\
&=\begin{cases}4i^m\left(\sum_{a>0}a^{-m}\right)\left(\sum_{b>0}b^{m-k}\right)\frac{(m-1)!(k-m-1)!}{(k-1)!}\text{ si $m$ pair}\\0 \text{ sinon.}\end{cases}
\end{align*}

\begin{rem}
L'équation fonctionnelle de la fonction $\zeta$:
\begin{equation}
\zeta(s)=2^s\pi^{s-1}\sin\left(\frac{\pi s}{2}\right)\Gamma(1-s)\zeta(1-s),
\end{equation}
permet de retrouver le lien entre les applications $L$ et $\Lambda$ à savoir: $\Lambda(E_k,s)=\frac{\Gamma(s)}{(2\pi)^s}L(E_k,s).$
Ainsi que l'équation fonctionnelle : $\Lambda(E_k,k-s)=i^k\Lambda(E_k,s).$
\end{rem}
\end{ex}

\section{Polynômes des périodes}

\begin{defi}
On appelle \textit{périodes} d'une forme modulaire parabolique $f\in S_k(SL_2(\Z))$ les valeurs aux entiers critiques de la fonction $L$ de $f$:
$$L(f,m)\text{ pour } 1\leq m \leq k-1.$$
On rassemble ces valeurs dans le \textit{polynôme des périodes}:
\begin{equation}
P_f(X)=-\sum_{m=1}^{k-1} \binom{k-2}{m-1} \frac{\Lambda(f,m)}{i^{m}}X^{k-m-1}=\int_0^{i\infty} f(z) (X-z)^{k-2}\d z.
\end{equation}
C'est un élément de $V_k=\C_{k-2}[X]$ l'espace des polynômes de degré au plus $k-2$ que l'on muni de l'action $|_{2-k}$ de $SL_2(\Z)$-module.
\end{defi}

\nomenclature{$V_k^K$}{Espaces des polynômes sur un corps $K$ de degré au plus $k-2$ }
\nomenclature{$P_f$}{Polynôme des périodes de la forme $f$}

\begin{rem}
On peut étendre cette définition aux séries d'Eisenstein et donc à toute forme modulaire $f\in M_k(SL_2(\Z))$ en posant:
\begin{multline}\label{intpolnonmod}
P_f(X)=\int_{z_0}^{i\infty} \left(f(z)-a_0(f)\right) (X-z)^{k-2}\d z+\int_{0}^{z_0} \left(f(z)-a_0(f)z^{-k}\right) (X-z)^{k-2}\d z\\
+\frac{a_0(f)}{k-1}\left((X-z_0)^{k-1}+\frac{1}{X}(1-\frac{X}{z_0})^{k-1}\right),
\end{multline}
où $z_0\in \mathfrak{H}$ est un paramètre quelconque qui permet d'assurer la convergence.\\
Cette définition est bien indépendante de $z_0$ et étend la précédente. L'espace parcouru est ici $\frac{1}{X}\C_k[X]$ qui n'est pas un $SL_2(\Z)$-module sous l'action précédente.\par
On dispose aussi d'une écriture des coefficients :
\begin{equation}
P_f(X)=-\sum_{m=0}^{k} \lim_{s\to m}\frac{\Lambda(f,s)}{\Gamma(s)\Gamma(k-s)} \frac{(k-2)!}{i^{m}}X^{k-m-1}.
\end{equation}
\end{rem}

\begin{ex}[Polynôme étendu des périodes de $E_k$] Pour tout poids $k\geq 4$, on a :
\begin{equation}
P_{E_k}(X)=-\frac{(2i\pi)^k}{k-1}\sum_{n=0}^k\frac{B_n}{n!}\frac{B_{k-n}}{(k-n)!}X^{n-1}-(2i\pi)\zeta(k-1)(1-X^{k-2}),
\end{equation}
où $B_n$ sont les nombres de Bernoulli donnés par la série génératrice $\sum B_n\frac{t^n}{n!}=\frac{t}{e^t-1}$.
\end{ex}

\begin{thm}[Relation de Manin]
On dispose de la relation de modularité en la variable $X$ pour tout $\gamma\in SL_2(\Z)$ et pour tout $(z_1,z_2)\in\H^2$:
\begin{equation}
\left(\int_{\gamma z_0}^{\gamma z_1} f(z)(X-z)^{k-2}\d z\right)|_{2-k}\gamma=\int_{z_0}^{z_1} f(z)(X-z)^{k-2}\d z,.
\end{equation}
Les polynômes des périodes vérifient les relations de Manin, pour tout $f\in M_k(SL_2(\Z))$:
\begin{equation}
P_f(X)|_{(1+S)}=P_f(X)|_{(1+U+U^2)}=0,
\end{equation}
où $S=\mat{\phantom{-}0}{1}{-1}{0}$ et $U=\mat{\phantom{-}0}{1}{-1}{1}$.
\end{thm}

\begin{proof}
En effet pour toutes bornes d'intégration $z_1,z_2$ et toute matrice $\gamma=\mat{a}{b}{c}{d}$:
\begin{align*}
\left(\int_{\gamma z_1}^{\gamma z_2} f(z)(X-z)^{k-2}\d z\right)|_{2-k}\gamma&=\int_{z_1}^{z_2} f(\gamma z)(\gamma X-\gamma z)^{k-2}(cX+d)^{k-2}\d (\gamma z)\\
&=\int_{z_1}^{z_2} f(z)(cz+d)^k\left(\frac{X-z}{(cX+d)(cz+d)}\right)^{k-2}(cX+d)^{k-2}\frac{\d z}{(cz+d)^2}\\
&=\int_{z_0}^{z_1} f(z)(X-z)^{k-2}\d z.
\end{align*}
On peut alors appliquer ce résultat à:
$$P_f|_{2-k}\gamma=\left(\int_{\gamma\gamma^{-1}0}^{\gamma\gamma^{-1}i\infty} f(z)(X-z)^{k-2}\d z\right)|_{2-k}\gamma=\int_{\gamma^{-1}0}^{\gamma^{-1}i\infty}f(z)(X-z)^{k-2}\d z.$$
Puis $P_f|_{2-k}(1+S)=\int_0^{\infty}+\int_{S0}^{Si\infty}$ et $P_f|_{2-k}(1+U+U^2)=\int_0^{\infty}+\int_{U^20}^{U^2i\infty}+\int_{U0}^{Ui\infty}$.\par
Tout les deux nulles car elles constituent des intégrales sur un contour fermé d'une application holomorphe. Cette démonstration s'adapte bien lorsque la forme est non modulaire en utilisant la formule (\ref{intpolnonmod}).
\end{proof}

On cherche à déterminer si les relations de Manin sont suffisantes pour déterminer l'espace des polynômes des périodes. Pour cela on va comparer l'espace image de $M_k(SL_2(\Z))$ par $f\mapsto P_f$ avec l'ensemble $W_k=\{P\in\C_{k-2}[X];P|_{(1+S)}=P|_{(1+U+U^2)}=0\}$ des polynômes vérifiant les relations de Manin. Ce dernier peut se scinder en $W_k=W_k^+\oplus W_k^-$ suivants les parités des polynômes.\par
Les relations de Manin se traduisent sur les coefficients des polynômes de $W_k$ par des relations de dépendance linéaire. 
\begin{prop}
Notons $w=k-2$. Soit $P(X)=\sum_{m=0}^w a_mX^m\in W_k$.\\
Alors pour tout $0\leq m \leq w$ :
\begin{equation}\label{eqcoefwk}
a_m+(-1)^ma_{w-m}=0\text{ et }\sum_{i=0}^w C_w(m,i) a_i=0,
\end{equation}
où on a défini: $C_w(m,i)=\begin{cases}\binom{m}{i}&\text{ si }m>i\\\binom{w-m}{w-i}&\text{ sinon.}\end{cases}$\par
\end{prop}

\begin{proof}
Pour cela, il suffit de réaliser le calcul :
\begin{align*}
\left(\sum_{m=0}^w a_mX^m\right)|_{1+S}&=\sum_{m=0}^w a_m\left(X^m+(-X)^{w-m}\right)\\
&=\sum_{m=0}^w \left(a_m+(-1)^ma_{w-m}\right)X^m.
\end{align*}
Ce qui nous donne bien la famille des premières relations. Puis on obtient de même : 
\begin{align*}
\left(\sum_{m=0}^w a_mX^m\right)|_{1+U+U^2}&=\sum_{m=0}^w a_m\left(X^m+(-X+1)^{w-m}+(X-1)^mX^{w-m}\right)\\
&=\sum_{m=0}^w a_m\left(X^m+\sum_j\binom{w-m}{w-j}(-X)^{w-j}+\sum_j\binom{m}{j}(-X)^{w-j}\right)\\
&=\sum_{j=0}^w\left(\sum_{m=0}^w C_w(m,j)a_m\right)(-X)^{w-j},
\end{align*}
où on a utilisé $a_{w-m}=-(-1)^ma_m$ pour substituer $a_mX^m$ dans la dernière égalité.
\end{proof}

On remarque que pour $P_f\in W_k$ les relations peuvent se scinder pour les coefficients pairs et impairs en spécialisant les parties réelle et imaginaire car $a_m\in i^{m-1}\R$. Ceci se retrouve dans la proposition suivante liant le produit scalaire sur $S_k(SL_2(\Z))$ à un sur $W_k$.

\begin{prop}\label{prop123}
Pour tout $f,g\in S_k(SL_2(\Z))$, on dispose de la relation :
\begin{equation}
\langle f,g \rangle=\frac{-1}{6(2i)^{k-1}}\sum_{a,b}(i^{a-b}-i^{b-a})\binom{k-2}{a+b}\binom{a+b}{a}\Lambda(f,a)\overline{\Lambda(g,b)}.
\end{equation}
\end{prop}

\begin{proof}
Pour la démonstration, on introduit $F(z)=\int_z^{i\infty} f(u) (u-\bar{z})^{k-2}\d u$, de sorte que $dF=f(z)\left(2iIm(z)\right)^{k-2}\d z$. Et ainsi:
$$\langle f,g \rangle=\frac{i}{2}\int_{SL_2(\Z)\backslash \H} f(z)\overline{g(z)}Im(z)^{k-2}\d z\d \bar{z}
=\frac{-1}{(2i)^{k-1}}\int_{SL_2(\Z)\backslash \H} \d (F(z)\overline{g(z)}\d \bar{z}).$$
On considère alors $D_0=\{z\in \mathfrak{H};|Re(z)|<1/2\text{ et }|z|>1\}$ un système de représentant de $\mathfrak{H}$ modulo $SL_2(\Z)$ au bord près. Ceci nous permet d'appliquer le Théorème de Stokes pour obtenir:
$$\int_{D_0}\d (F(z)g(\bar{z})\d \bar{z})=\int_{\partial D_0}F(z)g(\bar{z})\d \bar{z}.$$
On commence par remarquer que $F$ est $1$-périodique et donc que: 
$$\int_{i\infty}^{\rho^2}F(z)\overline{g(z)}\d \bar{z}=-\int_{\rho}^{i\infty}F(z)\overline{g(z)}\d \bar{z},$$
où $\rho=e^{i\frac{\pi}{6}}$ et $\rho^2=\rho-1$. 
De plus on remarque que la modularité de $f$ et $g$ permet d'obtenir pour tout $a,b,c\in\H$ et pour toute matrice $\gamma\in\G$ :
$$\int_{\gamma a}^{\gamma b}\int_{\gamma c}^z f(u)g(\bar{z})(u-\bar{z})^{k-2}\d u\d \bar{z}=\int_{a}^{b}\int_{c}^z f(u)g(\bar{z})(u-\bar{z})^{k-2}\d u\d \bar{z}.$$
Cette propriété de modularité permet de calculer:
\begin{align*}
-(2i)^{k-1}\langle f,g \rangle &=\int_{\rho^2}^{\rho} F(z)g(\bar{z})\d \bar{z}\\
&=-\int_{\rho^2}^{\rho} \int_{i\infty}^{z} f(u)g(\bar{z})(u-\bar{z})^{k-2}\d u\d \bar{z}\\
&=\int_{Si}^{S\rho} \int_{S0}^{z} f(u)g(\bar{z})(u-\bar{z})^{k-2}\d u\d \bar{z}-\int_{i}^{\rho} \int_{i\infty}^{z} f(u)g(\bar{z})(u-\bar{z})^{k-2}\d u\d \bar{z}\\
&=\int_{i}^{\rho} \int_{0}^{i\infty} f(u)g(\bar{z})(u-\bar{z})^{k-2}\d u\d \bar{z}\\
&=\int_{i}^{\infty} \int_{0}^{i\infty} f(u)g(\bar{z})(u-\bar{z})^{k-2}\d u\d \bar{z}-\int_{\rho}^{\infty} \int_{0}^{i\infty} f(u)g(\bar{z})(u-\bar{z})^{k-2}\d u\d \bar{z}
\end{align*}
Puis on a, d'une part:
\begin{align*}
\int_{i}^{i\infty} \int_{0}^{i\infty} f(u)g(\bar{z})(u-\bar{z})^{k-2}\d u\d \bar{z}
&=\int_{i}^{i\infty} P_f(\bar{z})g(\bar{z})\d \bar{z}\\
&=\frac{1}{2}\left(\int_{i}^{i\infty} P_f(\bar{z})g(\bar{z})\d \bar{z}+\int_{Si}^{Si\infty} P_f|_{2-k}S(\bar{z})g(\bar{z})\d \bar{z}\right)\\
&=\frac{1}{2}\int_{0}^{i\infty} P_f(\bar{z})g(\bar{z})\d \bar{z}\text{ car }P_f|S=-P_f.
\end{align*}
Et d'autre part: 
\begin{align*}
\int_{\rho}^{i\infty} P_f(\bar{z})g(\bar{z})\d \bar{z}
&=\frac{1}{3}\left(\int_{\rho}^{i\infty} P_f(\bar{z})g(\bar{z})\d \bar{z}+\int_{U\rho}^{Ui\infty} P_f|_{2-k}U(\bar{z})g(\bar{z})\d \bar{z}+\int_{U^2\rho}^{U^2i\infty} P_f|_{2-k}U^2(\bar{z})g(\bar{z})\d \bar{z}\right)\\
&=\frac{1}{3}\left(\int_{0}^{i\infty} P_f|_{2-k}(1-U)(\bar{z})g(\bar{z})\d \bar{z}\right)\text{ car }P_f|U^2=-P_f-P_f|U.
\end{align*}
On obtient ainsi $\int_0^{\infty} P_f|(1/6Id+1/3U)(\bar{z})g(\bar{z})\d \bar{z}=\frac{1}{6}\int_0^{\infty} P_f|(U-U^2)(\bar{z})g(\bar{z})\d \bar{z}$.
Enfin on conclu :
\begin{align*}
\int_0^{i\infty} P_f|(U-U^2)(\bar{z})g(\bar{z})\d \bar{z}&=\int_0^{i\infty} P_f|U(\bar{z})g(\bar{z})d\bar{z}-\int_0^{i\infty} f(z)\overline{P_g|U(z)}\d z\\
&=-\int_0^{i\infty} P_f|T^{-1}(\bar{z})g(\bar{z})\d \bar{z}+\int_0^{i\infty} f(z)\overline{P_g|T^{-1}(z)}\d z\text{ car }U=ST^{-1}\\
&=\sum_{a+b+c=k-2}\frac{(k-2)!}{a!b!c!}\left(i^a\Lambda(f,a)\overline{i^b\Lambda(g,b)}-(-i)^{a}\Lambda(f,a)\overline{(-i)^b\Lambda(g,b)}\right).
\end{align*}
\end{proof}

\begin{rem}\label{prodvk}
Une manière de récrire ce résultat est de considérer la forme bilinéaire symétrique définie sur $\Q$ et qu'on étendra au corps des complexes:
\begin{equation}
V_k^{\Q}\times V_k^{\Q}\to \Q,\quad (P,Q)\mapsto [P,Q],
\end{equation}
définit sur les monômes par $[X^m,X^n]=\delta_{(m+n=k-2)}(-1)^n\binom{k-2}{m}^{-1}$. Elle est non dégénérée et vérifie les propriétés suivantes:\\
(i) Pour tout $\gamma\in PSL_2(\Z), [P|_{\gamma},Q|_{\gamma}]=[P,Q]$, on dira $P\G$-invariant.\\
(ii) Pour tout $\alpha\in\C, [P,(X-\alpha)^{k-2}]=P(\alpha)$.\par
On peut alors reformuler la Proposition \ref{prop123} par:
\begin{equation}
\langle f,g \rangle=\frac{1}{6(2i)^{k-1}}[P_f|_{(T-T^{-1})},\overline{P_g}].
\end{equation}
\end{rem}

\begin{coro}
Les applications suivantes sont injectives:
$$S_k(SL_2(\Z))\to V_k, f\mapsto P_f^+,$$
$$\text{ et }S_k(SL_2(\Z))\to V_k,f\mapsto P_f^-,$$
les parties paire et impaire du polynôme des périodes.
\end{coro}

On a alors le théorème suivant dû à Eichler et Shimura:

\begin{thm}[Eichler-Shimura]
Les applications suivantes sont des isomorphismes de $\C$-espaces vectoriels:
$$S_k(SL_2(\Z))\to W_k^-, f\mapsto P_f^-$$
$$\text{et }M_k(SL_2(\Z))\to W_k^+, f\mapsto P_f^+.$$
\end{thm}

\begin{proof}
On dispose déjà de l'injectivité des deux applications. En effet, la partie paire des séries d'Eisenstein vérifie aussi ce qui précède car elle est dans $V_k$. Il suffit alors de démontrer l'égalité des dimensions. Ceci se déduit par exemple de l'étude des relations (\ref{eqcoefwk}) vérifiées par les coefficients.
\end{proof}

\begin{rem}Il existe une démonstration dans un cadre plus large de ce résultat. En effet, pour tout $\Gamma$-module $M$ on peut introduire une cohomologie parabolique:
$$Z_{par}^1(\Gamma,M)=\{\varphi:\Gamma\to M;\varphi(T)=0,\varphi(\gamma_1\gamma_2)=\varphi(\gamma_1)|\gamma_2+\varphi(\gamma_2)\},$$
$$B_{par}^1(\Gamma,M)=\{\varphi:\gamma\mapsto m|_{(1-\gamma)},m\in M^T\},$$
$$\text{ et }H_{par}^1(\Gamma,M)=Z_{par}^1(\Gamma,M)/B_{par}^1(\Gamma,M).$$
Alors l'application $\varphi\mapsto\varphi(S)$ identifie $Z_{par}^1(\Gamma,M)$ à $W=\{m\in M,m|_{1+S}=m|_{1+U+U^2}=0\}$.
En effet, $\varphi(S)\in W$ car on peut vérifier que: 
\begin{align*}
\varphi(S)|(1+S)&=\varphi(S^2)=0,\\
\text{et }\varphi(S)|(1+U+U^2)&=\varphi(S)|(1+(TS)^2+TS)=0.
\end{align*}
De plus cette application est bien bijective car $\varphi(T)=0$ et $\Gamma$ est engendré par $S$ et $T$. Et l'image de $B_{par}^1(\Gamma,M)$ est simplement $M^{T}|_{1-S}$. Démontrant ainsi:
$$H_{par}^1(SL_2(\Z),M)\oplus M^{T}|_{1-S}=W=M/\left(M^S+M^U\right).$$
On peut utiliser ceci dans le cas où $M=V_k$. En effet, il nous reste à lier $S_k(\Gamma)$ et $H_{par}^1(\Gamma,V_k)$, pour cela on associe à $f$ une primitive d'ordre $k-1$: 
$$F(z)=\int^{i\infty}_z f(u)(u-z)^{k-2}\d u,$$ 
qui permet de définir les bijections:
$$S_k(SL_2(\Z))\to H_{par}^1(\Gamma,V_k^{\pm}),\quad f\mapsto\varphi_f^{\pm}(\gamma)=F|_{(\gamma-1)}^{\pm}(z).$$
La condition de modularité équivaut à l'appartenance à $V_k$ et la périodicité $f|T=f$ équivaut à $\varphi_f(T)=0$. 
\end{rem}

\section{Structures rationnelles de $M_k(\Gamma)$}

On peut définir deux structures rationnelles sur $M_k(\Gamma)$. La décomposition suivant les coefficients de Fourier permet de définir:
\begin{equation}
M_k^{\Z}(\Gamma)=\{f\in M_k(\Gamma)\text{ telle que }a_n(f)\in\Z,\text{ pour tout } n\}.
\end{equation}
On peut montrer que $M_k^{\Z}(\Gamma)$ est un $\Z$-module de rang la dimension de $M_k(\Gamma)$. Ceci permet notamment de déduire que les valeurs propres des opérateurs de Hecke sont des entiers algébriques totalement réels. 
Ceci permet de définir $M_k^{\Q}(\Gamma)=M_k^{\Z}(\Gamma)\otimes\Q$ et plus généralement $M_k^{K}(\Gamma)=M_k^{\Z}(\Gamma)\otimes K$ où $K$ est un corps de nombre. ceci permet d'obtenir:
\begin{equation}
M_k(\Gamma)=M_k^{\Z}(\Gamma)\otimes_{\Z}\C=S_k^{\Z}(\Gamma)\otimes_{\Z}\C\oplus\C E_k,
\end{equation}
où on a adapté la définition aux formes paraboliques. Ceci est possible car la condition $a_0(f)=0$ est bien indépendante de l'anneau de base.\par
Une structure duale est liée à celle-ci grâce au produit de Petersson:
\begin{equation}
\widehat{M_k^{K}(\Gamma)}=\{f\in M_k(\Gamma)\text{ telle que } \langle f,g\rangle\in K,\text{ pour tout } g\in M_k^{K}(\Gamma)\}.
\end{equation}

Le travail de Rankin sur des formules liant les périodes à un calcul de produit de Petersson permet d'obtenir un Théorème de Eichler-Shimura plus précis:

\begin{thm}
Soit $f\in M_k(\Gamma)$ une forme de Hecke normalisée. Notons $\Q_f$ le corps de nombres engendrés par les $(a_n(f))_{n\geq 0}$.
Il existe deux nombres, $\omega_f^+\in i\R$ et $\omega_f^-\in\R$, tels que:
\begin{equation}
P_f^+\in \omega_f^+\Q_f[X]\quad\text{ et }\quad P_f^-\in\omega_f^-\Q_f[X].
\end{equation}
De plus, le carré scalaire de Petersson vérifie : 
\begin{equation}
i\langle f,f\rangle\in \omega_f^+\omega_f^-\Q_f.
\end{equation}
\end{thm}

\begin{proof}
Ceci peut être déduit d'une généralisation d'un résultat de Rankin, qui donne pour $f$ forme de Hecke et pour des entiers pairs vérifiant $m_1+m_2=k$:
\begin{equation}\label{eqrank}
\langle f,E_{m_1}E_{m_2}\rangle=\left(\frac{i}{2}\right)^{k-1}\Lambda(f,1)\Lambda(f,m_1).
\end{equation}
On peut élargir ce résultat en introduisant les crochets de Rankin:
\begin{equation}
(2i\pi)^n[E_{m_1},E_{m_2}]_n=\sum_{n_1+n_2=n}(-1)^{n_1}\binom{m_1+n-1}{m_1+n_1-1}\binom{m_2+n-1}{m_2+n_2-1}E_{m_1}^{(n_1)}E_{m_2}^{(n_2)},
\end{equation}
qui est une forme modulaire de $M_{m_1+m_2+2n}(\Gamma)$. Ceci permet d'obtenir pour une forme $f$ de Hecke parabolique et des entiers pairs vérifiant $m_1+m_2=k-2n$:
\begin{equation}
\langle f,[E_{m_1}E_{m_2}]_n\rangle=\left(\frac{i}{2}\right)^{k-1}\binom{k-2}{n-1}\Lambda(f,n+1)\Lambda(f,m_1+n).
\end{equation}
Ceci permet de ramener le question de la rationalité à des opérations sur les crochet de Rankin indépendamment de $f$. Et on démontre par le calcul que les formes modulaires $[E_{m_1},E_{m_2}]_n$ sont à coefficients de Fourier rationnels.
\end{proof}

On peut ainsi définir les espaces:
\begin{equation}
S_{k,\Gamma}^{\pm}(K)=\{f\in S_k(\Gamma)\text{ telle que }P_f^{\pm}\in K[X]\}.
\end{equation}
Et on dispose à nouveau de la décomposition:
\begin{equation}
S_k(\Gamma)=S_{k,\Gamma}^{\pm}(K)\otimes_{K}\C,\text{ pour tout corps de nombres }K.
\end{equation}
De plus, les espaces $S_{k,\Gamma}^+(\Q)$ et $S_{k,\Gamma}^-(\Q)$ sont duaux l'un de l'autre par rapport au produit de Petersson, c'est à dire:
\begin{align}
S_{k,\Gamma}^+(\Q)&=\{f\in S_k(\Gamma)\text{ telle que }\langle f ,g \rangle\in i\Q,\text{ pour tout } g\in S_{k,\Gamma}^-(\Q)\},\\
\text{ et }S_{k,\Gamma}^-(\Q)&=\{f\in S_k(\Gamma)\text{ telle que }\langle f ,g \rangle\in i\Q,\text{ pour tout } g\in S_{k,\Gamma}^+(\Q)\}.
\end{align}
Dans la partie suivante, on montre un exemple d'éléments de cette structure rationnelle.

\section{Une famille de formes génératrices de $S_k(\G)$}

Notons $\Gamma=SL_2(\mathbb{Z})$. Soit $f\in S_k(\Gamma)$ une forme modulaire parabolique de niveau $1$ de poids $k\geq 4$. La forme $f$ est entièrement déterminée par ses périodes, les $\Lambda(f,m)$ pour les entiers critiques $m$ (i.e $1\leq m\leq k-1$). De plus le produit de Petersson donne une famille de formes modulaires paraboliques $F_m^k$ représentant ces évaluations:
\begin{equation}
\Lambda(f,m)=\langle f,F_m^k\rangle_k=\int_{\Gamma\backslash\H} f(z)\overline{F_m^k(z)}y^{k-2}\d x\d y.
\end{equation}
Nous noterons $F_m=F_m^k$ s'il n'y pas ambiguïté et $\tilde{m}=k-m$.\par
Dans leur article \cite{KZ}, Kohnen et Zagier calculent $\Lambda(F_m,n)$ pour $m+n$ impaire. Ceci permet, la famille $F_m$ étant génératrice de l'espace des formes paraboliques, de faire les calculs des périodes pour toute forme modulaire parabolique.
\nomenclature{$F_m^k$}{Forme de Kohnen-Zagier pour le poids $k$}

\subsection{Décomposition de $F_m$}
Une formule de Cohen \cite{Cohen81} donne une expression simple de $F_m$:

\begin{prop}[Cohen]\label{propcoh}
La forme modulaire $F_m$ est donnée par la formule:
\begin{equation}
F_m(z)=C_m\sum_{\left(\begin{smallmatrix} a & b \\ c & d \end{smallmatrix}\right)\in \Gamma} (az+b)^{-m}(cz+d)^{-\tilde{m}}\in S_k(\Gamma),
\end{equation}
où on a défini la constante $C_m=-\frac{(2i)^{k-2}i^m}{\pi\binom{k-2}{m-1}}$.
\end{prop}

\begin{rem}
La convergence normale de cette série est assurée car $k\geq 4$.\par
Pour simplifier les notations, nous introduisons de manière générique les coefficients d'une matrice $\gamma=\mat{a}{b}{c}{d}$, quitte à ajouter des indices éventuels. De plus, on introduit la notation: 
$$F_m\left[\mat{a}{b}{c}{d}\right](z)=(az+b)^{-m}(cz+d)^{-\tilde{m}},$$
de manière à obtenir:
$$F_m=C_m\sum_{\gamma\in\Gamma}F_m[\gamma].$$\par
La modularité de $F_m$ peut alors simplement s'observer sur la formule:
$$F_m[\gamma'](\gamma z)=(cz+d)^kF_m[\gamma'\gamma](z).$$
Le caractère parabolique se montre à nouveau grâce au lemme d'Hurwitz.
\end{rem}

\begin{proof}
Un calcul direct montre ce résultat. En effet, pour $f\in S_k(\Gamma)$:
\begin{align*}
i^k\Lambda(f,m)&=\Lambda(f,\tilde{m})\quad\text{ d'après l'équation fonctionnelle}\\
&=\int_0^{\infty} f(iu)u^{\tilde{m}-1}\d u\\
&=\int_0^{\infty} \frac{(-i)^{m-1}(\tilde{m}-1)!}{(k-2)!}f^{(m-1)}(iu)u^{k-2}\d u
\end{align*}

\begin{align*}
&=\frac{(-i)^{m-1}(\tilde{m}-1)!}{(k-2)!}\int_0^{\infty}\left(\frac{(m-1)!}{2i\pi}\int_{-\infty}^{\infty} \frac{f(x+iu/2)}{(x-iu/2)^m}\d x\right)u^{k-2}\d u\\
&=\frac{(-i)^{m}2^{k-2}}{\pi\binom{k-2}{m-1}}\int_{\H} \frac{f(z)}{\bar{z}^m}y^{k-2}\d x\d y\\
&=\frac{(-i)^{m}2^{k-2}}{\pi\binom{k-2}{m-1}}\int_{\Gamma\backslash\H} \sum_{\gamma\in\Gamma}\frac{f(\gamma z)Im(\gamma z)^k}{(\gamma\bar{z})^m}\frac{\d x\d y}{y^2}\\
&=\frac{(-i)^{m}2^{k-2}}{\pi\binom{k-2}{m-1}}\int_{\Gamma\backslash\H} f(z)Im(z)^k\sum_{\gamma\in\Gamma}\frac{(cz+d)^k(c\bar{z}+d)^m}{|cz+d|^{2k}(a\bar{z}+b)^m}\frac{\d x\d y}{y^2}\\
&=i^k\bar{C_m}\int_{\Gamma\backslash\H}f(z)\sum_{\gamma\in\Gamma}(a\bar{z}+b)^{-m}(c\bar{z}+d)^{-\tilde{m}}y^{k-2}\d x\d y
\end{align*}
Ceci donnant bien $F_m(z)=C_m\sum(az+b)^{-m}(cz+d)^{-\tilde{m}}$. 
\end{proof}

\subsection{Calcul de $\Lambda(F_m,n)$}

\begin{thm}[Kohnen-Zagier]\label{thmKZ}
Soient $m$ et $n$ des entiers critiques (i.e. $1\leq m,n \leq k-1$). S'ils sont de parités opposées alors:
$$\Lambda(F_m,n)=\frac{i^k2^{k-1}}{(k-2)!}(Q(m,n)+i^kQ(m,\tilde{n})+i^kQ(\tilde{m},n)+Q(\tilde{m},\tilde{n})+R_1),$$
où:
$$Q(m,n)=(\tilde{m}-1)!(n-1)!\frac{\zeta(n-m+1)}{(2\pi)^{n-m+1}},$$
 et:
 \begin{align*}
 R_1=&(\delta_1(n)+i^k\delta_1(\tilde{n}))(m-1)!(\tilde{m}-1)!\frac{i^m\zeta(m)\zeta(\tilde{m})}{(k-1)\zeta(k)}\\
 &+(\delta_1(m)+i^k\delta_1(\tilde{m}))(n-1)!(\tilde{n}-1)!\frac{i^{n}\zeta(n)\zeta(\tilde{n})}{(k-1)\zeta(k)}.
 \end{align*}
où l'on a noté $\delta_1$ le symbole de Kronecker en $1$.
\end{thm}

\begin{rem}
D'après le calcul de la fonction $\zeta$ en un entier pair, on obtient que les $\Lambda(F_m,n)$ sont des rationnels. 
\begin{equation}
\frac{\zeta(2n)}{(2\pi)^{2n}}=\frac{(-1)^{n-1}}{2}\frac{B_{2n}}{(2n)!}.
\end{equation}
Ceci était anticipé par le résultat sur la rationalité des périodes.\par
Pour effectuer la démonstration de ce théorème, nous remarquons que nous pouvons nous ramener au cas $1\leq n< m\leq k/2$. Il suffit de prendre éventuellement les symétriques dans la première demi-bande, puis d'inverser les deux paramètres par symétrie du résultat.
\end{rem}

\subsubsection{Termes d'interversion}

Nous recherchons à intervertir la somme et l'intégrale dans l'expression:
$$\int_0^{\infty} \sum_{\gamma\in\Gamma} F_m[\gamma](iu)u^{n-1}\d u.$$\par
Pour se faire nous exprimons l'intégrale comme une limite sur un paramètre $\varepsilon>0$:
$$\lim_{\varepsilon\to 0} \int_{\varepsilon}^{1/\varepsilon} \sum_{\gamma\in\Gamma} F_m[\gamma](iu)u^{n-1}\d u.$$\par
La convergence absolue nous permet d'intervertir sans problème dans ce cas. Le terme d'erreur est donc:
$$\lim_{\varepsilon\to 0} \left(\int_0^{\varepsilon}+\int_{1/\varepsilon}^{\infty}\right) \sum_{\gamma\in\Gamma} F_m[\gamma](iu)u^{n-1}\d u.$$\par
Les deux termes apparaissant se calculent de la même manière. En effet, après changement de variables $u\to1/u$ et utilisation de la modularité, on obtient:
$$\int_{1/\varepsilon}^{\infty} F_m(iu)u^{n-1}\d u=i^k\int_{0}^{1/\varepsilon} F_m(iu)u^{\tilde{n}-1}\d u.$$  
Il faut toutefois séparer la somme pour $bd\neq 0$ et $bd=0$:

\textsl{Cas $bd\neq 0$}
\begin{align*}
\int_0^{\varepsilon}& \sum_{bd\neq 0} F_m[\gamma](iu)u^{n-1}\d u\\
&=\varepsilon^{n}\int_0^1\sum_{bd\neq 0} (aiu\varepsilon+b)^{-m}(ciu\varepsilon+d)^{-\tilde{m}}u^{n-1}\d u\\
&=\varepsilon^{n}\int_0^1\sum_{b\wedge d=1}\sum_{\alpha\in\mathbb{Z}} ((a+\alpha b)iu\varepsilon+b)^{-m}((c+\alpha d)iu\varepsilon+d)^{-\tilde{m}}u^{n-1}\d u\\
&=\int_0^1\sum_{b\wedge d=1}b^{-m}d^{-\tilde{m}}\left(\varepsilon^n\sum_{\alpha\in\mathbb{Z}}(1+iu\alpha\varepsilon+\varepsilon aiu/b)^{-m}(1+iu\alpha\varepsilon+\varepsilon ciu/d)^{-\tilde{m}}\right)u^{n-1}\d u\\
&\to_{\varepsilon\to 0}\int_0^1\sum_{b\wedge d=1} b^{-m}d^{-\tilde{m}}\left(\delta_1(n)\int_{-\infty}^{\infty}(iux+1)^{-k}dx\right)\d u\\
&=\delta_1(n)\frac{2\zeta(m)\zeta(\tilde{m})}{\zeta(k)}\frac{\pi}{k-1}.
\end{align*}
En multipliant par $C_m$ et en ajoutant le terme issue du changement de variables $u\to1/u$, nous obtenons bien $\frac{i^k2^{k-1}}{(k-2)!}R_1$. Car nous avons fait l'hypothèse que $m$ ne pouvait être $1$ ou $k-1$ ($1\leq n< m\leq k/2$).

\textsl{Cas $bd=0$}\\
Le calcul de ce terme est analogue et donne:
\begin{align*}
\int_0^{\varepsilon}& \sum_{bd=0} F_m[\gamma](iu)u^{n-1}\d u\\
&=\int_0^{\varepsilon}\sum_{\alpha\in\mathbb{Z}} \left(F_m[(\begin{smallmatrix}\alpha & 1\\ -1 & 0\end{smallmatrix})](iu) + F_m[(\begin{smallmatrix}1 & 0\\ \alpha & 1\end{smallmatrix})](iu)\right)u^{n-1}\d u\\
&=\int_0^1\sum_{\alpha\in\mathbb{Z}} (\alpha iu\varepsilon+1)^{-m}(-iu\varepsilon)^{-\tilde{m}}\varepsilon^{n}u^{n-1}\d u\\
&\quad+\int_0^1\sum_{\alpha\in\mathbb{Z}} (iu\varepsilon)^{-m}(\alpha iu\varepsilon+1)^{-\tilde{m}}\varepsilon^{n}u^{n-1}\d u.
\end{align*}
Puis en passant à la limite sur le paramètre $\varepsilon$ nous obtenons des intégrales de Riemann:
\begin{align*}
&\to_{\varepsilon\to 0}\delta_1(n-\tilde{m})\int_0^1\int_{-\infty}^{\infty} (iux+1)^{-m}i^{\tilde{m}}\d u\d x\\
&\quad+\delta_1(n-m)\int_0^1\int_{-\infty}^{\infty} (iux+1)^{-\tilde{m}}i^{-m}\d u\d x\\
&=i^{-m}\left(i^{k}\delta_1(n-\tilde{m})\frac{\pi}{m-1}+\delta_1(n-m)\frac{\pi}{\tilde{m}-1}\right).
\end{align*}
La multiplication par $C_m=-i^m\frac{2^{k-2}i^k}{\pi \binom{k-2}{m-1}}$ permet d'obtenir le résultat car $\zeta(0)=-\frac{1}{2}$.

\subsubsection{Termes principaux}

Nous recherchons à simplifier la somme des:
$$I(\gamma)=\int_0^{\infty}F_m[\gamma](iu)u^{n-1}du.$$
En effet, pour $\gamma^{-}=\big(\begin{smallmatrix} \phantom{-}a & -b \\ -c & \phantom{-}d \end{smallmatrix}\big)$, nous avons:
\begin{align*}
I(\gamma^{-})&=\int_0^{\infty}\frac{u^{n-1}\d u}{(aiu-b)^m(-ciu+d)^{\tilde{m}}}\\
&=(-1)^{m+n-1}\int_{-\infty}^0\frac{u^{n-1}\d u}{(aiu+b)^m(ciu+d)^{\tilde{m}}}.
\end{align*}
D'après l'hypothèse de parité sur $m+n$, nous obtenons:
$$I(\gamma)+I(\gamma^{-})=\int_{-\infty}^{\infty}\frac{u^{n-1}\d u}{(aiu+b)^m(ciu+d)^{\tilde{m}}}.$$\par
Or les pôles de l'intégrande sont $ib/a$ et $id/c$ qui sont dans le même demi-plan strict dès lors que $bd\neq0$. Et comme $n<k=m+\tilde{m}$ alors l'intégrande est contrôlée à l'infini. Ainsi, d'après le Théorème des résidus, l'intégrale sur cette portion de bord est nulle.\par

Il nous reste à calculer la somme des termes avec $bd=0$. Nous la séparons selon deux sommes celle des termes tels que $b=0$ d'une part et ceux tels que $d=0$ d'autre part. Le calcul de ces sommes est similaire, quitte à changer $m$ en $\tilde{m}$, et repose sur la formule d'Hurwitz:
$$\sum_{\alpha\in\mathbb{Z}} (z+\alpha)^{-m}=\frac{(-2i\pi)^m}{(m-1)!}\sum_{\beta>0} \beta^{m-1} \exp(2i\pi\beta z).$$\par
Ainsi 
\begin{align*}
\sum_{b=0}I(\gamma)&=\int_0^{\infty}\sum_{\alpha\in\mathbb{Z}}\frac{u^{n}}{(iu)^m(\alpha iu+1)^{\tilde{m}}}\frac{\d u}{u}\\
&=i^k\int_0^{\infty}\sum_{\alpha\in\mathbb{Z}}\frac{u^{\tilde{n}}}{(\alpha-iu)^{\tilde{m}}}\frac{\d u}{u}\\
&=i^k(-1)^{\tilde{m}}\frac{(-2i\pi)^{\tilde{m}}}{(\tilde{m}-1)!}\sum_{\beta>0} \beta^{\tilde{m}-1} \int_0^{\infty}\exp(-2\pi\beta u)u^{\tilde{n}-1}\d u\\
&=i^k(-1)^m\frac{(-2i\pi)^{\tilde{m}}}{(\tilde{m}-1)!}\sum_{\beta>0} \beta^{\tilde{m}-1} (2\pi\beta)^{-\tilde{n}}\Gamma(\tilde{n})\\
&=\frac{\zeta(\tilde{n}-\tilde{m}+1)(\tilde{n}-1)!}{i^m(\tilde{m}-1)!(2\pi)^{\tilde{n}-\tilde{m}}}.
\end{align*}
Il nous reste à multiplier par $C_m$ pour obtenir $\frac{i^k2^{k-1}}{(k-2)!}Q(\tilde{m},\tilde{n})$.\par

\subsection{Calcul du polynôme des période de $F_m$}

On peut écrire les polynômes des périodes que l'on vient de calculer de manière simple. Ils dépendent des polynômes de Bernoulli que nous introduisons.

\begin{defi}
On peut définir les \textit{nombres de Bernoulli} par la formule:
$$\sum_{n\geq 0} B_n \frac{t^n}{n!}=\frac{t}{e^t-1}.$$
Et les \textit{polynômes de Bernoulli} par:
$$\sum_{n\geq 0} B_n(X) \frac{t^n}{n!}=\frac{e^{tX}t}{e^t-1}.$$
On pose $B_n^0(X)=B_n(X)-B_1 X^{n-1}$, pour tout $n\geq 1$.
\end{defi}
\nomenclature{$B_n, B_n(X)$}{Nombres et polynômes de Bernoulli}
\begin{prop}On dispose de formule explicite selon les parité:\par
1) Pour tout entier critique impair $n$, nous obtenons:
$$P_{F_n}^-(X)=2^{k-2}(Q_n(iX)+Q_{k-n}(iX)+R_-(X))$$
$$\text{avec }nQ_n(X)=B^0_n(X)-X^{k-2}B^0_n(1/X)$$
$$\text{et }R_-(X)=\frac{\delta_1(n)-\delta_{k-1}(n)}{(k-1)B_k}\sum_{\alpha=2}^{k-2}\binom{k}{\alpha}B_{\alpha}B_{k-\alpha}X^{\alpha-1}.$$\par
2) Pour tout entier critique pair $n$, nous obtenons:
$$P_{F_n}^+(X)=2^{k-2}(Q_n(iX)+Q_{k-n}(iX)+R_+(iX))$$
$$\text{avec }R_+(X)=\frac{k}{B_k}\frac{B_n}{n}\frac{B_{k-n}}{k-n}(1-X^{k-2}).$$
\end{prop}

\begin{proof} 
Les formules suivantes démontrent le résultat:
$$B_n(X)=\sum_{\alpha=0}^n\binom{n}{\alpha}B_{n-\alpha}X^{\alpha}\text{ et }
\frac{\zeta(2N)}{(2\pi)^{2N}}=\frac{(-1)^N}{2}\frac{B_{2N}}{(2N)!}.$$
Ceci permet de déduire la proposition du calcul de Kohnen et Zagier. Par exemple lorsque $n$ est pair:
$$P_{F_n}^+(X) =-\sum_{\substack{\alpha=1\\ \alpha \text{ impair}}}^{k-1} \binom{k-2}{\alpha-1} \frac{\Lambda(F_n,\alpha)}{i^{\alpha}}X^{k-\alpha-1}.$$
Nous obtenons quatre termes similaires et un terme de reste. Ce dernier est donné par:
\begin{align*}
2^{k-2}R_+(iX)&=\frac{i^k2^{k-1}}{(k-2)!}(\Lambda(F_n,1)+\Lambda(F_n,k-1)X^{k-2})\\
&=\frac{i^k 2^{k-1}}{(k-2)!}(n-1)!(k-n-1)!\frac{\zeta(n)\zeta(k-n)}{(k-1)\zeta(k)}(1+i^kX^{k-2})\\
&=\frac{i^k 2^{k-2}kB_nB_{k-n}}{n(k-n)B_k}(1-(iX)^{k-2}).
\end{align*}
\end{proof}

\section{Symboles modulaires}
On va voir dans cette partie que nous pouvons nous abstraire de l'espace $\H$ en nous limitant seulement à ses pointes $\pte$.\par
En effet, si l'on prend deux chemins $c_1$ et $c_2$ de $\H$ reliant $\alpha$ à $\beta$ deux pointes. Alors l'holomorphie d'une forme modulaire $f$ donne l'égalité:
$\int_{c_1} \omega_f=\int_{c_2} \omega_f$ car la chaîne $c_1-c_2$ est fermée et d'intérieur sans pôle. Ainsi ce nombre ne dépend pas du chemin dans $\H$ mais uniquement de ses extrémités.

\begin{defi}
Pour tout $\alpha,\beta\in\pte$, on définit l'application:
\begin{equation}
\{\alpha,\beta\}_k:S_k(SL_2(\Z))\to V_k,\quad f\mapsto \int_{\alpha}^{\beta} f(z)(X-z)^{k-2}\d z.
\end{equation}
On dit que $\{\alpha,\beta\}_k$ est le \textit{symbole modulaire} du chemin reliant $\alpha$ à $\beta$ dans $SL_2(\Z)\backslash \H$.
\end{defi}
\nomenclature{$\{\alpha,\beta\}_k$}{Symbole modulaire du chemin reliant $\alpha$ et $\beta$}
\begin{prop}
Pour tout $\alpha,\beta$ et $\eta$ dans $\pte$, on a:
\begin{equation}
\{\alpha,\beta\}_k+\{\beta,\eta\}_k=\{\alpha,\eta\}_k\text{ et }\{\alpha,\alpha\}_k=0.
\end{equation}
De plus pour tout $g\in SL_2(\Z)$,
\begin{equation}
\{g\alpha,g\beta\}_k=g\{\alpha,\beta\}_k.
\end{equation}
\end{prop}

\begin{proof}
Les relations de Chasles sont triviales une fois les définitions écrites. La propriété de modularité provient de la relation:
$$\int_{g\alpha}^{g\beta}f(z)(X-z)^{k-2}\d z=\int_{\alpha}^{\beta}f(z)(X-z)^{k-2}\d z|_{2-k}g^{-1}=\{\alpha,\beta\}_k(f)(X)|_{2-k}g^{-1}.$$
Ce qui donne bien l'action à gauche $g\{\alpha,\beta\}_k(f)(X)=\{\alpha,\beta\}_k(f)(X)|_{2-k}g^{-1}$.
\end{proof}

Ceci nous permet de donner une structure aux symboles modulaires de $SL_2(\Z)$-module. Les relations de Chasles peuvent se retrouver dans le noyau de l'application:
$$\Z[\pte]\to \Z,\quad \sum \lambda_i (p_i)\mapsto \sum \lambda_i.$$
Ainsi si l'on note $\Z[\pte]^0$ ce noyau, on peut l'envoyer dans le groupe des symboles modulaires.
Le symbole modulaire $\{i\infty,0\}$ image de $(i\infty)-(0)\in \Z[\pte]^0$ joue un rôle central dans cette théorie.

\begin{prop}
L'application $\Z$-linéaire suivante:
\begin{equation}
\Theta_1:\Z[SL_2(\Z)]\to \Z[\pte]^0,\quad \gamma\mapsto (\gamma.i\infty)-(\gamma.0)
\end{equation}
est surjective et son noyau est donné par les relations de Manin:
\begin{equation}
\Ker(\Theta_1)=\mathcal{I}_1=(1+S,1+U+U^2)\Z[SL_2(\Z)].
\end{equation}
En particulier, la suite suivante est exacte:
\begin{equation}
0\to \mathcal{I}_1 \to \Z[SL_2(\Z)]\to \Z[\pte]^0\to 0.
\end{equation}
\end{prop}
\nomenclature{$\mathcal{I}_1$}{Idéal de $\Z[\Gamma]$ des relations de Manin}

\begin{proof}
Pour démontrer la surjectivité on peut se servir de l'écriture en fractions continues des nombres rationnelles:
$$\frac{p}{q}=[a_0,a_1,...,a_N]=a_0+\frac{1}{a_1+\frac{1}{a_2+...}}\text{ et on définit }\frac{p_i}{q_i}=[a_0,...,a_i].$$
Définissons la famille de matrice, $\gamma_{i+1}=\mat{(-1)^ip_{i+1}}{p_i}{(-1)^iq_{i+1}}{q_i}\in\G$ et $\gamma_1=\mat{1}{0}{a_0}{1}$ vérifiant:
$$p/q=\gamma_N.i\infty\to\gamma_N.0=\gamma_{N-1}.i\infty\to...\to\gamma_2.0=\gamma_1.i\infty\to\gamma_1.0=0.$$
Ainsi on peut télescoper de la manière suivante:
$$(p/q)-(0)=\sum_{i=1}^N \left[(\gamma_i.i\infty)-(\gamma_i.0)\right]=\sum_{i=1}^N\Theta_1(\gamma_i).$$
Et la famille des $(\frac{p}{q})-(0)$ engendre bien $Z[\pte]^0$.\\
Pour le noyau, il est simple de vérifier que les relations de Manin annulent bien $\Theta_1$. L'inclusion réciproque demande plus de travail et nous présentons une démonstration dans le Chapitre 3.
\end{proof}

\begin{coro}
Le polynôme des périodes de $T_nf$ en fonction de celui de $f$:
\begin{equation}
P_{T_nf}(X)=\sum_{\gamma\in M_n^0} P_f(X)|_{\gamma},
\end{equation}
où $M_n^0$ est l'ensemble des matrices $\mat{a}{b}{c}{d}\in M_n(\Z)$ vérifiant l'une des trois conditions:
\begin{itemize}
\item $bc<0$, $0<|b|<a$ et $0<|c|<d$.
\item $b=0$ et $|c|<d/2$.
\item $c=0$ et $|b|<a/2$.
\end{itemize}
\end{coro}

\begin{proof}
On va associer à toute forme modulaire $f$ la forme différentielle présente dans la définition des symboles modulaires $\omega_f(z,X)=f(z)(X-z)^{k-2}\d z$.\par
De la définition, des opérateurs de Hecke:
$$T_n f=n^{k/2-1}\sum_{\gamma\in\Gamma\backslash M_n(\Z)} f|_{\gamma},$$
on scinde $\omega_{T_n f}(z,X)$ selon les classes $\mat{a}{b}{c}{d}\gamma\in\Gamma\backslash M_n(\Z)$ en:
\begin{align*}
\omega_{f|_{\gamma}}(z,X)&=n^{k/2-1}\frac{(ad-bc)^{k/2}}{(cz+d)^k}f(\gamma.z)(X-z)^{k-2}\d z\\
&=\frac{n^{k-1}}{(cz+d)^k}f(\gamma.z)\left(\frac{\gamma.X-\gamma.z}{ad-bc}\right)^{k-2}(cz+d)^{k-2}(cX+d)^{k-2}(ad-bc)(cz+d)^2\d(\gamma.z)\\
&=\left[f(\gamma.z)(\gamma.X-\gamma.z)\d(\gamma.z)\right](cX+d)^{k-2}.
\end{align*}
On obtient ainsi une décomposition de la forme différentielle:
\begin{equation}
\omega_{T_nf}(z,X)=\sum_{\gamma\in\Gamma\backslash M_n(\Z)}\omega_{f|\gamma}(z,X)=\sum_{\gamma\in\Gamma\backslash M_n(\Z)}\omega_{f}(\gamma.z,X|_{\gamma}).
\end{equation}
Et ainsi du polynôme des périodes:
$$P_{T_nf}(X)=\{i\infty,0\}(\omega_{T_nf}(z,X))=\sum_{\gamma\in\Gamma\backslash M_n(\Z)} \{i\infty,0\}(\omega_{f}(\gamma.z,X|_{\gamma})).$$
On connait l'action de $\gamma$ sur une telle image:
$$\{i\infty,0\}(\omega_{f}(\gamma.z,X|_{\gamma}))=\{\gamma.i\infty,\gamma.0\}(\omega_{f}(z,X|_{\gamma})).$$
Le symbole modulaire $\{\gamma.i\infty,\gamma.0\}$ se décompose en $\sum_i \{g_i.i\infty,g_i.0\}$ pour des matrices $g_i\in SL_2(\Z)$ et étant donnée que $f|_{g_i}=f$ alors on peut déduire:
$$\{i\infty,0\}(\omega_{f}(\gamma.z,X|_{\gamma}))=\sum_i \{i\infty,0\}(\omega_{f}(z,X|_{\gamma.g_i^{-1}}))=\sum_i P_f(X)|_{\gamma.g_i^{-1}}.$$
Ceci donne une formule liant $P_{T_nf}$ et $P_f$.
\end{proof}

\begin{rem}
On remarque que l'ensemble $M_n^0$ respecte l'échange de $X$ en $-X$ ainsi on peut déduire la stabilité des espaces $S_{k,\Gamma}^+(\Q)$ et $S_{k,\Gamma}^-(\Q)$ par les opérateurs de Hecke.
\end{rem}

\chapter{Intégrale itérée}
\minitoc
\bigskip
La première idée de généralisation de fonction $L$ en plusieurs variables est de construire un analogue de la série de Dirichlet tronquée définissant les valeurs multiples de zêta:
\begin{equation}
\zeta(s_1,...,s_n)=\sum_{l_1,...,l_n> 0}(l_1+...+l_n)^{-s_1}(l_2+...+l_n)^{-s_2}...l_n^{-s_n}\nomenclature{$\zeta(s_1,...,s_n)$}{Valeurs multiples de zêta}
\end{equation}
Nous allons ensuite voir que pour exploiter les propriétés de modularité il est plus commode d'utiliser un analogue des transformées de Mellin grâce aux intégrales itérées.

\section{Séries de Dirichlet multiples à plusieurs variables}

Soient $n\geq 1$ et $k_1,...,k_n\geq 0$ des entiers. Soient $f_1,...,f_n$ des formes modulaires paraboliques de niveau $1$ et de poids respectifs $k_1,...,k_n$. Elles ont des développements en $q$-série de la forme:
\begin{equation}
f_j(z)=\sum_{l\geq 0} a_m(f_j)\exp(2i\pi lz)\text{ pour tout }1\leq j\leq n.
\end{equation}
On suppose de plus que les coefficients $a_m(f_j)$ sont rationnels pour tout $1\leq j\leq n$ et $m>0$.

\begin{defi}
On appelle \textit{série de Dirichlet multiple} associée à la famille $(f_1,...,f_n)$ la fonction en les variables complexes $s_1,...,s_n$ définie par:
\begin{equation}
L(f_1,...,f_n;s_1,...,s_n)=\sum_{l_1,...,l_n> 0}\frac{a_{l_1}(f_1)...a_{l_n}(f_n)}{(l_1+...+l_n)^{s_1}(l_2+...+l_n)^{s_2}...l_n^{s_n}},\nomenclature{$L(f_1,...,f_n;s_1,...,s_n)$}{Valeurs multiples de fonction $L$}
\end{equation}
lorsque $\Re(s_n)>k_n/2+1,...,\Re(s_1+...+s_n)>k_1/2+...+k_n/2+n$.\\
On appelle \textit{transformée de Mellin} associée à la famille $(f_1,...,f_n)$ la fonction en les variables complexes $s_1,...,s_n$ définie par:
\begin{equation}
\Lambda(f_1,...,f_n;s_1,...,s_n)=\int_{0<t_1<...<t_n} f_1(it_1)t_1^{s_1-1}...f_n(it_n)t_n^{s_n-1}\d t_1...\d t_n,\nomenclature{$\Lambda(f_1,...,f_n;s_1,...,s_n)$}{Valeurs multiples de fonction $\Lambda$}
\end{equation}
lorsque $\Re(s_n)>k_n/2+1,...,\Re(s_1+...+s_n)>k_1/2+...+k_n/2+n$.
\end{defi}

\begin{prop}
Pour le domaine de $\C^n$ considéré dans la définition, la série de Dirichlet multiple et la transformée de Mellin converge.
\end{prop}

\begin{proof}
On vérifie la convergence de ces définitions. On a démontré précédemment le résultat de l'estimée de Hecke donnant l'existence de constante $C_1,...,C_n$ tel que:
$$|a_{l}(f_j)|\leq C_jl^{k_j/2}\text{ pour tout }1\leq j\leq n\text{ et }l>0.$$
Il permet d'obtenir la convergence uniforme de la série de Dirichlet sur un certain domaine:
$$\left|\frac{a_{l_1}(f_1)...a_{l_n}(f_n)}{(l_1+...+l_n)^{s_1}...l_n^{s_n}}\right|\leq C_1...C_n \frac{l_1^{k_1/2}...l_n^{k_n/2}}{(l_1+...+l_n)^{\Re(s_1)}...l_n^{\Re(s_n)}}.$$
Ceci est à nouveau majorée par des inégalités arithmético-géométrique par:
$$\frac{C_1...C_n\prod_{j=1}^nl_j^{k_j/2}}{\prod_{j=1}^n (n-j+1)^{\Re(s_j)}(l_j...l_n)^{\Re(s_j)/(n-j+1)}}=O(l_1^{k_1/2-\Re(s_1)/n}...l_n^{k_n/2-\Re(s_n)-\Re(s_{n-1})/2-...-\Re(s_1)/n}).$$
Cette série de réels converge bien pour les valeurs complexe de $(s_1,...,s_n)$ vérifiant:
$$\Re(k_1/2-s_1/n)<-1,...,\Re(k_n/2-s_n-...-s_1/n)<-1.$$
Donc en particulier pour le domaine inclue définie par: 
$$\Re(s_n)>k_n/2+1,...,\Re(s_1+...+s_n)>k_1/2+...+k_n/2+n.$$
On notera que la majoration de Ramanujan donne un domaine de converge plus grand.\par
On obtient de même la convergence de l'intégrale de Mellin à partir de $f(it)=O_{t\to 0}(t^{-k/2-1})$ provenant de l'estimée de Hecke:
\begin{multline*}
|f(it)|=\left|\sum_{l>0}a_l(f)\exp(-2\pi lt)\right|\leq C\sum_{l>0} l^{k/2} \exp(-2\pi lt)\\
=\frac{C}{(-2\pi)^{k/2}}\left(\frac{d}{dt}\right)^{k/2}\left[\sum_{l>0}\exp(-2\pi lt)\right]=\frac{C}{(-2\pi)^{k/2}}\left(\frac{d}{dt}\right)^{k/2}\left[\frac{\exp(-2\pi t)}{1-\exp(-2\pi t)}\right]=O_{t\to 0}(t^{-k/2-1}).
\end{multline*}
La modularité de la fonction donne aussi $f(it)=O_{t\to \infty}(t^{-k/2+1})$, pour tout $f\in S_k(\G)$. Ceci nous donnant bien la convergence de l'intégrale itérée sur le domaine donné par la définition.
\end{proof}

\begin{prop}
L'application $\Lambda$ admet un prolongement holomorphes à $\C^n$ et vérifie l'équation fonctionnelle:
\begin{equation}
\Lambda(f_1,...,f_n;s_1,...,s_n)=i^{\sum k_j}\Lambda(f_n,...,f_1;k_n-s_n,...,k_1-s_1),
\end{equation}
pour tout $(s_1,...,s_n)\in\C^n$.
\end{prop}

\begin{proof}
Pour $(s_1,...,s_n)\in\C^n$ vérifiant $\Re(s_n)>k_n/2+1,...,\Re(s_1+...+s_n)>k_1/2+...+k_n/2+n$ et pour un paramètre réel $T>0$, on a:
\begin{align*}
\Lambda(f_1,...,f_n;s_1,...,s_n)&=\int_{0<t_1<...<t_n}f_1(it_1)t_1^{s_1}...f_n(it_n)t_n^{s_n}\frac{\d t_1}{t_1}...\frac{\d t_n}{t_n}\\
&=\sum_{j=0}^n\int_{0<t_1<...<t_j<T<t_{j+1}<...<t_n}f_1(it_1)t_1^{s_1}...f_n(it_n)t_n^{s_n}\frac{\d t_1}{t_1}...\frac{\d t_n}{t_n}.
\end{align*}
Pour $T>0$, posons :
\begin{equation}
\Lambda_T(g_1,...,g_m;z_1,...,z_m)=\int_{T<t_1<...<t_m}g_1(it_1)t_1^{z_1}...g_1(it_m)t_m^{z_m}\frac{\d t_1}{t_1}...\frac{\d t_m}{t_m}.
\end{equation}
Alors le changement de variable $(t_1,...,t_j)\mapsto (1/t_1,...,1/t_j)$ et la propriété de modularité $f(i/t)=(it)^{k}f(it)$ donne:
$$\Lambda(f_1,...,f_n;s_1,...,s_n)=\sum_{j=0}^ni^{\sum_{a=1}^j k_a}\Lambda_{1/T}(f_j,...,f_1;k_j-s_j,...,k_1-s_1)\Lambda_T(f_{j+1},...,f_n;s_{j+1},...,s_n).$$
Cette famille de fonctions admet un prolongement holomorphe.\par
Ainsi il en va de même pour $\Lambda$. On trouve l'équation fonctionnelle en faisant le même calcul pour $\Lambda(f_n,...,f_1;k_n-s_n,...,k_1-s_1)$ et en obtenant le même résultat.
\end{proof}

Dans son article \cite{Ma1}, Manin considère le prolongement de la série de Dirichlet multiple en prenant la limite pour $z\to 0$ des applications:
$$L_z(f_1,...,f_n;s_1,...,s_n)=\sum_{l_1,...,l_n> 0}\frac{a_{l_1}(f_1)...a_{l_n}(f_n)\exp(2i\pi(l_1+...+l_n)z)}{(l_1+...+l_n)^{s_1}(l_2+...+l_n)^{s_2}...l_n^{s_n}}.$$
Il montre ensuite que ceci convient pour les valeurs entières de $(s_1,...,s_n)\in\Z_{>0}^n$. Nous proposons ici de les relier à la transformée de Mellin que l'on vient d'étendre pour ces valeurs.

\begin{prop}\label{prop24}
Soient $(m_1,...,m_n)\in\Z_{>0}^n$ vérifiant $m_n>k_n/2+1,...,m_1+...+m_n>k_1/2+...+k_n/2+n$. Considérons la fonction $L$ complétée par:
$$L^*(f_1,...,f_n;s_1,...,s_n)=\frac{\Gamma(s_1)...\Gamma(s_n)}{(2\pi)^{s_1+...+s_n}}L(f_1,...,f_n;s_1,...,s_n),$$
et les familles de coefficients entiers:
\begin{align}
A_{m_1,...,m_n}(M_1,...,M_n)&=\sum_{\substack{\alpha_{a,b}\geq 0\\ \sum_{a=1}^b\alpha_{a,b}=m_b-1\\ \sum_{b=j}^n\alpha_{j,b}=M_j-1}}\frac{(m_1-1)!...(m_n-1)!}{\prod_{1\leq a\leq b\leq n} \alpha_{a,b}!},\\
\text{et }B_{m_1,...,m_n}(M_1,...,M_n)&=(-1)^{\sum_{j=1}^n (n+1-j)(M_j-m_j)}\prod_{j=1}^{n-1}\binom{m_{j+1}-1}{\sum_{a=1}^j (M_a-m_a)}.
\end{align}
Alors on a les relations:
\begin{align}
\Lambda(f_1,...,f_n;m_1,...,m_n)&=\sum_{\substack{M_1,...,M_n\geq 0\\ \sum M_j=\sum m_j}} A_{m_1,...,m_n}(M_1,...,M_n) L^*(f_1,...,f_n;M_1,...,M_n),\\
L^*(f_1,...,f_n;m_1,...,m_n)&=\sum_{\substack{M_1,...,M_n\geq 0\\ \sum M_j=\sum m_j}} B_{m_1,...,m_n}(M_1,...,M_n) \Lambda(f_1,...,f_n;M_1,...,M_n).\label{defln}
\end{align}
Cette dernière permet d'étendre la définition de $L^*$ à tout $n$-uplet d'entiers naturels.\nomenclature{$L^*(f_1,...,f_n;s_1,...,s_n)$}{Valeurs multiples de fonction $L$ complétée}
\end{prop}

\begin{proof}
Pour des entiers $m_1,...,m_n$ comme dans l'énoncé, on a:
\begin{align*}
&L^*(f_1,...,f_n;m_1,...,m_n)=(2\pi)^{-m_1-...-m_n}\Gamma(m_1)...\Gamma(m_n)L(f_1,...,f_n;m_1,...,m_n)\\
&=\int_0^{\infty}e^{-2\pi u_1}u_1^{m_1-1}\d u_1...\int_0^{\infty}e^{-2\pi u_n}u_n^{m_n-1}\d u_n\sum_{l_1,...,l_n>0}\frac{a_{l_1}(f_1)...a_{l_n}(f_n)}{(l_1+...+l_n)^{m_1}...l_n^{m_n}}\\
&=\sum_{l_1,...,l_n>0}a_{l_1}(f_1)...a_{l_n}(f_n)\int_0^{\infty}e^{-2\pi(l_1+...+l_n) u_1}u_1^{m_1-1}\d u_1...\int_0^{\infty}e^{-2\pi l_nu_n}u_n^{m_n-1}\d u_n\\
&=\sum_{l_1,...,l_n>0}a_{l_1}(f_1)...a_{l_n}(f_n)\int_{0<u_1,...,u_n}e^{-2\pi l_1 u_1}u_1^{m_1-1}...e^{-2\pi l_n(u_1+...+u_n)}u_n^{m_n-1}\d u_1...\d u_n\\
&=\sum_{l_1,...,l_n>0}a_{l_1}(f_1)...a_{l_n}(f_n)\int_{0<t_1<...<t_n}e^{-2\pi l_1 t_1}t_1^{m_1-1}...e^{-2\pi l_nt_n}(t_n-t_{n-1})^{m_n-1}\d t_1...\d t_n\\
&=\int_{0<t_1<...<t_n} f_1(it_1)...f_n(it_n)t_1^{m_1-1}\sum_{\substack{a_j=0\\ j=2..n}}^{m_j-1}\binom{m_j-1}{a_j}t_j^{m_j-1-a_j}(-t_{j-1})^{a_j}\d t_1...\d t_n\\
&=\sum_{a_2=0}^{m_2-1}\binom{m_2-1}{a_2}(-1)^{a_2}...\sum_{a_n=0}^{m_n-1}\binom{m_n-1}{a_n}(-1)^{a_n}\Lambda(f_1,...,f_n;m_1+a_2,m_2-a_2+a_3,...,m_n-a_n).
\end{align*}
Et d'autre part:
\begin{align*}
\Lambda&(f_1,...,f_n;m_1,...,m_n)\\
=&\int_{0<t_1<...<t_n} f_1(it_1)t_1^{m_1-1}...f_n(it_n)t_n^{m_n-1}\d t_1...\d t_n\\
=&\int_{0<u_1,...,u_n} f_1(it_1)u_1^{m_1-1}...f_n(i(u_1+...+u_n)(u_1+...+u_n)^{m_n-1}\d u_1...\d u_n\\
=&\sum_{l_1,...,l_n>0}a_{l_1}(f_1)...a_{l_n}(f_n)\int_{0<u_1,...,u_n} e^{-2\pi l_1 u_1}u_1^{m_1-1}...e^{-2\pi l_n (u_1+...+u_n)}(u_1+...+u_n)^{m_n-1}\d u_1...\d u_n\\
=&\sum_{l_1,...,l_n>0}\sum_{\substack{\alpha_{a,b}\geq 0\\ \sum_{a=1}^b\alpha_{a,b}=m_b-1}}\frac{(m_1-1)!...(m_n-1)!a_{l_1}(f_1)...a_{l_n}(f_n)}{\prod_{1\leq a\leq b\leq n} \alpha_{a,b}!}\\
&\int_0^{\infty} e^{-2\pi (l_1+...+l_n) u_1}u_1^{\sum_{b=1}^n \alpha_{1,b}}\d u_1...\int_0^{\infty} e^{-2\pi l_n u_n}u_n^{\alpha_{n,n}}\d u_n\\
=&\sum_{\substack{\alpha_{a,b}\geq 0\\ \sum_{a=1}^b\alpha_{a,b}=m_b-1}}\frac{(m_1-1)!...(m_n-1)!}{\prod_{1\leq a\leq b\leq n} \alpha_{a,b}!}
L^*(f_1,...,f_n;\sum_{b=1}^n \alpha_{1,b}+1,...,\alpha_{n,n}+1).
\end{align*}
\end{proof}

\begin{ex}
Pour $n=1$, nous retrouvons la relation: $\Lambda(f;m)=L^*(f;m)=\frac{(m-1)!}{(2\pi)^m}L(f,m)$.\par
Pour $n=2$, nous obtenons les relations suivantes:
\begin{align}
\frac{\Lambda(f_1,f_2;m_1,m_2)}{(m_1-1)!(m_2-1)!}&=\sum_{\alpha=0}^{m_2-1}\binom{m_1-1+\alpha}{\alpha}\frac{L(f_1,f_2;m_1+\alpha,m_2-\alpha)}{(2\pi)^{m_1+m_2}},\\
\text{et }L^*(f_1,f_2;m_1,m_2)&=\sum_{\alpha=0}^{m_2-1}(-1)^{\alpha}\binom{m_2-1}{\alpha}\Lambda(f_1,f_2;m_1+\alpha,m_2-\alpha).
\end{align}
\end{ex}

\section{Intégrales itérées et séries génératrices (d'après Manin)}

Manin introduit dans \cite{Ma1}, un formalisme de série génératrice qui permet notamment de rassembler les valeurs multiples de zêta. Il donne ensuite un certain nombre de propriétés liées à la structure. Nous les rappelons ici brièvement.

\begin{defi}
Soit $V$ un ensemble fini. Soit $X_V=(X_v)_{v\in V}$ une famille d'indéterminées indexées par $V$.\\
Posons $\Ser(V)=\C\left\langle\langle X_v\right\rangle\rangle_{v\in V}$ \nomenclature{$\Ser(V)$}{Espace des séries entières non commutatives}
 l'algèbre sur $\C$ des séries entières non commutatives en les indéterminées $(X_v)_{v\in V}$ et de coefficient constant égal à $1$.\\
Soit $E$ une surface de Riemann connexe. Soit $\mathcal{O}_E$ son faisceau de fonctions holomorphes et $\Omega^1_E$ son faisceau de $1$-formes holomorphes.\\
\nomenclature{$\mathcal{O}_E$}{Faisceau des fonctions holomorphes sur $E$}
\nomenclature{$\Omega^1_E$}{Faisceau des $1$-formes holomorphes sur $E$}
Pour toute famille $\omega_V=(\omega_v)_{v\in V}$ de $1$-formes de $\Omega_E^1$ et tout chemin $\varphi: (0,1)\to E$ continue, on définit la série entière de variables non commutatives:
\begin{equation}
J_{\varphi}(\omega_V,X_V)=\sum_{n\geq 0} \sum_{v_1,...,v_n\in V} \int_0^1 \varphi^*\omega_{v_1}(t_1)\int_0^{t_1} \varphi^*\omega_{v_2}(t_2)...\int_0^{t_{n-1}} \varphi^*\omega_{v_n}(t_n)\d t_1...\d t_nX_{v_1}...X_{v_n}
\end{equation}
Pour simplifier les notations, on écrit: $\Omega=\langle \omega_V,X_V\rangle=\sum_{v\in V} \omega_v X_v\in \Omega_E^1\langle\langle X_V\rangle\rangle$ et $J_{\varphi}(\Omega)=J_{\varphi}(\omega_V,X_V)$.
\end{defi}

Si $\varphi$ et $\varphi'$ sont deux chemins homotopes alors $J_{\varphi}(\Omega)=J_{\varphi'}(\Omega)$. Soient $U\subset E$ un ouvert connexe et simplement connexe et $a\in U$ un point base alors on définit une fonction de $\mathcal{O}_U\langle\langle X_V\rangle\rangle$ par:
\begin{equation}
z\mapsto J_a^z(\Omega)=J_{\varphi_z}(\Omega),
\end{equation}
pour tout $\varphi_z$ chemin dans la classe d'homotopie vérifiant $\varphi_z(0)=a$ et $\varphi_z(1)=z$.

\begin{ex}
1) On dispose de l'exemple fondamental suivant. Soient $r\geq 1$ un entier et $m_1,...,m_r$ une famille d'entiers. Pour assurer la convergence on suppose ces entiers $m_j\geq 1$ et $m_1>1$. Les valeurs multiples de zêta sont définies par:
$$\zeta(m_1,...,m_r)=\sum_{n_1>...>n_r}\prod_{i=1}^r n_i^{-m_i}.$$
On peut les regrouper dans la série $J_0^1(\Omega)$ où $U=\R\subset\C$ et $\Omega=\frac{dt}{t}X_0+\frac{dt}{1-t}X_1$.\par
En effet, on dispose de l'expression suivante en intégrale itérée :
\begin{multline}
\zeta(m_1,...,m_r)=\int_0^1\frac{\d t_{1,1}}{t_{1,1}}...\int_0^{t_{m_1-2,1}}\frac{\d t_{m_1-1,1}}{t_{m_1-1,1}}\int_0^{t_{m_1-1,1}}\frac{\d t_{m_1,1}}{1-t_{m_1,1}}\int_0^{t_{m_1,1}}\frac{\d t_{1,2}}{t_{1,2}}...\\
...\int_0^{t_{m_{r-1},r-1}}\frac{\d t_{1,r}}{t_{1,r}}...\int_0^{t_{m_r-2,r}}\frac{\d t_{m_r-1,r}}{t_{m_r-1,r}}\int_0^{t_{m_r-1,r}}\frac{\d t_{m_r,r}}{1-t_{m_r,r}}.
\end{multline}
Nous permettant d'obtenir:
\begin{equation}
J_0^1(\Omega)=1+\sum_{r\geq 1}\sum_{\substack{m_1>1\\ m_2,...,m_r\geq 1}}\zeta(m_1,...,m_r)X_0^{m_1-1}X_1...X_0^{m_r-1}X_1+\text{termes non convergents}.
\end{equation}
En effet, les termes qui nous intéressent sont les coefficients des monômes $X_0...X_1$. Commencer par $X_0$ donne $m_1>1$ et finir par $X_1$ donne les types d'intégrales attendues. On se restreint donc dans ce cas à un idéal de $\Ser(V)$ vérifiant de bonnes propriétés. Par exemple, le changement de variable $t\mapsto 1-t$ échange le sens d'intégration et les indéterminées $X_0$ et $X_1$ donnant les relations sur les valeurs multiples de zêta:
\begin{equation}
\zeta(m_1,...,m_r)=\zeta(a_1,...,a_l)\text{ tel que }X_0^{a_1-1}X_1...X_0^{a_l-1}X_1=X_0X_1^{m_r-1}...X_0X_1^{m_1-1}.
\end{equation}
Ceci donne par exemple $\zeta(m)=\zeta(2,1,...,1)$, pour tout $m\geq 2$.\par
2) Soit $k\geq 2$ un entier pair. Posons $E=\H$ et $V=\{1,...,k-1\}$. Pour une forme modulaire parabolique $f\in S_k(SL_2(\Z))$ de poids $k$ et de niveau $1$, posons :
\begin{equation}\label{defl1}
\Omega_f=\sum_{m=1}^{k-1} f(z) z^{m-1} \d z X_{m}\text{ et }J(f)=J_{i\infty}^{0}(\Omega_f).
\end{equation}
La série obtenue contient notamment à l'ordre $1$ les périodes $\Lambda(f,m)$ :
\begin{equation}
J(f)=1+\sum_{m=1}^{k-1} \Lambda(f,m)X_m+\text{ termes de degré }\geq 2.
\end{equation}
L'application $S_k(SL_2(\Z))\to \C^{k-1},f\mapsto(\Lambda(f,m))_{1\leq m\leq k-1}$ est injective donc $J:S_k(Sl_2(\Z))\to \Ser(V)$ aussi.\par
3) Soit $\Gamma\subset SL_2(\Z)$ un groupe de congruence. Posons $E=\H$ et $V=\Gamma\backslash SL_2(\Z)$. Pour une forme modulaire $f\in S_2(\Gamma)$, définissons:
\begin{equation}\label{defw2}
\Omega_f=\sum_{g\in \Gamma\backslash SL_2(\Z)} f|_g(z) \d z X_g\text{ et }J(f)=J_{i\infty}^{0}\left(\Omega_f\right).
\end{equation}
On obtient à nouveau le début de développement:
\begin{equation}
J(f)=1+\sum_{g\in V}\Lambda(f|_g,1)X_g+\text{ termes de degré }\geq 2,
\end{equation}
permettant de déduire l'injectivité de $J:S_2(\Gamma)\to \Ser(V)$ à partir de celle de l'application:
$$S_2(\Gamma)\to \C^V,f\mapsto \left(\int_0^{i\infty}f|_g(z)\d z\right)_{g\in V}.$$
\end{ex}

On rappel des propriétés liées à la structure démontrées par Manin dans \cite{Ma1}.

\begin{prop}Soit $\Omega=\sum_{v\in V} \omega_v X_v$ une famille de $1$-formes holomorphes sur une surface de Riemann $E$ connexe et simplement connexe.
On a:
\begin{equation}
J_a^b(\Omega)=\sum_{n\geq 0} \int_a^b \Omega(z_1) \int_a^{z_1} \Omega(z_2)... \int_a^{z_{n-1}} \Omega(z_n).
\end{equation}
On dispose alors des deux relations:
\begin{align}
\partial_z J_a^z(\Omega)&=\Omega(z)J_a^z(\Omega)\in \Omega^1_E\langle\langle A_V\rangle\rangle,\\ 
J_{a}^{c}(\Omega)&=J_{b}^{c}(\Omega)J_{a}^{b}(\Omega)\in \Ser(V),\label{chasles}
\end{align}
où $a,b,c$ et $z$ sont des points de $E$.
\end{prop}

\begin{proof} L'équation différentielle se vérifie par une simple dérivation et permet de déduire la deuxième relation. En effet, les applications $z\mapsto J_a^z(\Omega)$ et $z\mapsto J_b^z(\Omega)J_a^b(\Omega)$ sont deux solutions du problème de Cauchy:
$$\partial_z \varphi(z)=\Omega(z)\varphi(z)\text{ et }\varphi(b)=J_a^b(\Omega)\text{ pour }\varphi\in\mathcal{O}_E\langle\langle X_V\rangle\rangle.$$
Donc d'après le Théorème de Cauchy-Lipschitz, il existe bien une unique solution. Et l'égalité des fonctions nous donnent la relation (\ref{chasles}).
\end{proof}

\begin{prop}
On a une comultiplication $\Delta:\Ser(V)\to \Ser(V)\otimes \Ser(V)$. C'est l'homomorphisme d'algèbre graduée déterminé par les formules:
\begin{equation}
\Delta X_v=X_v\otimes 1+1\otimes X_v,\text{ pour tout }v\in V.
\end{equation}
Soit $\Omega=\sum_{v\in V}\omega_v X_v$ une famille de $1$-formes holomorphes sur $E$. Soient $a,b\in E$. On a alors une relation de mélange:
\begin{equation}\label{relmel}
\Delta(J_a^b(\Omega))=J_a^b(\Omega)\otimes J_a^b(\Omega).
\end{equation}
\end{prop}

\begin{proof}
L'équation différentielle est à nouveau la clé de la démonstration. En effet les applications $z\mapsto \Delta(J_a^z(\Omega))$ et $z\mapsto J_a^z(\Omega)\otimes J_a^z(\Omega)$ sont tous les deux solutions du problème de Cauchy:
$$\partial_z\varphi(z)=\Delta(\Omega)\varphi(z)\text{ et }\varphi(a)=1\otimes 1\text{ pour }\varphi\in\mathcal{O}_E\langle\langle X_V\otimes X_V\rangle\rangle.$$
Les deux fonctions sont donc identiques et on obtient la relation de mélange.
\end{proof}

Lorsque la surface $E=\H$, on peut transporter l'action de $SL_2(\Z)$ en une action à gauche sur les séries construites:\\
Soient $a,b\in\H$, $\gamma\in SL_2(\Z)$ et $\Omega=\sum_{v\in V} \omega_v X_v$ une famille de $1$-formes holomorphes sur $\H$.
On définit une action à gauche de $SL_2(\Z)$ par:
\begin{equation}\label{defact}
\gamma.J_a^b(\Omega)=J_{\gamma.a}^{\gamma.b}(\Omega)\text{ et }J_{\gamma}(\Omega)=J_{\gamma.i\infty}^{\gamma.0}(\Omega).
\end{equation}

\begin{prop}
Pour toute famille de $1$-formes holomorphes $\Omega$, on a des relations de Manin non commutatives:
\begin{align}
J(\Omega)J_S(\Omega)&=J_S(\Omega)J(\Omega)=1,\\
J_{U^2}(\Omega)J_U(\Omega)J(\Omega)=J_U(\Omega)&J(\Omega)J_{U^2}(\Omega)=J(\Omega)J_{U^2}(\Omega)J_U(\Omega)=1.
\end{align}
\end{prop}

\begin{proof}
Pour cela il suffit de regarder l'action de $S$ et $U$ sur les pointes $i\infty$ et $0$:
$$i\infty\stackrel{S}{\to} 0\stackrel{S}{\to} i\infty\text{ et }i\infty\stackrel{U}{\to} 0 \stackrel{U}{\to} 1 \stackrel{U}{\to} i\infty.$$
Le relation (\ref{chasles}) donne alors le résultat car $J_{i\infty}^{i\infty}(\Omega)=1$.
\end{proof}

\section{Séries génératrices des formes modulaires paraboliques de niveau $1$}

\subsection{Le cas d'une forme modulaire parabolique de niveau $1$}

On reprend plus précisément l'exemple (\ref{defl1}).\\
Soient $k$ un entier pair et $f\in S_k(SL_2(\Z))$ une forme modulaire parabolique de poids $k$ et de niveau $1$. On s'intéresse à l'action de $SL_2(\Z)$ sur $J(f)$. On va construire une action sur les séries formelles de $\Ser(V)$ qui sera duale de celle définie en (\ref{defact}).\par

Posons $\omega_i(z)=f(z)z^{i-1}\d z$ pour $1\leq i\leq k-1$.\\ 
On commence par remarquer que l'action $J_{\gamma}(f)=J_{\gamma i\infty}^{\gamma 0}(\Omega_f)$ se traduit en une action sur la famille des formes $\left(\omega_i\right)_{1\leq i\leq k-1}$. 

\begin{prop}
Pour toute matrice $\gamma\in SL_2(\Z)$, il existe une matrice $M_k(\gamma)\in SL_{k-1}(\Z)$ ne dépendant pas de $f$ telle que pour tout $1\leq i \leq k-1$:
\begin{equation}
\gamma^*(\omega_i(z))=\omega_i(\gamma z)=\sum_{j=1}^{k-1} M_k(\gamma)_{i,j} \omega_j(z).
\end{equation}
\end{prop}

\begin{proof}
Soit $\gamma=\mat{a}{b}{c}{d}\in SL_2(\Z)$, on a: 
\begin{multline*}
\mat{a}{b}{c}{d}^*(f(z)z^{i-1}\d z)=f\left(\frac{az+b}{cz+d}\right)\left(\frac{az+b}{cz+d}\right)^{i-1}\d\gamma z\\
=(cz+d)^kf(z)(az+b)^{i-1}(cz+d)^{1-i}\frac{\d z}{(cz+d)^2}=f(z)(az+b)^{i-1}(cz+d)^{k-i-1}\d z.
\end{multline*}
Ainsi posons $M_k(\gamma)_{i,j}=\sum_{\alpha+\beta=j-1}\binom{i-1}{\alpha}\binom{k-i-1}{\beta}a^{\alpha}b^{i-1-\alpha}c^{\beta}d^{k-i-1-\beta}$ pour $1\leq i,j \leq k-1$. Ces coefficients sont entiers et constituent bien une matrice indépendante de la famille des $(\omega_i)_{1\leq i\leq k-1}$. Pour démontrer son appartenance à $SL_{k-1}(\Z)$, nous allons montrer une propriété de morphisme donnant $M_k(\gamma^{-1})$ comme inverse de $M_k(\gamma)$.
\end{proof}

\begin{prop}\label{mkmorph}
L'application $\gamma\mapsto M_k(\gamma)$ est un homomorphisme de groupe, c'est à dire:
$$M_k(\gamma_1\gamma_2)=M_k(\gamma_1)M_k(\gamma_2)\text{ pour tout }\gamma_1,\gamma_2\in SL_2(\Z).$$
\end{prop}

\begin{proof}
Soit $1\leq i\leq k-1$ une entier. On a la série d'égalité: 
$$\sum_{l=1}^{k-1}M_k(\gamma_1\gamma_2)_{i,l}\omega_l=(\gamma_1\gamma_2)^*\omega_i=\gamma_2^*\gamma_1^*\omega_i=\gamma_2^*\left(\sum_{j=1}^{k-1} M_k(\gamma_1)_{i,j}\omega_j\right)$$
$$=\sum_{j=1}^{k-1} M_k(\gamma_1)_{i,j}\gamma_2^*\omega_j=\sum_{j,l=1}^{k-1}M_k(\gamma_1)_{i,j}M_k(\gamma_2)_{j,l}\omega_l.$$
Ceci donne l'égalité des coefficients des matrices par identification.
\end{proof}

Posons $S=\left(\begin{smallmatrix} 0 & -1 \\1&\phantom{-}0\end{smallmatrix}\right)$, $T=\left(\begin{smallmatrix} 1 & 1 \\0&1\end{smallmatrix}\right)$ et $U=ST^{-1}=\left(\begin{smallmatrix} \phantom{-}0 & 1 \\-1&1\end{smallmatrix}\right).$ 
Le groupe $SL_2(\Z)$ est engendré par deux de ces trois éléments.\par

$M_k(\gamma)$ est ainsi entièrement déterminé par deux des trois valeurs en $S$, $T$ ou $U$:
\begin{align}
M_k(S)_{i,j}&=(-1)^{i-1}\delta_{(i+j=k)},\\
M_k(T)_{i,j}&=\binom{i-1}{j-1}\\
\text{ et }M_k(U)_{i,j}&=(-1)^{j-1}\binom{k-i-1}{j-1}.
\end{align}

On étend la notation classique en posant:
\begin{align}
\gamma^*\Omega_f&=\sum_{i=1}^{k-1}(\gamma^*\omega_i) X_{i}=\sum_{i,j=1}^{k-1} M_k(\gamma)_{i,j}\omega_jX_{i}\\
&=\langle M_k(\gamma)\omega_V,X_V\rangle=\langle\omega_V,^tM_k(\gamma)X_V\rangle.
\end{align}
Donc on voit que si $M_k(\gamma)$ agit sur la famille des $(\omega_i)_{1\leq i\leq k-1}$ alors $^tM_k(\gamma)$ agit de manière duale sur la famille des $(X_j)_{1\leq j\leq k-1}$:
\begin{equation}
J_{\gamma}(f)=J_{\gamma i\infty}^{\gamma 0}(\Omega_f)
=J_{i\infty}^0\left(\langle M_k(\gamma)\omega_V,X_V\rangle\right)
=J_{i\infty}^0\left(\langle\omega_V,^tM_k(\gamma)X_V\rangle\right).
\end{equation}
On a ainsi définit une action de $SL_2(\Z)$-module sur l'espace des séries formelles $J\in \Ser(V)$ par:
\begin{equation}
J(X_V)|_{\gamma}=J({}^tM_k(\gamma)X_V),
\end{equation}
On a bien une action de groupe, d'après la Proposition \ref{mkmorph}

\subsection{Le cas de $n$ formes modulaires paraboliques de niveau $1$}

Soit $k=(k_1,...,k_n)$ une famille d'entiers naturels pairs. 
Soit $f_1,...,f_n$ une famille de formes modulaires paraboliques de niveau $1$ et de poids respectifs $k_1,...,k_n$. Posons:
$$V=\{(i,j);1\leq j \leq n,1\leq i \leq k_j-1\}.$$
Définissons la famille de $1$-formes de $\Omega_{\H}^1$ indexée par $V$ associée à $(f_1,...,f_n)$ par:
\begin{align}
\Omega_{f_1,...,f_n}(z)&=\sum_{j=1}^n\Omega_{f_j}(z)Y_j\\
&=\sum_{j=1}^n\sum_{i=1}^{k_j-1}f_j(z)z^{i-1}\d zX_iY_j.
\end{align}
On définit une série formelle en les indéterminées non commutatives $A_{i,j}=X_iY_j$ pour $(i,j)\in V$ par:
\begin{equation}
J(f_1,...,f_n)=J_{i\infty}^0\left(\Omega_{f_1,...,f_n}\right).
\end{equation}

On définit une action à gauche de $\gamma\in\Gamma$ sur $J(f_1,...,f_n)$ :
\begin{equation}
J_{\gamma}(f_1,...,f_n)=J_{\gamma i\infty}^{\gamma 0}\left(\Omega_{f_1,...,f_n}\right).
\end{equation}
Il est naturel de chercher à étendre cette action à toute série de $\Ser(V)$,
 
Ici encore cette action se transcrit en une action sur la famille $\omega_V$ des :
$$\omega_{i,j}(z)=f_j(z)z^{i-1}\d z\text{ indexé par les }(i,j)\in V.$$

Cette action est diagonale par bloc et est donnée par les actions précédentes:
$$\gamma^*\omega_{i,j}=\sum_{\alpha=1}^{k_j-1}M_{k_j}(\gamma)_{i,\alpha}\omega_{\alpha,j}\text{ pour tout }1\leq j\leq n.$$

Ainsi la matrice diagonale par bloc $M_{(k_1,...,k_n)}(\gamma)$ décrit l'action de $\gamma$ sur $\omega_V$:
\begin{equation}
M_{(k_1,...,k_n)}(\gamma)=\left(\begin{smallmatrix}M_{k_1}(\gamma) & & 0\\ & \ddots & \\ 0 & & M_{k_n}(\gamma)\end{smallmatrix}\right).
\end{equation}
On en déduit que $\gamma\mapsto M_{(k_1,...,k_n)}(\gamma)$ est un morphisme de groupe. De plus la transposée construit une action compatible sur $\Ser(V)$:
\begin{equation}
J(A_V)|_{\gamma}=J({}^tM_{(k_1,...,k_n)}(\gamma)A_V),\text{ pour tout }J\in \Ser(V).
\end{equation}
En effet, cette action est compatible à la précédente sur les éléments $J(f_1,...,f_n)\in \Ser(V)$:
$$J_{\gamma}(f_1,...,f_n)(A_V)=J_{i\infty}^0\left(\sum_{j=1}^n\gamma^*(\Omega_{f_j})Y_j\right)
=J_{i\infty}^0\left(\langle M_{(k_1,...,k_n)}(\gamma)\omega_V,A_V\rangle\right)$$
$$=J_{i\infty}^0\left(\langle \omega_V,{}^tM_{(k_1,...,k_n)}(\gamma)A_V\rangle\right)=J(f_1,...,f_n)({}^tM_{(k_1,...,k_n)}(\gamma)A_V)$$

\begin{thm}[Relations de Manin non commutatives]
Pour toute famille de formes modulaires $f=(f_1,...,f_n)$ de niveau $1$, on dispose des relations:
$$J(f)J_S(f)=J_S(f)J(f)=1,$$
$$J_{U^2}(f)J_U(f)J(f)=J_U(f)J(f)J_{U^2}(f)=J(f)J_{U^2}(f)J_U(f)=1.$$
\end{thm}

\section{Polynômes des multipériodes des formes de niveau $1$}

Soient $k=(k_1,...,k_n)$ une famille d'entiers naturels pairs et $f_1,...,f_n$ des formes modulaires paraboliques de niveau $1$ et de poids respectifs $k_1,...,k_n$. Soit $N$ une entier positif. On pose à nouveau $V=\{(i,j);1\leq j\leq n,1\leq i\leq k_j-1\}$ et $A_V$ la famille des indéterminées.\par
On va regarder de manière itérée les termes homogènes de degré $N$ de $J(f_1,...,f_n)$:
$$J(f_1,...,f_n)_N=\int_{i\infty}^0\sum_{j=1}^n\sum_{i=1}^{k_j-1} f_j(z_1)z_1^{i-1}\d z_1A_{i,j}\int_{i\infty}^{z_1}...\int_{i\infty}^{z_{n-1}}\sum_{j=1}^n\sum_{i=1}^{k_j-1} f_j(z_N)z_N^{i-1}\d z_N A_{i,j}.$$
Nous allons les transformer en polynômes en plusieurs variables commutatives.\par
Par construction le degré $0$ est toujours égale à $1$, on regarde alors par degré croissant.

\subsection{Polynôme des périodes}

Soit $f$ une forme modulaire parabolique de niveau $1$ et de poids $k$ un entier pair. Le terme de degré $1$ de $J(f)$ est donné par:
\begin{equation}
J(f)_1=\sum_{m=1}^{k-1}\int_{i\infty}^0f(z)z^{m-1}\d z X_{m}.
\end{equation}
On reconnait, à proportionnalité près, les périodes de la forme modulaire.
En effet, ce sont les valeurs aux entiers critiques de sa fonction $L$:
\begin{equation}
L(f,m)=\sum_{n>0} a_n(f) n^{-m}=\frac{(2\pi)^m}{\Gamma(m)}\int_0^{\infty} f(it) t^{m-1}dt,
\end{equation}
pour tout entier $m$ tel que $1\leq m\leq k-1.$\par
Après renormalisation, nous les regroupons dans un polynôme de $V_k=\C_{k-2}[X]$ via la formule:
\begin{equation}
P_f(X)=\int_0^{\infty} f(it)(X-it)^{k-2}dt.
\end{equation}
On peut voir ce polynôme comme étant l'image de $J(f)_1$ par l'application $\Ser(\llbracket 1,k-1\rrbracket)_1\to V_k$ définie par:
$$X_m\mapsto \binom{k-2}{m-1}(-X)^{k-m-1},\text{ pour tout }1\leq m\leq k-1.$$

Lorsque nous avons plusieurs formes modulaires $f_j$ de poids $k_j$, on peut définir l'application $\Psi_1:\Ser(V)_1\to\prod_{j=1}^n V_{k_j}$ définit par:
$$A_{m,a}=X_mY_a\mapsto \left(\delta_j(a)\binom{k_j-2}{m-1}(-X)^{k_j-m-1}\right)_{1\leq j\leq n},$$
pour tout $1\leq a\leq n$ et $1\leq m\leq k_a-1$.\par
Elle nous donne que $J(f_1,...,f_n)_1$ est déterminé par les $n$-uplets de polynômes:
$$(P_{f_1}(X),...,P_{f_n}(X)).$$
Ce sont les polynômes des périodes des $f_j$ pour $1\leq j\leq n$.

\subsection{Polynôme des bipériodes}

Pour une famille de formes modulaires paraboliques de niveau $1$, $f_j$ de poids $k_j$ pour $1\leq j\leq n$. Le terme de degré $2$ de $J(f_1,...,f_n)$ est donné par:
\begin{equation}
J(f_1,...,f_n)_2=\sum_{(a,b)\in\llbracket 1,n\rrbracket^2}\sum_{m_1=1}^{k_a-1}\sum_{m_2=1}^{k_b-1}\int_{i\infty}^0\int_{i\infty}^{z_1}f_a(z_1)z_1^{m_1-1}\d z_1f_b(z_2)z_2^{m_2-1}\d z_2 X_{m_1}Y_aX_{m_2}Y_b.
\end{equation}
Ceci peut être vu comme la somme de $n^2$ polynômes des bipériodes selon les valeurs de $(a,b)\in\llbracket 1,n\rrbracket^2$.\par
En effet, on peut introduire une fonction $\Lambda$ à plusieurs variables complexes $(s_1,s_2)\in\C^2$ pour tout couple de formes modulaires $(f_a,f_b)$ via la définition:
\begin{equation}
\Lambda(f_a,f_b;s_1,s_2)=\int_{\infty}^0f_a(it_1)t_1^{s_1-1}\d t_1\int_{\infty}^{t_1}f_b(it_2)t_2^{s_2-1}\d t_2,
\end{equation}
La convergence de l'intégrale est assurée pour $\Re(s_1)>k_1$ et $\Re(s_2)>k_2$ par les mêmes théorèmes de comparaison que dans le cas d'une intégrale de Mellin d'une forme parabolique. On peut alors prolonger à $\C^2$ par méromorphie.\par
Et ainsi on peut considérer les bipériodes comme étant les valeurs aux entiers de la bande critique et de les regrouper dans un polynôme des bipériodes par:
\begin{equation}
P_{f_a,f_b}(X,Y)=\int_{0<t_1<t_2} f_a(it_1)f_b(it_2)(X-it_1)^{k_1-2}(Y-it_2)^{k_2-2}\d t_1\d t_2.
\end{equation}
\nomenclature{$V_{k_1,k_2}^A$}{Espaces des polynômes en deux variables sur un anneau $A$ de degré au plus $(k_1-2,k_2-2)$}
\nomenclature{$P_{f_1,f_2}$}{Polynôme des bipériodes des formes $f_1$ et $f_2$}
Ceci peut être vu comme étant l'image de $J(f_1,...,f_n)_2$ par l'application :
\begin{align*}
\Psi_2:\Ser(V)_2&\to\prod_{(a,b)\in\llbracket 1,n\rrbracket^2}\C_{(k_a-2,k_b-2)}[X,Y]\text{ définit par:}\\
A_{m_1,c}A_{m_2,d}&\mapsto \left(\delta_{(a,b)=(c,d)}\binom{k_a-2}{m_1}(-X)^{k_a-m_1-2}\binom{k_b-2}{m_2}(-Y)^{k_b-m_2-2}\right)_{(a,b)\in\llbracket 1,n\rrbracket^2},
\end{align*}
pour tout $c,d\in\llbracket 1,n\rrbracket$, $0\leq m_1\leq c$ et $0\leq m_2\leq d$.\\
On peut ainsi réécrire de manière plus compact cette application en utilisant l'isomorphisme : 
$$\C_{(k_a-2,k_b-2)}[X,Y]\cong\C_{k_a-2}[X]\otimes\C_{k_b-2}[Y].$$
On remarque alors que l'isomorphisme $\Psi_2:\Ser(V)_2\to \left(\prod_{j=1}^n V_k\right)^{\otimes 2}$ peut aussi être définit comme le carré tensoriel de l'application précédente $\Psi_1:\Ser(V)_1\to\prod_{j=1}^n V_k$:
\begin{equation}
\Psi_2(A_{i_1,j_1}A_{i_2,j_2})(X,Y)=\Psi_1(A_{i_1,j_1})(X)\Psi_1(A_{i_2,j_2})(Y),\text{ pour tout }(i_1,j_1),(i_2,j_2)\in V.
\end{equation}

\subsection{Polynôme des multipériodes}

Soient $k_1,...,k_N$ des entiers positifs pairs et soient $f_1,...,f_N$ des formes modulaires paraboliques de poids respectifs $k_1,...,k_N$ et de niveau $1$.\par

La définition de fonction $\Lambda$ comme transformée de Mellin s'étend en:
\begin{equation}
\Lambda(f_1,...,f_N;s_1,...,s_N)=\int_{0<t_1<...<t_N}f_1(it_1)t_1^{s_1-1}\d t_1...f_N(it_N)t_N^{s_N-1}\d t_N,
\end{equation}
pour des variables complexes $(s_1,...,s_N)\in\C^N$.\par
Cette définition est à nouveau valide pour $s_j>k_j,1\leq j\leq N$ puis on étend par méromorphie la fonction à $\C^N$.\par
Le terme de degré $N$ de $J(f_1,...,f_n)$ est donnée par :
\begin{equation}
J(f_1,...,f_n)_N=\sum_{(a_1,...,a_N)\in\llbracket 1,n\rrbracket^N}\sum_{m_i=0}^{k_{a_i}-2}\Lambda(f_{a_1},...,f_{a_N};m_1,...,m_N)i^{\sum m_i} X_{m_1}Y_{a_1}...X_{m_N}Y_{a_N}.
\end{equation}

Les valeurs obtenues aux points de coordonnées entières $1\leq m_i \leq k_i-1$ pour $1\leq i\leq N$ sont ainsi appelées \textit{multipériodes} des $N$ formes modulaires. On les regroupe dans le \textsl{polynôme des multipériodes}:
\begin{equation}
P_{f_1,...,f_N}(X_1,...,X_N)=\int_{0<t_1<...<t_N} f_1(it_1)...f_N(it_N)(X_1-it_1)^{k_1-2}...(X_N-it_N)^{k_N-2}\d t_1...\d t_N.
\end{equation}
\nomenclature{$V_{k_1,...,k_n}^A$}{Espaces des polynômes en $n$ variables sur un anneau $A$ de degré respectifs bornée par $k_1-2,...,k_n-2$}
\nomenclature{$P_{f_1,...,f_n}$}{Polynôme des multipériodes des formes $f_1,...,f_n$}

\begin{prop}
L'application induite par $\Psi_1:\Ser(V)_1\to\prod_{j=1}^n V_{k_j}$:
\begin{equation}
\Psi_N=\left(\Psi_1\right)^{\otimes N}:\Ser(V)_N\to\left(\prod_{j=1}^n V_{k_j}\right)^{\otimes N}
\end{equation}
est un isomorphisme de $\C$-espaces vectoriels. Et on a:
\begin{equation}
\Psi_N(J(f_1,...,f_n)_N)=(P_{f_{a_1},...,f_{a_N}})_{(a_1,...,a_N)\in\llbracket 1,n\rrbracket^N}.
\end{equation}
\end{prop}

\subsection{Ecriture uniformisée de ces polynômes}

\begin{prop}
On peut réécrire la définition des polynômes des multipériodes grâce à l'accouplement:
$$\Omega^n(\H^n,\C)\times M_n(\H^n,\pte^n,\Z)\to\C,$$
$$(\omega,C)\mapsto \left\langle \omega,C \right\rangle=\int_C \omega.$$
Soit $f\in S_k(SL_2(\Z))$. Posons $\omega_f(z,X)=f(z)(z-X)^{k-2}\d z\in\Omega^1(\H,\C)\otimes V_k$.\\
\nomenclature{$\omega_f(z,X)$}{Famille des $1$-formes différentielles des périodes de $f$}
Il existe un $n$-simplexe $\tau_n\in M_n(\H^n,\pte^n,\Z)$ tel que:
$$P_{f_1,...,f_n}(X_1,...,X_n)=\left\langle \omega_{f_1}(z_1,X_1)\wedge ...\wedge \omega_{f_n}(z_n,X_n), \tau_n\right\rangle.$$
\end{prop}

On précisera ensuite la définition de ce $n$-cycle. On peut toutefois préciser son support:
$$Supp(\tau_n)=\{(it_1,...,it_n)\text{ tel que }0<t_1...<t_n\}\subset\H^n.$$

Et par ailleurs on remarque la propriété liée à la modularité de $f\in S_k(SL_2(\Z))$. Pour tout $\gamma=\mat{a}{b}{c}{d}\in SL_2(\Z)$, on a:
\begin{multline}\label{modl1}
\gamma^*\omega_f(z,X)=f(\gamma z)(\gamma z-X)^{k-2}d(\gamma z)\\
=f(z)(z-\gamma^{-1}X)^{k-2}\d z(-cX+a)^{k-2}\d z=\omega_f(z,X|_{\gamma^{-1}}).
\end{multline}

\section{Relations élémentaires vérifiées par les multipériodes}

Soient $n\geq 1$ et $k_1,...,k_n$ des entiers. Considérons $f_1,...,f_n$ une famille de formes modulaires paraboliques de poids respectifs $k_1,...,k_n$ et de niveau $1$. Posons $V=\{(i,j);1\leq j\leq n,1\leq i \leq k_j-1\}$, $A_V$ une famille d'indéterminées indexée par $V$ et $V_{k_j}=\C_{k_j-2}[X]$.\par

Dans la partie précédente, on a construit un isomorphisme d'algèbres graduées:
$$\Psi:\Ser(V)=\bigoplus_{N\geq 0} \Ser(V)_N\to\bigoplus_{N\geq 0}\left(\prod_{j=1}^nV_{k_j}\right)^{\otimes N}.$$
Les projections de l'image de $J(f_1,...,f_n)$ sont par définition les polynômes des multipériodes.\par
Nous présentons ici les relations induites par celle vérifiée par $J(f_1,...,f_n)$.

\subsection{Relations de mélange}\label{secmel}

Manin a démontré dans un cadre plus général (voir \cite{Ma1}) que :
\begin{equation}\label{seriemelange}
\Delta(J(f_1,...,f_n))=J(f_1,...,f_n)\otimes J(f_1,...,f_n).
\end{equation}
Pour décrire les relations induites sur les polynômes des multipériodes on introduit la famille des \textit{permutées}:
\begin{equation}
P^{\sigma}_{f_1,...,f_n}(X_1,...,X_n)=P_{f_{\sigma(1)},...,f_{\sigma(n)}}(X_{\sigma(1)},...,X_{\sigma(n)})\text{ pour tout }\sigma\in\S_n.
\end{equation}
\nomenclature{$P^{\sigma}_{f_1,...,f_n}$}{Polynôme des permutés des multipériodes des formes $f_1,...,f_n$}
Définissons les sous-ensembles de permutations de $\mathfrak{S}_n$ de battage. Soit $a$ et $b$ des entiers positifs tel que $a+b=n$, on définit $\mathfrak{S}_{a,b}$ par:
$$\sigma\in\mathfrak{S}_{a,b}\subset\mathfrak{S}_{a+b}$$
$$\Longleftrightarrow\sigma(1)<...<\sigma(a)\text{ et }\sigma(a+1)<...<\sigma(a+b).$$ 
La relation (\ref{seriemelange}) se traduit alors en des relations de mélange.
\begin{prop}
Pour tout couple d'entiers naturels $(a,b)$ vérifiant $a+b=n$, on a:
\begin{equation}
\sum_{\sigma\in\mathfrak{S}_{a,b}}P^{\sigma}_{f_1,...,f_{n}}(X_{1},...,X_{n})
=P_{f_1,...,f_a}(X_1,...,X_a)P_{f_{a+1},...,f_{a+b}}(X_{a+1},...,X_{n}).
\end{equation}
De plus, notons $(n,...,1)\in\S_n$ la permutation totale alors:
\begin{equation}
P_{f_1,...,f_n}^{(n,...,1)}=P_{f_1,...,f_n}|_{(S,...,S)}.
\end{equation}
\end{prop}

\subsection{Relations de Manin}

On s'intéresse ici à l'action de $SL_2(\Z)$ sur $\bigoplus_{N\geq 0}\left(\prod_{j=1}^nV_{k_j}\right)^{\otimes N}$.\\
On sait munir $V_k$ de l'action $|_{2-k}$ de $SL_2(\Z)$-module. Elle induit bien une action diagonale sur $\bigoplus_{N\geq 0}\left(\prod_{j=1}^nV_{k_j}\right)^{\otimes N}$.\par
Et on a construit une structure de $SL_2(\Z)$-module sur $\Ser(V)$.\par
On dispose alors de la propriété suivante:

\begin{prop}
L'application $\Psi:\Ser(V)\to\bigoplus_{N\geq 0}\left(\prod_{j=1}^nV_{k_j}\right)^{\otimes N}$ est un morphisme de $SL_2(\Z)$-module.
\end{prop}

\begin{proof}
L'expression $\Psi_N=\Psi_1^{\otimes N}$ permet de réduire la démonstration à montrer que $\Psi_1:\C^V\to\prod_{j=1}^nV_{k_j}$ est un morphisme de $SL_2(\Z)$-module. En effet, les actions sont compatibles aux deux structures d'algèbres graduées et $\Psi$ est un morphisme d'algèbres graduées.\par
On sait que l'action est diagonale suivant les $1\leq j\leq n$ donc il suffit de le démontrer pour un seul poids $k$, c'est à dire l'application suivante est un morphisme de $SL_2(\Z)$-module:
$$\Psi:\quad X_m\mapsto \binom{k-2}{m-1}(-X)^{k-m-1},\text{ pour tout }1\leq m \leq k-1.$$
Or on peut calculer l'image par $\Psi$ de la forme $\Omega_f=\sum_{m=1}^{k-1}f(z)z^{m-1}X_{m}$:
$$\Psi(\Omega_f)=\sum_{m=1}^{k-1}f(z)\binom{k-2}{m}z^{m-1}(-X)^{k-1-m}=f(z)(X-z)^{k-2}=\omega_f(z,X).$$
Et (\ref{modl1}) montre que l'action d'un élément de $SL_2(\Z)$ sur la variable $X$ est duale dans les deux cas à celle sur $\H$:
$$\Psi(\Omega_f)(z)|_{\gamma}=\Psi(\Omega_f)(\gamma^{-1}z)=\omega_f(\gamma^{-1}z,X)=\omega_f(z,X|_{\gamma}).$$
\end{proof}

Ceci permet de déduire des relations de Manin généralisées sur les polynômes des multipériodes:

\begin{coro}
Soient $n$ un entier et $f_1,...,f_n$ une famille de formes modulaires paraboliques de niveau $1$. On dispose des relations:
$$\sum_{\substack{a,b\geq 0\\ a+b=n}}P_{f_1,...,f_a}|_{(S,...,S)}P_{f_{a+1},...,f_{a+b}}=0\text{ et }$$
$$\sum_{\substack{a,b,c\geq 0\\ a+b+c=n}}P_{f_1,...,f_a}|_{(U^2,...,U^2)}P_{f_{a+1},...,f_{a+b}}|_{(U,...,U)}P_{f_{a+b+1},...,f_{a+b+c}}=0.$$
où on pose par convention $P_{\emptyset}=1$ pour simplifier les notations.
\end{coro}

\begin{rem}
1) Dans le cas $n=1$, des polynômes des périodes, les relations de mélanges sont triviales: $P_{f_j}=P_{f_j}$ et on retrouve les relations de Manin classiques:
$$P_f|_{1+S}=P_f|_{1+U+U^2}=0.$$
2) Dès le cas $n=2$, des polynômes des bipériodes, les relations obtenues sont assez riches.\\
Les relations de mélanges sont:
\begin{align*}
P_{f_1,f_2}(X_1,X_2)+P_{f_2,f_1}(X_2,X_1)&=P_{f_1}(X_1)P_{f_2}(X_2),\\
P_{f_1,f_1}(X_1,X_2)+P_{f_1,f_1}(X_2,X_1)&=P_{f_1}(X_1)P_{f_1}(X_2),\\
\text{ et }P_{f_2,f_2}(X_1,X_2)+P_{f_2,f_2}(X_2,X_1)&=P_{f_2}(X_1)P_{f_2}(X_2).
\end{align*}
Et les relations de Manin s'écrivent:
$$P_{f_1,f_2}|_{(S,S)}+P_{f_1}|_SP_{f_2}+P_{f_1,f_2}=0,$$
$$P_{f_1,f_2}|_{(U^2,U^2)}+P_{f_1}|_{U^2}P_{f_2}|_U+P_{f_1,f_2}|_{(U,U)}+P_{f_1}|_{U}P_{f_2}+P_{f_1,f_2}=0.$$
Dans le chapitre suivant, on étudie plus en détail cette famille de relations.\\
3) On remarque qu'on peut déduire des relations du type:
$$P_{f,f}(X,X)=\frac{1}{2}P_f(X)^2.$$
Ceci se généralise pour tout entier $n\geq 1$ en :
\begin{equation}
P_{f,...,f}(X,...,X)=\frac{1}{n!}P_f(X)^n,
\end{equation}
et ainsi donnant la série génératrice en les indéterminées $X$ et $Z$:
$$\sum_{n\geq 0} P_{f,...,f}(X,...,X)Z^n=\exp\left(ZP_f(X)\right).$$
\end{rem}

\section{Généralisation aux formes non paraboliques de niveau $1$}

Soient $n\geq 1$ et $k_1,...,k_n\geq 2$ des entiers.
Soient $(f_1,...,f_n)\in \prod_{j=1}^n M_{k_j}(SL_2(\Z))$ une famille de formes modulaires holomorphes de niveau $1$. On dispose des \textit{formes modulaires épointées} définies par:
\begin{equation}
f_j^*(z)=f_j(z)-a_0(f_j)(1+z^{-k})\text{ pour }1\leq j\leq n.
\end{equation}
\nomenclature{$f^*$}{Forme modulaires épointées associée à une forme $f$ non parabolique}
On pose alors:
\begin{equation}
\Lambda(f_1,...,f_n;s_1,...,s_n)=\int_{0<t_1<...<t_n}f^*_1(it_1)t_1^{s_1-1}\d t_1...f_n^*(it_n)t_n^{s_n-1}\d t_n
\end{equation}
Cette transformée de Mellin converge pour les $(s_1,...,s_n)$ tels que $0<\Re(s_j)<k_j$ pour tout $1\leq j\leq n$.
De plus, elle admet un prolongement méromorphe à $\C^n$ qui vérifie à nouveau l'équation fonctionnelle suivante car $f_j^*|_S=f_j^*$ pour tout $1\leq j\leq n$:
\begin{equation}
\Lambda(f_1,...,f_n;s_1,...,s_n)=i^{\sum k_j}\Lambda(f_n,...,f_1;k_n-s_n,...,k_1-s_1),\text{ pour tout }(s_1,...,s_n)\in\C^n.
\end{equation}

\subsection{Calcul des pôles}

Dans le cas $n=1$, les pôles de $s\mapsto \Lambda(f;s)$ se situe en $0$ et $k$ dès lors que $a_0(f)\neq 0$. On dispose de la généralisation suivante:
\begin{prop}\label{propres}
Les pôles de $\Lambda(f_1,...,f_n;s_1,...,s_n)$ se situent dans les hyperplans de coordonnées $s_j\in\{0,k_j\}$ pour tout $1\leq j\leq n$ tel que $a_0(f_j)\neq 0$. De plus, les résidus simples non nuls sont données par:
\begin{equation}
\Res_{s_1= 0}\Lambda(f_1,...,f_n;s_1,...,s_n)=a_0(f_1)\Lambda(f_2,...,f_n;s_2,...,s_n),
\end{equation}
pour toute famille de paramètres $(s_2,...,s_n)\in\C^{n-1}$ vérifiant $s_2\neq 0$ et $s_n\neq k_n$.
\begin{equation}
\Res_{s_n= k_n}\Lambda(f_1,...,f_n;s_1,...,s_n)=\Lambda(f_1,...,f_n;s_1,...,s_{n-1})a_0(f_n),
\end{equation}
pour toute famille de paramètres $(s_1,...,s_{n-1})\in\C^{n-1}$ vérifiant $s_1\neq 0$ et $s_{n-1}\neq k_{n-1}$.\par
Et on a l'équivalent au voisinage de $(s_1,...,s_n)=(0,..,0,u_{a+1},...,u_b,k_{b+1},...,k_n)$ pour toute famille de paramètres complexes $(u_{a+1},...,u_b)$ vérifiant $u_{a+1}\neq 0$ et $u_{b}\neq k_{b}$:
\begin{equation}
\Lambda(f_1,...,f_n;s_1,...,s_n)\sim \frac{a_0(f_1)}{s_1}...\frac{a_0(f_a)}{s_a}\Lambda(f_{a+1},...,f_b;u_{a+1},...,u_b)\frac{a_0(f_{b+1})}{s_{b+1}-k_{b+1}}...\frac{a_0(f_{n})}{s_{n}-k_{n}}.
\end{equation}
\end{prop}

\begin{proof}
L'intégrale considérée dans la définition de $\Lambda$ ne converge pas normalement en $0$ et en $\infty$. Pour y remédier, on découpe le domaine d'intégration pour obtenir la formule:
\begin{equation}
\Lambda(f_1,...,f_n;s_1,...,s_n)=\sum_{j=0}^n\Lambda_1(f_j^*,...,f_1^*;k_j-s_j,...,k_1-s_1)\Lambda_1(f_{j+1}^*,...,f_n^*;s_{j+1},...,s_n),
\end{equation}
où on a noté pour toute famille $(g_1,...,g_m)$ de fonctions continues de $\H$ vers $\C$:
\begin{equation}
\Lambda_1(g_1,...,g_m;s_1,...,s_m)=\int_{1<t_1<...<t_m}g_1(it_1)t_1^{s_1-1}\d t_1...g_m(it_m)t_m^{s_m-1}\d t_m.
\end{equation}

On utilise alors le résultat intermédiaire suivant:\\

Soient $1\leq a\leq m$ des entiers. Soit $g_1,...,g_m$ des applications continues $\H\to\C$. Soit $(s_1,...,s_m)\in\C^m$. On suppose que $g_a(z)=z^{-k}$ alors l'application:
$$s_a\mapsto\Lambda_1(g_1,...,g_{a-1},z^{-k},g_{a+1},...,g_m;s_1,...,s_m),$$
admet au plus un pôle en $s_a=k$ et de plus le résidu est:
$$\Res_{s_a=k}\left(\Lambda_1(g_1,...,g_{a-1},z^{-k},g_{a+1},...,g_m;s_1,...,s_m)\right)=
\begin{cases}
\Lambda_1(g_1,...,g_{m-1},s_1,...,s_{m-1})&\text{ lorsque }a=m\\
0&\text{ sinon.}
\end{cases}$$

Pour cela, il suffit d'intégrer la forme différentielle $t_a^{s_a-k-1}\d t_a$ entre $t_{a-1}$ et $t_{a+1}$. Pour ne pas perdre en généralité, on note ici $t_0=1$ et $t_{m+1}=\infty$. On obtient:
$$\int_{t_{a-1}}^{t_{a+1}}t_a^{s_a-k-1}\d t_a=\frac{t_{a+1}^{s-k}-t_{a-1}^{s_a-k}}{s_a-k}.$$
Ainsi lorsque $1\leq a<m$ le résidu obtenu en $s_a=k$ est:
$$\Lambda_1(...,g_{a-1},g_{a+1},...;...,s_{a-1}+0,s_{a+1},...)-\Lambda_1(...,g_{a-1},g_{a+1},...;...,s_{a-1},s_{a+1}+0,...)=0,$$
et lorsque $m=a$ on trouve bien $\Lambda_1(g_1,...,g_{m-1},s_1,...,s_{m-1})-0$.\par

Ceci permet de déduire les pôles de la fonction $\Lambda$ ainsi que les résidus obtenus dans la proposition.
\end{proof}

\subsection{Polynôme des multipériodes étendu}

Zagier propose dans \cite{Za1} d'étendre la définition de polynôme des périodes aux formes non paraboliques. Nous considèrerons ainsi \textit{le 'polynôme' des multipériodes} de longueur $n$ d'une famille de formes modulaires non nécessairement paraboliques:
\begin{equation}
\frac{P_{f_1,...,f_n}(X_1,...,X_n)}{(k_1-2)!...(k_n-2)!}
=\sum_{\substack{0\leq m_j\leq k_j\\1\leq j\leq n}} \lim_{s\to 0}\frac{\Lambda(f_1,...,f_n,m_1+s,...,m_n+s)X_1^{k_1-m_1-1}/i^{m_1}...X_n^{k_n-m_n-1}/i^{m_n}}{\Gamma(m_1+s)\Gamma(k_1-m_1-s)...\Gamma(m_n+s)\Gamma(k_n-m_n-s)}.
\end{equation}

Cette définition étend bien celle donnée lorsque les formes $f_1,...,f_n$ sont paraboliques. C'est un élément de $\tilde{V}_k=\bigotimes_{j=1}^n\tilde{V}_{k_j}$ où $\tilde{V}_{k_j}=\frac{1}{X_j}\C_{k_j}[X_j]$. Cette espace n'est plus stable par l'action de $\Gamma$. Pourtant dans le cas $n=1$ Zagier \cite{Za1} montre que l'on peut vérifier les relations Manin:
$$P_f|_{1+S}=P_f|_{1+U+U^2}=0,\text{ pour tout }f\in M_{k}(SL_2(\Z)).$$\par
La proposition \ref{propres} permet d'écrire:
\begin{equation}\label{defnopara}
P_{f_1,...,f_n}(X_1,...,X_n)=\sum_{\substack{a,b,c\geq 0\\ a+b+c=n}}\frac{1}{a!c!} \prod_{j=1}^a \frac{a_0(f_j)}{(k_j-1)X_j}P_{f_{a+1}^*,...,f_{a+b}^*}(X_{a+1},...,X_{a+b})\prod_{j=a+b+1}^n a_0(f_j)\frac{X_j^{k_j-1}}{(k_j-1)}.
\end{equation}
Les expressions de cette décomposition de la forme $P_{f_{a+1}^*,...,f_{a+b}^*}(X_{a+1},...,X_{a+b})$ sont alors des polynômes. 

\begin{ex}
Dans le cas $n=1$, on retrouve la formule usuelle pour tout $f\in M_{k}(\G)$:
$$P_f(X)=\frac{a_0(f)}{(k-1)X}+P_{f^*}(X)+\frac{a_0(f)X^{k-1}}{(k-1)}.$$
Dans le cas $n=2$, nous obtenons pour $(f_1,f_2)\in M_{k_1}(\G)\times M_{k_2}(\G)$:
\begin{align*}
P_{f_1,f_2}(X_1,X_2)=\frac{a_0(f_1)}{2(k_1-1)X_1}\frac{a_0(f_2)}{(k_2-1)X_2}&+\frac{a_0(f_1)}{(k_1-1)X_1}P_{f_2^*}(X_2)&+\frac{a_0(f_1)}{(k_1-1)X_1}\frac{a_0(f_2)X_2^{k_2-1}}{(k_2-1)}\\
&+P_{f_1^*,f_2^*}(X_1,X_2)&+P_{f_1^*}(X_1)\frac{a_0(f_2)X_2^{k_2-1}}{(k_2-1)}\\
&&+\frac{a_0(f_1)X_1^{k_1-1}}{2(k_1-1)}\frac{a_0(f_2)X_2^{k_2-1}}{(k_2-1)}.
\end{align*}
On obtient notamment une formule simple, analogue à celle Rankin \ref{eqrank}, pour obtenir le polynôme des périodes à partir du polynôme des bipériodes. Pour tout $(f_1,f_2)\in M_{k_1}(\G)\times M_{k_2}(\G)$:
\begin{align*}
\lim_{x_1\to 0} x_1P_{f_1,f_2}(x_1,X_2)&=\frac{a_0(f_1)}{(k_1-1)}\left(P_{f_2}(X_2)-\frac{a_0(f_2)}{2(k_2-1)X_2}\right),\\
\text{et }\lim_{x_2\to \infty} x_2^{1-k_2}P_{f_1,f_2}(X_1,x_2)&=\frac{a_0(f_2)}{(k_2-1)}\left(P_{f_1}(X_1)-\frac{a_0(f_1)X_1^{k_1-1}}{2(k_1-1)}\right).
\end{align*}
Ainsi dès lors qu'une des deux formes n'est pas paraboliques nous avons un moyen simple de déduire le polynôme des périodes de la seconde. Dans le cas général, ceci nécessite d'effectuer une symétrisation et l'utilisation de formules de mélange.
\end{ex}

Considérons la famille des permutés comme définie dans la section \ref{secmel} par:
\begin{equation}
P_{f_1,...,f_n}^{\sigma}(X_1,...,X_n)=P_{f_{\sigma(1)},...,f_{\sigma(n)}}(X_{\sigma(1)},...,X_{\sigma(n)}),\text{ pour }\sigma\in\S_n.
\end{equation}
Elle est à valeur dans $\tilde{V}_k$. Et on peut vérifier les relations de mélange:

\begin{prop} Pour tout couple d'entiers $a,b$ tel que $a+b=n$, on a:
\begin{equation}
\sum_{\sigma\in\S_{a,b}}P^{\sigma}_{f_1,...,f_n}(X_1,...,X_n)=P_{f_1,...,f_a}(X_1,...,X_a)P_{f_{a+1},...,f_n}(X_{a+1},...,X_n).
\end{equation} 
De plus, considérons la permutation $(n,...,1)\in\S_n$ alors:
\begin{equation}
P_{f_1,...,f_n}|_{(S,...,S)}=P^{(n,...,1)}_{f_1,..,f_n}.
\end{equation}
\end{prop}

\begin{proof} Notons $a_0^*(f)=\frac{a_0(f)}{k-1}$ pour les formes $f\in M_k(SL_2(\Z))$. On a:
\begin{align*}
&P_{f_1,...,f_{\alpha_1}}(X_1,...,X_{\alpha_1})P_{f_{{\alpha_1}+1},...,f_{\alpha_1+\alpha_2}}(X_{{\alpha_1}+1},...,X_{\alpha_1+\alpha_2})\\
&=\sum_{a_1+b_1+c_1=\alpha_1}\frac{1}{a_1!c_1!} \prod_{j=1}^{a_1} \frac{a_0^*(f_j)}{X_j}P_{f_{a_1+1}^*,...,f_{a_1+b_1}^*}(X_{a_1+1},...,X_{a_1+b_1})\prod_{j=a_1+b_1+1}^n a_0^*(f_j)X_j^{k_j-1}\\
&\sum_{a_2+b_2+c_2=\alpha_2}\frac{1}{a_2!c_2!} \prod_{j=\alpha_1+1}^{\alpha_1+a_2} \frac{a_0^*(f_j)}{X_j}P_{f_{\alpha_1+a_2+1}^*,...,f_{\alpha_1+a_2+b_2}^*}(X_{\alpha_1+a_2+1},...,X_{\alpha_1+a_2+b_2})\prod_{j=\alpha_1+a_2+b_2+1}^n a_0^*(f_j)X_j^{k_j-1}\\
&=\sum_{\substack{a_1+b_1+c_1=\alpha_1\\ a_2+b_2+c_2=\alpha_2}} 
\sum_{\sigma_1\in\S_{a_1,a_2}}\frac{1}{(a_1+a_2)!}\prod_{j=1}^{a_1+a_2}\frac{a_0^*(f_{\sigma_1(j)})}{X_{\sigma_1(j)}}\\
&\sum_{\sigma_2\in\S_{b_1,b_2}}P^{\sigma_2}_{f_{a_1+1},...,f_{a_1+b_1},f_{\alpha_1+a_2+1},...,f_{\alpha_1+a_2+b_2}}
\sum_{\sigma_3\in\S_{c_1,c_2}}\frac{1}{(c_1+c_2)!}\prod_{j=1}^{c_1+c_2}\frac{a_0^*(f_{\sigma_3(j)})}{X_{\sigma_3(j)}}\\
&=\sum_{\sigma\in\S_{\alpha_1,\alpha_2}}P^{\sigma}_{f_1,...,f_n}(X_1,...,X_n).
\end{align*}
Pour la deuxième relation, on peut identifier certain partie stable. En effet, $f_j^*|_S=f_j^*$ donne:
$$P_{f_{a+1}^*,...,f_{a+b}^*}|_{(S,...,S)}=P_{f_{a+1}^*,...,f_{a+b}^*}^{(b,...,1)}.$$
Puis il reste à remarquer que $1/X_j$ et $X_{n-j+1}^{k_{n-j+1}-1}$ sont échangés par $S$ et $(n,...,1)$.\\
Ainsi chaque terme de la somme (\ref{defnopara}) vérifie la relation donc le polynôme étendu des multipériodes aussi.
\end{proof}

%
%

\section{Calculs de bipériodes des formes génératrices de Cohen}

Nous déterminons certaines bipériodes d'un couple de formes modulaires $(F_{m_1}^{k_1},F_{m_2}^{k_2})$ de niveau $1$ et de poids $(k_1,k_2)$ quelconques. 
Nous nous limitons ici au cas $n=2$ ainsi qu'au niveau $1$ pour obtenir un énoncé de complexité raisonnable. Notre méthode s'adapte pour $n>2$ et pour des formes modulaires de niveaux quelconques.

On rappelle brièvement la définition de la familles des formes génératrices de Cohen.
Soient $k_1,k_2\geq 4$ des entiers pairs et $m_1,n_1,m_2$ et $n_2$ des entiers des bandes critiques respectives. C'est à dire:
$$1\leq m_j\leq k_j-1\text{ et }1\leq n_j\leq k_j-1\text{ pour }j=1,2.$$
Nous noterons $\tilde{m_j}=k_j-m_j$ et $\tilde{n_j}=k_j-n_j$ pour $j=1,2$.\par
On rappelle l'existence et l'unicité d'application $F_m^k\in S_k(\G)$ pour tout entier $1\leq m \leq k-1$ représentant les périodes, c'est-à-dire vérifiant:
$$L(f;m)=\langle f, F_m^k \rangle,\text{ pour tout }f\in S_k(\G).$$
Notons $\Gamma=P\G$ et rappelons la notation pour $\pm\mat{a}{b}{c}{d}\in\Gamma$: 
$$F_m[\pm\mat{a}{b}{c}{d}](z)=(az+b)^{-m}(cz+d)^{-\tilde{m}}.$$
Elle permet d'écrire la formule de Cohen, voir \cite{Cohen81}, de la Proposition \ref{propcoh}:
$$F_m=C_m\sum_{\gamma\in\Gamma}F_m[\gamma]\text{ avec }C_m=-\frac{(2i)^{k-2}i^m}{\pi\binom{k-2}{m-1}}.$$

\subsection{Enoncé et démonstration}

Le théorème ci-dessous est inspiré par une formule de Kohnen et Zagier \cite{KZ} du calcul de certaines périodes de ces formes génératrices (voir chapitre $1$).

\begin{thm}\label{calculprincipal}
Supposons que les entiers $m_1$, $n_1$, $m_2$ et $n_2$ vérifient les propriétés suivantes:
\begin{itemize}
\item Les entiers $m_1$ et $n_1$ sont de parités différentes.
\item Les entiers $m_2$ et $n_2$ sont de parités différentes.
\item Soit $n_1<m_1$, $n_1<\tilde{m_1}$, $n_2<m_2$ et $n_2<\tilde{m_2}$,\\
soit $m_1<n_1$, $\tilde{m_1}<n_1$, $m_2<n_2$ et $\tilde{m_2}<n_2$.
\end{itemize}
Alors on a:
\begin{align*}
\Lambda(F_{m_1}^{k_1},&F_{m_2}^{k_2};n_1,n_2)=\frac{i^{k_1+k_2}2^{k_1+k_2-2}}{(k_1-2)!(k_2-2)!}(Q(m_1,n_1,m_2,n_2)+i^{k_1}Q(\tilde{m_1},n_1,m_2,n_2)\\
&+i^{k_2}Q(m_1,n_1,\tilde{m_2},n_2)+i^{k_1+k_2}Q(\tilde{m_1},n_1,\tilde{m_2},n_2)+i^{k_1+k_2}Q(m_2,\tilde{n_2},m_1,\tilde{n_1})\\
&+i^{k_1}Q(\tilde{m_2},\tilde{n_2},m_1,\tilde{n_1})+i^{k_2}Q(m_2,\tilde{n_2},\tilde{m_1},\tilde{n_1})+Q(\tilde{m_2},\tilde{n_2},\tilde{m_1},\tilde{n_1})+R_2),
\end{align*}
où: $$\frac{Q(m_1,n_1,m_2,n_2)}{(\tilde{m_1}-1)!(n_1-1)!(\tilde{m_2}-1)!(n_2-1)!}$$
$$=\begin{cases}
\sum_{a=0}^{n_1-1}\sum_{b=0}^{m_2-1}(-1)^b\binom{m_2-1}{b}\binom{n_2-1+a}{a}\frac{\zeta(n_2-m_2+1+a+b,n_1-m_1+1-a-b)}{(2\pi)^{n_1+n_2-m_1-m_2+2}} &\text{si } n_1+n_2>m_1+m_2\\
C_{m_1}C_{m_2}J(m_2,m_1)&\text{si } n_1+n_2=m_1+m_2-2\\
 0&\text{sinon,}
 \end{cases}$$
avec :
$$\zeta(A,B)=\sum_{0<l_2<l_1} l_1^{-A}l_2^{-B},$$
$$J(\alpha,\beta)=\int_0^1 \d u\int_{-\infty}^{\infty}(1+iux)^{-\alpha}\d x\int_0^u\d v \int_{-\infty}^{\infty}(1+ivy)^{-\beta}\d y\in\frac{\pi^2}{2(\alpha-1)(\beta-1)}+\Q,$$
et enfin le nombre rationnel:
\begin{align*}
R_2&=1/2(\delta_{(n_1+n_2=2)}+i^{k_1}i^{k_2}\delta_{(\tilde{n_1}+\tilde{n_2}=2)})\frac{i^{m_1}\zeta(m_1)\zeta(\tilde{m_1})}{(k_1-1)\zeta(k_1)}\frac{i^{m_2}\zeta(m_2)\zeta(\tilde{m_2})}{(k_2-1)\zeta(k_2)}\\
&+L(F_{m_1},n_1)\delta_{(n_2=1)}\frac{i^{m_2}\zeta(m_2)\zeta(\tilde{m_2})}{(k_2-1)\zeta(k_2)}\\
&+L(F_{m_2},n_2)\delta_{(\tilde{n_1}=1)}\frac{i^{\tilde{m_1}}\zeta(m_1)\zeta(\tilde{m_1})}{(k_1-1)\zeta(k_1)}\\
&+\delta_{(\tilde{n_1}=1)}\frac{i^{\tilde{m_1}}\zeta(m_1)\zeta(\tilde{m_1})}{(k_1-1)\zeta(k_1)}\delta_{(n_2=1)}\frac{\zeta(m_2)\zeta(\tilde{m_2})}{(k_2-1)\zeta(k_2)}.
\end{align*}
\end{thm}

\begin{proof}
Notre démonstration s'inspire de celle de Kohnen et Zagier \cite{KZ}. Les hypothèses de parités et d'inégalités interviennent de manière essentielle dans le Lemme \ref{lemess}.
Commençons par remarquer qu'au vu des hypothèses et du résultat attendu, nous pouvons nous limiter au cas:
$$1\leq n_1<m_1\leq k_1/2\text{ et }1\leq n_2<m_2\leq k_2/2.$$
Pour cela, il suffit d'utiliser la formule:
$\Lambda(f,g;n_1,n_2)=i^{k_1}i^{k_2}\Lambda(g,f;\tilde{n_2},\tilde{n_1}),$
puis le fait que: $F_m^k=(-1)^mF_{\tilde{m}}^k$.\par

Commençons par décomposer $F_m$ suivant :
\begin{align*}
F_m(z)/C_m&=\sum_{bd\neq 0}F_m[\mat{a}{b}{c}{d}](z)+\sum_{N\in\Z} F_m[\mat{1}{0}{N}{1}](z)+\sum_{N\in\Z}F_m[\mat{N}{1}{-1}{0}](z)\\
&=\sum_{bd\neq 0}F_m[\mat{a}{b}{c}{d}](z)+F_{\tilde{m}}^*(z)+(-1)^mF_m^*(z),
\end{align*}
où nous définissons $F_m^*(z)=\sum_{N\in\Z}(Nz+1)^{-m}z^{-\tilde{m}}$.
Nous disposons des équivalents suivants en $0$ pour une variable réelle $u>0$. D'une part, on a:
$$F_m[\mat{a}{b}{c}{d}](iu)\longrightarrow_{u\to 0} b^{-m}d^{-\tilde{m}},\text{ pour }bd\neq 0.$$
D'autre part, la formule d'Hurwitz, (voir par exemple Lang \cite{Lang}):
$$\sum_{N\in\mathbb{Z}} (z+N)^{-m}=\frac{(-2i\pi)^m}{(m-1)!}\sum_{\beta>0} \beta^{m-1} \exp(2i\pi\beta z),$$
permet d'obtenir :
$$F_m^*(iu)\sim_{u\to 0}\frac{(-2\pi)^m}{(m-1)!}u^{-k}\exp(-2\pi/u).$$

Nous cherchons à calculer:
$$I=\int_0^{\infty} \int_0^{u_1}\sum_{\gamma_1\in\Gamma} \sum_{\gamma_2\in\Gamma} F_{m_1}[\gamma_1](iu_1)u_1^{n_1-1}F_{m_2}[\gamma_2](iu_2)u_2^{n_2-1}\d u_1\d u_2.$$
Soient $\gamma_1=\mat{a_1}{b_1}{c_1}{d_1}$ et $\gamma_2=\mat{a_2}{b_2}{c_2}{d_2}$ dans $\Gamma$ vérifiant $b_1d_1b_2d_2\neq 0$. Posons:
\begin{align*}
I(\gamma_1,\gamma_2)&=\int_0^{\infty}\int_0^{u_1} F_{m_1}[\gamma_1](iu_1)u_1^{n_1-1}F_{m_2}[\gamma_2](iu_2)u_2^{n_2-1}\d u_1\d u_2,\\
I_{m_1}(\gamma_2)&=\int_0^{\infty}\int_0^{u_1} F_{m_1}^*(iu_1)u_1^{n_1-1}F_{m_2}[\gamma_2](iu_2)u_2^{n_2-1}\d u_1\d u_2,\\
\text{et }I_{m_2}'(\gamma_1)&=\int_0^{\infty}\int_0^{u_1} F_{m_1}[\gamma_1](iu_1)u_1^{n_1-1}F_{m_2}^*(iu_2)u_2^{n_2-1}\d u_1\d u_2.
\end{align*}
Notons $\gamma_j^-=\mat{\phantom{-}a_j}{-b_j}{-c_j}{\phantom{-}d_j}$ pour $j=1,2$.

\begin{lem}\label{lemess}
Supposons que les entiers $n_1, m_1, n_2$ et $m_2$ vérifient les hypothèses suivantes:
\begin{itemize}
\item Les entiers $m_1+n_1$ et $m_2+n_2$ sont impairs,
\item les inégalités $1\leq n_1<m_1\leq \tilde{m_1}\leq k_1-1$,
\item et $1\leq n_2< m_2\leq \tilde{m_2} \leq k_2-1$. 
\end{itemize}
Alors, on a les annulations: 
$$I(\gamma_1,\gamma_2)+I(\gamma_1,\gamma_2^{-})=0,$$
$$I_{m_1}(\gamma_2)+I_{m_1}(\gamma_2^-)=0\text{ et }I_{m_2}'(\gamma_1)+I_{m_2}'(\gamma_1^-)=0.$$
\end{lem}

\begin{proof}
On dispose des relations $F_m[\gamma^-](z)=(-1)^mF_m[\gamma](-z)$ et $F_m^*(z)=(-1)^mF_m^*(-z)$. Ainsi on peut écrire d'une part:
\begin{align*}
I(\gamma_1,\gamma_2)+I(\gamma_1,\gamma_2^{-})&=\big(\int_0^{\infty}\int_0^1+(-1)^{m_2+n_2-1}\int_0^{\infty}\int_{-1}^0\big)F_{m_1}[\gamma_1](iu)u^{n_1+n_2-1}F_{m_2}[\gamma_2](iut)t^{n_2-1}\d u\d t\\
&=\int_{0}^{\infty}\int_{-1}^1F_{m_1}[\gamma_1](iu)F_{m_2}[\gamma_2](iut)u^{n_1+n_2-1}t^{n_2-1}\d u\d t.
\end{align*}

Soit $\varepsilon >0$ un paramètre réel, nous découpons l'intégrale précédente suivant la variable $u$:
$$I(\gamma_1,\gamma_2)+I(\gamma_1,\gamma_2^{-})=\left(\int_{0}^{\varepsilon}+\int_{\varepsilon}^{\infty}\right)\int_{-1}^1 F_{m_1}[\gamma_1](iu)F_{m_2}[\gamma_2](iut)u^{n_1+n_2-1}t^{n_2-1}\d u\d t=T_1^{\varepsilon}+T_2^{\varepsilon},$$
où on a, d'une part:
\begin{align*}
T_1^{\varepsilon}&=\int_0^{\varepsilon}\int_{-1}^1 F_{m_1}[\gamma_1](iu)F_{m_2}[\gamma_2](iut)u^{n_1+n_2-1}t^{n_2-1}\d u\d t\\
&=\int_0^1\int_{-1}^1 \varepsilon^{n_1+n_2}F_{m_1}[\gamma_1](i\varepsilon u)F_{m_2}[\gamma_2](i\varepsilon ut)u^{n_1+n_2-1}t^{n_2-1}\d u\d t\\
&\sim_{\varepsilon\to 0}\int_0^1\int_{-1}^1 \varepsilon^{n_1+n_2}b_1^{-m_1}d_1^{-\tilde{m_1}}b_2^{-m_2}d_2^{-\tilde{m_2}}u^{n_1+n_2-1}t^{n_2-1}\d u\d t\\
&\stackrel{\varepsilon\to 0}{\longrightarrow} 0\text{ car }n_1+n_2>0\text{ et }n_2>0,
\end{align*}

et d'autre part, après changement de variables:
\begin{align*}
T_2^{\varepsilon}&=\int_{\varepsilon}^{\infty}\int_{-1}^1 F_{m_1}[\gamma_1](iu)F_{m_2}[\gamma_2](iut)u^{n_1+n_2-1}t^{n_2-1}\d u\d t\\
&=\int_1^{\infty}\int_{-1/\varepsilon}^{1/\varepsilon} \varepsilon^{n_1}F_{m_1}[\gamma_1](i\varepsilon u)F_{m_2}[\gamma_2](iut)u^{n_1+n_2-1}t^{n_2-1}\d u\d t.
\end{align*}
L'hypothèse $b_2d_2\neq 0$ permet d'observer que pour tout réel $u>1$ les pôles de $t\mapsto F_{m_2}[\gamma_2](iut)$ sont $ib_2/ua_2$ et $id_2/uc_2$. Ils sont dans la même composante connexe de $\C-\R$ car $a_2d_2-b_2c_2=1$ impose $a_2b_2c_2d_2\geq 0$. Ainsi le théorème des résidus permet de substituer l'intégrale sur le segment $[-1/\varepsilon,1/\varepsilon]$ par l'intégrale le long d'un arc de cercle de rayon $1/\varepsilon$:
$$T_2^{\varepsilon}=\int_1^{\infty}\int_{0}^{\pi} \varepsilon^{n_1-n_2}F_{m_1}[\gamma_1](i\varepsilon u)F_{m_2}[\gamma_2](iue^{i\theta}/\varepsilon)u^{n_1+n_2-1}ie^{n_2 i\theta}\d u\d \theta.$$
Lorsque $\varepsilon\to 0$, nous obtenons:
$$T_2^{\varepsilon}\sim_{\varepsilon\to 0} \int_1^{\infty}\int_{0}^{\pi} \varepsilon^{n_1-n_2}b_1^{-m_1}d_1^{-\tilde{m_1}}C_2(iue^{i\theta}/\varepsilon)^{-\alpha_2}u^{n_1+n_2-1}ie^{n_2 i\theta}\d u\d \theta,$$
où $\alpha_2\in\{k_2,m_2,\tilde{m_2}\}$ suivant l'annulation de $a_2$ et $c_2$ et où $C_2$ est une constante non nulle.
On a bien convergence de l'intégrale suivant les deux variables car $n_1+n_2<m_2\leq\tilde{m_2}<k_2$ et l'exposant de $\varepsilon$ est positif car $m_2-n_2+n_1>0$ donc la limite pour $\varepsilon\to 0$ est $0$. Ceci montre bien que: $I(\gamma_1,\gamma_2)+I(\gamma_1,\gamma_2^{-})=0$.

On calcule de façon analogue :
\begin{align*}
I_{m_1}(\gamma_2)+I_{m_1}(\gamma_2^{-})&=\big(\int_0^{\infty}\int_0^1+(-1)^{m_2+n_2-1}\int_0^{\infty}\int_{-1}^0\big)F_{m_1}^*(iu)u^{n_1+n_2-1}F_{m_2}[\gamma_2](iut)t^{n_2-1}\d u\d t\\
&=\big(\int_{0}^{\varepsilon}+\int_{\varepsilon}^{\infty}\big)\int_{-1}^1F_{m_1}^*(iu)F_{m_2}[\gamma_2](iut)u^{n_1+n_2-1}t^{n_2-1}\d u\d t=T_3^{\varepsilon}+T_4^{\varepsilon},
\end{align*}
pour un paramètre réel $\varepsilon >0$. D'une part, on calcule
\begin{align*}
T_3^{\varepsilon}&=\int_0^1\int_{-1}^1\varepsilon^{n_1+n_2}F_{m_1}^*(i\varepsilon u)F_{m_2}[\gamma_2](i\varepsilon ut)u^{n_1+n_2-1}t^{n_2-1}\d u\d t\\
&\sim_{\varepsilon\to 0}\varepsilon^{n_1+n_2}\int_0^1\int_{-1}^1\frac{(-2\pi)^{m_1}}{(m_1-1)!}(u\varepsilon)^{-k_1}\exp(-2\pi/u\varepsilon)b_2^{-m_2}d_2^{-\tilde{m_2}}u^{n_1+n_2-1}t^{n_2-1}\d u\d t\\
&\stackrel{\varepsilon\to 0}{\longrightarrow} 0\text{ sans condition car }\exp(-2\pi/u\varepsilon)<\exp(-2\pi/\varepsilon)\text{ pour }0<u<1.
\end{align*}
Et d'autre part, on a:
\begin{align*}
T_4^{\varepsilon}&=\int_{\varepsilon}^{\infty}\int_{-1}^1 F_{m_1}^*(iu)F_{m_2}[\gamma_2](iut)u^{n_1+n_2-1}t^{n_2-1}\d u\d t\\
&=\int_1^{\infty}\int_{-1/\varepsilon}^{1/\varepsilon} \varepsilon^{n_1}F_{m_1}^*(i\varepsilon u)F_{m_2}[\gamma_2](iut)u^{n_1+n_2-1}t^{n_2-1}\d u\d t\\
&=\int_1^{\infty}\int_{0}^{\pi} \varepsilon^{n_1-n_2}F_{m_1}^*(i\varepsilon u)F_{m_2}[\gamma_2](iue^{i\theta}/\varepsilon)u^{n_1+n_2-1}ie^{n_2 i\theta}\d u\d \theta\\
&\sim_{\varepsilon\to 0} \int_1^{\infty}\int_{0}^{\pi} \varepsilon^{n_1-n_2}\frac{(-2\pi)^{m_1}}{(m_1-1)!}(u\varepsilon)^{-k_1}\exp(-2\pi/u\varepsilon)C_2(iue^{i\theta}/\varepsilon)^{-\alpha_2}u^{n_1+n_2-1}ie^{n_2 i\theta}\d u\d \theta,
\end{align*}
où on a, de même, $\alpha_2\in\{k_2,m_2,\tilde{m_2}\}$ et $C_2>0$ une constante et tend vers $0$ car on a supposé: $n_1+n_2-k_1-m_2<0.$

Enfin la dernière annulation peut se déduire en échangeant les couples $(m_1,n_1)$ et $(m_2,n_2)$ dans ce qui précède et en notant que lorsque $m_1$ et $n_1$ sont de parités opposées et vérifiant $1\leq n_1<m_1\leq\tilde{m_1}\leq k_1-1$, Kohnen et Zagier montrent dans \cite{KZ} de manière analogue que:
$$\int_0^{\infty}F_{m_1}[\gamma_1](iu)u^{n_1-1}du+\int_0^{\infty}F_{m_1}[\gamma_1^-](iu)u^{n_1-1}du=0.$$
\end{proof}

Pour démontrer le théorème nous allons comparer $I$ à $\mathcal{S}=(-1)^{m_1+m_2}S(m_1,m_2)+(-1)^{m_1}S(m_1,\tilde{m_2})+(-1)^{m_2}S(\tilde{m_1},m_2)+S(\tilde{m_1},\tilde{m_2})$ où:
$$S(m_1,m_2)=\int_0^{\infty} \int_0^{u_1} F_{m_1}^*(iu_1)u_1^{n_1-1}F_{m_2}^*(iu_2)u_2^{n_2-1}\d u_1\d u_2.$$
En effet, on a d'une part:
\begin{multline*}
I=\int_0^{\infty}\int_0^{u_1} \left(\sum_{b_1d_1\neq 0} F_{m_1}[\mat{a_1}{b_1}{c_1}{d_1}](iu_1)+(-1)^{m_1}F_{m_1}^*(iu_1)+F_{\tilde{m_1}}^*(iu_1)\right)u_1^{n_1-1}\\
\left(\sum_{b_2d_2\neq 0} F_{m_2}[\mat{a_2}{b_2}{c_2}{d_2}](iu_2)+(-1)^{m_2}F_{m_2}^*(iu_2)+F_{\tilde{m_2}}^*(iu_2)\right)u_2^{n_2-1}\d u_1\d u_2.
\end{multline*}
Et d'autre part, d'après le Lemme \ref{lemess}:
\begin{multline*}
\mathcal{S}=\sum_{b_1d_1b_2d_2\neq 0} I(\mat{a_1}{b_1}{c_1}{d_1},\mat{a_2}{b_2}{c_2}{d_2})+\sum_{b_2d_2\neq 0}\left((-1)^{m_1}I_{m_1}(\mat{a_2}{b_2}{c_2}{d_2})+I_{\tilde{m_1}}(\mat{a_2}{b_2}{c_2}{d_2})\right)+\sum_{b_1d_1\neq 0}\left((-1)^{m_2}I_{m_2}'(\mat{a_1}{b_1}{c_1}{d_1})+I_{\tilde{m_2}}'(\mat{a_1}{b_1}{c_1}{d_1})\right)\\
+(-1)^{m_1+m_2}S(m_1,m_2)+(-1)^{m_1}S(m_1,\tilde{m_2})+(-1)^{m_2}S(\tilde{m_1},m_2)+S(\tilde{m_1},\tilde{m_2}).
\end{multline*}

Il s'agit ainsi d'échanger l'intégration et la sommation dans $I$. Soit un paramètre réel $\varepsilon>0$, posons:
$$I_{\varepsilon}= \int_{\varepsilon}^{1/\varepsilon} \int_{\varepsilon}^{u_1}\sum_{\gamma_1\in\Gamma} \sum_{\gamma_2\in\Gamma} F_{m_1}[\gamma_1](iu_1)u_1^{n_1-1}F_{m_2}[\gamma_2](iu_2)u_2^{n_2-1}\d u_1\d u_2.$$
Nous pouvons représenter ce domaine d'intégration sur la figure suivante:\\ \\
\begin{tabular}{ll}
\multirow{10}*{
\setlength{\unitlength}{0.8cm}
\begin{picture}(8,6)(-1,-1)
\put(0,-1){\vector(0,1){6}}
\put(-1,0){\vector(1,0){6.5}}
\put(0,0){\line(1,1){5}}
\put(1,1){\line(1,0){4}}
\put(1,1){\line(0,-1){1}}
\put(4,4){\line(0,-1){4}}
\put(4,4){\line(1,0){1}}
\put(-0.5,-0.5){0}
\put(-0.5,1){$\varepsilon$}
\put(-0.8,4){$1/\varepsilon$}
\put(-0.7,4.8){$u_2$}
\put(1,-0.5){$\varepsilon$}
\put(4,-0.5){$1/\varepsilon$}
\put(5,-0.5){$u_1$}
\put(0.3,0.2){$D_1^{\varepsilon}$}
\put(2.2,0.2){$D_2^{\varepsilon}$}
\put(4.5,0.2){$D_3^{\varepsilon}$}
\put(4.5,2.2){$D_4^{\varepsilon}$}
\put(4.5,4.2){$D_5^{\varepsilon}$}
\put(1.5,1.5){\line(1,-1){0.5}}
\put(2,2){\line(1,-1){1}}
\put(2.5,2.5){\line(1,-1){1.5}}
\put(3,3){\line(1,-1){1}}
\put(3.5,3.5){\line(1,-1){0.5}}
\end{picture}
} & \\
& où nous notons les domaines:\\
& $D_1^{\varepsilon}=\{0<u_2<u_1<\varepsilon\},$\\
& $D_2^{\varepsilon}=\{\varepsilon<u_1<1/\varepsilon;0<u_2<\varepsilon\},$\\
& $D_3^{\varepsilon}=\{1/\varepsilon<u_1;0<u_2<\varepsilon\},$\\
& $D_4^{\varepsilon}=\{1/\varepsilon<u_1;\varepsilon<u_2<1/\varepsilon\},$\\
& et $D_5^{\varepsilon}=\{1/\varepsilon<u_2<u_1\}.$\\
& \\
& \\
& \\
\end{tabular}

La convergence normale nous permet d'intervertir dans ce cas:
$$I_{\varepsilon}= \sum_{\gamma_1\in\Gamma} \sum_{\gamma_2\in\Gamma}\int_{\varepsilon}^{1/\varepsilon} \int_{\varepsilon}^{u_1} F_{m_1}[\gamma_1](iu_1)u_1^{n_1-1}F_{m_2}[\gamma_2](iu_2)u_2^{n_2-1}\d u_1\d u_2.$$
Le terme d'erreur $I-\mathcal{S}$ est donc la limite pour $\varepsilon$ tendant vers $0$ de:
$$I-I_{\varepsilon}=\int\int_{D_1^{\varepsilon}\cup D_2^{\varepsilon}\cup D_3^{\varepsilon}\cup D_4^{\varepsilon}\cup D_5^{\varepsilon}}\sum_{\gamma_1\in\Gamma} \sum_{\gamma_2\in\Gamma} F_{m_1}[\gamma_1](iu_1)u_1^{n_1-1}F_{m_2}[\gamma_2](iu_2)u_2^{n_2-1}\d u_1\d u_2.$$
Notons pour $j=1,...5$, ces limites:
$$S_j=\lim_{\varepsilon\to 0}S_j^{\varepsilon}\text{ avec }S_j^{\varepsilon}=\int\int_{D_j^{\varepsilon}}\sum_{\gamma_1\in\Gamma} \sum_{\gamma_2\in\Gamma} F_{m_1}[\gamma_1](iu_1)u_1^{n_1-1}F_{m_2}[\gamma_2](iu_2)u_2^{n_2-1}\d u_1\d u_2.$$
Les intégrales $S_2^{\varepsilon}$, $S_3^{\varepsilon}$ et $S_4^{\varepsilon}$ peuvent se calculer en scindant suivant les deux variables d'intégration. Le résultat obtenu est ainsi une conséquence directe du calcul effectué dans le cas du calcul des périodes d'une forme:
$$\lim_{\varepsilon\to 0}\left(\int_0^{\varepsilon}\sum_{bd\neq 0}F_m[\mat{a}{b}{c}{d}](iu)u^{n-1}du\right)=\frac{\zeta(m)\zeta(k-m)}{\zeta(k)}\frac{\delta_{(n=1)}\pi}{k-1},$$
$$\text{et }\lim_{\varepsilon\to 0}\left(\int_0^{\varepsilon}F_m^*(iu)u^{n-1}du\right)=\frac{\delta_{(n-\tilde{m}=1)}\pi}{m-1}.$$
Ceci permet d'obtenir après multiplication par $C_{m_1}C_{m_2}$, on obtient le nombre rationnel:
$$C_{m_1}C_{m_2}(S_2+S_3+S_4)=R_2.$$\par

Les intégrales sur $D_1$ et $D_5$  se calculent toutes les deux de la même manière après changement de variables $(u_1,u_2)\to(1/u_2,1/u_1)$ et utilisation de la modularité. En effet:
\begin{align*}
S_5^{\varepsilon}&=\int_{1/\varepsilon}^{\infty}\int_{1/\varepsilon}^{u_1}F_{m_1}(iu_1)u_1^{n_1-1}F_{m_2}(iu_2)u_2^{n_2-1}\d u_1\d u_2\\
&=i^{k_1+k_2}\int_{0}^{\varepsilon}\int_{u_1}^{\varepsilon}F_{m_1}(iu_1)u_1^{\tilde{n_1}-1}F_{m_2}(iu_2)u_2^{\tilde{n_2}-1}\d u_1\d u_2\\
&=i^{k_1+k_2}\int_{0}^{\varepsilon}\int_{0}^{u_2}F_{m_2}(iu_2)u_2^{\tilde{n_2}-1}F_{m_1}(iu_1)u_1^{\tilde{n_1}-1}\d u_2\d u_1.
\end{align*}\par
Le calcul de l'intégrale $S_1^{\varepsilon}$ se fait par une disjonction suivant l'annulation des produits $b_1d_1$ et $b_2d_2$.
En effet, on dispose du résultat général pour des entiers $m$ et $n$ du segment $\llbracket 1,k-1\rrbracket$ et tout réel $u>0$:
\begin{align*}
u^{n-1}\varepsilon^{n}\sum_{bd=0} F_m[\gamma](iu\varepsilon)&=u^{n-1}\varepsilon^n\sum_{N\in\Z}(Niu\varepsilon+1)^{-m}(-iu\varepsilon)^{-\tilde{m}}+\sum_{N\in\Z}(iu\varepsilon)^{-m}(Niu\varepsilon+1)^{-\tilde{m}}\\
&\stackrel{\varepsilon\to 0}{\longrightarrow} i^{-m}\int_{-\infty}^{\infty}\delta_{(n-\tilde{m}=1)}i^k(1+iux)^{-m}+\delta_{(n-m=1)}(1+iux)^{-\tilde{m}}\d x.
\end{align*}
La limite provenant de l'expression en somme de Riemann de l'intégrale. Et d'autre part:
\begin{align*}
\varepsilon^{n}\sum_{bd\neq 0} F_m[\gamma](iu\varepsilon)
&=\varepsilon^n\sum_{b\wedge d=1}\sum_{N\in\mathbb{Z}}((a+Nb)iu\varepsilon+b)^{-m}((c+Nd)iu\varepsilon+d)^{-\tilde{m}}\\
&=\varepsilon^{n}\sum_{b\wedge d=1}b^{-m}d^{-\tilde{m}}\sum_{N\in\mathbb{Z}}((a/b+N)iu\varepsilon+1)^{-m}((c/d+N)iu\varepsilon+1)^{-\tilde{m}}\\
&\stackrel{\varepsilon\to 0}{\longrightarrow} \delta_{(n=1)}\frac{\zeta(m)\zeta(k-m)}{\zeta(k)}\int_{-\infty}^{\infty}(1+iux)^{-k}\d x.
\end{align*}
Ici on a commencé par décomposer l'espace $P\G$ suivant ses classes à droite modulo $\Gamma_{\infty}=\{\mat{1}{N}{0}{1};N\in\Z\}$. Puis on a à nouveau mis en évidence une somme de Riemann.
Ceci nous donne ainsi les neuf termes suivants. On a:
\begin{align*}
S_1&=\delta_{(n_1-\tilde{m_1}=1)}\delta_{(n_2-\tilde{m_2}=1)}i^{\tilde{m_1}+\tilde{m_2}}J(m_1,m_2)
+\delta_{(n_1-\tilde{m_1}=1)}\delta_{(n_2-m_2=1)}i^{\tilde{m_1}-m_2}J(m_1,\tilde{m_2})\\
&+\delta_{(n_1-m_1=1)}\delta_{(n_2-\tilde{m_2}=1)}i^{-m_1+\tilde{m_2}}J(\tilde{m_1},m_2)
+\delta_{(n_1-m_1=1)}\delta_{(n_2-m_2=1)}i^{-m_1-m_2}J(\tilde{m_1},\tilde{m_2})\\
&+\delta_{(n_1=1)}Z'(m_1)\left[\delta_{(n_2-\tilde{m_2}=1)}i^{\tilde{m_2}}J(k_1,m_2)+\delta_{(n_2-m_2=1)}i^{-m_2}J(k_1,\tilde{m_2})\right]\\
&+\delta_{(n_2=1)}Z'(m_2)\left[\delta_{(n_1-\tilde{m_1}=1)}i^{\tilde{m_1}}J(m_1,k_2)+\delta_{(n_1-m_1=1)}i^{-m_1}J(\tilde{m_1},k_2)\right]\\
&+\delta_{(n_1=1)}\delta_{(n_2=1)}Z'(m_1)Z'(m_2)J(k_1,k_2),
\end{align*}

où nous avons noté: $Z'(m)=\frac{\zeta(m)\zeta(k-m)}{\zeta(k)}$ (rationnel si $m$ est un entier pair).
\end{proof}

Il reste ainsi à calculer $\mathcal{S}$ ou plus simplement après utilisation des symétries:
$$S(m_1,m_2)=\int_0^{\infty} \int_0^{u_1} F_{m_1}^*(iu_1)u_1^{n_1-1}F_{m_2}^*(iu_2)u_2^{n_2-1}\d u_1\d u_2=\Lambda(F_{m_1}^*,F_{m_2}^*;n_1,n_2).$$
Afin d'utiliser plus simplement la formule d'Hurwitz, à savoir:
$$\sum_{N\in\mathbb{Z}} (z+N)^{-m}=\frac{(-2i\pi)^m}{(m-1)!}\sum_{\beta>0} \beta^{m-1} \exp(2i\pi\beta z).$$
Notons $F_m^0=\sum_{N\in\Z}F_m[\mat{1}{N}{0}{1}]$. On a $F_m^*=(-1)^m\sum_{N\in\Z}F_m[\mat{N}{1}{-1}{0}]$ et ainsi on remarque que :
$$F_m^0(-1/(iu))=(iu)^kF_m^*(iu).$$

\begin{lem}
On a:
$$\frac{S(m_1,m_2)}{(\tilde{n_1}-1)!(\tilde{n_2}-1)!}=\sum_{j=0}^{\tilde{n_2}-1}\binom{\tilde{n_1}-1+j}{j} \frac{L(F_{m_2}^0,F_{m_1}^0;\tilde{n_2}-j,\tilde{n_1}+j)}{(2\pi)^{\tilde{n_1}+\tilde{n_2}}},$$
et pour tout entier $M_1$ et $M_2$, on a:
$$L(F_{m_2}^0,F_{m_1}^0;M_2,M_1)=\frac{(-2i\pi)^{m_1+{m_2}}}{({m_1}-1)!({m_2}-1)!} Z(M_2-{m_2}+1,1-{m_1},M_1),$$
où : $Z(a,b,c)=\sum_{\alpha>0}\sum_{\beta>0}\alpha^{-a}\beta^{-b}(\alpha+\beta)^{-c}$.
\end{lem}

\begin{proof}
La définition de $F_m^0$ permet d'obtenir la réduction après changement des variables $(u_1,u_2)\to (1/u_2,1/u_1)$: 
$$S(m_1,m_2)=\Lambda(F_{m_1}^*,F_{m_2}^*;n_1,n_2)=(-1)^{m_1+m_2}i^{k_1+k_2}\Lambda(F_{m_2}^0,F_{m_1}^0;\tilde{n_2},\tilde{n_1}).$$\par
Pour simplifier les calculs nous rappelons que la connaissance des fonctions $L$ est équivalente à celle des fonctions $\Lambda$ d'après la Proposition \ref{prop24}. Puis en substituant par la formule d'Hurwitz, on obtient par calcul direct:
\begin{align*}
\frac{L(F_{m_2}^0,F_{m_1}^0;M_2,M_1)}{(2\pi)^{M_1+M_2}}&=\sum_{\alpha,\beta>0}\int_0^{\infty}\d u_1...\int_{u_{M_1-1}}^{\infty}\frac{(-2i\pi)^{m_1}}{(m_1-1)!}\alpha^{{m_1}-1}\exp(-2\pi\alpha u_{M_1})\d u_{M_1}\\
&\int_{u_{M_1}}^{\infty}dv_1...\int_{v_{M_2-1}}^{\infty}\frac{(-2i\pi)^{{m_2}}}{({m_2}-1)!}\beta^{{m_2}-1}\exp(-2\pi\beta v_{M_2})dv_{M_2}\\
&=\frac{(-2i\pi)^{{m_1}+{m_2}}}{(2\pi)^{M_1+M_2}({m_1}-1)!({m_2}-1)!} Z(M_2-{m_2}+1,1-{m_1},M_1).
\end{align*}
Les séries $Z(a,b,c)$ convergent pour tous les entiers vérifiant $a+b+c>2$. 
\end{proof}

Zagier introduit les nombres $Z(a,b,c)$ dans \cite{Za94} et s'intéresse à leurs algébricités lorsque le poids $a+b+c$ est impair. Nous avons ici au contraire un poids $\tilde{n_1}+\tilde{n_2}-m_1-m_2+2$ pair. Pourtant étant donné que $b=1-{m_1}\leq 0$, on dispose d'une expression en fonction des valeurs multiples de zêta.

\begin{prop}
Soit $n\in\mathbb{Z}_{\geq 0}$ et $a,b\in\mathbb{Z}$ vérifiant $a+b-n>2$ alors:
$$Z(a,-n,b)=\sum_{j=0}^n (-1)^j \binom{n}{j} \zeta(b-n+j,a-j).$$
\end{prop}

\begin{proof}
Il suffit d'écrire $\beta=(\alpha+\beta)-\alpha$ et de développer:
\begin{align*}
Z(a,-n,b)&=\sum_{\alpha,\beta>0} ((\alpha+\beta)-\alpha)^n \alpha^{-a} (\alpha+\beta)^{-b}\\
&=\sum_{\alpha,\beta>0}\sum_{j=0}^n (-1)^j \binom{n}{j} \alpha^j(\alpha+\beta)^{n-j} \alpha^{-a} (\alpha+\beta)^{-b}\\
&=\sum_{j=0}^n (-1)^j \binom{n}{j} \zeta(b-n+j,a-j).
\end{align*}
\end{proof}

\begin{coro}
Sous les hypothèses précédentes sur les entiers $n_1,n_2,m_1$ et $m_2$, on déduit l'expression de $S(m_1,m_2)=\Lambda(F_{m_1}^*,F_{m_2}^*;n_1,n_2)$ en fonction des bizêtas classiques:
\begin{multline}\label{eqpartiel}
C_{m_1}C_{m_2}\Lambda(F_{m_1}^*,F_{m_2}^*;n_1,n_2)=i^{k_1+k_2}\frac{2^{k_1+k_2-2}(\tilde{m_1}-1)!(\tilde{m_2}-1)!(\tilde{n_1}-1)!(\tilde{n_2}-1)!}{(k_1-2)!(k_2-2)!}\\
\times\sum_{a=0}^{\tilde{n_2}-1}\sum_{b=0}^{{m_1}-1}(-1)^b\binom{{m_1}-1}{b}\binom{\tilde{n_1}-1+a}{a} 
\frac{\zeta(\tilde{n_1}-{m_1}+1+a+b,\tilde{n_2}-{m_2}+1-a-b)}{(2\pi)^{\tilde{n_1}+\tilde{n_2}-{m_1}-{m_2}+2}}.
\end{multline}
\end{coro}

Ceci conclut la démonstration du théorème.\par
\vspace*{.5cm}
Dans l'expression qui figure dans (\ref{eqpartiel}), on peut distinguer deux types de bizêtas $\zeta(b,a)$ suivant le signe de $a$.
Lorsque $a$ est négatif, on a la formule suivante pour laquelle nous n'avons pas rencontré de référence étendant la suivante:
$$\zeta(b,0)=\sum_{0<m<n}n^{-b}=\sum_{0<n}(n-1)n^{-b}=\zeta(b-1)-\zeta(b).$$

\begin{lem}\label{lemneg}
Soit $a,b\in\Z_{>0}$ tels que $b-a>2$.
Les valeurs zeta multiples avec un argument négatif s'expriment grâce aux valeurs zêta:
$$\zeta(b,-a)=\frac{1}{a+1}\sum_{n=0}^{a}\binom{a+1}{n}B_n\zeta(b-a-1+n)$$ 
\end{lem}

\begin{proof}
On dispose de la formule bien connue donnant la somme des puissances des entiers consécutifs:
$$\sum_{k=1}^{n-1} k^m=\frac{1}{m+1}\sum_{\alpha=0}^{m} \binom{m+1}{\alpha}B_{\alpha}n^{m+1-\alpha}.$$
On peut alors appliquer cette formule à notre définition de bizêtas:
\begin{align*}
\zeta(b,-a)&=\sum_{\alpha<\beta} \frac{\alpha^a}{\beta^b}=\sum_{\beta=1}^{\infty}\frac{\sum_{\alpha=1}^{\beta-1}\alpha^a}{\beta^b}\\
&=\sum_{\beta=1}^{\infty}\frac{\sum_{n=0}^{a}\binom{a+1}{n}B_n\beta^{a+1-n}}{(a+1)\beta^b}\\
&=\frac{1}{a+1}\sum_{n=0}^{a}\binom{a+1}{n}B_n\zeta(b-a-1+n).
\end{align*}
\end{proof}

Dans la formule du théorème \ref{calculprincipal}, les nombres intervenants sont bien compris. Il s'agit de rationnels ou de valeurs multiples de zêta. Seuls les quantités $J(\alpha,\beta)$ restent encore indéterminées. Pourtant on dispose tout de même de l'équation :
\begin{equation}\label{Jsym}
J(\alpha,\beta)+J(\beta,\alpha)=J(\alpha)J(\beta),
\end{equation}
où $J(\alpha)=\int_0^1\int_{-\infty}^{\infty}(1+iux)^{-\alpha}\d u\d x=\frac{\pi}{\alpha-1}$. Ceci permet notamment de calculer les $J(\alpha,\alpha)=\frac{\pi^2}{2(\alpha-1)^2}$. Nous démontrons tout de même le résultat suivant:

\begin{prop}
Soit $\alpha,\beta\geq 2$ des entiers, on a:
\begin{equation}
(\alpha-1)(\beta-1)J(\alpha,\beta)=\frac{\pi^2}{2}+B_{\alpha,\beta},
\end{equation}
avec les $B_{\alpha,\beta}\in\Q$ et vérifiant pour les entiers $a,N\geq 0$ par:
$$B_{a,a+N}=\sum_{k=0}^{N-1} \frac{A_{a,a+k}}{a+k+1}\text{ et }B_{a,a-N}=-\sum_{k=0}^{N-1}\frac{A_{a,a-k}}{a-k+1},$$
et où les rationnels $A_{\alpha,\beta}$ sont donnés par la série génératrice:
$$\sum_{\alpha,\beta\geq 0}A_{\alpha,\beta}X^{\alpha}Y^{\beta}=\frac{4}{(2-X-Y)(Y-X)}\log\left(\frac{1-X}{1-Y}\right)\in\Q[[X,Y]].$$
\end{prop}

\begin{proof}
Considérons la série génératrice des $J(\alpha)$ suivante:
\begin{align*}
h(X)&=\sum_{\alpha\geq 2} (\alpha-1)J(\alpha)X^{\alpha-2}=\frac{\d}{\d X}\left(\int_{\R}\int_0^1\sum_{a\geq 0}\frac{X^{a+1}}{(1+iux)^{a+2}}\d x\d u\right)\\
&=\int_{\R}\int_0^1\frac{\d}{\d X}\left(\frac{X\d u\d x}{(1+iux)(1+iux-X)}\right)=\int_{\R}\int_0^1\frac{\d u\d x}{(1-X+iux)^2}\\
&=\int_{\R}\frac{\d x}{(1-X)(1-X+ix)}=\int_0^{\infty}\frac{2dx}{(1-X)^2+x^2}=\frac{\pi}{1-X}.
\end{align*}
Les calculs sont valides car la série a un rayon de converge de $1$.
De manière analogue, considérons la série génératrice des $J(\alpha,\beta)$:
$$H(X,Y)=\sum_{\alpha,\beta\geq 2}(\alpha-1)(\beta-1)J(\alpha,\beta)X^{\alpha-2}Y^{\beta-2}.$$

D'une part, la relation (\ref{Jsym}) donne :
$$H(X,Y)+H(Y,X)=h(X)h(Y).$$

D'autre part, on obtient par le calcul:
\begin{align*}
H(X,Y)&=\int_{\R^2}\int_0^1\int_0^u\frac{\d^2}{\d X\d Y}\left(\sum_{a\geq 0}\frac{X^{a+1}\d x\d u}{(1+iux)^{a+2}}\sum_{b\geq 0}\frac{Y^{b+1}\d x\d u}{(1+ivy)^{b+2}}\right)\\
&=\int_{\R^2}\int_0^1\int_0^u\frac{\d^2}{\d X\d Y}\left(\frac{X\d u\d x}{(1+iux)(1+iux-X)}\frac{Y\d v\d y}{(1+ivy)(1+ivy-Y)}\right)\\
&=\int_{\R^2}\int_0^1\int_0^u\frac{\d u\d x}{(1-X+iux)^2}\frac{\d v\d y}{(1-Y+ivy)^2}\\
&=\int_{\R^2}\int_0^1\frac{u\d u\d x\d y}{(1-X+iux)^2(1-Y)(1-Y+iuy)}.
\end{align*}

On remarque alors que cette expression peut être symétrisée par:
\begin{align*}
\frac{\d}{\d Y}\left[(1-Y)H(X,Y)\right]&=\int_{\R}\int_0^1\frac{u\d u\d x \d y}{(1-X+iux)^2(1-Y+iuy)^2}\\
&=\lim_{M\to\infty}\int_{-M}^M\int_{-M}^M\int_0^1\frac{u\d u\d x\d y}{(1-X+iux)^2(1-Y+iuy)^2}\\
&=\lim_{M\to\infty}\int_0^1\frac{4M^2u\d u}{((1-X)^2+u^2M^2)((1-Y)^2+u^2M^2)}\\
&=\int_0^{\infty}\frac{4udu}{((1-X)^2+u^2)((1-Y)^2+u^2)}=\int_0^{\infty}\frac{2\d w}{((1-X)^2+w)((1-Y)^2+w)}\\
&=\left[\frac{2}{(2-X-Y)(X-Y)}\log\left(\frac{(1-X)^2+w}{(1-Y)^2+w}\right)\right]_{w=0}^{\infty}\\
&=\frac{4}{(2-X-Y)(Y-X)}\log\left(\frac{1-X}{1-Y}\right)\\
&=2\sum_{k\geq 2}\frac{(X+Y)^k}{2^k}\sum_{a,b\geq 0}\frac{X^aY^b}{a+b+1}=\sum_{a,b\geq 0}A_{a,b}X^aY^b.
\end{align*}

D'autre part, notons $B_{a,b}=(a-1)(b-1)J(a,b)$ les coefficients de $H$, on a:
$$\frac{\d}{\d Y}\left[(1-Y)H(X,Y)\right]=\sum_{a,b\geq 0} (b+1)(B_{a,b+1}-B_{a,b})X^aY^b.$$
Cette dernière donne une relation de récurrence simple donnant pour tout les entiers $a\geq 0$ et $N>0$:
$$B_{a,a+N}=\frac{\pi^2}{2}+\sum_{k=0}^{N-1} \frac{A_{a,a+k}}{a+k+1}\text{ et }B_{a,a-N}=\frac{\pi^2}{2}-\sum_{k=0}^{N-1}\frac{A_{a,a-k}}{a-k+1}.$$
\end{proof}

\subsection{Discussion et application du théorème \ref{calculprincipal}}

a) Les termes $Q(m_1,m_2,n_1,n_2)$ se décomposent comme suit:
$$Q(m_1,m_2,n_1,n_2)=Q_1(m_1,m_2,n_1,n_2)+Q_2(m_1,m_2,n_1,n_2),$$
où le terme dépendant de bizêtas classiques homogènes en poids est:
$$Q_1(m_1,m_2,n_1,n_2)=\sum_{\gamma=0}^{\tilde{n_2}-\tilde{m_2}} \binom{\tilde{n_1}+\tilde{n_2}-\tilde{m_1}-\tilde{m_2}+1-\gamma}{\tilde{n_1}-\tilde{m_1}} \zeta(\tilde{n_1}+\tilde{n_2}-\tilde{m_1}-\tilde{m_2}+1-\gamma,\gamma+1),$$
et le terme en zêtas simples est:
$$Q_2(m_1,m_2,n_1,n_2)=\sum_{n=0}^{\tilde{m_1}+\tilde{m_2}-2}A_n\zeta(\tilde{n_1}+\tilde{n_2}-\tilde{m_1}-\tilde{m_2}+1+n),$$
dans lequel les $A_n$ sont des coefficients rationnels obtenus grâce au lemme \ref{lemneg} par réécriture des $\zeta(b,-a)$ pour $a\geq 0$.
Ces coefficients vérifient des propriétés de symétrie, en effet, on observe que, d'une part:
$$Q_2(m_1,m_2,n_1,n_2)+Q_2(m_2,m_1,n_2,n_1)=0.$$
Et d'autre part, la formule sur les bizêtas, voir \cite{DZV} par exemple:
\begin{equation}\label{bizetform}
\sum_{r=2}^{k-1}\left[\binom{r-1}{j-1}+\binom{r-1}{k-j-1}\right]\zeta(r,k-r)=\zeta(j)\zeta(k-j),\text{ pour tout }2\leq j\leq k-2,
\end{equation}
permet d'obtenir:
$$Q_2(m_1,m_2,n_1,n_2)+Q_2(m_2,m_1,n_2,n_1)=\zeta(\tilde{n_1}-\tilde{m_1}+1)\zeta(\tilde{n_2}-\tilde{m_2}+1).$$

On peut alors notamment vérifier la relation de mélange:
\begin{equation}\label{mel2}
\Lambda(F_{m_1},F_{m_2};n_1,n_2)+\Lambda(F_{m_2},F_{m_1};n_2,n_1)=\Lambda(F_{m_1},n_1)\Lambda(F_{m_2},n_2).
\end{equation}
Le membre de droite est obtenu par la formule de Kohnen-Zagier \cite{KZ}. Le membre de gauche est donné par le théorème \ref{calculprincipal}.\par
L'expression proposée dans le théorème nous semble ainsi irréductible. En effet, les termes $Q_1$ et $Q_2$ ne semble pas être dans le même $\Q$-espace vectoriel puisque le poids des valeurs multiples de zêta en jeu sont étrangers. Une version simplifiée de $Q_1$ donnerait une version simplifiée de la formule (\ref{bizetform}). L'indépendance supposée des $\zeta(2k+1)$ pour $k\in\Z_+^*$ suggère la minimalité de l'écriture de $Q_2$.

b) \textit{Hypothèses de parité :} Nous n'espérons pas que les hypothèses du théorème:
$$m_1+n_1\text{ et }m_2+n_2 \text{ impairs}, $$
puissent être levées. En effet, si on disposait d'une formule pour $\Lambda(F_{m_1},F_{m_2};n_1,n_2)$ avec $m_1$ et $n_1$ de même parité alors la formule (\ref{mel2}) permettrait de trouver une formule fermée pour $\Lambda(F_{m_1},n_1)$ avec $n_1+m_1$ pair. Ce calcul n'ayant à l'heure actuelle pas été résolu, ils semblent donc que ces hypothèses de parités soient essentielles.\par

c) \textit{Hypothèses d'inégalités :} Il n'est pas clair qu'on ne puisse obtenir une formule analogue lorsque les hypothèses:
$$\text{soit }n_1<m_1,\quad n_1<\tilde{m_1},\quad n_2<m_2\text{ et }n_2<\tilde{m_2},$$
$$\text{soit }m_1<n_1,\quad \tilde{m_1}<n_1,\quad m_2<n_2\text{ et }\tilde{m_2}<n_2,$$
ne sont pas vérifiées. Ces conditions apparaissent comment une limite technique dans la démonstration mais nous ne percevons pas d'obstruction conceptuelle. Voici une piste pour résoudre cette indétermination. La complexité de ce procédé augmente avec les poids $k_1$ et $k_2$ mais semble bien se généraliser. Le polynôme $\sum_{m=0}^{k-2} \binom{k-2}{m} F_{m+1} X^m\in S_k\otimes \C_{k-2}[X]$ vérifie les relations de Manin, si bien qu'on dispose de relations linéaires reliant les nombres:
$$\Lambda(F_{m_1},F_{m_2};n_1,n_2)\text{ pour tout }m_1+n_1\text{ et }m_2+n_2\text{ impairs}.$$
Certains de ces nombres sont donnés par le théorème \ref{calculprincipal}. On observe dans des cas de petits poids ($k_1,k_2\in\{12,16\}$) que les relations suffisent pour obtenir les valeurs:
$$\Lambda(F_{m_1},F_{m_2};n_1,n_2)\text{ lorsque }(n_1,n_2)\in \llbracket 1,k_1/2-1\rrbracket\times \llbracket 1,k_2/2-1\rrbracket\cup \llbracket k_1/2+1,k_1-1\rrbracket\times \llbracket k_2/2+1,k_2-1\rrbracket,$$
sans condition supplémentaire que celles de parités sur le couple $(m_1,m_2)$. Ceci nous conduit à la remarque suivante.\par

d) Lorsque $k_1=k_2=12$, le théorème \ref{calculprincipal} donne une forme explicite des nombres:
$$\frac{\Lambda(\Delta,\Delta;n_1,n_2)}{\Lambda(\Delta;n_1)\Lambda(\Delta;n_2)},\text{ pour }(n_1,n_2)\in \llbracket 1,5\rrbracket^2\cup \llbracket 7,11\rrbracket^2. $$
En effet, la quantité $\Lambda(f,g;n_1,n_2)/\Lambda(f;n_1)\Lambda(g;n_2)$ est indépendante du choix de $f,g\in S_{12}\setminus\{0\}$ car $\dim_{\C} S_{12}=1$. On peut ainsi trouver des fonctions $F_{m_1}^{12}$ et $F_{m_2}^{12}$ vérifiant les hypothèses du théorème en vu d'effectuer le calcul. On retrouve notamment que pour tout entier $1\leq n\leq 12$, $\Lambda(f,g;n,n)/\Lambda(f;n)\Lambda(g;n)=1/2$. Or les nombres $\Lambda(\Delta;n)$ sont assez bien compris. En effet, il existe des périodes ( Voir \cite{KZ}):
$$\Omega^-_{\Delta}=0.0214460667...\text{ et }\Omega^+_{\Delta}=0.0000482774800..,$$
tels que $\Lambda(\Delta,n)=\frac{(n-1)!}{(2\pi)^n}L(\Delta,n)=\alpha_n \Omega^{(-1)^n}_{\Delta}$ avec les rationnels $\alpha_n\in\Q$ données par:
$$\alpha_1=-\alpha_{11}=\frac{192}{691},\alpha_2=\alpha_{10}=\frac{384}{5},\alpha_3=-\alpha_9=\frac{16}{135},\alpha_4=\alpha_8=40,\alpha_5=-\alpha_7=\frac{8}{105}\text{ et }\alpha_6=32.$$
Ceci démontre le résultat suivant de manière constructive:

\begin{thm}\label{thm37}
Soit $(n_1,n_2)\in \llbracket 1,6\rrbracket^2\cup \llbracket 6,11\rrbracket^2$. Alors on a:
$$\Lambda(\Delta,\Delta;n_1,n_2)\in \Omega^{(-1)^{n_1}}_{\Delta}\Omega^{(-1)^{n_2}}_{\Delta}\sum_{\substack{2<a+b\text{ pair}\\1<a}}\Q\frac{\zeta(a,b)}{\pi^{a+b}}.$$
\end{thm}

\textbf{Application explicite}\\
On va utiliser le théorème \ref{calculprincipal} pour obtenir une formule fermée pour $\Lambda(\Delta,\Delta;2,3)$. Pour cela on choisit de calculer d'une part:
\begin{align*}
\Lambda(F_5,F_6;2,3)&=\frac{2^{22}}{10!^2}\left(2Q(5,10,6,9)+2Q(7,10,6,9)\right)\\
&=2^{26}(4862Z(-10, 18)-(342628/5)Z(-7, 15)+60203Z(-8, 16)-26884Z(-9, 17)\\
&-9680Z(-5, 13)+(199628/5)Z(-6, 14)+7Z(0, 8)+(56/5)Z(-1, 9)+(84/5)Z(-2, 10)\\
&+24Z(-3, 11)+33Z(-4, 12)+(1/5)Z(4, 4)+(4/5)Z(3, 5)+2Z(2, 6)+4Z(1, 7)\\
&-514800Z(-7, 17)+145860Z(-8, 18)+4752Z(-3, 13)+7722Z(-4, 14)\\
&-279708Z(-5, 15)+630630Z(-6, 16)+36Z(3, 7)+126Z(2, 8)+336Z(1, 9)\\
&+756Z(0, 10)+1512Z(-1, 11)+2772Z(-2, 12)+6Z(4, 6)),
\end{align*}
où $Z(a,b)=\zeta(b,a)/(2\pi)^{a+b}$.
D'autre part, nous disposons de:
$$\Lambda(F_5;2)=2^{11}\left(72\frac{\zeta(6)}{(2\pi)^6}+\frac{12}{5}\frac{\zeta(4)}{(2\pi)^4}\right)=\frac{1024}{175}\text{ et }\Lambda(F_6:3)=\frac{2^{14}\zeta(4)}{3(2\pi)^4}=\frac{512}{135}.$$
Ceci permet d'obtenir:
\begin{align*}
\Lambda(\Delta,\Delta&;2,3)=2^7 \Lambda(\Delta;2)\Lambda(\Delta;3)\big[-635134500Z(-9, 17)+114864750Z(-10, 18)\\
&-1618917300Z(-7, 15)+1422295875Z(-8, 16)-228690000Z(-5, 13)+943242300Z(-6, 14)\\
&+567000Z(-3, 11)+779625Z(-4, 12)+4725Z(4, 4)+18900Z(3, 5)+47250Z(2, 6)\\
&+94500Z(1, 7)+165375Z(0, 8)+264600Z(-1, 9)+396900Z(-2, 10)\\
&+3445942500Z(-8, 18)-6608101500Z(-5, 15)+14898633750Z(-6, 16)\\
&-12162150000Z(-7, 17)+112266000Z(-3, 13)+182432250Z(-4, 14)\\
&+850500Z(3, 7)+2976750Z(2, 8)+7938000Z(1, 9)+17860500Z(0, 10)\\
&+35721000Z(-1, 11)+65488500Z(-2, 12)+141750Z(4, 6)\big].
\end{align*}
Puis en appliquant le lemme \ref{lemneg}, on obtient:
\begin{align}\nonumber
\Lambda(\Delta,\Delta&;2,3)=2^73^35^27 \Lambda(\Delta;2)\Lambda(\Delta;3)\Big[30Z(4, 6)+180Z(3, 7)+630Z(2, 8)+1680Z(1, 9)\\ \nonumber
&+Z(4, 4)+4Z(3, 5)+10Z(2, 6)+20Z(1, 7)-630\frac{\zeta(10)}{(2\pi)^{10}}\\ \nonumber
&-\frac{1105}{126}\frac{\zeta(7)}{(2\pi)^{8}}+\frac{12155}{63}\frac{\zeta(13)}{(2\pi)^{8}}-\frac{12155}{6}\frac{\zeta(15)}{(2\pi)^{8}}+\frac{5525}{3}\frac{\zeta(17)}{(2\pi)^{8}}\\ 
&+\frac{7456}{3}\frac{\zeta(9)}{(2\pi)^{10}}-330\frac{\zeta(11)}{(2\pi)^{10}}+143\frac{\zeta(13)}{(2\pi)^{10}}+\frac{67925}{3}\frac{\zeta(15)}{(2\pi)^{10}}-24310\frac{\zeta(17)}{(2\pi)^{10}}\Big]. \label{formzeta23}
\end{align}
Dans le membre de droite de l'expression ci-dessus la sous-somme des termes de dénominateurs $(2\pi)^8$ (respectivement $(2\pi)^{10}$) est nulle lorsqu'on remplace $\zeta(k)$ par $1$ pour $7\leq k \leq 17$ entier. On peut alors récrire:
\begin{align}\label{formzeta23bis}
&\frac{\Lambda(\Delta,\Delta;2,3)}{\Lambda(\Delta;2)\Lambda(\Delta;3)}=\frac{2^73^35^27}{(2\pi)^{10}} \Big[30\zeta(6,4)+180\zeta(7,3)+630\zeta(8,2)+1680\zeta(9,1)\\ \nonumber
&+(2\pi)^2\left(\zeta(4,4)+4\zeta(5,3)+10\zeta(6,2)+20\zeta(7,1)\right)-630(\zeta(10)-1)\\ \nonumber
&+(2\pi)^2\left(-\frac{1105}{126}(\zeta(7)-1)+\frac{12155}{63}(\zeta(13)-1)-\frac{12155}{6}(\zeta(15)-1)+\frac{5525}{3}(\zeta(17)-1)\right)\\ \nonumber
&+\frac{7456}{3}(\zeta(9)-1)-330(\zeta(11)-1)+143(\zeta(13)-1)+\frac{67925}{3}(\zeta(15)-1)-24310(\zeta(17)-1)\Big].
\end{align}

\begin{rem}
1) On a vérifié cette dernière formule par ordinateur avec $15$ chiffres après la virgule. Pour calculer $\Lambda(\Delta,\Delta;2,3)$ nous transformons la série de Dirichlet double associée à la fonction $\tau$ de Ramanujan. Pour rendre convergent la série on choisit de découper son expression d'intégrale itérée suivant les domaines:
$$\{0<t_1<t_2\}=\{0<t_1<t_2<1\}\cup \{0<t_1<1\}\times\{1<t_2\}\cup\{1<t_1<t_2\}.$$
On a une convergence rapide sur le domaine $\{1<t_1<t_2\}$ et l'équation fonctionnelle permet de s'y ramener sans difficulté.\par
D'autre part, une expression tronquée des valeurs multiples de zêta donne une bonne approximation. Chacun des dix-huit termes qui interviennent dans le membre de droite de l'expression (\ref{formzeta23bis}) sont minorés en valeurs absolue par $10^{-3}$. Le calcul numérique de chacun des deux membres donnent à $10^{-15}$ près:
$$\frac{\Lambda(\Delta,\Delta;2,3)}{\Lambda(\Delta;2)\Lambda(\Delta;3)}=0.588554464393778...$$
2) Le théorème \ref{thm37} peut naturellement être étendu aux valeurs de $\Lambda(f_1,f_2;n_1,n_2)$ pour $f_1\otimes f_2\in S_{k_1}\otimes S_{k_2}$ avec $\dim_{\C}\left(S_{k_1}\otimes S_{k_2}\right)=1$.\\
3) Il parait raisonnable d'attendre un théorème analogue pour $f_1$ et $f_2$ des formes primitives de poids et de niveaux quelconques.
\end{rem}

\section{Polynôme des bipériodes de deux séries d'Eisenstein}

Nous calculons le polynôme étendu des bipériodes d'un couple quelconque de série d'Eisenstein de niveau $1$.

Soit $(f_1,f_2)\in M_{k_1}\times M_{k_2}$.
Le polynôme des bipériodes du couple $(f_1,f_2)$ est donné par les formes épointées:
$$\frac{P_{f_1,f_2}(X_1,X_2)}{(k_1-2)!(k_2-2)!}
=\sum_{m_1,m_2} \lim_{(s_1,s_2)\to(m_1,m_2)}\frac{\Lambda(f_1,f_2;s_1,s_2)}{\Gamma(s_1)\Gamma(k_1-s_1)\Gamma(s_2)\Gamma(k_2-s_2)}\frac{X_1^{k_1-m_1-1}}{i^{m_1}}\frac{X_2^{k_2-m_2-1}}{i^{m_2}},$$
$$\text{où }\Lambda(f_1,f_2;s_1,s_2)=\int_{0<t_2<t_1} f_1^*(t_1)t_1^{s_1-1}\d t_1f_2^*(t_2)t_2^{s_2-1}\d t_2.$$
Rappelons la notation des séries d'Eisenstein $E_k\in M_k$ pour un entier pair $k\geq 4$:
$$E_k(z)=\sum_{(m,n)\in\Z^2\setminus\{(0,0)\}} (mz+n)^{-k},\text{ pour }z\in\C.$$

Introduisons des nombres particuliers dépendant de paramètres entiers $e_1,e_2,f_1,f_2$ et $g$ tels que $e_1+e_2+f_1+f_2+g\geq 3$ par:
\begin{equation}
Z_g\mat{e_1}{f_1}{e_2}{f_2}=\sum'_{a_1,b_1,a_2,b_2\in\Z}\frac{a_1^{-e_1}b_1^{-f_1}a_2^{-e_2}b_2^{-f_2}}{(a_1b_2-a_2b_1)^{g}},
\end{equation}
où la somme est prise sur les entiers relatifs $(a_1,b,a_2,b_2)$ vérifiant $a_1b_1a_2b_2\neq 0$ et $a_1b_2\neq a_2b_1$.

\begin{prop}\label{biperEk}
La partie impaire du polynôme des bipériodes de deux séries d'Eisenstein est donnée par:
$\frac{1}{2}\left(P_{E_{k_1},E_{k_2}}(X_1,X_2)-P_{E_{k_1},E_{k_2}}(-X_1,-X_2)\right)=$
\begin{equation}
\frac{i\pi}{(k_1-1)(k_2-1)}\sum_{\substack{m\geq 0\\ m\text{ impair}}} \sum_{e_1,e_2\geq 0} 
\frac{(m+1)!Z_{m+2}\mat{e_1}{m-k_1-e_1+2}{e_2}{m-k_2-e_2+2}X_1^{e_1}X_2^{e_2}(X_2-X_1)^{k_1+k_2-4-m}}{e_1!e_2!(m-k_1+2-e_1)!(m-k_2+2-e_2)!(k_1+k_2-4-m)!}.
\end{equation}
\end{prop}

Commençons par rappeler comment Zagier propose de faire un calcul analogue pour les polynômes des périodes des séries d'Eisenstein dans \cite{Za91}. Il considère la série génératrice suivante:
$$P_1(X,Y)=\sum_{k\geq 4} P_{E_k}(X)Y^{k-2}.$$
On s'intéresse à la partie impaire en $X$ c'est à dire $Q_1(X,Y)=1/2(P_1(X,Y)-P_1(-X,Y))$. On peut ainsi ajouter de manière artificielle le terme pour $k=2$ et ceux pour $k$ impair car ces termes s'annulent:
$$P_1(X,Y)=\sum_{k\geq 2}\int_0^{\infty}\sum_{mn\neq 0}\frac{(X-it)^{k-2}Y^{k-2}idt}{(mit+n)^k}=\int_0^{\infty}\sum_{mn\neq 0}\frac{idt}{(mit+n)(mit+n-Y(X-it))}.$$
On calcul alors:
$$\frac{\d}{\d Y}(YP_1(X,Y))=\int_0^{\infty}\sum_{mn\neq 0}\frac{idt}{(mit+n-Y(X-it))^2}$$
Ceci permet d'avoir après intervention de l'intégrale valide par convergence normale de la série:
$$\frac{\d}{\d Y}(YP_1(X,Y))==\sum_{mn\neq 0}\frac{1}{(m+Y)(n-XY)}.$$
Considérons alors la partie impaire suivant $X$ puis la partie paire suivant $Y$ donne:
\begin{multline*}
\frac{\d}{\d Y}(YQ_1(X,Y))=\sum_{\substack{k\geq 4\\k\text{ pair}}} P^-_{E_k}(X)(k-1)Y^{k-2}=\sum_{mn\neq 0}\frac{XY^2}{(m^2-Y^2)(n^2-(XY)^2)}\\
=4XY^2\sum_{m,n>0}\sum_{\alpha,\beta\geq 0} \frac{Y}{m}^{2\alpha}\frac{XY}{n}^{2\beta}=\sum_{\alpha,\beta\geq 0}\zeta(2\alpha+2)\zeta(2\beta+2)X^{2\beta+1}Y^{2\alpha+2\beta+2}
\end{multline*}
Ceci permet d'obtenir le résultat bien connu:
$$P_{E_k}^-(X)=\sum_{m\text{ impair}}4\frac{\zeta(k-m+1)\zeta(m+1)}{k-1}X^m.$$

Venons en à la démonstration de la proposition \ref{biperEk}.

\begin{proof}
Pour faciliter le calcul on regroupe, par analogie avec ce qui précède, sous forme de série génératrice l'ensemble des polynômes $P_{E_{k_1},E_{k_2}}$ pour les couples d'entiers $(k_1,k_2)$. 
Le polynôme des bipériodes $P_{E_{k_1},E_{k_2}}(X_1,X_2)$ est défini par:
$$P_{E_{k_1},E_{k_2}}(X_1,X_2)=\int_0^{\infty}E_{k_1}^*(it_1)(it_1-X_1)^{k_1-2}i\d t_1\int_{t_1}^{\infty}E_{k_2}^*(it_2)(it_2-X_2)^{k_2-2}i\d t_2.$$
Introduisons ainsi la série génératrice pour les polynômes des bipériodes:
$$P_2(X_1,X_2,Y_1,Y_2)=\sum_{k_1,k_2\geq 4}P_{E_{k_1},E_{k_2}}(X_1,X_2)Y_1^{k_1-2}Y_2^{k_2-2}.$$
On peut ajouter les termes nuls obtenus pour $k_i=2$ et $k_i$ impair donnant:
\begin{align*}
&\sum_{k_1,k_2\geq 2}\int_0^{\infty}E_{k_1}^*(it_1)(it_1-X_1)^{k_1-2}i\d t_1\int_{t_1}^{\infty}E_{k_2}^*(it_2)(it_2-X_2)^{k_2-2}i\d t_2 Y_1^{k_1-2}Y_2^{k_2-2}\\
=&\sum_{a_1b_1a_2b_2\neq 0}\int_0^{\infty}\sum_{k_1\geq 2}\frac{(Y_1(it_1-X_1))^{k_1-2}}{(a_1it_1+b_1)^{k_1}}i\d t_1\int_{t_1}^{\infty}\sum_{k_2\geq 2}\frac{(Y_2(it_2-X_2))^{k_2-2}}{(a_2it_2+b_2)^{k_2}}i\d t_2\\
=&\sum_{a_1b_1a_2b_2\neq 0}\int_0^{\infty}\frac{i\d t_1}{(a_1it_1+b_1)(a_1it_1+b_1-Y_1(it_1-X_1))}\int_{t_1}^{\infty}\frac{i\d t_2}{(a_2it_2+b_2)(a_2it_2+b_2-Y_2(it_2-X_2))}.
\end{align*}

Pour simplifier la suite du calcul, nous appliquons l'opération suivante:
\begin{align*}
\frac{\d^2}{\d Y_1\d Y_2}&(Y_1Y_2P_2(X_1,X_2,Y_1,Y_2))\\
&=\sum_{a_1b_1a_2b_2\neq 0}\int_0^{\infty}\frac{i\d t_1}{(a_1it_1+b_1-Y_1(it_1-X_1))^2}\int_{t_1}^{\infty}\frac{i\d t_2}{(a_2it_2+b_2-Y_2(it_2-X_2))^2}\\
&=\sum_{a_1b_1a_2b_2\neq 0}\int_0^{\infty}\frac{i\d t_1}{(a_1it_1+b_1-Y_1(it_1-X_1))^2}\frac{1}{(a_2-Y_2)(a_2it_1+b_2-Y_2(it_1-X_2))}\\
&=\sum_{a_1b_1a_2b_2\neq 0}\int_0^{\infty}\frac{idt}{(A_1it+B_1)^2A_2(A_2it+B_2)}.
\end{align*}
où on définit $A_i=a_i-Y_i$ et $B_i=b_i+X_iY_i$ pour $i=1,2$.\par
On dispose de la décomposition en éléments simples suivante:
\begin{multline*}
\frac{1}{(it+B_1/A_1)^2(it+B_2/A_2)}=\frac{1}{it+B_1/A_1}\left(\frac{1}{it+B_1/A_1}-\frac{1}{it+B_2/A_2}\right)\frac{1}{B_2/A_2-B_1/A_1}\\
=\frac{1}{(B_2/A_2-B_1/A_1)(it+B_1/A_1)^2}-\frac{1}{(B_2/A_2-B_1/A_1)^2}\left(\frac{1}{it+B_1/A_1}-\frac{1}{it+B_2/A_2}\right).
\end{multline*}
Elle nous permet d'effectuer le calcule suivant des paramètre $(a_1,b_1,a_2,b_2)$ non nuls et vérifiant $a_1b_2\neq a_2b_2$:
\begin{multline*}
\int_0^{\infty}\frac{idt}{(A_1it+B_1)^2A_2(A_2it+B_2)}=\frac{1}{(A_1B_2-A_2B_1)A_2B_1}\\
+\frac{1}{(A_1B_2-A_2B_1)^2}\left(ln\left|\frac{A_2B_1}{A_1B_2}\right|+i\frac{\pi}{2} sgn\left(\frac{b_2}{a_2}-\frac{b_1}{a_1}\right)\right).
\end{multline*}

Pour retrouver les polynômes des bipériodes il suffit de développer par rapport à $Y_1$ et $Y_2$.

La partie impaire, correspond à la partie imaginaire en effet:
$$Q_2(X_1,X_2,Y_1,Y_2)=\frac{1}{2}\left(P_2(X_1,X_2,Y_1,Y_2)-P_2(-X_1,-X_2,-Y_1,-Y_2)\right)=\Im(P_2(X_1,X_2,Y_1,Y_2)).$$
et est donnée par les termes de la forme:
\begin{equation*}
\frac{1}{(A_1B_2-A_2B_1)^2}=\sum_{m\geq 0}(m+1)\frac{\left[ Y_1(a_2X_1+b_2)-Y_2(a_1X_2+b_1)+Y_1Y_2(X_1-X_2)\right]^m}{(a_1b_2-a_2b_1)^{m+2}}
\end{equation*}
Car on peut écrire:
\begin{align*}
A_1B_2-A_2B_1&=(a_1-Y_1)(b_2+X_2Y_2)-(a_2-Y_2)(b_1+X_1Y_1)\\
&=(a_1b_2-a_2b_1)-Y_1(a_2X_1+b_2)+Y_2(a_1X_2+b_1)-Y_1Y_2(X_1-X_2).
\end{align*}

Et ainsi il ne reste plus qu'à sommer pour obtenir:
\begin{equation*}
Q_2=\sum_{e_1+e_2+f_1+f_2+\alpha=m\geq 0} \frac{(m+1)!}{e_1!e_2!f_1!f_2!\alpha!}Z_{m+2}\mat{e_1}{f_1}{e_2}{f_2} X_1^{e_2}Y_1^{e_2+f_2}X_2^{e_1}(-Y_2)^{e_1+f_1}[Y_1Y_2(X_1-X_2)]^{\alpha}.
\end{equation*}

\end{proof}

\begin{rem}
On peut faire un calcul similaire pour la partie paire malgré l'apparition de termes plus complexe. Celle-ci correspond à la partie réelle obtenue. On a, d'une part:
\begin{align*}
&\frac{1}{(A_1B_2-A_2B_1)A_2B_1}\\
&=\sum_{\alpha=0}^{\infty}\frac{(Y_1(a_2X_1+b_2)-Y_2(a_1X_2+b_1)+Y_1Y_2(X_1-X_2))^{\alpha}}{(a_1b_2-a_2b_1)^{\alpha+1}}\frac{1}{(a_2-Y_2)(b_1+X_1Y_1)}\\
&=\sum_{k_1,k_2}\sum_{\alpha_1,\alpha_2,\alpha_3}\frac{(\alpha_1+\alpha_2+\alpha_3)!}{\alpha_1!\alpha_2!\alpha_3!}\frac{(a_2X_1+b_2)^{\alpha_1}(-a_1X_2-b_1)^{\alpha_2}(X_1-X_2)^{\alpha_3}}{(a_1b_2-a_2b_1)^{\alpha_1+\alpha_2+\alpha_3+1}}\frac{(-X_1)^{k_1-2-\alpha_1-\alpha_3}Y_1^{k_1-2}Y_2^{k_2-2}}{b_1^{k_1-1-\alpha_1-\alpha_3}a_2^{k_2-1-\alpha_2-\alpha_3}}
\end{align*}
En développant selon $X_1$ et $X_2$, on trouve des coefficients en fonction de $Z_g\mat{e_1}{f_1}{e_2}{f_2}$ donnée par
\begin{equation}
\frac{1}{(k_1-1)(k_2-1)}\sum_{\substack{m\geq 0\\ m\text{ impair}}} \sum_{e_1,e_2,e_3,e_4\geq 0} \frac{m!Z_{m+1}\mat{-e_3}{k_1-1+e_3-m}{k_2-1+e_2-m}{-e_2}X_1^{k_1-2+e_1+e_3+e_4-m}X_2^{e_3}}{e_1!e_2!e_3!e_4!(m-e_1-e_2-e_3-e_4)!}.
\end{equation}

Et d'autre part, on a:
\begin{align*}
&\frac{1}{(A_1B_2-A_2B_1)^2}ln\left|\frac{A_2B_1}{A_1B_2}\right|\\
&=\sum_{m\geq 0}(m+1)\frac{\left[ Y_1(a_2X_1+b_2)-Y_2(a_1X_2+b_1)+Y_1Y_2(X_1-X_2)\right]^m}{(a_1b_2-a_2b_1)^{m+1}}\\
&\times\left(ln\left|\frac{a_2b_1}{a_1b_2}\right|+\sum_{\alpha_2>0}\frac{(Y_2/a_2)^{\alpha_2}}{\alpha_2}+\sum_{\beta_1>0}\frac{(-X_1Y_1/b_1)^{\beta_1}}{\beta_1}
-\sum_{\alpha_1>0}\frac{(Y_1/a_1)^{\alpha_1}}{\alpha_1}-\sum_{\beta_2>0}\frac{(-X_2Y_2/b_2)^{\beta_2}}{\beta_2}\right).
\end{align*}
Elle fait apparaitre une partie avec une somme sur de logarithmes ainsi que termes du type $Z_g\mat{e_1}{f_1}{e_2}{f_2}$.\par
\end{rem}

La famille des nombres $Z_g\mat{e_1}{f_1}{e_2}{f_2}$ est une généralisation particulière des valeurs multiples de zêta. Bien que la détermination de ces nombres ne semble pas se ramener directement aux valeurs $\zeta(n)$ ou $\zeta(a,b)$, un certain nombre de relations assez simple les relient.

\begin{prop}
On peut calculer explicitement les $Z_g\mat{e_1}{f_1}{e_2}{f_2}$ pour $g=0$.\\
Soient $e_1,e_2,f_1$ et $f_2$ des entiers naturels. On a:
$$Z_0\left(\begin{smallmatrix}e_1 & f_1\\ e_2 & f_2\end{smallmatrix}\right)=
\begin{cases}
16\zeta(e_1)\zeta(f_1)\zeta(e_2)\zeta(f_2)-\frac{8\zeta(e_1+f_1)\zeta(e_2+f_2)\zeta(e_1+e_2)\zeta(f_1+f_2)}{\zeta(e_1+e_2+f_1+f_2)}&\text{si }e_1,e_2,f_1\text{ et }f_2\text{ pairs},\\
-\frac{8\zeta(e_1+f_1)\zeta(e_2+f_2)\zeta(e_1+e_2)\zeta(f_1+f_2)}{\zeta(e_1+e_2+f_1+f_2)}&\text{si }e_1,e_2,f_1\text{ et }f_2\text{ impairs},\\
0&\text{sinon.}
\end{cases}$$
De plus, on dispose des relations suivantes:
$$Z_g\left(\begin{smallmatrix}e_1 & f_1\\ e_2 & f_2\end{smallmatrix}\right)=
Z_{g+1}\left(\begin{smallmatrix}e_1-1 & f_1\\ e_2 & f_2-1\end{smallmatrix}\right)
-Z_{g+1}\left(\begin{smallmatrix}e_1 & f_1-1\\ e_2-1 & f_2\end{smallmatrix}\right).$$
et:
$$Z_g\left(\begin{smallmatrix}e_1 & f_1\\ e_2 & f_2\end{smallmatrix}\right)
=Z_g\left(\begin{smallmatrix}f_2 & f_1\\ e_2 & e_1\end{smallmatrix}\right)
=Z_g\left(\begin{smallmatrix}e_1 & e_2\\ f_1 & f_2\end{smallmatrix}\right)
=(-1)^gZ_g\left(\begin{smallmatrix} e_2 & f_2\\e_1 & f_1\end{smallmatrix}\right).$$
\end{prop}

\chapter{Polynôme des bipériodes}
\minitoc
\section*{Enoncé des résultats}
Soient $k_1,k_2\geq 2$ des entiers pairs.\par
Soit $V_{k_1,k_2}^{\Z}$ l'espace des polynômes en deux indéterminées $X_1$ et $X_2$, à coefficients entiers et de degrés bornés par $deg(X_j)\leq k_j-2$. On l'identifiera librement à $V_{k_1}^{\Z}\otimes V_{k_2}^{\Z}$, où l'on a posé: $V_{k_j}^{\Z}=\Z_{k_j-2}[X_j]$.\par
Pour tout anneau commutatif $A$, définissons le $A$-module $V_{k_1,k_2}^A=V_{k_1,k_2}^{\Z}\otimes A$. Lorsque $A=\C$, nous noterons plus simplement $V_{k_1,k_2}^{\C}=V_{k_1,k_2}$. Il hérite d'une $\R$-structure donnée par $V_{k_1,k_2}^{\R}$.\par
Notons $\Gamma=\G/\{\pm id\}$ le groupe modulaire. Le groupe $\Gamma^2$ opère diagonalement sur $V_{k_1,k_2}^{\Z}$. Soient $S_{k_1}$ et $S_{k_2}$ les espaces des formes holomorphes modulaires paraboliques pour $\Gamma$ et de poids $k_1$ et $k_2$ respectivement. Pour tout anneaux $A\subset \C$, notons $S_{k_1}^{A}\subset S_{k_1}$ et $S_{k_2}^{A}\subset S_{k_2}$ les sous-espaces des formes à coefficients de Fourier dans l'anneau $A$. On a vu notamment que pour tout corps de nombre $K\subset\C$, on a: $S_{k_j}^K=S_{k_j}^{\Z}\otimes_{\Z} K$ pour $j=1,2$.\par
Soit $(f_1,f_2)\in S_{k_1}\times S_{k_2}$, le \textit{polynôme des bipériodes} $P_{f_1,f_2}\in V_{k_1,k_2}$ a été défini dans le Chapitre $2$ section $2.4.2$ par:
\begin{align}
P_{f_1,f_2}(X_1,X_2)&=\int_{0<t_1<t_2} f_1(it_1)f_2(it_2)(X_1-it_1)^{k_1-2}(X_2-it_2)^{k_2-2}\d t_1\d t_2\\
&=\sum_{m_1=1}^{k_1-1}\sum_{m_2=1}^{k_2-1}\binom{k_1-2}{m_1-1}\binom{k_2-2}{m_2-1}\Lambda(f_1,f_2;m_1,m_2)\frac{X_1^{k_1-m_1-1}X_2^{k_2-m_2-1}}{i^{m_1+m_2}},
\end{align}
où les \textit{bipériodes} sont les valeurs aux bi-entiers critiques (\textit{i.e.} $1\leq m_1\leq k_1-1$ et $1\leq m_2\leq k_2-1$) du prolongement analytique à $\C^2$ de l'application définie pour $\Re(s_j)>k_j$ par:
\begin{equation}
\Lambda(f_1,f_2;s_1,s_2)=\int_{0<t_1<t_2} f_1(it_1)f_2(it_2)t_1^{s_1-1}t_2^{s_2-1}\d t_1\d t_2.
\end{equation}

\begin{prop}
1) Le polynôme de bipériodes vérifie les deux systèmes d'équations :\\
De type récursive, c'est-à-dire dépendant des polynômes des périodes:
\begin{align}
P_{f_1,f_2}|&_{(1,1)+(S,S)}=P_{f_1}\otimes P_{f_2},\label{relSS}\\
P_{f_1,f_2}|&_{[(1,1)-(S,S)][(1,1)+(U,U)+(U^2,U^2)]}=P_{f_1}\otimes P_{f_2}|_{(U^2,U)-(U,U^2)}.\label{relUU}
\end{align}
De type linéaire:
\begin{align}
P_{f_1,f_2}|&_{(1+S,1+S)}=P_{f_1,f_2}|_{(1+U+U^2,1+U+U^2)}=0,\label{intral1}\\
P_{f_1,f_2}|&_{[(1,1)+(S,S)](1+U+U^2,1)}=P_{f_1,f_2}|_{[(1,1)+(S,S)](1,1+U+U^2)}=0,\nonumber\\
P_{f_1,f_2}|&_{(S,S)+(S,SU^2)+(SU^2,SU^2)+(1,U^2)-(U,U)}=0,\nonumber
\end{align}
$$\text{où }S=\left(\begin{smallmatrix} 0 & -1\\ 1 & \phantom{-}0\end{smallmatrix}\right)
\text{ et }U=\left(\begin{smallmatrix} \phantom{-}0 & 1\\-1 & 1\end{smallmatrix}\right).$$
2) Toute conjonction de deux des trois assertions (\ref{relSS}), (\ref{relUU}) et (\ref{intral1}) implique la troisième. En particulier, la minimalité du système d'équation (\ref{intral1}) implique la minimalité des relations de Manin généralisés. 
\end{prop}

Les relations de type linéaire permettent de définir un sous-module de $V_{k_1,k_2}^{\Z}$:
\begin{equation}
W_{k_1,k_2}^{\Z}=\{P\in V_{k_1,k_2}^{\Z}\text{ vérifiant les relations linéaires (\ref{intral1}) }\}.
\end{equation}
Posons $W_{k_1,k_2}=W_{k_1,k_2}^{\Z}\otimes\C$. Cet espace contient le $\Q$-espace vectoriel des polynômes des bipériodes à coefficients de Fourier rationnelles:
\begin{equation}
\Per_{k_1,k_2}^{\Q}=\{P_{f_1,f_2}\text{ tel que }(f_1,f_2)\in S_{k_1}^{\Q}\times S_{k_2}^{\Q}\}\subset W_{k_1,k_2}.
\end{equation}
\nomenclature{$\Per_{k_1,k_2}^{\Q}$}{Espace des polynômes des bipériodes des formes de $S_{k_1}^{\Q}\otimes S_{k_2}^{\Q}$}
\nomenclature{$W_{k_1,k_2}^{A}$}{Sous-module de $V_{k_1,k_2}^A$ annulé par $\mathcal{I}_2$}

La $\Q$-structure de $S_{k_1}\times S_{k_2}$ donnée par $S_{k_1}^{\Q}\times S_{k_2}^{\Q}$ permet de considérer, pour tout corps de nombres $K$, l'ensemble des polynômes des bipériodes à coefficients de Fourier dans $K$, noté $\Per_{k_1,k_2}^K$ et défini par:
\begin{equation}
\Per_{k_1,k_2}^K=\Per_{k_1,k_2}^{\Q}\otimes_{\Q} K=\{P_{f_1,f_2}\text{ tel que }(f_1,f_2)\in S_{k_1}^{K}\times S_{k_2}^{K}\}.
\end{equation}
Nous noterons simplement $\Per_{k_1,k_2}=\Per_{k_1,k_2}^{\C}$ l'ensemble des polynômes des périodes et on a ainsi $\Per_{k_1,k_2}\subset W_{k_1,k_2}$.\par

Lorsque $n=1$, on dispose pour tout entier pair $k\geq 4$ du sous-groupe de $V_k^{\Z}$:
$$W_{k}^{\Z}=\{P\in\Z_{k-2}[X]\text{ tel que }P|_{1+S}=P|_{1+U+U^2}=0\}.$$
Le groupe $W_{k}^{\Z}$ contient un élément distingué $1-X^{k-2}$ proportionnel sur $\C$ à $P_{G_k}^+$ la partie paire du polynôme des périodes de la $k$-ième série d'Eisenstein. Posons $E_k^{\Z}=(1-X^{k-2})\Z\subset W_k^{\Z}$. De plus, comme vu le Chapitre $1$ section $1.6$, le $\Z$-module $W_k^{\Z}$ se scinde suivant les parties paires et impaires des polynômes: 
$$W_k^{\pm ,\Z}=\{P\in W_k^{\Z}\text{ tel que }P(-X)=\pm P(X)\}\text{ et }W_k^{\Z}=W_k^{+,\Z}\oplus W_k^{-,\Z}.$$
\nomenclature{$\Per_{k}^{\Q}$}{Espace des polynômes des bipériodes des formes de $S_{k}^{\Q}$}
\nomenclature{$W_k^{A}$}{Sous-module de $V_k^A$ annulé par $\mathcal{I}_1$}
\nomenclature{$E_k^{A}$}{Sous-module de rang $1$ de $W_k^A$ engendré par $1-X^{k-2}$}

Définissons le sous-groupe de $V_{k_1,k_2}^{\Z}$ annulé par les relations de Manin diagonales par:
\begin{equation}
V_{k_1,k_2}^{\Z}[I_D]=\{P\in V_{k_1,k_2}^{\Z}\text{ tel que }P|_{(1,1)+(S,S)}=P|_{(1,1)+(U,U)+(U^2,U^2)}=0\}.
\end{equation}
Par analogie avec le cas où $n=1$ nécessitant les formes non paraboliques, posons dans $V_{k_1,k_2}^{\Z}$ :
\begin{equation}
E_{k_1,k_2}^{\Z}=\left(W_{k_1}^{\Z}\otimes 1\right)+ V_{k_1,k_2}^{\Z}[I_D]+ \left(X_1^{k_1-2}\otimes W_{k_2}^{\Z}\right)\subset W_{k_1,k_2}^{\Z}.
\end{equation}
\nomenclature{$E_{k_1,k_2}^{A}$}{Partie Eisenstein d'ordre $2$ de $W_{k_1,k_2}^A$}

\begin{thm}
L'espace vectoriel $W_{k_1,k_2}^{\Q}$ est le plus petit $\Q$-sous-espace-vectoriel de $V_{k_1,k_2}^{\Q}$ contenant $E_{k_1,k_2}^{\Q}$ tel que son extension au corps des complexes contient $\Per_{k_1,k_2}$.
\end{thm}

\begin{rem}
Pourtant la dépendance linéaire sur $\Q$ des bipériodes d'un couple de formes paraboliques holomorphes est lié à la minimalité de $W_{k_1,k_2}^{\Q}$ parmi les $\Q$-espaces vectoriels $W$ tel que $\Per_{k_1,k_2}\subset W\otimes\C$. Comme dans le cas $n=1$, il y a donc un terme d'erreur provenant des séries d'Eisenstein. Nous pouvons préciser à nouveau l'écart entre $\Per_{k_1,k_2}$ et $W_{k_1,k_2}$.\par
\end{rem}

On peut décomposer toute partie $U^{\Z}\subset V_{k_1,k_2}^{\Z}$ en parties paire et impaire, $U^{\pm,\Z}=U^{\Z}\cap V_{k_1,k_2}^{\pm,\Z}$ où:
\begin{equation}
V_{k_1,k_2}^{\pm,\Z}=\{P\in V_{k_1,k_2}^{\Z}\text{ tel que }P(-X_1,-X_2)=\pm P(X_1,X_2)\}.
\end{equation}

\begin{thm}
On a une suite exacte de groupes abéliens:
$$0\to\Z\to W_{k_1}^{\Z} \times V_{k_1,k_2}^{\Z}[I_D]\times W_{k_2}^{\Z}\to W_{k_1,k_2}^{\Z} \stackrel{\Phi_{S}}{\longrightarrow} W_{k_1}^{\Z}/E_{k_1}^{\Z} \otimes W_{k_2}^{\Z}/E_{k_2}^{\Z}\to 0,$$
où les applications non triviales sont données dans l'ordre par les applications:
$$1\mapsto \left(1-X_1^{k_1-2},1-X_1^{k_1-2}X_2^{k_2-2},X_1^{k_1-2}(1-X_2^{k_2-2})\right),$$
$$(p_1,p_2,p_3)\mapsto p_1-p_2+p_3\quad\text{ et }\quad \Phi_{S} : P\mapsto\left[P|_{(1,1)+(S,S)}\right].$$
De plus, l'application $\Phi_{S}\otimes \C$ peut être restreinte aux $\C$-espaces vectoriels $W_{k_1,k_2}^{\pm}$ et pour tout couple de signes $(\epsilon_1,\epsilon_2)\in\{\pm 1\}^2$, les applications suivantes sont des bijections:
$$[\Phi_S]_{\epsilon_1,\epsilon_2}: W_{k_1,k_2}^{\epsilon_1\epsilon_2}/E_{k_1,k_2}^{\epsilon_1\epsilon_2}\to \left(W_{k_1}^{\epsilon_1}/E_{k_1}^{\epsilon_1}\right)\otimes \left(W_{k_2}^{\epsilon_2}/E_{k_2}^{\epsilon_2}\right).$$
\end{thm}

\begin{prop}
On a les décompositions en sommes directes de $\C$-espaces vectoriels:
\begin{equation}
W_{k_1,k_2}=\Per_{k_1,k_2}\oplus\overline{\Per_{k_1,k_2}}\oplus E_{k_1,k_2},
\end{equation}
\begin{equation}
W_{k_1,k_2}^+=\Per_{k_1,k_2}^{+}\oplus E_{k_1,k_2}^{+}\text{ et }
W_{k_1,k_2}^-=\Per_{k_1,k_2}^{-}\oplus E_{k_1,k_2}^{-}.
\end{equation}
\end{prop}

\section{Relations vérifiées par les polynômes des bipériodes}

Nous nous proposons d'examiner si les relations (\ref{intral1}) définissant $W_{k_1,k_2}$ sont optimales. C'est-à-dire nous allons étudier l'idéal à droite de $\Z[\Gamma^2]$:
\begin{equation}
\mathcal{J}(k_1,k_2)=\{g\in\Z[\Gamma^2]\text{ tel que }P_{f_1,f_2}|_{g}=0,\text{ pour tout }(f_1,f_2)\in S_{k_1}^{\Q}\times S_{k_2}^{\Q}\}
\end{equation}
\nomenclature{$\mathcal{J}(k_1,k_2)$}{Idéal de $\Z[\Gamma^2]$ annulateur de $Per_{k_1,k_2}^{\Q}$}
Pour tout idéal $I$ à droite (resp. à gauche) de $\Z[\Gamma^2]$, notons $\widetilde{I}\subset\Z[\Gamma^2]$ l'idéal à gauche (resp. à droite) image par l'antiautomorphisme: $\sum_{\gamma_1,\gamma_2} \lambda_{\gamma_1,\gamma_2}[\gamma_1,\gamma_2]\mapsto \sum_{\gamma_1,\gamma_2} \lambda_{\gamma_1,\gamma_2}[\gamma_1^{-1},\gamma_2^{-1}]$.\par
On va construire un idéal à gauche $\mathcal{I}_2\subset\widetilde{\mathcal{J}(k_1,k_2)}$, indépendant des poids $k_1$ et $k_2$, défini entièrement par les propriétés topologiques de $\H^2$ sous l'action de $\Gamma^2$.\par
\nomenclature{$\widetilde{I}$}{Idéal à droite (resp. à gauche) de $\Z[\Gamma^n]$ associé à l'idéal à gauche (resp. à droite) $I$}
Notre démarche peut être comprise comme la généralisation à $n=2$ de l'étude accomplie par Manin \cite{Ma2} pour $n=1$. Nous redécrirons brièvement dans le cadre mis en place cette recherche des relations vérifiées par les polynômes des périodes.

\subsection{Homologie singulière relative aux pointes}

Nous introduisons une homologie singulière. Pour référence, on pourra se rapporter au livre de Hatcher \cite{Hat}.
Soit $X$ un espace topologique et $X_0$ une partie fermée de $X$. Soit $m\geq 0$ un entier. On définit le simplexe fondamental de dimension $m$ par:
$$\Delta_m=\{(t_0,...,t_m)\in [0,1]^{m+1}\text{ tel que }\sum_{j=0}^m t_j =1\}.$$
Et l'ensemble de ses sommets par: $\Delta_m^0=\Delta_m\cap \Z^{m+1}=\{e_j^m=(...,0,1,0,...),j=0...n\}$.\par
Définissons $M_m(X,X_0,\Z)$ le $\Z$-module libre engendré par les $m$-cycles de $X$ aux sommets dans $X_0$ à homotopie près. C'est-à-dire la classes des applications continues:
$$C:\Delta_m\to X,\text{ telle que }C(\Delta_m^0)\subset X_0.$$
Et $C_0\equiv C_1$ si et seulement si il existe une application continue:
$$h:\Delta_m\times [0,1]\to X\text{ tel que }h(u,0)=C_0(u)\text{ et }h(u,1)=C_1(u)\text{ pour tout }u\in\Delta_m,$$
et pour tout $t\in [0,1]$ et $p_0\in\Delta_m^0$, on a $h(p_0,t)\in X_0$.\par
Ceci permet de considérer les applications de bord:
$$\delta_m: M_{m}(X,X_0,\Z)\to M_{m-1}(X,X_0,\Z),\quad [C] \mapsto \sum_{j=0}^{m} (-1)^j [C\circ\delta_j^m],$$
où $\delta_j^m:\Delta_{m-1}\to\Delta_{m}, (t_0,...,t_{m-1})\mapsto (t_0,...,t_{j-1},0,t_{j},...,t_{m-1})$.\par
Elles vérifient pour tout entier $m\geq 0$, $\delta_m\circ\delta_{m+1}=0$. Ceci permet de considérer les groupes d'homologie singulière:
$$H_{m}(X,X_0,\Z)=\Ker\left(\delta_m|M_{m}(X,X_0,\Z)\right)/Im\left(\delta_{m+1}|M_{m}(X,X_0,\Z)\right).$$
\nomenclature{$M_m(X,X_0,\Z)$}{Module des $m$-chaînes de $X$ relatives à $X_0$}
\nomenclature{$H_m(X,X_0,\Z)$}{Groupe d'homologie singulières de $X$ relatives à $X_0$}
Dans notre cadre, nous étudierons des parties $X\subset\H^n$ pour un entier $n\geq 1$ donné et elles seront associées à des ensembles de sommets $X_0=\pte^n\cap X$. Ainsi pour toute partie de ce type nous noterons simplement:
$$M_m^{pte}(X,\Z)=M_m(X,X_0,\Z)\text{ et }H_m^{pte}(X,\Z)=H_m(X,X_0,\Z).$$\par
De plus, nous remarquons que ces espaces de sommets $X_0$ sont discrets ainsi la continuité des restrictions $h_0:\Delta_m^0\times [0,1]\to X_0$ démontre qu'elles sont constantes suivants la seconde variable. Ainsi dans le cadre de notre étude, les images des sommets ne dépend pas du représentant mais seulement de la classes d'homotopie.

\subsection{Cas du polynôme des périodes (Rappel)}

Nous rappelons les résultats des travaux de Manin \cite{Ma2}, Eichler \cite{Eich57} et Shimura \cite{Sh59} que nous généralisons pour $n=2$. Ils sont traités plus en détails dans le Chapitre $1$.\par

Soit $k\geq 4$ un entier pair. Définissons l'idéal à droite de $\Z[\Gamma]$:
\begin{equation*}
\mathcal{J}(k)=\{g\in\Z[\Gamma]\text{ tel que }P_f|_{g}=0\text{ pour tout }f\in S_k^{\Q}
\text{ et }(1-X^{k-2})|_{g}=0\}.
\end{equation*}
Notons $\Per_k=\{P_f;f\in S_k \}$ l'ensemble des polynômes des périodes.
Notons $\tau_1=\{0,i\infty\}\in M_1^{pte}(\H,\Z)$ la classe d'homotopie d'un chemin de $\H$ reliant les pointes $0$ et $i\infty$. 
\nomenclature{$\mathcal{J}(k)$}{Idéal annulateur de $1-X^{k-2}$ et $Per_k^{\Q}$}
\begin{thm}[Relations de Manin]
Définissons $\mathcal{I}_1=\{g\in\Z[\Gamma]\text{ tel que }g.\delta_1\tau_1=0\}$. Alors pour tout entier $k$:
\begin{equation}
\Per_k\subset V_k[\mathcal{I}_1]\text{ et }V_k^{\Q}[\mathcal{I}_1]=V_k^{\Q}[\mathcal{J}(k)].
\end{equation}
De plus, l'idéal $\mathcal{I}_1$ est de type fini et est donné par:
\begin{equation}
\mathcal{I}_1=\Z[\Gamma](1+S,1+U+U^2).
\end{equation}
Les relations de Manin sur les polynômes des périodes sont ainsi:
$$P_f|_{1+S}=P_f|_{1+U+U^2}=0,\text{ pour tout }f\in S_k.$$
\end{thm}
\nomenclature{$\tau_1$}{La $1$-chaine reliant $i\infty$ à $0$}
\begin{proof}
On a $\tau_1=\{i\infty,0\}$ et donc $\delta_1 \tau_1=(i\infty)-(0)$. Or les actions de $S$ et $U$ sur ces pointes sont données par:
$$i\infty\stackrel{S}{\longrightarrow}0\stackrel{S}{\longrightarrow}i\infty\text{ et }i\infty\stackrel{U}{\longrightarrow}0\stackrel{U}{\longrightarrow}1\stackrel{U}{\longrightarrow}i\infty.$$
On obtient ainsi $\Z[\Gamma](1+S,1+U+U^2)= \mathcal{I}_1$. Comme $\tilde{S}=S, \tilde{U}=U^2$ et $\tilde{U^2}=U$ alors on a : $\widetilde{\mathcal{I}_1}=(1+S,1+U+U^2)\Z[\Gamma]$. On peut vérifier que ces relations annulent les polynômes des périodes ainsi que le polynôme $1-X^{k-2}$, donc $\widetilde{\mathcal{I}_1}\subset\mathcal{J}(k)$ et ainsi $V_{k}^{\Q}[\mathcal{J}(k)]\subset V_{k}^{\Q}[\mathcal{I}_1]$. L'analyse des dimensions de ces sous-espaces vectoriels de $V_k^{\Q}$ donne alors leurs égalités.
\end{proof}

\begin{rem}
a) La représentation $\Gamma\to GL(V_k^{\Q})$ est irréductible. En effet, la famille $(X^{k-2}|_{T^j})_{0\leq j\leq k-2}$ est une base de $V_k^{\Q}$. Ainsi on déduit la surjectivité du morphisme de $\Q[\Gamma]$-modules obtenue par $\Q$-linéarité: $\Q[\Gamma]\to End_{\Q}(V_k^{\Q})$. Ceci permet notamment de déduire que tout sous-$\Q$-espace vectoriel $W\subset V_{k}^{\Q}$ peut s'écrire comme l'annulateur $W=V_k^{\Q}[g]$ d'un élément de $g\in\Q[\Gamma]$. Puis si l'on suppose qu'un tel espace vérifie $\Per_k\subset W\otimes\C$ et $1-X^{k-2}\in W$ alors on a $\gamma\in \mathcal{J}(k)$ puis $W_k^{\Q}\subset W$. Ceci démontre que $W_k^{\Q}$ est le plus petit sous-$\Q$-espace vectoriel vérifiant ces propriétés.\par
b) Posons $\varepsilon=\left(\begin{smallmatrix}-1 & 0\\ \phantom{-}0 & 1 \end{smallmatrix}\right)$. Alors $V_k^+=V_k[1-\varepsilon]$, $V_k^-=V_k[1+\varepsilon]$ et $W_k=V_k[1+S,1+U+U^2]$. Ce sont des éléments de $\Z[\Gamma]$ qui représentent des opérateurs sur $V_k$ qui commutent, en effet:
$$\varepsilon(1+S)\varepsilon=1+S\in\mathcal{I}_1$$
$$\text{et }\varepsilon(1+U+U^2)\varepsilon=S(1+U+U^2)(1+S)-S(1+U+U^2)\in\mathcal{I}_1.$$
Ainsi $\varepsilon\mathcal{I}_1\varepsilon=\mathcal{I}_1$. Ceci nous permet de considérer les parties paires et impaires des polynômes tout en conservant la structure des relations de Manin. Et ainsi le $\Z$-module $W_k^{\Z}=\{P\in V_k^{\Z}\text{ tel que }P|_{1+S}=P|_{1+U+U^2}=0\}$ se scinde bien en parties paire et impaire.\par
\end{rem}

\begin{thm}[Eichler-Shimura]
On dispose des décompositions en somme directe de $\C$-espaces vectoriels:
\begin{align}
W_k&=\Per_k\oplus \overline{\Per_k}\oplus E_k,\\
W_k^+&=\Per_k^{+}\oplus E_k\text{ et }W_k^-=\Per_k^{-}.
\end{align}
\end{thm}

\subsection{Dualité des actions de $\Z[\Gamma^2]$ sur $V_{k_1,k_2}$ et sur $M_2^{pte}(\H^2,\Z)$}

Le groupe $M_2^{pte}(\H^2,\Z)$ est le $\Z$-module librement engendré par les $2$-cycles de $\H^2$ aux sommets dans $\pte^2$. Rappelons qu'un tel $2$-cycle est une classe d'équivalence aux homotopies fixant les sommets près des applications continues :
$$C:\Delta_2=\{(t_0,t_1,t_2)\in[0,1]^3\text{ tel que }t_0+t_1+t_2=1\}\to \H^2,$$
vérifiant $C(1,0,0),C(0,1,0),C(0,0,1)\in\pte^2$. L'application:
\begin{equation}
\Delta_2\to\H^2,\quad(t_0,t_1,t_2)\mapsto \left(-i\log(t_0),-i\log(t_0+t_1)\right),
\end{equation}
définit un élément $\tau_2$ de $M_2^{pte}(\H^2,\Z)$ car on a dans $\pte^2$ :
$$\tau_2(1,0,0)=(0,0),\quad \tau_2(0,1,0)=(i\infty,0)\text{ et }\tau_2(0,0,1)=(i\infty,i\infty).$$
L'action diagonale de $\Gamma^2$ sur $\H^2$ produit une action à gauche de $\Z[\Gamma^2]$ sur $M_2^{pte}(\H^2,\Z)$.\par

\nomenclature{$\tau_2$}{La $2$-chaîne associée au simplexe de $\H^2$}

Soit $\Omega_{par}^2(\H^2,\C)$ le $\C$-espace vectoriel des $2$-formes différentielles harmoniques de $\H^2$ et nulles sur le bord $\partial\H^2=\pte\times\H\cup\H\times\pte$. Cette propriété permet notamment d'obtenir la convergence de l'intégration d'une telle forme le long de $\tau_2$.\par
Pour tout couple $(f_1,f_2)\in S_{k_1}\times S_{k_2}$ et tout couple d'entiers critiques $(m_1,m_2)$ (i.e. vérifiant $1\leq m_j\leq k_j-1$), considérons la $2$-forme :
$$f_1(z_1)z_1^{m_1-1}f_2(z_2)z_2^{m_2-1}dz_1\wedge dz_2\in\Omega_{par}^2(\H^2,\C).$$
Cette famille peut être indexée par $X_1^{k_1-m_1-1}X_2^{k_2-m_2-1}$ et après renormalisation, posons:
\begin{equation}
\omega_{f_1,f_2}(z_1,z_2;X_1,X_2)=\omega_{f_1}(z_1,X_1)\wedge\omega_{f_2}(z_2,X_2)\in\Omega_{par}^2(\H^2,\C)\otimes_{\C} V_{k_1,k_2}=\Omega_{par}^2(\H^2,V_{k_1,k_2}),
\end{equation} 
où l'on rappelle la notation pour $f\in S_k $ de la forme : $\omega_f(z,X)=f(z)(X-z)^{k-2}dz.$\par

Considérons l'accouplement défini par:
\begin{align*}
\Omega_{par}^2(\H^2,\C)\times M_2^{pte}(\H^2,\Z)&\to\C,\\
(\omega,C)\mapsto \langle\omega, C\rangle=\int_{C}\omega.
\end{align*}

On rappelle l'écriture du polynôme des bipériodes introduit dans le chapitre $2$ à la section $2.4.4$ en fonction de la $2$-forme associée $\omega_{f_1,f_2}$ et du $2$-cycle $\tau_2$:
\begin{equation}
P_{f_1,f_2}(X_1,X_2)=\langle \omega_{f_1,f_2},\tau_2\rangle.
\end{equation}\par
Le groupe $\Gamma^2$ agit à la fois sur $V_{k_1,k_2}$ et sur $M_2^{pte}(\H^2,\Z)$. Ces actions peuvent être reliées grâce aux propriétés d'invariances suivantes.\par

Pour tout $\gamma\in\Gamma$ et toute forme $f\in S_k $, la forme différentielle associée $\omega_f$ vérifie:
\begin{equation}
\omega_f(\gamma.z,X|_{\gamma})=\omega_f(z,X).\label{relmod31}
\end{equation}\par
De plus, pour toute $2$-forme différentielle $\omega\in\Omega^2_{par}(\H^2,\C)$ et toute $2$-chaîne $C\in M_2^{pte}(\H^2,\Z)$, on dispose de la formule:
\begin{equation}
\langle\omega(\gamma.z),C \rangle=\int_{C}\omega(\gamma.z)=\int_{\gamma.C}\omega(z)=\langle\omega(z),\gamma.C\rangle.\label{relacc31}
\end{equation}

\begin{prop}\label{prop37}
Les actions de $\Gamma^2$ sur $V_{k_1,k_2}$ et $ M_2^{pte}(\H^2,\Z)$ sont duales:
\begin{equation}
P_{f_1,f_2}(X_1|_{\gamma_1},X_2|_{\gamma_2})=\langle\omega_{f_1,f_2}, (\gamma_1^{-1},\gamma_2^{-1})\tau_2\rangle,
\end{equation}
pour tout couple $(\gamma_1,\gamma_2)\in\Gamma^2$.
\end{prop}

\begin{proof}
Soit $(\gamma_1,\gamma_2)\in\Gamma^2$. Son action sur le polynôme des bipériodes devient:
\begin{align*}
P_{f_1,f_2}(X_1,X_2)|_{(\gamma_1,\gamma_2)}
&=\langle \omega_{f_1}(z_1,X_1|_{\gamma_1})\wedge\omega_{f_2}(z_2,X_2|_{\gamma_2}),\tau_2\rangle\\
&=\langle \omega_{f_1}((\gamma_1)^{-1}.z_1,X_1)\wedge\omega_{f_2}((\gamma_2)^{-1}.z_2,X_2),\tau_2\rangle\text{ par (\ref{relmod31})}\\
&=\langle\omega_{f_1}(z_1,X_1)\wedge\omega_{f_2}(z_2,X_2),(\gamma_1^{-1},\gamma_2^{-1}).\tau_2\rangle\text{ par (\ref{relacc31})}.
\end{align*}
La proposition s'étend par linéarité à $\Z[\Gamma^2]$.
\end{proof}

Définissons le sous-espace $\Omega_{k_1,k_2}^{\holo}$ de $\Omega_{par}^2(\H^2,V_{k_1,k_2})$ comme l'image de l'application $\Q$-linéaire et injective suivante:
\nomenclature{$\Omega_{k_1,k_2}^{\holo}$}{Sous-espace des $2$-formes modulaires biholomorphes de poids $(k_1,k_2)$}
\nomenclature{$\Omega_{k_1,k_2}^{\Q}$}{Sous-espace des $2$-formes modulaires biholomorphes et bi-anti-holomorphes de poids $(k_1,k_2)$}
\begin{equation*}
S_{k_1}^{\Q} \otimes S_{k_2}^{\Q} \to\Omega_{par}^2(\H^2,V_{k_1,k_2}),\quad f_1\otimes f_2\mapsto \omega_{f_1}\wedge\omega_{f_2}.
\end{equation*}
Nous noterons $\overline{\Omega_{k_1,k_2}^{\holo}}\subset \Omega_{par}^2(\H^2,V_{k_1,k_2})$ son conjugué complexe, image de l'application $\Q$-linéaire et injective:
\begin{equation*}
S_{k_1}^{\Q} \otimes S_{k_2}^{\Q} \to\Omega_{par}^2(\H^2,V_{k_1,k_2}),\quad f_1\otimes f_2\mapsto \overline{\omega_{f_1}}\wedge\overline{\omega_{f_2}}.
\end{equation*}
Considérons ainsi la somme directe stable par conjugaison complexe suivante:
\begin{equation}
\Omega_{k_1,k_2}^{\Q}=\Omega_{k_1,k_2}^{\holo}\oplus\overline{\Omega_{k_1,k_2}^{\holo}}.\label{sumomega2}
\end{equation}\par

Définissons pour tout ensemble $\Omega$ de $2$-formes harmoniques sur $\H^2$, son orthogonal dans $M_2^{pte}(\H^2,\Z)$ par:
\begin{equation}
\Omega^{\bot}=\{C\in M_2^{pte}(\H^2,\Z)\text{ tel que }\langle \omega,C\rangle=0,\text{ pour tout }\omega\in\Omega\}.
\end{equation}
On remarquera que $\overline{\Omega}^{\bot}=\Omega^{\bot}$ d'après le simple calcul:
$$\forall\omega\in\Omega^2,\forall C\in M_2,\langle \omega,C\rangle=0 \Leftrightarrow \langle \overline{\omega},C\rangle=\overline{\langle \omega, C\rangle}=0.$$
 Ceci permet de récrire l'idéal $\mathcal{J}(k_1,k_2)$:

\begin{prop}
On a $\widetilde{\mathcal{J}(k_1,k_2)}=\{g\in \Z[\Gamma^2]\text{ tel que }g.\tau_2\in \left(\Omega_{k_1,k_2}^{\Q}\right)^{\bot}\}$.
\end{prop}

\begin{proof}
La Proposition \ref{prop37} donne pour tout élément $\sum \lambda_{\gamma_1,\gamma_2}(\gamma_1,\gamma_2)\in\Z[\Gamma^2]$:
\begin{align*}
&\forall (f_1,f_2)\in S_{k_1}^{\Q} \times S_{k_2}^{\Q} , P_{f_1,f_2}|_{\sum \lambda_{\gamma_1,\gamma_2}(\gamma_1,\gamma_2)}=0\\
&\Leftrightarrow \forall (f_1,f_2)\in S_{k_1}^{\Q} \times S_{k_2}^{\Q} , \sum \lambda_{\gamma_1,\gamma_2}\langle \omega_{f_1}\wedge\omega_{f_2}, (\gamma_1^{-1},\gamma_2^{-1}).\tau_2\rangle=0\\
&\Leftrightarrow \forall \omega\in \Omega_{k_1,k_2}^{\holo}, \sum \lambda_{\gamma_1,\gamma_2}\langle \omega, (\gamma_1^{-1},\gamma_2^{-1}).\tau_2\rangle=0\\
&\Leftrightarrow \sum \lambda_{\gamma_1,\gamma_2}(\gamma_1^{-1},\gamma_2^{-1})\tau_2\in \left(\Omega_{k_1,k_2}^{\holo}\right)^{\bot}.
\end{align*}
Or on a remarqué que: $\left(\Omega_{k_1,k_2}^{\holo}\right)^{\bot}=\left(\overline{\Omega_{k_1,k_2}^{\holo}}\right)^{\bot}.$ Permettant de déduire:
$$\left(\Omega_{k_1,k_2}^{\holo}\right)^{\bot}=\left(\Omega_{k_1,k_2}^{\holo}\right)^{\bot}\cap\left(\overline{\Omega_{k_1,k_2}^{\holo}}\right)^{\bot}=\left(\Omega_{k_1,k_2}^{\Q}\right)^{\bot}.$$
La dernière égalité provenant de la définition (\ref{sumomega2}).
\end{proof}

Il nous reste alors à déterminer $\left(\Omega_{k_1,k_2}^{\Q}\right)^{\bot}$.

\subsection{Espaces transverses de $ M_1^{pte}(\H^2,\Z)$}

Pour calculer $(\Omega_{k_1,k_2}^{\Q})^{\bot}$, nous introduisons des \textit{espaces transverses} de $M_1^{pte}(\H^2,\Z)$. Soient $g\in\Gamma$ et $c\in M_1^{pte}(\H,\Z)$. 
Définissons $H_g$, $V_g$ et $D_g$ des applications de $M_1^{pte}(\H,\Z)$ dans $M_1^{pte}(\H^2,\Z)$ données respectivement par les classes des applications:
\begin{align}
H_g(c):\Delta_1\to\H^2,\quad& (t_0,t_1) \mapsto (c(t_0,t_1),g.i\infty),\\
V_g(c):\Delta_1\to\H^2,\quad& (t_0,t_1) \mapsto (g.0,c(t_0,t_1)),\\
\text{et }D_g(c):\Delta_1\to\H^2,\quad& (t_0,t_1) \mapsto (c(t_0,t_1),g.c(t_0,t_1)).
\end{align}
Considérons leurs images:
\begin{align}
H=&\{H_g(c)\text{ pour }c\in M_1^{pte}(\H,\Z)\text{ et }g\in\Gamma\},\\
V=&\{V_g(c)\text{ pour }c\in M_1^{pte}(\H,\Z)\text{ et }g\in\Gamma\},\\
\text{et }D=&\{D_g(c)\text{ pour }c\in M_1^{pte}(\H,\Z)\text{ et }g\in\Gamma\}.
\end{align}
On qualifiera de \textit{transverses} les chaînes de $H+V+D\subset M_1^{pte}(\H^2,\Z)$.\\
Notons $H^0$, $V^0$ et $D^0$ les sous-groupe constitués par les chaînes fermées, c'est-à-dire de bord nul, de $H$, $V$ et $D$ respectivement.

Posons:
\begin{equation}
\mathcal{I}_2=\{g\in\Z[\Gamma^2]\text{ tel que }g.\delta_2\tau_2\in H^0+V^0+D^0\}.
\end{equation}

Pour $0\leq j\leq 2$, posons $\varphi_j:\H\to\H^2$ défini pour $z\in\H$ par:
\begin{equation}
\varphi_0(z)=(i\infty,z),\quad\varphi_1(z)=(z,z)\text{ et}\quad\varphi_2(z)=(z,0).
\end{equation}
Et pour $(\gamma_1,\gamma_2)\in\Gamma^2, \psi_{\gamma_1,\gamma_2}:\H^2\to\H^2$ défini pour $(z_1,z_2)\in\H^2$ par :
\begin{equation}
\psi_{\gamma_1,\gamma_2}(z_1,z_2)=(\gamma_1 z_1,\gamma_2 z_2).
\end{equation}

Définissons le sous-espace de $\Omega_{par}^2(\H^2,V_{k_1,k_2})$ par:
\begin{multline*}
\Omega^2_{par}(\H^2,V_{k_1,k_2})^{\Gamma^2}=\Omega^2_{par}(\H^2,\C)\otimes_{\C[\Gamma^2]}V_{k_1,k_2}\\
=\{\omega\in\Omega_{par}^2(\H^2,V_{k_1,k_2})\text{ tel que }(\psi_{\gamma}^*\omega)|_{\gamma}=\omega\text{ pour tout }\gamma\in\Gamma^2\}\\
=\{\omega\in\Omega_{par}^2(\H^2,V_{k_1,k_2});\omega(\gamma_1.z_1,\gamma_2.z_2,X_1|_{\gamma_1},X_2|_{\gamma_2})=\omega(z_1,z_2,X_1,X_2),\forall(\gamma_1,\gamma_2)\in\Gamma^2\}.
\end{multline*}
Alors on a le lemme essentiel suivant:

\begin{lem}
Soit $\omega\in\Omega_{par}^2(\H^2,V_{k_1,k_2})^{\Gamma^2}$ vérifiant $d\omega=\varphi_0^*\omega=\varphi_1^*\omega=\varphi_2^*\omega=0$ alors, on a :
\begin{equation}
\omega\in\Omega_{k_1,k_2}=\Omega_{k_1,k_2}^{\Q}\otimes \C.
\end{equation}
\end{lem}

\begin{proof}
En effet, une $2$-forme harmonique et exacte sur $\H^2$ s'écrit comme somme de six termes:
\begin{align*}
&F_1(z_1,\bar{z_1})dz_1\wedge d\bar{z_1}+F_2(z_1,z_2)dz_1\wedge dz_2+F_3(z_1,\bar{z_2})dz_1\wedge d\bar{z_2}\\
&F_4(\bar{z_1},z_2)d\bar{z_1}\wedge dz_2+F_5(\bar{z_1},\bar{z_2})d\bar{z_1}\wedge d\bar{z_2}+F_6(z_2,\bar{z_2})dz_2\wedge d\bar{z_2},
\end{align*}
où $F_j:\H^2\to V_{k_1,k_2}$, pour $j=1...6$, sont des applications holomorphes en les deux variables de sortes à obtenir des $2$-formes exactes.\par
La condition $0=\varphi_0^*\omega=F_1(z,\bar{z})dz\wedge d\bar{z}$ donne l'annulation des $2$-formes invariantes par $\Gamma^2$ :
$$f_{l_1}(z_1)g_{l_2}(\bar{z_1})(X_1-z_1)^{a}(X_1-\bar{z_1})^{b}(z_1-\bar{z_1})^c dz_1\wedge d\bar{z_1}P(X_2),$$
où $a+b=k_1-2$, $l_1=a+c+2$ et $l_2=b+c+2$ est une décomposition en entiers, $(f_{l_1},g_{l_2})\in M_{l_1} \times M_{l_2} $ et $P\in V_{k_2}^{\Gamma}$. Ce dernier nul car $k_2>2$.\\
De même, la condition $0=\varphi_2^*\omega=F_6(z,\bar{z})dz\wedge d\bar{z}$ donne l'annulation des $2$-formes invariantes par $\Gamma^2$ :
$$P(X_1)f_{l_1}(z_2)g_{l_2}(\bar{z_2})(X_2-z_2)^{a}(X_2-\bar{z_2})^{b}(z_2-\bar{z_2})^c dz_2\wedge d\bar{z_2},$$
où $a+b=k_2-2$, $l_1=a+c+2$ et $l_2=b+c+2$ est une décomposition en entiers, $(f_{l_1},g_{l_2})\in M_{l_1} \times M_{l_2} $ et $P\in V_{k_1}^{\Gamma}$. De la même manière, nul car $k_1>2$.\\
La condition $0=\varphi_1^*\omega=\left[F_3(z,\bar{z})-F_4(\bar{z},z)\right]dz\wedge d\bar{z}$ permet d'obtenir l'annulation des termes croisés invariants par $\Gamma^2$ de la forme:
$$f_1(z_1)(X_1-z_1)^{k_1-2}g_2(\bar{z_2})(X_2-\bar{z_2})^{k_2-2}dz_1\wedge d\bar{z_2}
+g_1(\bar{z_1})(X_1-\bar{z_1})^{k_1-2}f_2(z_2)(X_2-z_2)^{k_2-2}d\bar{z_1}\wedge dz_2,$$
où $f_1,g_1\in M_{k_1} $ et $f_2,g_2\in M_{k_2} $. Pour pouvoir séparer les termes par annulation des coefficients des monômes en $X_1$ et $X_2$ et ainsi déduire l'annulation de $F_3$ et $F_4$, il suffit d'avoir $k_1+k_2>4$ qui est vrai par hypothèse.\\
Il reste donc plus que les termes invariants par $\Gamma^2$ de la forme:
$$f_1(z_1)(X_1-z_1)^{k_1-2}f_2(z_2)(X_2-z_2)^{k_2-2}dz_1\wedge d\bar z_2
+g_1(\bar{z_1})(X_1-\bar{z_1})^{k_1-2}g_2(\bar{z_2})(X_2-\bar{z_2})^{k_2-2}d\bar{z_1}\wedge d\bar{z_2},$$
avec $f_1,g_1\in M_{k_1} $ et $f_2,g_2\in M_{k_2} $. L'annulation sur le bord $\partial\H^2$ donne alors le caractère parabolique des quatre formes modulaires. En effet, en prenant la limite pour $z_1\to i\infty$, on obtient:
$$0=a_0(f_1)f_2(z_2)(X_2-z_2)^{k_2-2}dz_2+a_0(g_1)g_2(\bar{z_2})(X_2-\bar{z_2})^{k_2-2}d\bar{z_2}.$$
C'est à dire $f_1$ et $g_1$ sont paraboliques. Et on obtient ainsi bien toutes les formes différentielles de $\Omega_{k_1,k_2}$.
\end{proof}

\begin{prop}\label{incljkk}
On a les inclusions des idéaux:
\begin{equation}
\mathcal{I}_2\subset\widetilde{\mathcal{J}(k_1,k_2)}.
\end{equation}
\end{prop}

\begin{proof}
On remarque que pour tout $0\leq j \leq 2$, $\{\varphi_j^*\omega=0\}^{\bot}=\varphi_j( M_2^{pte}(\H,\Z)).$\\
Par ailleurs, on a $\{d\omega=0\}^{\bot}=\Im(\delta_3)=\Ker(\delta_2)$. Et ainsi:
$$\Ker(\delta_2)+\sum_{j=0}^2 \varphi_j( M_2^{pte}(\H,\Z))\subset \{\varphi_0^*\omega=\varphi_1^*\omega=\varphi_2^*\omega=d\omega=0\}^{\bot}\subset\Omega_{k_1,k_2}^{\bot}.$$
On utilise alors l'invariance de $\Omega_{k_1,k_2}^{\bot}$ par $\Gamma^2$ en prenant la somme suivante:
$$\sum_{\gamma\in\Gamma^2}\psi_{\gamma}\left(\Ker(\delta_2)+\sum_{j=0}^2 \varphi_j( M_2^{pte}(\H,\Z))\right)\subset \Omega_{k_1,k_2}^{\bot}.$$
On a $\psi_{\gamma}(\Ker(\delta_2))=\Ker(\delta_2)$ et ainsi pour tout $g\in\Z[\Gamma^2]$:
\begin{align*}
g\tau_2&\in \Ker(\delta_2)+\sum_{\gamma\in\Gamma^2}\sum_{j=0}^2 \psi_{\gamma}\varphi_j( M_2^{pte}(\H,\Z))\\
\Leftrightarrow g\delta_2\tau_2=\delta_2 (g\tau_2)&\in H^0+D^0+V^0.
\end{align*}
Ce qui conclut bien la démonstration par définition des idéaux.
\end{proof}

\section{Calcul de l'idéal des relations $\mathcal{I}_2$}

\subsection{Décomposition en fonction de $\mathcal{I}_1$}

On peut décomposer $\delta_2 \tau_2$ en fonction de $\tau_1$. Ceci permet de lier les calculs de $\mathcal{I}_1$ et $\mathcal{I}_2$.
\begin{prop}\label{dtau2}
La $1$-chaîne $\delta_2 \tau_2$ est la somme alternée de trois $1$-cycles de $M_1^{pte}(\H^2,\Z)$:
\begin{equation}\label{dtau2eq}
\delta_2 \tau_2=H_{Id}(\tau_1)-D_{Id}(\tau_1)+V_{Id}(\tau_1).
\end{equation}
\end{prop}

\begin{proof}
Les bords d'une $2$-chaine sont donnés par la formule $\sum_{j=0}^2(-1)^j\delta_2^j$ où $\delta_2^j$ consiste à imposer $t_j=0$, ainsi le bord de $\tau_2$ est déterminé par:
\begin{align*}
\tau_2(0,t_1,t_2)&=\left(i\infty,-i\log(t_1)\right)=\left(i\infty,\tau_1(t_1,t_2)\right),\\
\tau_2(t_0,0,t_2)&=\left(-i\log(t_0),-i\log(t_0)\right)=\left(\tau_1(t_0,t_2),\tau_1(t_0,t_2)\right),\\
\text{et }\tau_2(t_0,t_1,0)&=\left(-i\log(t_0),0\right)=\left(\tau_1(t_0,t_1),0\right).
\end{align*}
Ce qui donne le résultat après récriture.
\end{proof}

Notons $\Gamma_{i\infty}$ et $\Gamma_0$ les sous-groupes de $\Gamma$ stabilisant les pointes $i\infty$ et $0$ respectivement.
\begin{lem}
a) L'action de $\Gamma^2$ sur $ M_1^{pte}(\H^2,\Z)$ respecte les ensembles $H$, $D$ et $V$ et on a:
\begin{align}
(\gamma_1,\gamma_2).H_g(c)&=H_{\gamma_2 g}(\gamma_1.c),\label{acth}\\
(\gamma_1,\gamma_2).V_g(c)&=V_{\gamma_1 g}(\gamma_2.c),\\
\text{et }(\gamma_1,\gamma_2).D_g(c)&=D_{\gamma_2 g \gamma_1^{-1}}(\gamma_1.c).
\end{align}
b) L'action de $\Gamma$ scinde chacun des $\C$-espaces vectoriels:
\begin{align}
H&=\bigoplus_{g\in\Gamma/\Gamma_{i\infty}} H_g( M_1^{pte}(\H,\Z)),\label{sumh}\\
V&=\bigoplus_{g\in\Gamma/\Gamma_{0}} V_g( M_1^{pte}(\H,\Z)),\\
\text{et }D&=\bigoplus_{g\in\Gamma} D_g( M_1^{pte}(\H,\Z)).\label{sumd}
\end{align}
c) Ces sous-groupes sont deux à deux disjoints.
\end{lem}

\begin{proof}
La démonstration de a) consiste à récrire la définition des objets.\\
Pour démontrer b), il suffit de montrer que les espaces sont deux à deux disjoints car la famille des espaces est génératrice par définition. Or les équations suivantes se résolvent :
\begin{align*}
(g_1.i\infty,c_1)=(g_2.i\infty,c_2)&\Rightarrow g_1\in g_2\Gamma_{i\infty},\\
(c_1,g_1.0)=(c_2,g_2.0)&\Rightarrow g_1\in g_2\Gamma_{0},\\
\text{et }(c_1,g_1.c_1)=(c_2,g_2.c_2)&\Rightarrow g_1=g_2.
\end{align*}
La démonstration de c) provient de la résolution des systèmes du type :
\begin{align*}
(g_1.i\infty,c_1(t_0,t_1))&=(c_2(t_0,t_1),g_2.0), \forall (t_0,t_1)\in\Delta_1 &\text{pour }V\cap H,\\
(g_1.i\infty,c_1(t_0,t_1))&=(c_2(t_0,t_1),g_2.c_2(t_0,t_1)), \forall (t_0,t_1)\in\Delta_1 &\text{pour }V\cap D,\\
\text{et }(c_1(t_0,t_1),g_1.c_1(t_0,t_1))&=(c_2(t_0,t_1),g_2.0), \forall (t_0,t_1)\in\Delta_1 &\text{pour }D\cap H.
\end{align*}
fixant dans chaque cas les chemins $c_1$ et $c_2$ comme étant constant. Donc la classe de $(c_1,c_2)$ est nulle dans $M_1^{pte}(\H^2,\Z)$.
\end{proof}

On définit les idéaux annulateurs des chemins transverses par:
\begin{align}
I_H&=\{g\in\Z[\Gamma^2]\text{ tel que }g.H_{id}(\{i\infty,0\})\in H^0\},\\ 
I_V&=\{g\in\Z[\Gamma^2]\text{ tel que }g.V_{id}(\{i\infty,0\})\in V^0\},\\
\text{et }I_D&=\{g\in\Z[\Gamma^2]\text{ tel que }g.D_{id}(\{i\infty,0\})\in D^0\}.
\end{align}

\begin{prop}\label{inclhvd}
L'idéal à gauche $\mathcal{I}_2$ de $\Z[\Gamma^2]$ est déterminé par l'intersection des idéaux annulateurs de ces trois segments:
\begin{equation}
\mathcal{I}_2=I_H\cap I_V \cap I_D.
\end{equation}
\end{prop}

\begin{proof}
L'inclusion $I_H\cap I_V \cap I_D\subset \mathcal{I}_2$ est une conséquence directe de la décomposition de $\delta_2 \tau_2$ vu en (\ref{dtau2eq}). Réciproquement, soit $g\in\mathcal{I}_2$ alors $g.\delta_2 \tau_2\in (H^0+V^0+D^0)$. On a:
$$g.H_{id}(\{i\infty,0\})=g.\delta_2\tau_2-g.V_{id}(\{i\infty,0\})+g.D_{id}(\{i\infty,0\})\in H\cap\left(H^0+V+D\right).$$
car ils sont stables par les chaînes transverses sont stables par $\Gamma$. Puis on a :
$$H\cap\left(H^0+V+D\right)=H^0+(H\cap V)+(H\cap D)=H^0.$$
En effet, les $\Z$-modules $H$, $V$ et $D$ sont $2$ à deux disjoints. De plus, ils sont sans torsions donc quitte à prendre les $\Q$-espaces vectoriels associés, on obtient bien l'égalité. Et ainsi on obtient $g\in I_H$. Ce raisonnement se symétrise et on obtient bien le résultat.
\end{proof}

\subsection{Calcul de $I_H,I_V$ et $I_D$}

Déterminons $I_H,I_V$ et $I_D$ par calcul direct.\par

\begin{prop}
Ces idéaux sont de type fini et on a:
\begin{align}
I_H&=\Big((1+S,1),(1+U+U^2,1),(1,1-T)\Big),\\
I_V&=\Big((1,1+S),(1,1+U+U^2),(1-US,1)\Big),\\
\text{et }I_D&=\Big((1,1)+(S,S),(1,1)+(U,U)+(U^2,U^2)\Big).
\end{align}
\end{prop}

\begin{proof}
Commençons par le calcul de $I_H$. Soit $\sum_i \lambda_i(\gamma_1^i,\gamma_2^i)\in\Z[\Gamma^2]$. Son action sur $H_{id}(\{i\infty,0\})$ s'exprime ainsi:
$$\sum_i \lambda_i(\gamma_1^i,\gamma_2^i). H_{id}(\{i\infty,0\})=\sum_i \lambda_i H_{\gamma_2^i}(\gamma_1^i.\{i\infty,0\})\text{ par (\ref{acth})}.$$
Supposons cet élément dans $I_H$ alors:
$$\sum_{g\in\Gamma/\Gamma_{i\infty}}\sum_{\gamma_2^i\in g\Gamma_{i\infty}}\lambda_i H_{g}(\gamma_1^i.\{i\infty,0\})
=\sum_{g\in\Gamma/\Gamma_{i\infty}}H_{g}\left(\sum_{\gamma_2^i\in g\Gamma_{i\infty}} \lambda_i \gamma_1^i. \{i\infty,0\}\right)\in H^0$$
Les espaces $H_g(M_1(\H,\Z))$ sont en somme directe d'après (\ref{sumh}). Donc pour chaque orbite $g\Gamma_{i\infty}\in\Gamma/\Gamma_{i\infty}$, on a:
$$H_{g}\left(\sum_{\gamma_2^i\in g\Gamma_{i\infty}} \lambda_i \gamma_1^i. \{i\infty,0\}\right)\in H^0,$$
et ainsi, on obtient:
$$\delta_1\sum_{\gamma_2^i\in g\Gamma_{i\infty}} \lambda_i \gamma_1^i. \{i\infty,0\}=\sum_{\gamma_2^i\in g\Gamma_{i\infty}} \lambda_i \gamma_1^i. \delta_1\{i\infty,0\}=0.$$

On déduit alors $\sum_{\gamma_2^i\in g\Gamma_{i\infty}} \lambda_i \gamma_1^i\in\mathcal{I}_1=(1+S,1+U+U^2)$.
De plus, le stabilisateur de la pointe $i\infty$ dans $\Gamma$ est:
$$\Gamma_{i\infty}=<T^a>_{a\in\Z}\text{ où }T=\left(\begin{smallmatrix}1 & 1\\ 0 & 1 \end{smallmatrix}\right)=U^2S.$$
On obtient ainsi:
\begin{equation*}
I_H=\Big((1,1)+(S,T^a),(1,1)+(U,T^a)+(U^2,T^b)\Big)_{a,b\in\Z}
\end{equation*}

Cet idéal est en faite de type fini. Les relations de Manin agissant sur le première coordonnées et la stabilité de la pointe $i\infty$ sur la seconde, on peut les scinder suivant:
$$(1,1)+(S,T^a)=(1+S,1)-(S,1-T^a),$$
$$\text{et }(1,1)+(U,T^a)+(U^2,T^b)=(1+U+U^2,1)-(U,1-T^a)-(U^2,1-T^b).$$
Et donc les éléments $(1+S,1)$, $(1+U+U^2,1$ et $(1,1-T)$ engendre $I_H$.\par

De la même manière, on réduit le calcul de $I_V$ à celui du stabilisateur de $0$:
$$\Gamma_0=<(US)^a>_{a\in\Z}\text{ où }US=\left(\begin{smallmatrix}1 & 0\\ 1 & 1 \end{smallmatrix}\right).$$
On en déduit:
\begin{equation*}
I_V=\Big((1,1)+((US)^a,S),(1,1)+((US)^a,U)+((US)^b,U^2)\Big)_{a,b\in\Z}
\end{equation*}
Une récriture de cette famille de ces générateurs donne de la même manière le résultat.\par

Pour le calcul de $I_D$, on rappel que $(\gamma_1,\gamma_2).D_I(c)=D_{\gamma_2\gamma_1^{-1}}(\gamma_1.c)$.
On peut alors réduire l'action de $\Z[\Gamma^2]$ à celle de $\Z[\Gamma]$ via:
$$\sum \lambda_i (\gamma_1^i,\gamma_2^i).D_I(\{i\infty,0\})=\sum_g D_g\left(\sum_{g=\gamma_2\gamma_1^{-1}}\lambda_i \gamma_1^i.\{i\infty,0\}\right).$$

Il suffit dans ce cas de faire agir de manière diagonale $\mathcal{I}_1$ pour obtenir le résultat.
\end{proof}

Dans la suite, nous considèrerons les parties paire et impaire du polynôme des bipériodes. 
Ainsi on s'intéressera à l'action par conjugaison de $\varepsilon=\mat{-1}{0}{\phantom{-}0}{1}$. Notons $c_+=id_{\Gamma}$ et $c_-:g\mapsto\varepsilon g\varepsilon$. 
Posons pour tout choix de signes $(\epsilon_1,\epsilon_2)\in\{\pm 1\}^2$, et toute partie $I\subset \Z[\Gamma^2]$:
\begin{align}
I^{(\epsilon_1,\epsilon_2)}=\left(c_{\epsilon_1}\times c_{\epsilon_2}\right)(I).
\end{align}

\begin{prop}
Les conjugaisons de $I_H$ et $I_V$ n'apportent pas de nouveaux termes:
\begin{align}
I_H&=I_H^{(+,+)}=I_H^{(-,+)}=I_H^{(+,-)}=I_H^{(-,-)},\\
\text{et }I_V&=I_V^{(+,+)}=I_V^{(-,+)}=I_V^{(+,-)}=I_V^{(-,-)}.
\end{align}
La conjugaison de $I_D$ donne lieu à de nouveaux termes:
\begin{align}
I_D=I_D^{(+,+)}=I_D^{(-,-)}&=\Big((1,1)+(S,S),(1,1)+(U,U)+(U^2,U^2)\Big),\\
\text{ et }I_D^{(-,+)}=I_D^{(+,-)}&=\Big((1,1)+(S,S),(1,1)+(U,V)+(U^2,V^2)\Big),
\end{align}
où $V=\left(\begin{smallmatrix}0 & -1\\ 1 & \phantom{-}1 \end{smallmatrix}\right)=\varepsilon U\varepsilon=SU^2S$.
\end{prop}

\begin{proof}
Dans le cas de $I_H$, les conjugaisons non triviales des générateurs sont:
\begin{align*}
(\varepsilon,1)(1+S,1)(\varepsilon,1)&=(1+S^{-1},1)=(1+S,1)\in I_H,\\
(\varepsilon,1)(1+U+U^2,1)(\varepsilon,1)&=(1+SU^2S+SUS,1)\\
&=(S,1)(1+U+U^2,1)(1+S,1)-(S,1)(1+U+U^2,1)\in I_H,\\
\text{et }(1,\varepsilon)(1,1-T)(1,\varepsilon)&=(1,1-T^{-1})=-(1,T^{-1})(1,1-T)\in I_H.
\end{align*}
Pour $I_V$, les deux premiers générateurs se traitent de la même manière par symétrie et on a:
$$(\varepsilon,1)(1-US,1)(\varepsilon,1)=(1,1-(US)^{-1})=-(1,(US)^{-1})(1,1-US)\in I_V.$$
Pour $I_D$, $(1,1)+(S,S)$ est stabilisé par les conjugaisons et celles de $(1,1)+(U,U)+(U^2,U^2)$ apporte de nouveaux termes à savoir ceux engendrés par $(1,1)+(U,V)+(U^2,V^2)$. Cet élément n'appartient pas à $I_D$ car:
$$\left[(1,1)+(U,V)+(U^2,V^2)\right]D_{id}(\tau_1)=D_{id}(\tau_1)+D_{VU^{-1}}(U\tau_1)+D_{V^2U^{-2}}(U^2\tau_1).$$
Cet élément est non nul car il est somme de trois éléments non nuls et libres entre eux d'après (\ref{sumd}).
De plus, la conjugaison des deux coordonnées est bien stabilisée:
\begin{multline*}
(\varepsilon,\varepsilon)\left[(1,1)+(U,U)+(U^2,U^2)\right](\varepsilon,\varepsilon)=(S,S)\left[(1,1)+(U^2,U^2)+(U,U)\right](S,S)\\
=(S,S)\left[(1,1)+(U,U)+(U^2,U^2)\right]\left[(1,1)+(S,S)\right]-(S,S)\left[(1,1)+(U,U)+(U^2,U^2)\right]\in I_D.
\end{multline*}
\end{proof}

\subsection{Cohomologie relative de $SL_2(\Z)^2$}

Dans le cas d'une forme modulaire, on pouvait oublier l'espace $\H$ en étudiant uniquement l'action de $\Gamma$ sur les pointes. On rappelle brièvement cette construction pour ensuite la généraliser au cas de l'action de $\Gamma^2$ sur $\H^2$.\par

On définit le groupe $\Z[\pte]^0$ des $1$-chaines fermées comme étant le noyau de l'application:
$$\Z[\pte]\to\Z,\quad \sum_i \lambda_i p_i\mapsto \sum_i \lambda_i.$$
Le Théorème de Manin donne alors la surjectivité de l'application:
$$\Theta_1:\Z[\Gamma]\to\Z[\pte]^0,\quad \gamma\to (\gamma.i\infty)-(\gamma.0).$$
Le noyau de cette application est l'idéal $\mathcal{I}_1$ des relations de Manin.

\subsubsection{Construction du groupe formel des $2$-chaînes}

Notons $P_2=\pte^2$ les sommets de $\H^2$.
On dispose des applications de bord:
$$\delta_{m}:\Z[P_2^{m+1}]\to\Z[P_2^m],\quad ((p_1,q_1),...,(p_{m+1},q_{m+1}))\mapsto\sum_{j=1}^{m+1}(-1)^j(...,(p_{j-1},q_{j-1}),(p_{j+1},q_{j+1}),...).$$
On vérifie simplement que $\delta_{m}\circ\delta_{m+1}=0$ et on a même l'exactitude de la suite longue:
$$...\stackrel{\delta_{m+1}}{\longrightarrow}\Z[P_2^{m+1}]\stackrel{\delta_{m}}{\longrightarrow}\Z[P_2^m]\stackrel{\delta_{m-1}}{\longrightarrow}...
\stackrel{\delta_1}{\longrightarrow}\Z[P_2]\stackrel{\delta_{0}}{\longrightarrow}\Z\longrightarrow 0.$$
Ceci permet de définir pour tout entier $m\geq 1$ le sous-groupe:
\begin{equation}
\Z[P_2^m]^0=\Ker(\delta_{m-1})=\Im(\delta_{m}).
\end{equation}

Pour $0\leq j\leq 2$, on dispose de la famille d'applications $\varphi_j^g:\H\to\H^2$ dépendant d'un paramètre $g\in\Gamma$ par: 
\begin{equation*}
\varphi_0^g(z)=(g.i\infty,z),\quad\varphi_1^g(z)=(z,g.z)\text{ et }\varphi_2^g(z)=(z,g.0).
\end{equation*}

On a vu que les cycles contenus dans les sous-espaces transverses, image de $\varphi_j^g$, amène des annulations. Nous allons donc quotienter par ces cycles.
Ce sont ceux de la forme:
$$\varphi_j^g(c),\text{ pour }c\in\Z[P_1^m]^0,j=0,1,2\text{ et }g\in\Gamma.$$
Posons $\Z[P_2^m]_j^0=\sum_{g\in\Gamma}\varphi_j^g(\Z[P_1^m]^0)\subset\Z[P_2^m]^0$.\par
On définit alors le groupe quotient:
\begin{equation}
H_m(P_2)=\Z[P_2^m]^0/\left(\sum_{j=0}^2 \Z[P_2^m]^0_j\right).
\end{equation}

\begin{prop}
L'action de $\Gamma^2$ sur $P_2$ se transmet diagonalement aux groupes $\Z[P_2^m]$. Cette action passe alors au quotient sur $H_m(P_2)$. \par
\end{prop}

\begin{proof}
Pour définir l'action de $\Gamma^2$ sur $H_m(P_2)$, il faut vérifier que l'action sur $\Z[P_2^m]$ fixe les espaces $\Z[P_2^m]^0$ et $\Z[P_2^m]^0_j$ pour tout $j$.
Or on peut mettre en évidence les relations:
$$\delta_{m}((\gamma_1,\gamma_2).(p_1,q_1),...,(\gamma_1,\gamma_2)(p_{m+1},q_{m+1}))=(\gamma_1,\gamma_2)\delta_{m}((p_1,q_1),...,(p_{m+1},q_{m+1})),$$
$$\text{ et }\varphi_j^g(\gamma z)=
\begin{cases}
(g i\infty,\gamma z)=(g,\gamma).\varphi_0^{Id}(z)&\text{ si }j=0,\\
(\gamma z,g \gamma z)=(\gamma,g\gamma).\varphi_1^{Id}(z)&\text{ si }j=1,\\
(\gamma z,g 0)=(\gamma,g).\varphi_2^{Id}(z)&\text{ si }j=2.
\end{cases}$$
qui démontrent la stabilité des sous-groupes.
\end{proof}

\subsubsection{Le triangle fondamental}

On construit un élément $T_2$ de $\Z[P^3_2]$ associée à $\tau_2$ donné par:
\begin{align}
T_2&=[(i\infty,i\infty),(i\infty,0),(0,0)]\\
\text{et ainsi: }\delta_2 T_2&=[(i\infty,0),(0,0)]-[(i\infty,i\infty),(0,0)]+[(i\infty,i\infty),(i\infty,0)].\nonumber
\end{align}
On notera $\partial T_2$ la classe de $\delta_2 T_2$ dans $H_2(P_2)$.

Définissons le morphisme de $\Z[\Gamma^2]$-modules:
$$\Theta_2:\Z[\Gamma^2]\to H_2(P_2),\quad(\gamma_1,\gamma_2)\mapsto (\gamma_1,\gamma_2).\partial T_2.$$

\begin{prop}
Le noyau de $\Theta_2$ est l'idéal à gauche $\mathcal{I}_2$.\\
Son image, que nous noterons $H_2(P_2)^0$, est l'ensemble des classes des chaînes engendrées par celles de la forme:
$$\delta_2[(a,g.a),(a,g.b),(b,g.b)]\text{ où }a,b\in\pte\text{ et }g\in\Gamma.$$
\end{prop}

\begin{proof}
Par définition de $\mathcal{I}_2$ et la correspondance donnée pour tout entier $m\geq 0$:
$$M_m^{pte}(\H^2,\Z)\to\Z[P_2^{m+1}],\quad C\mapsto \left(C(e_0),...,C(e_m)\right),$$
permet d'obtenir par simple récriture des objets $\mathcal{I}_2=\Ker(\Theta_2)$.\par
Pour calculer l'image de l'application, on choisit un élément qui admet pour représentant:
$$\delta_2[(a,g.a),(a,g.b),(b,g.b)]\text{ avec }a,b\in\pte,g\in\Gamma.$$
Alors on peut commencer par remarquer qu'il est image de la multiplication par $(1,g)$ de:
$$\delta_2[(a,a),(a,b),(b,b)]$$
Puis on écrit comme dans le cas classique une suite fini de matrice $\gamma_i\in\Gamma$ telle que:
$$a=\gamma_1.i\infty\to\gamma_1.0=\gamma_2.i\infty\to...\to\gamma_{N-1}.0=\gamma_N.i\infty\to\gamma_N.0=b,$$
ceci se fait de manière constructive grâce aux fractions continues.\par
Il nous suffit alors de montrer le résultat de décomposition suivante pour $N=2$:
$$\delta_2[(g_1.i\infty,g_1.i\infty),(g_1.i\infty,g_2.0),(g_2.0,g_2.0)]=[(g_1,g_1)+(g_2,g_2)+(g_1,g_2)+(g_1S,g_2S)].\partial T_2$$
En effet, il suffit de simplifier le terme de droite qui est la classe dans $H_2(P_2)$ du bord d'un triangle qui se décompose simplement en quatre triangles:

\begin{tabular}{ll}
& \multirow{9}*{
\setlength{\unitlength}{0.6cm}
\begin{picture}(9,7)
\put(1,1){\line(1,0){6}}
\put(1,1){\line(1,1){6}}
\put(4,1){\line(0,1){3}}
\put(4,1){\line(1,1){3}}
\put(7,1){\line(0,1){6}}
\put(4,4){\line(1,0){3}}

\put(0.4,0){$g_1i\infty$}
\put(2.7,0){$g_10=g_2i\infty$}
\put(6.7,0){$g_20$}

\put(8,1){$g_1i\infty$}
\put(7.5,4){$g_10=g_2i\infty$}
\put(8,7){$g_20$}

\end{picture}}\\
$\delta_2[(g_1.i\infty,g_1.i\infty),(g_1.i\infty,g_1.0),(g_1.0,g_1.0)]$&\\
&\\
$+\delta_2[(g_2.i\infty,g_2.i\infty),(g_2.i\infty,g_2.0),(g_2.0,g_2.0)]$& \\
&\\
$+\delta_2[(g_1.i\infty,g_2.i\infty),(g_1.i\infty,g_2.0),(g_1.0,g_2.0)]$& \\
&\\
$+\delta_2[(g_1S.i\infty,g_2S.i\infty),(g_1S.i\infty,g_2S.0),(g_1S.0,g_2S.0)]$.& \\
&
\end{tabular}

Ce résultat peut s'itérer sur la famille des $\gamma_1,...,\gamma_N$ et on obtient bien alors un élément de $\Z[\Gamma^2]\partial T_2$ l'image de $\Theta_2$.\par

Pour démontrer que le noyau de $\Theta_2$ est $\mathcal{I}_2$, on utilise la formule de construction: 
$$\mathcal{I}_2=I_H\cap I_D\cap I_V.$$
et on vérifie que chacun de ces idéaux annulent tour à tour les côtés de :
$$\delta_2 T_2=[(i\infty,i\infty),(i\infty,0)]-[(i\infty,i\infty),(0,0)]+[(i\infty,0),(0,0)],$$
modulo $\sum_{j=0}^2 \Z[P_2^m]^0_j$.
\end{proof}

L'application $\Theta_2$ n'est pas surjective et on peut déterminer plus précisément $H_2(P_2)^0$:
\begin{prop}
Le sous-groupe $H_2(P_2)^0$ est celui des $2$-chaînes transverses:
\begin{align}
H_2(P_2)^0&=\left[\left(\sum_{j,g} \varphi_j^g(\Z[P_1^2])\right)\cap \Z[P_2^2]^0\right]/\sum_{j,g} \varphi_j^g(\Z[P_1^2]^0)\\
&\cong (H+D+V)^0/(H^0+D^0+V^0).\nonumber 
\end{align}
\end{prop}

\begin{proof}
On commence par montrer que les générateurs sont dans l'espace transverse:
$$\delta_2[(a,g.a),(a,g.b),(b,g.b)]=[(a,ga),(a,gb)]-[(a,ga),(b,bg)]+[(a,gb),(b,gb)]\in \Z[P_2^2]^0$$
vérifie les appartenance:
\begin{align*}
[(a,ga),(a,gb)]&=\varphi_0^a(\{ga,gb\})\in\varphi_0^a(\Z[P_1^2])\\
[(a,ga),(b,gb)]&=\varphi_1^g(\{a,b\})\in\varphi_1^g(\Z[P_1^2])\\
[(a,gb),(b,gb)]&=\varphi_2^{gb}(\{a,b\})\in\varphi_2^{gb}(\Z[P_1^2])
\end{align*}
Pour démontrer, l'inclusion réciproque on donne un algorithme de décomposition de l'espace $H_2(P_2)$ en triangle simple.
\end{proof}

\begin{prop}
Les cycles de $H_2(P_2)$ peuvent être décomposés comme combinaison d'une famille plus simple les cycles de la forme:
$$\delta_2[(a,g.a),(a,g.b),(c,g.b)]\text{ où }a,b,c\in\pte\text{ et }g\in\Gamma.$$
\end{prop}

\begin{proof}
On choisi un représentant d'une classe de $H_2(P_2)$:
$$\delta_2[(a_1,a_2),(b_1,b_2),(c_1,c_3)]\in\Z[P_2^2]^0.$$
Pour le réduire, on va trianguler pour obtenir une combinaison d'élément de la forme:
$$\delta_2[(\alpha_1,\alpha_2),(\alpha_1,\beta_2),(\beta_1,\beta_2)].$$
Ils sont bien de la forme voulue car si on écrit $\alpha_1=g_1.i\infty$ et $\alpha_2=g_2i\infty$ alors on obtient:
$$=\delta_2[(\alpha_1,g_2g_1^{-1}.\alpha_1),(\alpha_1,g_2g_1^{-1}.b),(g_2g_1^{-1}.c,g_2g_1^{-1}.b)].$$
Pour trianguler, il nous suffit d'utiliser le fait que $\delta_2\circ\delta_3=0$ en ajoutant une quatrième pointe.
Si on adjoint le point $(a_1,b_2)$, on obtient:
\begin{multline*}
0=\delta_2\circ\delta_3[(a_1,a_2),(a_1,b_2),(b_1,b_2),(c_1,c_2)]\\
=\delta_2[(a_1,b_2),(b_1,b_2),(c_1,c_2)]-\delta_2[(a_1,a_2),(b_1,b_2),(c_1,c_2)]\\
+\delta_2[(a_1,a_2),(a_1,b_2),(c_1,c_2)]-\delta_2[(a_1,a_2),(a_1,b_2),(b_1,b_2)].
\end{multline*}
Ceci permet d'obtenir une réécriture de $\delta_2[(a_1,a_2),(b_1,b_2),(c_1,c_2)]$:
$$=\delta_2[(a_1,b_2),(b_1,b_2),(c_1,c_2)]+\delta_2[(a_1,a_2),(a_1,b_2),(c_1,c_2)]-\delta_2[(a_1,a_2),(a_1,b_2),(b_1,b_2)].$$
Le troisième terme est de la forme voulue et les deux premiers ne dépendent plus que de cinq pointes.\par 
Ils se réduisent tout les deux de la même manière. Par exemple, pour calculer $\delta_2[(a_1,b_2),(b_1,b_2),(c_1,c_2)]$ on peut ajouter $(c_1,b_2)$ donnant:
$$=\delta_2[(c_1,b_2),(b_1,b_2),(c_1,c_2)]+\delta_2[(a_1,b_2),(c_1,b_2),(c_1,c_2)]-\delta_2[(a_1,b_2),(c_1,b_2),(b_1,b_2)].$$
Le premier et deuxième termes sont de la forme attendue. Le dernier est en faite un triangle à bord transverse:
$$\delta_2[(a_1,b_2),(c_1,b_2),(b_1,b_2)]=\varphi_2^{b_2}(\delta_2[a_1,c_1,b_1])\in\Z[P_2^2]^0_2,$$
et donc est nulle dans $H_2(P_2)$.\par
On peut désormais compléter la démonstration précédente. Supposons que la $2$-chaîne obtenue appartient à $H_2(P_2^2)^0$ alors 
$\delta_2[(a,ga),(a,gb),(c,gb)]$ a deux de ces côtés dans l'espace transverse imposant le troisième $[(a,ga),(c,gb)]\in\sum_{j,g}\varphi_j^g(\Z[P_1^2])$ c'est à dire $c\in\{a,b\}$. Ceci donne bien l'inclusion manquante.\par
\end{proof}

\begin{rem}
On dispose alors d'une suite exacte décomposant $\Z[\Gamma^2]$ :
\begin{equation}\label{suiexact2}
0\to\mathcal{I}_2\to\Z[\Gamma^2]\stackrel{\Theta_2}{\to}H_2(P_2^2)^0\to 0.
\end{equation}
\end{rem}

\subsubsection{Indépendance du poids de l'idéal annulateur}

On dispose désormais de tous les outils pour démontrer le résultat suivant:

\begin{thm}
Pour tout poids $k_1,k_2$, on a:
\begin{equation}
V_{k_1,k_2}^{\Q}[\mathcal{I}_2]=V_{k_1,k_2}^{\Q}[\mathcal{J}(k_1,k_2)]+E_{k_1,k_2}^{\Q}.
\end{equation}
Dans la suite, nous noterons $W_{k_1,k_2}^{\Q}$ ce sous-espace de $V_{k_1,k_2}^{\Q}$ dont l'extension au corps des complexes contient $\Per_{k_1,k_2}$.
\end{thm}

\begin{proof}
Rappelons la définition de $E_{k_1,k_2}^{\Z}$:
$$E_{k_1,k_2}^{\Z}=V_{k_1,k_2}^{\Z}[I_H]+V_{k_1,k_2}^{\Z}[I_D]+V_{k_1,k_2}^{\Z}[I_V].$$
On a démontrer dans la Proposition \ref{incljkk} que l'idéal à gauche $\mathcal{I}_2$ est inclus dans $\widetilde{\mathcal{J}(k_1,k_2)}$. D'autre la Proposition \ref{inclhvd}, montre que $\mathcal{I}_2$ est inclus dans les idéaux $I_H,I_D$ et $I_V$ par construction. On en déduit l'inclusion des groupes:
$$V_{k_1,k_2}^{\Z}[\mathcal{J}(k_1,k_2)]+E_{k_1,k_2}^{\Z}\subset V_{k_1,k_2}^{\Z}[\mathcal{I}_2].$$
Comme dans le cas $n=1$, on va montrer que ces groupes engendrent des $\Q$-espaces vectoriels de même dimension. Considérons l'application:
$$\Theta_2\otimes_{\Q[\Gamma^2]}V_{k_1,k_2}^{\Q}: V_{k_1,k_2}^{\Q}\to \left(H_2(P_2)^0 \otimes_{\Q[\Gamma^2]} V_{k_1,k_2}^{\Q}\right),\quad P\mapsto (\partial T_2)\otimes P,$$
où l'invariance par $\gamma\in\Q[\Gamma^2]$ est donnée par $(\gamma C)\otimes P|_{\gamma}=C\otimes P$.\par
Or les représentations $\Gamma\to GL(V_{k_j})$ étant irréductible, on déduit la surjectivité de $\Q[\Gamma]\to End_{\Q}(V_{k_j}^{\Q})$ puis celle de $\Q[\Gamma^2]\to End_{\Q}(V_{k_1,k_2}^{\Q})$. Donc $V_{k_1,k_2}^{\Q}$ est un $\Q[\Gamma^2]$-module simple. Ainsi la suite exacte (\ref{suiexact2}) reste exacte après tensorisation par $V_{k_1,k_2}^{\Q}$:
\begin{equation}
0\to V_{k_1,k_2}^{\Q}|_{\widetilde{\mathcal{I}_2}} \to V_{k_1,k_2}^{\Q} \to \left(H_2(P_2)^0 \otimes_{\Q[\Gamma^2]} V_{k_1,k_2}^{\Q}\right) \to 0.
\end{equation}

D'une part, on dispose de l'application bilinéaire symétrique, non dégénérée et invariante par $\Gamma^2$:
$$V_{k_1,k_2}^{\Q}\times V_{k_1,k_2}^{\Q}\to \Q,$$
induite par celles introduites dans la Remarque \ref{prodvk} du chapitre $1$ : $V_{k_j}^{\Q}\times V_{k_j}^{\Q}\to\Q$ invariante par $\Gamma$. Elle donne la dualité des $\Q$-espaces $V_{k_1,k_2}^{\Q}|_{\widetilde{\mathcal{I}_2}}$ et $V_{k_1,k_2}^{\Q}[\mathcal{I}_2]$ et ainsi:
$$\dim_{\Q}\left(V_{k_1,k_2}^{\Q}[\mathcal{I}_2]\right)+\dim_{\Q}(V_{k_1,k_2}^{\Q}|_{\widetilde{\mathcal{I}_{2}}})=\dim_{\Q}(V_{k_1,k_2}^{\Q}).$$

D'autre part, on a:
\begin{align*}
H_2(P_2)^0&\cong (H+D+V)^0/(H^0+D^0+V^0)\\
&=\{d\omega=\varphi_0^*\omega=\varphi_1^*\omega=\varphi_2^*\omega=0\}^{\bot}/\left(\sum_{j=0}^2\{d\omega=\varphi_j^*\omega=0\}^{\bot}\right).
\end{align*}
Or l'accouplement, $\Omega_{par}^2\times M_2^{pte}\to\C$ permet d'avoir les résultats de dualités:
\begin{equation*}
\mathcal{J}(k_1,k_2)=\{g\in\Z[\Gamma^2];\tilde{g}\tau_2\in\Omega_{k_1,k_2}^{\bot}\}\text{ et }\{d\omega=\varphi_0^*\omega=\varphi_1^*\omega=\varphi_2^*\omega=0\}\otimes_{\Q[\Gamma^2]} V_{k_1,k_2}^{\Q}=\Omega_{k_1,k_2}.
\end{equation*}
Et d'autre part:
\begin{align*}
I_H=\{g\in\Z[\Gamma^2];\tilde{g}\tau_2\in\Omega_{k_1,k_2}(I_H)^{\bot}\}&\text{ où }\Omega_{k_1,k_2}(I_H)=\{d\omega=\varphi_0^*\omega=0\}\otimes_{\Q[\Gamma^2]} V_{k_1,k_2}^{\Q},\\
I_D=\{g\in\Z[\Gamma^2];\tilde{g}\tau_2\in\Omega_{k_1,k_2}(I_D)^{\bot}\}&\text{ où }\Omega_{k_1,k_2}(I_D)=\{d\omega=\varphi_1^*\omega=0\}\otimes_{\Q[\Gamma^2]} V_{k_1,k_2}^{\Q},\\
I_V=\{g\in\Z[\Gamma^2];\tilde{g}\tau_2\in\Omega_{k_1,k_2}(I_V)^{\bot}\}&\text{ où }\Omega_{k_1,k_2}(I_V)=\{d\omega=\varphi_2^*\omega=0\}\otimes_{\Q[\Gamma^2]} V_{k_1,k_2}^{\Q}.
\end{align*}
Ainsi en tensorisant $H_2(P_2)^0$ par $V_{k_1,k_2}^{\Q}$ comme un $\Q[\Gamma^2]$-module, on obtient:
\begin{equation}
H_2(P_2)^0\otimes_{\Q[\Gamma^2]} V_{k_1,k_2}^{\Q}=V_{k_1,k_2}^{\Q}[\mathcal{J}(k_1,k_2)]+E_{k_1,k_2}^{\Q}.
\end{equation}
L'exactitude de la suite donne $\dim_{\Q}\left(V_{k_1,k_2}^{\Q}[\mathcal{J}(k_1,k_2)]+E_{k_1,k_2}^{\Q}\right)+\dim_{\Q}(V_{k_1,k_2}^{\Q}|_{\widetilde{\mathcal{I}_{2}}})=\dim_{\Q}(V_{k_1,k_2}^{\Q})$ et ainsi démontre le théorème.
\end{proof}

\subsection{Une famille finie de générateurs de $\mathcal{I}_2$}

\subsubsection{$\mathcal{I}_2$ est de type fini}

Nous montrons que l'idéal $\mathcal{I}_2$ est engendré par des combinaisons linéaires à support dans l'ensemble fini:
$$(U^{i_1},U^{i_2})(S,S)^{\delta}\text{ pour }(i_1,i_2)\in (\Z/3\Z)^2\text{ et }\delta\in \Z/2\Z.$$
Pour cela, nous allons introduire une hauteur sur $H_2(P_2^2)^0$ associée à une distance de $\pte$ invariante par $\G$.\par

Définissons l'application $d:\pte\times\pte\to \Z_{\geq 0}$ en posant pour $a\neq b\in\pte$:
\begin{equation}
d(a,b)=\min\{N\in\mathbb{N}\text{ tel qu'il existe }\gamma_1,...,\gamma_N\in\G\text{ vérifiant }a=\gamma_1 i\infty,\gamma_1 0=\gamma_2 i\infty,...,\gamma_N 0=b\},
\end{equation}
et $d(a,a)=0$.

\begin{prop}
L'application $d$ est une distance invariante par $P\G$, au sens où pour tout $a,b,c\in\pte$ et $\gamma\in P\G$:
\begin{align*}
d(a,b)=d(b,a),\quad &d(a,b)=0 \Leftrightarrow a=b,\\ 
d(a,c)\leq d(a,b)+d(b,c)\quad\text{ et }\quad &d(\gamma a,\gamma b)=d(a,b).
\end{align*}
\end{prop}

\begin{rem}
1) Pour la suite, on appellera \textit{chaîne} une telle suite de matrices et on notera:
\begin{equation}
a=\gamma_1 i\infty \stackrel{\gamma_1}{\to}\gamma_1 0\stackrel{\gamma_2}{\to}...\gamma_N i\infty\stackrel{\gamma_N}{\to} \gamma_N 0=b.
\end{equation}
On appellera \textit{longueur} de la chaîne le nombre $N$ de matrices. Les chaînes de longueurs minimales reliant deux points $a$ et $b$ déterminent ainsi la distance $d(a,b)$.\par
2) L'application $d$ est la distance associée au graphe orienté $\mathcal{G}_1$ dont l'ensemble des sommets est $\pte$ et l'ensemble des arrêtes est $PSL_2(\Z)$. En effet, pour tout $a/b,c/d\in\pte$, on a:
\begin{equation}
d(a/c,b/d)=1 \Leftrightarrow \mat{a}{b}{c}{d}\in PSL_2(\Z)\Leftrightarrow \mat{b}{-a}{d}{-c}\in PSL_2(\Z).
\end{equation}
Ainsi toute paire de sommet de ce type est relié par les matrices $\mat{a}{b}{c}{d}$ et $\mat{b}{-a}{d}{-c}$. Nous représentons ceci sur le schéma suivant où chaque trait représentent ainsi deux arrêtes orientées aller et retour.

\setlength{\unitlength}{0.6cm}
\begin{picture}(24,7)

\put(12.1,4.5){$id$}
\put(10,4.2){\line(1,0){4}}
\put(12.3,4.2){$_>$}

\put(11.9,3.2){$S$}
\put(10,4){\line(1,0){4}}
\put(11.7,4){$_<$}

\put(8.7,4){${1/0}$}
\put(14.5,4){${0/1}$}
\put(11.2,6.6){${-1/1}$}
\put(11.5,1.3){${1/1}$}
\put(9.5,4.7){\line(1,1){1.8}}
\put(14.5,4.7){\line(-1,1){1.8}}
\put(9.5,3.5){\line(1,-1){1.8}}
\put(14.5,3.5){\line(-1,-1){1.8}}

\put(15.5,3.5){\line(1,-2){0.8}}
\put(8.5,3.5){\line(-1,-2){0.8}}
\put(7.2,1.3){$2/1$}
\put(16.1,1.3){$1/2$}
\put(12.7,1.5){\line(1,0){3}}
\put(8.3,1.5){\line(1,0){2.9}}

\put(7.1,6.6){$-2/1$}
\put(15.8,6.6){$-1/2$}
\put(15.5,4.6){\line(1,2){0.8}}
\put(8.5,4.6){\line(-1,2){0.8}}
\put(12.9,6.8){\line(1,0){2.2}}
\put(9,6.8){\line(1,0){2}}

\put(14,0.2){$2/3$}
\put(13.5,0.7){\line(-2,1){0.7}}
\put(15.2,0.7){\line(2,1){0.7}}
\put(9.3,0.2){$3/2$}
\put(8.8,0.7){\line(-2,1){0.7}}
\put(10.5,0.7){\line(2,1){0.7}}

\put(17,2.8){$1/3$}
\put(17.3,2.5){\line(-1,-1){0.6}}
\put(16.8,3.3){\line(-2,1){0.8}}
\put(6,2.8){$3/1$}
\put(6.6,2.5){\line(1,-1){0.6}}
\put(6.8,3.4){\line(2,1){0.9}}

\put(17.2,4){$1/n$}
\put(17,4.2){\line(-1,0){0.8}}

\end{picture}

3) La valence de chaque sommet est infinie. Les propriétés de la distance données par la Proposition se traduisent sur le graphe par sa connexité et son invariance sous l'action de $\Gamma$.
\end{rem}

\begin{proof}
Tout d'abord cette application est bien définie car le développement en fraction continue donne l'existence d'une chaîne de longueur finie liant $0$ à tout point de $\pte$.\\
La condition de séparation est purement formelle, les chaînes de longueur $0$ sont $a=b$.\\
Soient $a,b\in\pte$. Si $d(a,b)=N$ alors il existe une chaîne minimale: $a\stackrel{\gamma_1}{\to}...\stackrel{\gamma_N}{\to}b$. Elle nous permet de construire la chaîne:
$b=\gamma_N S i\infty\stackrel{\gamma_N S}{\to}...\stackrel{\gamma_1 S}{\to} \gamma_1 S 0=a.$\\
Ceci démontre que $d(b,a)\leq d(a,b)$ et donc on obtient l'égalité en symétrisant.\\
Soient $a,b,c\in\pte$. On peut concaténer deux chaînes: 
$a\stackrel{\gamma_1}{\to}...\stackrel{\gamma_{d(a,b)}}{\to}b\stackrel{\gamma'_{1}}{\to}...\stackrel{\gamma'_{d(b,c)}}{\to}c.$\\
On obtient une chaîne de longueur $d(a,b)+d(b,c)$ et ainsi $d(a,c)\leq d(a,b)+d(b,c)$.\\
Soit $g\in\G$. Les chaînes: $a\stackrel{\gamma_1}{\to}...\stackrel{\gamma_{N}}{\to}b$ se translate en:
$g a\stackrel{g \gamma_1}{\to}...\stackrel{g \gamma_{d(a,b)}}{\to}g b.$\\
Ceci donne $d(g a,g b)\leq d(a ,b)$. Puis on applique ceci à $g^{-1}\in\Gamma$ et $ga,gb\in\pte$ pour obtenir l'inégalité inverse.
\end{proof}

Ceci permet de décomposer l'espace image de $\Theta_1:\Z[\Gamma]\to \Z[\pte]^0$ suivant une hauteur:
\begin{equation}
h:\pte\to\Z_{\geq 0}, p\mapsto max(d(p,i\infty),d(p,0)),
\end{equation}
en posant, pour tout entier $M>0$:
\begin{equation}
\Z[\pte]^0_M=\left\{\sum_p \lambda_p (p)\in\Z[\pte]^0\text{ tel que }\lambda_p\neq 0\Rightarrow h(p)\leq M\right\}.
\end{equation}

Définissons désormais les parties de $\pte$:
\begin{align*}
B_0&=\{i\infty,0,1,-1\}=\{p\in\pte\text{ tel que }h(p)\leq 1\},\\
B_1&=\{a/b;a< b\text{ et } ab> 0\},
B_2=\{a/b;a> b\text{ et } ab> 0\},\\
B_3&=\{a/b;a<b\text{ et } ab<0\}
\text{ et }B_4=\{a/b;a>b\text{ et } ab<0\}.
\end{align*}

Ceci nous permet d'obtenir la partition suivante de $\pte$:
\begin{equation}
\pte=B_0\sqcup B_1\sqcup B_2\sqcup B_3\sqcup B_4.\label{partpte}
\end{equation}

De plus, les actions de $S=\mat{0}{-1}{1}{0}$ et $\varepsilon=\mat{-1}{0}{0}{1}$ stabilisent $B_0$ et échangent les espaces $B_1,B_2,B_3$ et $B_4$ selon:
$$B_1=\varepsilon SB_2=\varepsilon B_3=SB_4.$$

L'ensemble $B_0$ est l'ensemble des points de hauteur $\leq 1$. Pour tout $j\in\{1,2,3,4\}$, l'ensemble des voisins des sommets de $B_j$ sur le graphe $\mathcal{G}_1$ est $B_j\cup B_0$. On peut ainsi réduire par symétrie l'étude du graphe à l'étude de $B_1\cup B_0$.

\begin{lem}\label{lemtri}
Soit $p_1/q_1,p_2/q_2\in B_1$ vérifiant $d\left(\frac{p_1}{q_1},\frac{p_2}{q_2}\right)=1$. Alors il existe exactement deux pointes à distance $1$ de $p_1/q_1$ et $p_2/q_2$ qui sont:
$$\frac{p_1-p_2}{q_1-q_2}\quad\text{ et }\quad\frac{p_1+p_2}{q_1+q_2}.$$
De plus, leurs hauteurs sont données par:
$$h\left(\frac{p_1-p_2}{q_1-q_2}\right)\leq min\left[h\left(\frac{p_1}{q_1}\right),h\left(\frac{p_2}{q_2}\right)\right]
\text{ et }h\left(\frac{p_1+p_2}{q_1+q_2}\right)= min\left[h\left(\frac{p_1}{q_1}\right),h\left(\frac{p_2}{q_2}\right)\right]+1.$$
\end{lem}

\begin{proof}
Posons $\gamma=\mat{p_1}{p_2}{q_1}{q_2}$, afin d'obtenir $p_1/q_1=\gamma i\infty$ et $p_2/q_2=\gamma 0$. L'hypothèse $d\left(\frac{p_1}{q_1},\frac{p_2}{q_2}\right)=1$ se traduit par $\gamma\in\Gamma$ et ainsi l'invariance de $d$ par $\Gamma$ se traduit par:
\begin{center}
Les pointes à distance $1$ de $\gamma i\infty$ et $\gamma 0$ sont $\gamma (1)=\frac{p_1+p_2}{q_1+q_2}$ et $\gamma (-1)=\frac{p_1-p_2}{q_1-q_2}$.
\end{center}\par
Pour calculer leurs hauteurs, on commence par remarquer que pour tout $\alpha \in B_1,h(\alpha)=d(i\infty,\alpha)$ et ainsi la pointe qui nous concerne après translation par $\gamma$ provient de $\gamma (\frac{-q_2}{q_1})=i\infty$. Or on sait que $\frac{-q_2}{q_1}\in B_3\cup B_4\cup \{-1\}$. Lorsque $\frac{-q_2}{q_1}=-1$ le résultat est vérifiable simplement car $d(-1,i\infty)=d(-1,0)=d(-1,1)-1$. Sinon, par symétrie entre $p_1/q_1$ et $p_2/q_2$, on peut supposer $\frac{-q_2}{q_1}\in B_3$ et on obtient:
$$d(\frac{-q_2}{q_1},-1)\leq d(\frac{-q_2}{q_1},0)\leq d(\frac{-q_2}{q_1},i\infty)\leq d(\frac{-q_2}{q_1},1).$$
La première inégalité se traduit par $h\left(\frac{p_1-p_2}{q_1-q_2}\right)\leq h\left(\frac{p_1}{q_1}\right)$ après translation par $\gamma$. Par l'absurde, on montre que parmi la deuxième et la troisième inégalité au moins une est strict faute de contredire le début du Lemme. Ainsi on obtient:
$$d(\frac{-q_2}{q_1},0)<d(\frac{-q_2}{q_1},1)\leq d(\frac{-q_2}{q_1},0) + d(0,1)=d(\frac{-q_2}{q_1},0)+1.$$
Ceci se translate par $\gamma$ en $h\left(\frac{p_1+p_2}{q_1+q_2}\right)= h\left(\frac{p_1}{q_1}\right)+1$.
\end{proof}

Nous avons désormais les outils pour montrer que $\mathcal{I}_1$ et $\mathcal{I}_2$ sont de type fini.

L'idéal $\mathcal{I}_1$ est le noyau de l'application $\Theta_1:\Z[\Gamma]\to \Z[\pte]^0$. Un élément du groupe $\Gamma$ est une arrête orientée de $\mathcal{G}_1$ dont les extrémités dans $\pte$ est donnée par l'application $\Theta_1$. Ainsi un élément du groupe $\Z[\Gamma]$ correspond à un $1$-cycle du graphe $\mathcal{G}_1$ et son image par $\Theta_1$ correspond à son bord. Les chemins de $\mathcal{I}_1$, c'est-à-dire annulant $\Theta_1$, sont ainsi les chemins fermées de $\mathcal{G}_1$.\par

Pour un élément $\sum \lambda_{\gamma} [\gamma]\in\Z[\Gamma]$, définissons son support la partie de $\pte$ donnée par: 
$$Supp(\sum \lambda_{\gamma} [\gamma])=\bigcup_{\lambda_{\gamma}\neq 0} \{\gamma i\infty,\gamma 0\}.$$
Notons $C^0(X)$ pour les $1$-cycles fermées de $\mathcal{G}_1$ à support dans $X$ une partie de $\pte$:
$$C^0(X)=\{g\in\mathcal{I}_1\text{ tel que }Supp(g)\subset X\}.$$
La partition de $\pte$ (\ref{partpte}) et les propriétés d'invariances élémentaires montrent que:
\begin{align*}
\mathcal{I}_1\cong C^0(\pte)&= C^0(B_0\cup B_1) + C^0(B_0\cup B_2) + C^0(B_0\cup B_3) + C^0(B_0\cup B_4)\\
&=C^0(B_0\cup B_1) + \varepsilon S.C^0(B_0\cup B_1) + \varepsilon .C^0(B_0\cup B_1) + S.C^0(B_0\cup B_1).
\end{align*}

Une méthode de descente sur la hauteur donne alors:

\begin{prop}
L'idéal $\mathcal{I}_1$ est engendré comme $\Z[\Gamma]$-module par les chemins fermés de $\mathcal{G}_1$ à support dans $B_0$:
\begin{equation}
\mathcal{I}_1 = \Z[\Gamma] . C^0(B_0)=\Z[\Gamma] (1+S,1+U+U^2).
\end{equation}
\end{prop}


\begin{proof}
Soit $g=\sum \mu_{\gamma} [\gamma] \in \Ker(\Theta_1)$ tel que $Supp(g)\subset B_0\cup B_1$.\\
Si $g\neq 0$, il existe un point du support $p\in Supp(g)$ de hauteur maximale. Posons:
$$R_p=\sum_{\gamma;p\in Supp([\gamma])} \mu_{\gamma} \Theta_1(\gamma)\text{ et }\Theta_1(\gamma)=\pm[(p)-(a_{\gamma})].$$
On va remplacer $p$ pour réduire la hauteur maximale du support. Pour $\gamma_1,\gamma_2$ distincts vérifiant $p\in Supp(\gamma_j)$, on peut supposer que $\Theta_1(\gamma_1)=p-a_1$ et $\Theta_1(\gamma_2)=p-a_2$ quitte à changer $\gamma$ par $\gamma S$ et $\mu_{\gamma S}=-\mu_{\gamma}$. On a $a_1\neq a_2$ donc $d(a_1,a_2)\in\{1,2\}$ et on distingue alors deux cas.\par
Si $d(a_1,a_2)=1$ alors d'après le Lemme \ref{lemtri}, il existe une unique pointe $q$ tel que $p$ et $q$ soit à distance $1$ de $a_1$ et $a_2$. De plus, la hauteur de $p$ étant maximal on obtient $h(q)\leq min\left[h(a_1),h(a_2)\right]<h(p)$.\par
Si $d(a_1,a_2)=2$ alors la situation est analogue et $a_1=\alpha_1/\beta_1$ est l'image de $1$ et $a_2=\alpha_2/\beta_2$ l'image de $-1$ par la matrice $\gamma=\mat{\alpha_1+\alpha_2}{\alpha_1-\alpha_2}{\beta_1+\beta_2}{\beta_1-\beta_2}$. On obtient alors $p\in\{\gamma 0,\gamma i\infty\}$ car seules deux pointes sont à distance $1$ des pointes $1$ et $-1$ et notons $q$ la seconde différente de $p$. On obtient à nouveau une configuration du type du Lemme \ref{lemtri}. Et ainsi on a: $h(q)+1=max[h(a_1),h(a_2)]\leq h(p)$.\par
Ainsi dans les deux cas, on a:
$$\Theta_1\left([\gamma_1]-[\gamma_2]\right)=(p-a_1)-(p-a_2)=(q-a_1)-(q-a_2)=\Theta_1\left([\gamma_1']-[\gamma_2']\right).$$
Les relations reliant les chemins ces quatre matrices sont à support dans un même translaté de $B_0$. Et comme on a $\sum_{p\in Supp[\gamma]}\mu_{\gamma}=0$ alors on peut construire un élément $R_p'$ à support de hauteur inférieure tel que: $R_p-R_p'\in \Z[\Gamma] C^0(B_0)$ et décomposer:
$$g=\sum_{\gamma;p\notin Supp[\gamma]}\mu_{\gamma}[\gamma]+R_p'+\left(R_p-R_p'\right).$$
Le support de $g$ étant fini on réduit bien celui-ci à une somme d'éléments de $\Gamma C^0(B_0)$ par récurrence sur la hauteur maximale. On remarquera notamment que dans le déroulement de la démonstration $p$ n'est pas nécessairement l'unique élément de hauteur maximale. Mais cet ensemble est fini et son cardinal diminue strictement car le seul point introduit vérifie $h(q)<h(p)$.
\end{proof}

Par analogie, il est alors naturel d'introduire le graphe $\mathcal{G}_2$ dont les sommets sont $P_2=\pte^2$ et les arêtes sont $H\cup D\cup V$. Un élément $(\gamma_1,\gamma_2)\in\Gamma^2$ est identifié dans ce graphe par le triangle orienté:
$$[(\gamma_1.i\infty,\gamma_2.i\infty),(\gamma_1 i\infty,\gamma_2 0),(\gamma_1 0,\gamma_2 0)].$$
Son bord est l'image par $\Theta_2$ c'est à dire la somme des trois arêtes orientées dans $H_2(P_2)^0$:
$$[(\gamma_1 i\infty,\gamma_2 i\infty),(\gamma_1 i\infty,\gamma_2 0)]+
[(\gamma_1 i\infty,\gamma_2 0),(\gamma_1 0,\gamma_2 0)]+
[(\gamma_1 0,\gamma_2 0),(\gamma_1 i\infty,\gamma_2 i\infty)].$$\par
Le support d'un élément de $g=\sum \lambda_{\gamma_1,\gamma_2}(\gamma_1,\gamma_2)\in\Z[\Gamma^2]$ est la partie de $\pte^2$ définie par:
$$Supp(g)=\bigcup_{\lambda_{\gamma_1,\gamma_2}\neq 0} \{(\gamma_1 i\infty,\gamma_2 i\infty),(\gamma_1 i\infty,\gamma_2 0),(\gamma_1 0,\gamma_2 0)\}.$$
Le schéma suivant fourni une représentation de la partie de $\mathcal{G}_2$ à support dans $\{i\infty,0,1\}^2\subset P_2$. Pour une question de lisibilité nous avons représenté ici huit fois les mêmes neuf points de $P_2$ pour illustrer l'ensemble des recouvrement possibles de cette partie.

\setlength{\unitlength}{0.8cm}
\begin{picture}(24,5)
\put(0,0){\line(1,0){4}}
\put(0,2){\line(1,0){4}}
\put(0,4){\line(1,0){4}}
\put(0,0){\line(0,1){4}}
\put(2,0){\line(0,1){4}}
\put(4,0){\line(0,1){4}}
\put(0,2){\line(1,1){2}}
\put(0,0){\line(1,1){4}}
\put(2,0){\line(1,1){2}}
\put(1,0.5){$_{(1,1)}$}
\put(0.2,1.5){$_{(S,S)}$}
\put(3,0.5){$_{(U,1)}$}
\put(2.2,1.5){$_{(SU,S)}$}
\put(1,2.5){$_{(1,U)}$}
\put(0.0,3.5){$_{(S,SU)}$}
\put(3,2.5){$_{(U,U)}$}
\put(2.0,3.5){$_{(SU,SU)}$}
\put(5,0){\line(0,1){4}}
\put(9,0){\line(0,1){4}}
\put(5,0){\line(1,0){4}}
\put(5,2){\line(1,0){4}}
\put(5,4){\line(1,0){4}}
\put(5,2){\line(2,-1){4}}
\put(5,4){\line(2,-1){4}}
\put(10,0){\line(0,1){4}}
\put(14,0){\line(0,1){4}}
\put(10,0){\line(1,0){4}}
\put(12,0){\line(0,1){4}}
\put(10,4){\line(1,0){4}}
\put(10,4){\line(1,-2){2}}
\put(12,4){\line(1,-2){2}}
\put(15,0){\line(0,1){4}}
\put(19,0){\line(0,1){4}}
\put(15,0){\line(1,0){4}}
\put(15,4){\line(1,0){4}}
\put(15,0){\line(1,1){4}}
\put(5.5,0.5){$_{(U^2,1)}$}
\put(5.5,2.5){$_{(U^2,U)}$}
\put(7,1.2){$_{(SU^2,S)}$}
\put(7,3.5){$_{(SU^2,SU)}$}
\put(10.4,3.5){$_{(S,SU^2)}$}
\put(12.4,3.5){$_{(SU,SU^2)}$}
\put(15.1,3.5){$_{(U^2,U^2)}$}
\put(10,0.5){$_{(1,U^2)}$}
\put(12.1,0.5){$_{(U,U^2)}$}
\put(17,0.5){$_{(SU^2,SU^2)}$}
\put(0,-0.4){$_{i\infty}$}
\put(1.8,-0.4){$_{0}$}
\put(3.9,-0.4){$_{1}$}
\put(4.8,-0.4){$_{i\infty}$}
\put(6.8,-0.4){$_{0}$}
\put(8.9,-0.4){$_{1}$}
\put(9.8,-0.4){$_{i\infty}$}
\put(11.8,-0.4){$_{0}$}
\put(13.9,-0.4){$_{1}$}
\put(14.8,-0.4){$_{i\infty}$}
\put(16.8,-0.4){$_{0}$}
\put(18.9,-0.4){$_{1}$}
\put(-0.8,0){$_{i\infty}$}
\put(-0.6,2){$_{0}$}
\put(-0.6,4){$_{1}$}

\end{picture}

\setlength{\unitlength}{0.8cm}
\begin{picture}(24,4)
\put(0,0){\line(1,0){4}}
\put(0,2){\line(1,0){4}}
\put(0,4){\line(1,0){4}}
\put(0,0){\line(0,1){4}}
\put(2,0){\line(0,1){4}}
\put(4,0){\line(0,1){4}}

\put(0,4){\line(1,-1){4}}
\put(0,2){\line(1,-1){2}}
\put(2,4){\line(1,-1){2}}

\put(0.2,0.5){$_{(1,S)}$}
\put(1,1.5){$_{(S,1)}$}
\put(2.2,0.5){$_{(U,S)}$}
\put(2.9,1.5){$_{(SU,1)}$}
\put(0.1,2.5){$_{(1,SU)}$}
\put(1,3.5){$_{(S,U)}$}
\put(2.1,2.5){$_{(U,SU)}$}
\put(2.7,3.5){$_{(SU,U)}$}

\put(5,0){\line(0,1){4}}
\put(9,0){\line(0,1){4}}
\put(5,0){\line(1,0){4}}
\put(5,2){\line(1,0){4}}
\put(5,4){\line(1,0){4}}
\put(5,2){\line(2,1){4}}
\put(5,0){\line(2,1){4}}

\put(10,0){\line(0,1){4}}
\put(14,0){\line(0,1){4}}
\put(10,0){\line(1,0){4}}
\put(12,0){\line(0,1){4}}
\put(10,4){\line(1,0){4}}
\put(10,0){\line(1,2){2}}
\put(12,0){\line(1,2){2}}

\put(15,0){\line(0,1){4}}
\put(19,0){\line(0,1){4}}
\put(15,0){\line(1,0){4}}
\put(15,4){\line(1,0){4}}
\put(15,4){\line(1,-1){4}}

\put(7,0.5){$_{(U^2,S)}$}
\put(7,2.5){$_{(U^2,SU)}$}
\put(5.5,1.2){$_{(SU^2,1)}$}
\put(5.5,3.5){$_{(SU^2,U)}$}

\put(10,3.5){$_{(S,U^2)}$}
\put(12.1,3.5){$_{(SU,U^2)}$}
\put(17,3.5){$_{(U^2,SU^2)}$}
\put(10.5,0.5){$_{(1,SU^2)}$}
\put(12.4,0.5){$_{(U,SU^2)}$}
\put(15.1,0.5){$_{(SU^2,U^2)}$}
\put(0,-0.4){$_{i\infty}$}
\put(1.8,-0.4){$_{0}$}
\put(3.9,-0.4){$_{1}$}
\put(4.8,-0.4){$_{i\infty}$}
\put(6.8,-0.4){$_{0}$}
\put(8.9,-0.4){$_{1}$}
\put(9.8,-0.4){$_{i\infty}$}
\put(11.8,-0.4){$_{0}$}
\put(13.9,-0.4){$_{1}$}
\put(14.8,-0.4){$_{i\infty}$}
\put(16.8,-0.4){$_{0}$}
\put(18.9,-0.4){$_{1}$}
\put(-0.8,0){$_{i\infty}$}
\put(-0.6,2){$_{0}$}
\put(-0.6,4){$_{1}$}

\end{picture}

L'idéal $\mathcal{I}_2$ est le noyau de l'application $\Theta_2:\Z[\Gamma^2]\to H_2(P_2)^0$. Un élément de $\mathcal{I}_2$ correspond donc à une combinaison linéaire de triangles orientés dont les bords sont nuls. On notera $C_2^0(X)\subset\Z[\Gamma^2]$ les tels recouvrements de $2$-chaînes fermées à support dans une partie $X\subset P_2$. La décomposition de $\pte$ nous permet d'obtenir:
$$C_2^0(P_2)=\sum_{i_1,i_2=1}^4 C_2^0\left((B_0\cup B_{i_1})\times (B_0\cup B_{i_2})\right)=\Z[\Gamma^2].C_2^0\left((B_0\cup B_1)\times(B_0\cup B_1)\right).$$\par
On réduit alors ceci par descente sur la hauteur. En effet, les projections suivants les coordonnées respectent la structure du graphe $\mathcal{G}_1$.

\begin{prop}
L'idéal $\mathcal{I}_2$ est engendré comme $\Z[\Gamma^2]$-module par les $2$-chaînes fermées de $\mathcal{G}_2$ à support dans $B_0^2$:
\begin{equation}
\mathcal{I}_2 = \Z[\Gamma^2] . C_2^0(B_0^2).
\end{equation}
\end{prop}

\begin{proof}
Comme $H_2(P_2^2)^0$ est le sous-groupe des chaînes transverses alors on va traiter une coordonnée puis l'autre par projection. Définissons pour tout couple d'entiers $M_1,M_2>0$:
\begin{equation*}
H_2(P_2)^0_{(M_1,M_2)}=\left\{C\in H_2(P_2)^0\text{ tel que }Supp(C)\subset (h\times h)^{-1}([0,M_1]\times[0,M_2])\right\}.
\end{equation*}
Donc les relations vérifiées par $H_2(P_2)^0_{(M_1,M_2)}$ sont les translatés par $(\Gamma,1)\subset\Gamma^2$, agissant sur la première coordonnées, de celles vérifiées par $H_2(P_2)^0_{(1,M_2)}$. Puis le même raisonnement sur la seconde coordonnée réduit l'étude de $H_2(P_2)^0_{(1,M_2)}$ à celle de $H_2(P_2)^0_{(1,1)}$. Ce dernier correspond bien à $C_2^0(B_0^2)$.
\end{proof}

\subsubsection{Générateurs de $\mathcal{I}_2$}

\begin{thm}
L'idéal annulateur dans $\Z[\Gamma^2]$ de $\partial T_2$ est:
\begin{multline*}
\mathcal{I}_2=\Big((1+S,1+S),(1+U+U^2,1)((1,1)+(S,S)),(1,1+U+U^2)((1,1)+(S,S)),\\
(1+U+U^2,1+U+U^2),(S,S)+(S,US)+(US,US)+(1,U)-(U^2,U^2)\Big)
\end{multline*}
\end{thm}

Le schéma précédent permet d'observer ces annulations des images de $\partial T_2$ dans le graphe $\mathcal{G}_2$. Il nous permet de décerner les décompositions utiles pour la démonstration. En effet, les recouvrements s'annulent si les segments horizontaux, verticaux et diagonaux s'annulent entre-eux respectivement. Ces annulations se traduisent par l'appartenance aux idéaux $I_H$, $I_V$ et $I_D$ respectivement.

\begin{proof}
On commence par démontrer que chacun des générateurs appartient bien à $I_H\cap I_D\cap I_V$.
Pour cela, on donne une écriture explicite:
\begin{align*}
(1+S,1+S)&=(1,1+S)(1+S,1)\in I_H\\
&=(1+S,1)(1,1+S)\in I_V\\
&=(1,1+S)[(1,1)+(S,S)]\in I_D 
\end{align*}
\begin{align*}
\text{et }(1+U+U^2,1+U+U^2)&=(1,1+U+U^2)(1+U+U^2,1)\in I_H\\
&=(1+U+U^2,1)(1,1+U+U^2)\in I_V\\
&=(1,1+U+U^2)[(1,1)+(U,U)+(U^2,U^2)]\in I_D.
\end{align*}

Les deux suivants sont dans $I_D$ comme multiple de $(1,1)+(S,S)$ et on a:
\begin{align*}
(1,1+U+U^2)&[(1,1)+(S,S)]=(1,S+US+U^2S)(1+S,1)+(1,1+U+U^2)(1,1-U^2S)\in I_H\\
&=(1,1+U+U^2)+(S,1)(1,1+U+U^2)(1,1+S)-(S,1)(1,1+U+U^2)\in I_V
\end{align*}
et
\begin{align*}
(1+U+U^2,1)&[(1,1)+(S,S)]=(S+US+U^2S,1)(1,1+S)+(1+U+U^2,1)(1,1-US)\in I_V\\
&=(1+U+U^2,1)+(1,S)(1+U+U^2,1)(1+S,1)-(1,S)(1+U+U^2,1)\in I_H.
\end{align*}
Finalement, on décompose $(S,S)+(S,US)+(US,US)+(1,U)-(U^2,U^2)$ en somme de deux éléments de $I_H$ en le découpant selon: 
\begin{align*}
(S,S)+(1,U)&=(1+S,1)+(1,U)(1,1-U^2S)\in I_H,\\
(S,US)+(US,US)-(U^2,U^2)&=(1+U,US)(1+S,1)\\
&-(1,US)(1+U+U^2,1)-(U^2,U^2)(1,1-U^2S)\in I_H.
\end{align*}
C'est aussi une somme d'éléments de $I_V$:
\begin{align*}
(US,US)+(1,U)&=(1,1+S)+(1,U)(1-US,1)\in I_V,\\
(S,S)+(S,US)-(U^2,U^2)&=(S,1+U)(1,1+S)-(S,1)(1,1+U+U^2)-(U^2,U^2)(1-US,1)\in I_V.
\end{align*}
et enfin une somme d'éléments de $I_D$:
\begin{align*}
(S,US)+(1,U)&=(1,U)[(1,1)+(S,S)]\in I_D,\\
(S,S)+(US,US)-(U^2,U^2)&=[(1,1)+(U,U)][(1,1)+(S,S)]-[(1,1)+(U,U)+(U^2,U^2)]\in I_D.
\end{align*}
Or on a réduit $\mathcal{I}_2$ à l'étude des relations vérifiées par les triangles à coordonnées parmi $\{0,i\infty,1,-1\}$. Une étude exhaustive à la main des combinaisons des $18$ triangles obtenus donne bien l'égalité des idéaux. Cette étude a été confirmée par ordinateur par un calcul exhaustif des combinaisons possibles.
\end{proof}

\begin{coro}
Le polynôme des bipériodes vérifie les équations:
\begin{align*}
P_{f_1,f_2}|&_{(1+S,1+S)}=P_{f_1,f_2}|_{[(1,1)+(S,S)](1+U+U^2,1)}=P_{f_1,f_2}|_{[(1,1)+(S,S)](1,1+U+U^2)}=0\\
P_{f_1,f_2}|&_{(1+U+U^2,1+U+U^2)}=P_{f_1,f_2}|_{(S,S)+(S,SU^2)+(SU^2,SU^2)+(1,U^2)-(U,U)}=0.
\end{align*}
\end{coro}

\begin{rem}
Le calcul de $\mathcal{I}_2^{(\epsilon_1,\epsilon_2)}$ est dû à la conjugaison de $\mathcal{I}_2=I_H\cap I_D\cap I_V$ par $\varepsilon$. Or $I_H$ et $I_V$ sont stables par cette conjugaison. Donc il résulte que :
$$\mathcal{I}_2^{(\epsilon_1,\epsilon_2)}=I_H\cap I_D^{(\epsilon_1,\epsilon_2)}\cap I_V.$$
Ainsi $\mathcal{I}_2^{(+,+)}=\mathcal{I}_2^{(-,-)}=\mathcal{I}_2$ et :
\begin{multline*}
\mathcal{I}_2^{(+,-)}=\mathcal{I}_2^{(-,+)}=\Big((1+S,1+S),(1+U+U^2,1)((1,1)+(S,S)),(1,1+U+U^2)((1,1)+(S,S)),\\
(1+U+U^2,1+V+V^2),(S,S)+(S,VS)+(US,VS)+(1,V)-(U^2,V^2)\Big).
\end{multline*}
Ceci permet de considérer les parties paires et impaires globales des polynômes des bipériodes définies comme:
\begin{align}
P^+(X_1,X_2)&=1/2\left(P(X_1,X_2)+P(-X_1,-X_2)\right)=\sum_{m_1+m_2\text{ pair}}A_{m_1,m_2}X_1^{m_1}X_2^{m_2},\\
P^-(X_1,X_2)&=1/2\left(P(X_1,X_2)-P(-X_1,-X_2)\right)=\sum_{m_1+m_2\text{ impair}}A_{m_1,m_2}X_1^{m_1}X_2^{m_2}.
\end{align}
Si $P\in V_{k_1,k_2}^{\Z}[\mathcal{I}_2]$ alors $P^+$ et $P^-$ sont aussi des éléments de $V_{k_1,k_2}^{\Z}[\mathcal{I}_2]$.
\end{rem}

\section{Contrôle de l'espace des bipériodes}

\subsection{Décomposition de $W_k$ (Rappel)}

Dans le cas classique, on sait que les polynômes des périodes vérifient les relations de Manin et ainsi appartiennent à:
$W_k=\{P\in V_k;P|_{1+S}=P|_{1+U+U^2}=0\}$.
Il devient alors naturel de définir les applications linéaires:
\begin{align*}
R_{k}:S_{k} &\to W_k,\quad
f\mapsto P_{f}(X),\\
R^{\pm}_{k}:S_{k} &\to W_k^{\pm},\quad
f\mapsto P^{\pm}_{f}(X).
\end{align*}
Posons $\Per_k=\Im R_k$ et $\Per_k^{\pm}=\Im R_k^{\pm}$. Le Théorème de Eichler-Shimura donne l'injectivité de ces trois applications et leur image est donnée par:
\begin{align}
W_k&=\Per_k\oplus \overline{\Per_k} \oplus \C P_{G_k}^{+},\\
W_k^+&=\Per_k^+\oplus \C P_{G_k}^{+}\text{ et }W_k^-=\Per_k^-.
\end{align}

Dans la suite, on préfèrera l'écriture suivante qui permet une généralisation plus simple:
\begin{equation}
P_f^{\pm}(X)=P_f(X)\pm P_f(-X)=\int_{\tau_1}\omega_f \pm \int_{\tau_1} \overline{\omega_f},\text{ pour tout }f\in S_k^{\R}.
\end{equation}
Définissons $\Omega_k^+\subset\left(\Omega^1_{par}(\H,\C)\otimes_{\C} V_k\right)^{\Gamma}$ le $\R$-espace vectoriel des formes différentielles holomorphes sur $\H$, nulles en $i\infty$ à valeurs dans $V_k$ et stable par l'action de $\Gamma$ et à coefficients de Fourier réels.
L'espace $\Omega_k^+$ est un isomorphe à $S_k^{\R}$ via l'application:
\begin{equation}
f\mapsto \omega_f(z,X)=f(z)(X-z)^{k-2}dz.
\end{equation}
L'action de la conjugaison sur $\Omega_k^+$ est donnée par la formule:
\begin{equation}
\overline{\omega_f(z,X)}=\omega_f(-\bar{z},-X).
\end{equation}
Cette propriété en faite équivalente à la condition $\{a_n(f);n>0\}\subset\R$. En effet, on a pour tout $f\in S_k$:
\begin{align*}
\overline{\omega_f(z,X)}&=\sum_{n>0} \overline{a_n(f)} exp(-2i\pi n\bar{z})(X-\bar{z})^{k-2}(-d\bar{z}),\\
\text{et }\omega_f(-\bar{z},-X)&=\sum_{n>0} a_n(f) exp(-2i\pi n\bar{z})(-X+\bar{z})^{k-2}(-d\bar{z}).
\end{align*}
Le conjugué complexe $\Omega_k^-=\overline{\Omega_k^+}$ est donc le $\R$-espace vectoriel des formes différentielles antiholomorphes sur $\H$, nulles en $i\infty$ à valeurs dans $V_k$ et stable par l'action de $\Gamma$ et à coefficients de Fourier réels.

Posons plus généralement, pour toute famille de poids $k_1,...,k_n\geq 2$ et de signes $\epsilon_1,...,\epsilon_n\in\{\pm 1\}^n$:
\begin{equation}
\Omega_{k_1,...,k_n}^{\epsilon_1,...,\epsilon_n}=\otimes_{j=1}^n\Omega_{k_n}^{\epsilon_n}\subset \Omega^n_{par}(\H^n,\C)\otimes_{\C[\Gamma^n]} V_{k_1,...,k_n}.
\end{equation}

La chaine $\tau_1$ est stable par l'involution de $\H$, $z\mapsto -\bar{z}$, donc pour tout $\gamma\in \Gamma$,
\begin{equation}
\langle\overline{\omega_f},\gamma. \tau_1\rangle=\langle \omega_f,(\varepsilon\gamma\varepsilon). \tau_1\rangle.
\end{equation}

Ceci nous permet de considérer les applications:
\begin{equation}
\Pi^{\epsilon}_k:\Omega_k^{\epsilon}\to V_k[\mathcal{I}_1^{\epsilon}],\quad
\omega\mapsto \int_{\tau_1}\omega.
\end{equation}
Et on a notamment $R_k^{\pm}(f)=\Pi_k^+(\omega_f)\pm\Pi_k^-(\overline{\omega_f})$ pour tout $f\in S_k^{\R}$. Ce système est clairement inversible et on obtient:
\begin{equation*}
\Pi_k^+(\omega_f)=1/2(R^+_k(f)+R^-_k(f))\text{ et }\Pi_k^-(\overline{\omega_f})=1/2(R^+_k(f)-R^-_k(f)).
\end{equation*}
Or on a démontré que $\mathcal{I}_1^+=\mathcal{I}_1^-$ et donc $V_k[\mathcal{I}_1^{\pm}]=W_k$.\par

\subsection{Décomposition de $W_{k_1,k_2}$}

Soit $(\epsilon_1,\epsilon_2)\in\{\pm 1\}^2$. Posons:
\begin{equation}
V_{k_1,k_2}^{\Q}[\mathcal{I}_2^{\epsilon_1,\epsilon_2}]=\{P\in V_{k_1,k_2}^{\Q}\text{ tel que }P|_{\gamma}=0\text{ pour } \gamma\in\widetilde{\mathcal{I}_2^{\epsilon_1,\epsilon_2}}\}.
\end{equation}
Nous avons ici adapté la notation d'annulateur à un idéal à gauche.\par
Définissons alors les applications:
\begin{equation}
\Pi_{k_1,k_2}^{\epsilon_1,\epsilon_2}:\Omega_{k_1,k_2}^{\epsilon_1,\epsilon_2}\to V_{k_1,k_2}[\mathcal{I}_2^{\epsilon_1,\epsilon_2}],\quad
\omega_1\wedge\omega_2\mapsto \int_{\tau_2}\omega_1(z_1,X_1)\wedge\omega_2(z_2,X_2).
\end{equation}
Comme $\tau_2$ est stable par les involutions $z_j\mapsto -\bar{z_j}$, pour $j=1,2$. Alors l'application,
$$P(X_1,X_2)\mapsto P(\epsilon_1X_1,\epsilon_2X_2),$$
est une involution respectant les espaces $V_{k_1,k_2}[\mathcal{I}_2^{\pm,\pm}]\to V_{k_1,k_2}[\mathcal{I}_2^{\pm\epsilon_1,\pm\epsilon_2}]$.\par
On rappel que:
$$\mathcal{I}_2^{+,+}=\mathcal{I}_2^{-,-}\text{ et }\mathcal{I}_2^{+,-}=\mathcal{I}_2^{-,+}.$$
Cette stabilité par conjugaison double permet de définir les espaces:
\begin{align}
W_{k_1,k_2}^{+,\Q}&=\{P(X_1,X_2)+ P(-X_1,-X_2);P\in V_{k_1,k_2}^{\Q}[\mathcal{I}_2]\}=V_{k_1,k_2}^{+,\Q}[\mathcal{I}_2]\subset V_{k_1,k_2}^{+,\Q},\\
\text{et }W_{k_1,k_2}^{-,\Q}&=\{P(X_1,X_2)-P(-X_1,-X_2);P\in V_{k_1,k_2}^{\Q}[\mathcal{I}_2]\}=V_{k_1,k_2}^{-,\Q}[\mathcal{I}_2]\subset V_{k_1,k_2}^{-,\Q}.
\end{align}
Puis de définir les applications indexées par un signe:
\begin{align}
&R_{k_1,k_2}^{\pm} : S_{k_1}^{\R} \otimes S_{k_2}^{\R} \to W_{k_1,k_2}^{\pm},\nonumber\\
&f_1\otimes f_2\mapsto 1/2\left(\Pi_{k_1,k_2}^{+,+}(\omega_{f_1}\wedge\omega_{f_2})\pm\Pi_{k_1,k_2}^{-,-}(\overline{\omega_{f_1}}\wedge\overline{\omega_{f_2}})\right).
\end{align}
On notera respectivement $Per^{+,\R}_{k_1,k_2}$ et $Per^{-,\R}_{k_1,k_2}$ les images de $R_{k_1,k_2}^{+}$ et $R_{k_1,k_2}^{-}$.

\begin{prop}Soient $(\epsilon_1,\epsilon_2)\in\{\pm 1\}^2$, $\omega_1\in\Omega_{k_1}^{\epsilon_1}$ et $\omega_2\in\Omega_{k_2}^{\epsilon_2}$. Soit $(f_1,f_2)\in S_{k_1}^{\R}\times S_{k_2}^{\R}$.\\
1) On dispose des formules liant les différentes applications:
\begin{align}
\Pi_{k_1,k_2}^{\epsilon_1,\epsilon_2}(\omega_1\wedge\omega_2)(X_1,X_2)|_{(1,1)+(S,S)}&=\Pi_{k_1}^{\epsilon_1}(\omega_1)(X_1)\Pi_{k_2}^{\epsilon_2}(\omega_2)(X_2),\label{pi2SS}\\
R_{k_1,k_2}^+(f_1\otimes f_2)(X_1,X_2)|_{(1,1)+(S,S)}&=R_{k_1}^+(f_1)(X_1)R_{k_2}^+(f_2)(X_2)+R_{k_1}^-(f_1)(X_1)R_{k_2}^-(f_2)(X_2),\\
R_{k_1,k_2}^-(f_1\otimes f_2)(X_1,X_2)|_{(1,1)+(S,S)}&=R_{k_1}^+(f_1)(X_1)R_{k_2}^-(f_2)(X_2)+R_{k_1}^-(f_1)(X_1)R_{k_2}^+(f_2)(X_2).
\end{align}
2) Les applications $\Pi_{k_1}^{\varepsilon_1}\otimes\Pi_{k_2}^{\varepsilon_2}$ et $\Pi_{k_1,k_2}^{\varepsilon_1,\varepsilon_2}$ sont injectives.\\
3) Les applications $R^+_{k_1,k_2}$ et $R_{k_1,k_2}^-$ sont injectives.
\end{prop}

\begin{proof}
La formule (\ref{pi2SS}) repose sur l'identité: $[(1,1)+(S,S)]\tau_2=\tau_1\times\tau_1$. De plus, $S$ et $\varepsilon$ commutent dans $PSL_2(\Z)$ donc $(1,1)+(S,S)$ envoie bien les différents espaces $V_{k_1,k_2}[\mathcal{I}_2^{\epsilon_1,\epsilon_2}]$ dans $V_{k_1}[\mathcal{I}_1^{\epsilon_1}]\otimes V_{k_2}[\mathcal{I}_1^{\epsilon_2}]$. Cette formule induit les deux autres par récriture des définitions.\par
L'injectivité de $\Pi_{k_1}^{\epsilon_1}\otimes\Pi_{k_2}^{\epsilon_2}$ se déduit de celle de $\Pi_{k_1}^{\varepsilon_1}$ et $\Pi_{k_2}^{\varepsilon_2}$, due au Théorème d'Eichler-Shimura. En effet, pour deux applications linéaires $f_1:E_1\to F_1$ et $f_2:E_2\to F_2$, on a: 
$$\Ker(f_1\otimes f_2)=\Ker(f_1)\otimes E_2 + E_1\otimes \Ker(f_2).$$
On en déduit ensuite l'injectivité de $\Pi_{k_1,k_2}^{\varepsilon_1,\varepsilon_2}$ et de $R^{\pm}_{k_1,k_2}$ d'après les formules du $1)$.
\end{proof}

Les résultats de cette proposition sont encore valides sur $\C$ car on a $S_{k_1}\times S_{k_2}=\left(S_{k_1}^{\R}\times S_{k_2}^{\R}\right)\otimes_{\R}\C$.

Les équations données par $\mathcal{I}_2$ permettent de définir le sous-groupe de $V_{k_1,k_2}^{\Z}$ :
\begin{equation}
W_{k_1,k_2}^{\Z}=\left\{P\in V_{k_1,k_2}^{\Z}\text{ tel que }P|_{g}=0,\text{ pour tout }g\in \widetilde{\mathcal{I}_2}\right\},
\end{equation}
qui se scinde à nouveau en parties paire et impaire.\par
L'image de $W_{k_1,k_2}^{\Z}$ par l'application $V_{k_1,k_2}^{\Z}\to V_{k_1,k_2}^{\Z}, P\mapsto P|_{(1,1)+(S,S)}$ est contenue dans:
\begin{equation}
V_{k_1}^{\Z}[\mathcal{I}_1]\otimes V_{k_2}^{\Z}[\mathcal{I}_1]=V_{k_1,k_2}^{\Z}\big[(1,1+S),(1,1+U+U^2),(1+S,1),(1+U+U^2,1)\big].\label{corowtw}
\end{equation}\par

De même, pour les idéaux $I_H,I_V$ et $I_D$ définissant $\mathcal{I}_2$ comme leurs intersections, on peut définir les sous-groupes de $V_{k_1,k_2}^{\Z}$ par:
\begin{align}
V_{k_1,k_2}^{\Z}[I_H]&=\{P\in V_{k_1,k_2}^{\Z};P|_{(1+S,1)}=P|_{(1+U+U^2,1)}=P|_{(1,1-T)}=0\},\\
V_{k_1,k_2}^{\Z}[I_V]&=\{P\in V_{k_1,k_2}^{\Z};P|_{(1,1+S)}=P|_{(1,1+U+U^2)}=P|_{(1-US,1)}=0\},\\
\text{et }V_{k_1,k_2}^{\Z}[I_D]&=\{P\in V_{k_1,k_2}^{\Z};P|_{(1,1)+(S,S)}=P|_{(1,1)+(U,U)+(U^2,U^2)}=0\}.
\end{align}
On peut à nouveau scinder ces sous-groupes suivants: $V_{k_1,k_2}^{\Z}=V_{k_1,k_2}^{+,\Z}\oplus V_{k_1,k_2}^{-,\Z}$ car la conjugaison par $(\varepsilon,\varepsilon)$ laisse stable ces idéaux de $\Z[\Gamma^2]$.

\begin{prop}
On peut décrire explicitement les sous-groupes $V_{k_1,k_2}^{\Z}[I_H]$ et $V_{k_1,k_2}^{\Z}[I_V]$:
\begin{equation}
V_{k_1,k_2}^{\Z}[I_H]=W_{k_1}^{\Z}\otimes 1\text{ et }V_{k_1,k_2}^{\Z}[I_V]=X_1^{k_1-2}\otimes W_{k_2}^{\Z}.
\end{equation}
\end{prop}

\begin{proof}
Les idéaux $I_H$ et $I_V$ sont engendrés par des éléments de la forme $(\gamma_1,1)$ et $(1,\gamma_2)$ qui commutent entre eux. Or dans ce cas, on a:
$$V_{k_1,k_2}^{\Z}[(I_1,1)+(1,I_2)]=V_{k_1,k_2}^{\Z}[(I_1,1)]\cap V_{k_1,k_2}^{\Z}[(1,I_2)].$$
Le calcul des quatre espaces obtenus se déduit d'une part de $W_k^{\Z}=V_k^{\Z}[(1+S,1+U+U^2)]$:
$$V_{k_1,k_2}^{\Z}\big[((1+S,1),(1+U+U^2,1))\big]=W_{k_1}^{\Z}\otimes V_{k_2}^{\Z},$$
$$V_{k_1,k_2}^{\Z}\big[((1,1+S),(1,1+U+U^2))\big]=V_{k_1}^{\Z}\otimes W_{k_2}^{\Z}.$$
Et d'autre part de $V_{k}^{\Z}[1-T]=\Z$ et $V_k^{\Z}[1-US]=\Z X^{k-2}$ donnant:
$$V_{k_1,k_2}^{\Z}[(1,1-T)]=V_{k_1}^{\Z}\otimes 1,$$
$$V_{k_1,k_2}^{\Z}[(1-US,1)]=X_1^{k_1-2}\otimes V_{k_2}^{\Z}.$$
Les intersections nous donne bien les sous-groupes attendus.
\end{proof}

\begin{thm}[Décomposition de $W_{k_1,k_2}$]
On a la décomposition sur $\C$:
\begin{equation}
W_{k_1,k_2}=\Per_{k_1,k_2}\oplus \overline{\Per_{k_1,k_2}}\oplus E_{k_1,k_2}^{\C},
\end{equation}
où $E_{k_1,k_2}^{\Z}=V_{k_1,k_2}^{\Z}[I_H]+ V_{k_1,k_2}^{\Z}[I_D]+ V_{k_1,k_2}^{\Z}[I_V]$. Ils sont deux à deux disjoints et l'unique relation de dépendance est donnée par:
$$1-X_1^{k_1-2}X_2^{k_2-2}\in V_{k_1,k_2}^{\Z}[I_D]\cap\left(V_{k_1,k_2}^{\Z}[I_H]\oplus V_{k_1,k_2}^{\Z}[I_V]\right).$$
Ceci se spécialise suivant les parités en:
\begin{equation}
W_{k_1,k_2}^{\pm}=\Per_{k_1,k_2}^{\pm}\oplus E_{k_1,k_2}^{\pm,\C}.
\end{equation}
Pour tout $(\epsilon_1,\epsilon_2)\in\{\pm 1\}^2$, on dispose de la suite exacte de $\Z$-modules:
$$0\rightarrow\Z\rightarrow V_{k_1,k_2}^{\Z}[I_H]\times V_{k_1,k_2}^{\Z}[I_D]\times V_{k_1,k_2}^{\Z}[I_V]
\stackrel{\sum}{\longrightarrow} W_{k_1,k_2}^{\Z}
\stackrel{(1,1)+(S,S)}{\longrightarrow} (W_{k_1}^{\Z}/E_{k_1}^{\Z})^{\epsilon_1}\otimes (W_{k_2}^{\Z}/E_{k_2}^{\Z})^{\epsilon_2}\rightarrow 0,$$
où on rappel $E_{k_j}^{\Z}=<1-X_j^{k_j-2}>_{\Z}$ pour $j=1,2$. 
\end{thm}

\begin{proof}
On introduit l'application:
$$\varphi_S^{\Z}: V_{k_1,k_2}^{\Z}[\mathcal{I}_2]+V_{k_1,k_2}^{\Z}[\mathcal{I}_2^{-,+}]\to \left(W_{k_1}^{\Z}/E_{k_1}^{\Z}\right)\otimes \left(W_{k_2}^{\Z}/E_{k_2}^{\Z}\right),P\mapsto \left[P|_{(1,1)+(S,S)}\right].$$
Elle est bien définie car la détermination de $\mathcal{I}_2$ montre que l'image par $(1,1)+(S,S)$ est dans:
$$V_{k_1,k_2}^{\Z}[(1+S,1),(1+U+U^2,1),(1,1+S),(1,1+U+U^2)]=W_{k_1}^{\Z}\otimes W_{k_2}^{\Z}\text{ d'après (\ref{corowtw})}.$$
Cet espace se projette bien dans $\left(W_{k_1}^{\Z}/E_{k_1}^{\Z}\right)\otimes \left(W_{k_2}^{\Z}/E_{k_2}^{\Z}\right)$ en prenant la classe de l'image $P|_{(1,1)+(S,S)}$ dans cet espace quotient.\par

On peut étendre l'application au corps des complexe et noter : $\varphi_S=\varphi_S^{\Z}\otimes \C$. On montre alors que $\varphi_S$ est surjective. 
Soit $P_1\otimes P_2\in W_{k_1}\otimes W_{k_2}$. 
Alors d'après le Théorème d'Eichler-Shimura les applications:
$$\Pi_{k_j}:\Omega_{k_j}^+\oplus \Omega_{k_j}^-\to W_{k_j}/E_{k_j}, \omega \mapsto \langle \omega,\tau_1 \rangle,$$
sont des bijections. Ainsi la classe de $P_j$ dans $W_{k_j}/E_{k_j}$ admet un unique antécédent $\omega_j\in\Omega_{k_j}^+\oplus \Omega_{k_j}^-$ pour $j=1,2$.
Donc posons $\omega=\omega_1\wedge\omega_2\in\bigoplus_{\epsilon_1,\epsilon_2}\Omega_{k_1,k_2}^{\epsilon_1,\epsilon_2}$, on a $\Pi_{k_1}\otimes\Pi_{k_2}(\omega)=\left[P_1\otimes P_2\right]$. Posons :
$$P=\Pi_{k_1,k_2}(\omega)\in V_{k_1,k_2}[\mathcal{I}_2]+V_{k_1,k_2}[\mathcal{I}_2^{-,+}].$$
On obtient :
$$\varphi_S(P)=\Pi_{k_1,k_2}(\omega)|_{(1,1)+(S,S)}=\Pi_{k_1}\otimes\Pi_{k_2}(\omega)=\left[P_1\otimes P_2\right].$$
Ceci nous donne une section uniquement valable sur le corps des complexes. On remarque de plus que l'image de cette section est l'image de $\Pi_{k_1,k_2}$, les polynômes des bipériodes de formes modulaires harmoniques.\par

Calculons désormais le noyau de l'application $\Z$-linéaire $\varphi_S^{\Z}$. 
On commence par regarder les éléments de $V_{k_1,k_2}^{\Z}[\mathcal{I}_2]$ annulés par $(1,1)+(S,S)$ dans $W_{k_1}^{\Z}\otimes W_{k_2}^{\Z}$. Ce sont les éléments annulés par l'idéal:
$$\mathcal{I}_2+[(1,1)+(S,S)]\Z[\Gamma^2]=\left[(1,1)+(S,S);(1,1)+(U,U)+(U^2,U^2)\right]\Z[\Gamma^2]=I_D.$$
Ceci démontre que $\{P\in V_{k_1,k_2}^{\Z}[\mathcal{I}_2]\text{ tel que }P|_{(1,1)+(S,S)}=0\}=V_{k_1,k_2}^{\Z}[I_D].$\par

Puis introduisons l'involution de $V_{k_1,k_2}^{\Z}$: $\delta:P(X_1,X_2)\mapsto P(-X_1,X_2)$. Pour tout idéal $I$ de $\Z[\Gamma^2]$, elle échange les espaces $V_{k_1,k_2}[I]$ et $V_{k_1,k_2}[I^{-,+}]$ . Comme $\varepsilon$ et $S$ commutent alors $\delta$ et $\varphi_S$ aussi et on en déduit que:
$$\{P\in V_{k_1,k_2}^{\Z}[\mathcal{I}_2^{-,+}]\text{ tel que }P|_{(1,1)+(S,S)}=0\}=\delta(V_{k_1,k_2}^{\Z}[I_D])=V_{k_1,k_2}^{\Z}[I_D^{-,+}].$$\par

Il nous reste à calculer les éléments qui s'envoient sur $W_{k_1}\otimes E_{k_2}+E_{k_1}\otimes W_{k_2}$. Pour cela, on décompose:
$$W_{k_1}^{\Z}\otimes E_{k_2}^{\Z}=W_{k_1}^{\Z}\otimes \left(V_{k_2}^{\Z}[1-US]|_{(1-S)}\right)=\left(W_{k_1}^{\Z}\otimes V_{k_2}^{\Z}[1-US]\right)|_{(1,1-S)}.$$
Et on obtient bien, pour $P\in V_{k_1,k_2}^{\Z}$ la chaine d'équivalence:
\begin{align*}
P\in V_{k_1,k_2}^{\Z}[I_H]&\Leftrightarrow P|_{(1+S,1)}=P_{(1+U+U^2,1)}=P|_{(1,1-US)}=0\\
&\Leftrightarrow P|_{(1,1)+(S,S)}=P|_{(1,1-S)}\text{ et }P|_{(1+S,1)}=P_{(1+U+U^2,1)}=P|_{(1,1-US)}=0\\
&\Leftrightarrow P|_{(1,1)+(S,S)}\in \left(V_{k_1}^{\Z}[1+S,1+U+U^2]\otimes V_{k_2}^{\Z}[1-US]\right)|_{(1,1-S)}.
\end{align*}
Ceci se symétrise pour $I_V$ sans difficulté car on a aussi $E_{k_1}^{\Z}=V_{k_1}^{\Z}[1-T]|_{(1-S)}$. 

On obtient ainsi:
$$\Ker \left(\varphi_S^{\Z}\right)=V_{k_1,k_2}^{\Z}[I_H]+ V_{k_1,k_2}^{\Z}[I_V]+ V_{k_1,k_2}^{\Z}[I_D]+ V_{k_1,k_2}^{\Z}[I_D^{-,+}]=E_{k_1,k_2}^{\Z}+\delta(E_{k_1,k_2}^{\Z}).$$\par

Le calcul de $V_{k_1,k_2}^{\Z}[I_H]$ et $V_{k_1,k_2}^{\Z}[I_D]$ permet d'observer qu'ils sont disjoints. Or on a:
$$1-X_1^{k_1-2}X_2^{k_2-2}=(1-X_1^{k_1-2})+X_1^{k_1-2}(1-X_2^{k_2-2}).$$
Ceci donne la première flèche de la suite exacte du théorème:
\begin{align*}
\Z\to &V_{k_1,k_2}^{\Z}[I_H]\times V_{k_1,k_2}^{\Z}[I_D]\times V_{k_1,k_2}^{\Z}[I_V],\\
\alpha\mapsto &\alpha\left((1-X_1^{k_1-2}),-1+X_1^{k_1-2}X_2^{k_2-2},X_1^{k_1-2}(1-X_2^{k_2-2})\right).
\end{align*}
Puis on a:
$$V_{k_1,k_2}^{\Z}[I_D]\cap\left(V_{k_1,k_2}^{\Z}[I_H]\oplus V_{k_1,k_2}^{\Z}[I_V]\right)=\langle 1-X_1^{k_1-2}X_2^{k_2-2}\rangle_{\Z}.$$
Ce dernier est stable par $\delta$ et on obtient aussi:
$$V_{k_1,k_2}^{\Z}[I_D^{-,+}]\cap\left(V_{k_1,k_2}^{\Z}[I_H]\oplus V_{k_1,k_2}^{\Z}[I_V]\right)=\langle 1-X_1^{k_1-2}X_2^{k_2-2}\rangle_{\Z}.$$\par
 
Le Théorème des noyaux appliqué à $\varphi_S$ sur $\C$ donne alors:
$$V_{k_1,k_2}[\mathcal{I}_2]+V_{k_1,k_2}[\mathcal{I}_2^{-,+}]
=\left(\Ker(\varphi_S^{\Z})\otimes\C\right)\oplus \bigoplus_{\epsilon_1,\epsilon_2} Per^{\epsilon_1,\epsilon_2}_{k_1,k_2}.$$
On obtient ainsi les égalités de $\C$-espaces vectoriels:
\begin{align*}
W_{k_1,k_2}&=\left(E_{k_1,k_2}^{\Z}\otimes\C\right)\oplus \Per_{k_1,k_2}\oplus \Per_{k_1,k_2}^{-,-}\\
\delta(W_{k_1,k_2})&=\left(\delta(E_{k_1,k_2}^{\Z})\otimes\C\right)\oplus \Per^{-,+}_{k_1,k_2}\oplus \Per_{k_1,k_2}^{+,-}.
\end{align*}
Et on observe que $\Per_{k_1,k_2}^{-,-}=\overline{\Per_{k_1,k_2}}$.
\end{proof}

\subsection{Calcul de $V_{k_1,k_2}^{\Q}[I_D]$}

Pour déterminer $E_{k_1,k_2}^{\Q}$, il nous reste à préciser le calcul de $V_{k_1,k_2}^{\Q}[I_D]$. Il peut se faire par une récurrence sur la somme des poids.

\begin{prop}
On dispose d'un isomorphisme :
$$V_{k_1,k_2}^{\Q}[I_D]\cong\bigoplus_{i=0}^{\min(k_1,k_2)} W_{k_1+k_2-2i}^{\Q}.$$
On peut spécialiser ce résultat selon les parités et on obtient:
$$V_{k_1,k_2}^{\Q}[I_D] \cap V_{k_1,k_2}^{\epsilon}\cong\bigoplus_{i=0}^{\min(k_1,k_2)} W_{k_1+k_2-2i}^{\Q,(-1)^i\varepsilon}.$$
\end{prop}

\begin{proof}
Par définition, on a : 
$$V_{k_1,k_2}^{\Q}[I_D]=\{P\in V_{k_1,k_2}(\Q)|P|_{(1,1)+(S,S)}=P|_{(1,1)+(U,U)+(U^2,U^2)}=0\}.$$
Ceci nous permet de construire la suite exacte de $\Q$-espaces vectoriels:
$$0 \longrightarrow V_{k_1-1,k_2-1}^{\Q}[I_D] \stackrel{\phi}{\longrightarrow} V_{k_1,k_2}^{\Q}[I_D] \stackrel{\psi}{\longrightarrow} W_{k_1+k_2-2}^{\Q} \longrightarrow 0$$
où on définit les $\Gamma$-morphismes par $\phi(P)(X,Y)=(X-Y)P(X,Y)$ et $\psi(P)(Z)=P(Z,Z)$.\par
L'application $\phi$ est bien un $\Gamma$-morphisme car $(\gamma X-\gamma Y)(cX+d)(cY+d)=(ad-bc)(X-Y)$ pour tout $\mat{a}{b}{c}{d}\in\Gamma$. De plus, elle est bien injective car nous pouvons l'inverser sur l'image en divisant par $X-Y$.\par
L'image de $\phi$ est bien l'ensemble des polynômes s'annulant diagonalement, c'est à dire le noyau de $\psi$.\par
Pour démontrer la surjectivité de $\psi$, nous construisons un antécédent de tout polynôme $P(Z)\in W_{k_1+k_2-2}^{\Z}$. On commence par le cas du polynôme particulier $1-Z^{k_1+k_2-4}$. Il est l'antécédent par $\psi$ de $1-X_1^{k_1-2}X_2^{k_2-2}\in V_{k_1,k_2}^{\Z}$. Cet élément appartient à $V_{k_1,k_2}^{\Z}[I_D]$. En effet, on a:
\begin{align*}
(1-X_1^{k_1-2}X_2^{k_2-2})|_{(1,1)+(S,S)}&=\left[1-X_1^{k_1-2}X_2^{k_2-2}\right]+\left[X_1^{k_1-2}X_2^{k_2-2}-1\right]=0\\
\text{et }(1-X_1^{k_1-2}X_2^{k_2-2})|_{(1,1)+(U,U)+(U^2,U^2)}&=\left[1-X_1^{k_1-2}X_2^{k_2-2}\right]+\left[(-X_1+1)^{k_1-2}(-X_2+1)^{k_2-2}-1\right]\\
&+\left[X_1^{k_1-2}X_2^{k_2-2}-(-X_1+1)^{k_1-2}(-X_2+1)^{k_2-2}\right]=0.
\end{align*}
Lorsque $P(Z)$ est quelconque, le Théorème d'Eichler-Shimura donne l'existence et l'unicité d'application $f,g\in S_{k_1+k_2-2}$ ainsi que d'un scalaire $\lambda\in\C$ telles que:
$$P_f^-(Z)=P^-(Z)\text{ et }P_g^+(Z)=P^+(Z)+\lambda( 1-Z^{k_1+k_2-4}).$$
On peut définir les polynômes suivant de $V_{k_1,k_2}$, pour $h\in S_{k_1+k_2-2}$:
$$Q_h(X_1,X_2)=\int_0^{\infty}h(it)(X_1-it)^{k_1-2}(X_2-it)^{k_2-2}dt.$$
Les polynômes $Q_f$ et $Q_g$ sont bien définis car $f$ et $g$ sont paraboliques. De plus, ils vérifient bien les relations de Manin diagonales car pour $h\in S_{k_1+k_2-2}$:
$$Q_h(X_1,X_2)|_{(\gamma,\gamma)}=i\int_{\gamma^{-1}i\infty}^{\gamma^{-1}0}h(z)(X_1-z)^{k_1-2}(X_2-z)^{k_2-2}dz.$$
Et ainsi on déduit : $Q_h|_{(1,1)+(S,S)}=Q_h|_{(1,1)+(U,U)+(U^2,U^2)}=0$ puis $Q_h\in V_{k_1,k_2}[I_D]$.\par
Les relations $Q_f(Z,Z)=P_f(Z)$ et $Q_g(Z,Z)=P_g(Z)$ se déduisent de cette construction.
Et ainsi: $P(Z)=\psi(Q_f^-(X_1,X_2)+Q_g^+(X_1,X_2))+\lambda\psi(1-X_1^{k_1-2}X_2^{k_2-2})$.
\end{proof}

\begin{ex}
On peut calculer explicitement ces espaces pour les petites dimensions:\par
1) Lorsque $k_1=2$ ou $k_2=2$, il n'y a qu'une variable et on a:
$$V_{2,k}^{\Q}[I_D]\cong V_{k,2}^{\Q}[I_D]\cong W_k^{\Q},\quad\text{pour tout } k\geq 2.$$\par
2) Pour tout $k\leq 10$, $W_k^{\Z}=<1-Z^{k-2}>_{\Z}$ donc on a pour $k_1+k_2\leq 12$:
\begin{align*}
V_{4,4}^{\Z}[I_D]&=<1-X_1^2X_2^2,(X_1-X_2)(1-X_1X_2)>,\\
V_{4,6}^{\Z}[I_D]&=<1-X_1^2X_2^4,(X_1-X_2)(1-X_1X_2^3),(X_1-X_2)^2(1-X_2^2)>,\\
V_{4,6}^{\Z}[I_D]&=<1-X_1^4X_2^2,(X_1-X_2)(1-X_1^3X_2),(X_1-X_2)^2(1-X_1^2)>,\\
V_{4,8}^{\Z}[I_D]&=<1-X_1^2X_2^6,(X_1-X_2)(1-X_1X_2^5),(X_1-X_2)^2(1-X_2^4)>,\\
V_{8,4}^{\Z}[I_D]&=<1-X_1^6X_2^2,(X_1-X_2)(1-X_1^5X_2),(X_1-X_2)^2(1-X_1^4)>,\\
V_{6,6}^{\Z}[I_D]&=<1-X_1^4X_2^4,(X_1-X_2)(1-X_3X_2^3),(X_1-X_2)^2(1-X_1^2X_2^2),(X_1-X_2)^3(1-X_1X_2)>.
\end{align*}
\end{ex}

Ce résultat de décomposition nous permet notamment de calculer la dimension des $\Q$-espaces vectoriels $V_{k_1,k_2}^{\Q}[I_D]$.\\
Pour tout poids $k$, définissons $a_k$ et $b_k$ les entiers tels que:
$$k=12 a_k+2b_k,\text{ avec }b_k\in\{0,1,2,3,4,5\}.$$ 
Notons le symbole de Kronecker d'une assertion $A$ par:
$$\delta_A=\begin{cases} 1 &\text{ si $A$ est vraie}\\0 &\text{ sinon.}\end{cases}$$

\begin{coro}
La dimension de $V_{k_1,k_2}^{\Q}[I_D]$ est donnée par:
\begin{equation}
\dim_{\Q}(V_{k_1,k_2}^{\Q}[I_D])=6 a_{k_1+k_2-2}^2 + 1 + (2a_{k_1+k_2-2}+1)b_{k_1+k_2-2}\delta_{b_{k_1+k_2-2}\neq 0}.
\end{equation}
\end{coro}

\begin{proof}
Si l'on note $d_k=\dim_{\Q} (W_k^{\Q})$, alors la proposition donne:
$$\dim_{\Q}(V_{k_1,k_2}^{\Q}[I_D])=\sum_{4\leq k\leq k_1+k_2-2} d_k.$$
D'autre part, un résultat classique donne: $d_k=2 a_k+1-2\delta_{b_k=2}$ (voir chapitre $1$ ou Serre, cours d'arithmétiques \cite{Se70}). Un calcul immédiat donne la formule du corollaire.
\end{proof}

\subsection{Description calculatoire de $V_{k_1,k_2}^{\Z}[I_D]$}

Pour implémenter sur un ordinateur l'appartenance d'un polynôme à $V_{k_1,k_2}^{\Z}[I_D]$, on dispose du résultat suivant sur les coefficients:
\begin{prop}
Soit $P=\sum A_{j_1,j_2} X_1^{j_1}X_2^{j_2}\in V_{k_1,k_2}$. Alors on a $P\in V_{k_1,k_2}[I_D]$ 
si et seulement si pour tout $0\leq j_1\leq k_1-2$ et $0\leq j_2 \leq k_2-2$, on a:
\begin{align*}
A_{j_1,j_2}+(-1)^{j_1+j_2}A_{w_1-j_1,w_2-j_2}&=0\\
\sum_{(i_1-j_1)(i_2-j_2)\geq 0} C_{k_1-2}(i_1,j_1)C_{k_2-2}(i_2,j_2)A_{i_1,i_2}&=0
\end{align*}
où pour tout triplet d'entiers naturels $(w,i,j)\in\Z_{\geq 0}^3$, on a posé : $C_w(i,j)=\begin{cases}\binom{i}{j}&\text{ si }i>j\\\binom{w-i}{w-j}&\text{ sinon.}\end{cases}$.
\end{prop}

\begin{proof}
Calcul de $W_{k_1,k_2}^0$. Posons $w_i=k_i-2$.\\
Soit $P(X_1,X_2)=\sum_{i_1=0}^{w_1}\sum_{i_2=0}^{w_2} A_{i_1,i_2} X_1^{i_1}X_2^{i_2}\in V_{k_1,k_2}[I_D].$ 
Alors on dispose des deux annulations:
\begin{footnotesize}
\begin{align*}
0=P_{(1,1)+(S,S)}&=\sum_{i_1=0}^{w_1}\sum_{i_2=0}^{w_2} A_{i_1,i_2} \left(X_1^{i_1}X_2^{i_2}+(-X_1)^{w_1-i_1}(-X_2)^{w_2-i_2}\right)\\
&=\sum_{i_1=0}^{w_1}\sum_{i_2=0}^{w_2} \left(A_{i_1,i_2}+(-1)^{i_1+i_2}A_{w_i-i_1,w_2-i_2}\right) X_1^{i_1}X_2^{i_2}.
\end{align*}

{\normalsize Puis } 
\begin{align*}
0&=P_{(1,1)+(U,U)+(U^2,U^2)}(X_1,X_2)\\
&=\sum_{i_1,i_2} A_{i_1,i_2} \left(X_1^{i_1}X_2^{i_2}+(-X_1+1)^{w_1-i_1}(-X_2+1)^{w_2-i_2}+(-X_1+1)^{i_1}(-X_1)^{w_1-i_1}(-X_2+1)^{i_2}(-X_2)^{w_2-i_2}\right)\\
&=\sum_{i_1,i_2}\sum_{j_1,j_2} A_{i_1,i_2} \left(X_1^{i_1}X_2^{i_2}
+\binom{w_1-i_1}{w_1-j_1}(-X_1)^{w_1-j_1}\binom{w_2-i_2}{w_2-j_2}(-X_2)^{w_2-j_2}
+\binom{i_1}{j_1}(-X_1)^{w_1-j_1}\binom{i_2}{j_2}(-X_2)^{w_2-j_2}\right)\\
&=\sum_{(j_1-i_1)(j_2-i_2)\geq 0}C_{w_1}(i_1,j_1)C_{w_2}(i_2,j_2)A_{i_1,i_2}(-X_1)^{w_1-j_1}(-X_2)^{w_2-j_2}.
\end{align*}
\end{footnotesize}
\end{proof}

On a vu que l'on pouvait décomposer:
$$V_{k_1,k_2}^{\Q}[I_D]\cong\bigoplus_{i=0}^{\min(k_1,k_2)} W_{k_1+k_2-2i}^{\Q}.$$
Cette décomposition a un intérêt calculatoire. Au cours de la démonstration, on a obtenu une méthode pour construire une famille de générateurs de ces différents espaces.

\begin{prop}
On dispose d'une formule explicite ne dépendant que des coefficients d'un antécédent par $\psi: V_{k_1,k_2}^{\Q}[I_D]\to W_{k_1+k_2-2}^{\Q}$ d'un élément:
$P(Z)=\sum_{m=0}^{k-2}\binom{k-2}{m}a_{m}(-Z)^{k-2-m}$ où $k=k_1+k_2-2$ donnée par:
\begin{equation}
Q(X_1,X_2)=\sum_{m_1=0}^{k_1-2}\sum_{m_2=0}^{k_2-2} \binom{k_1-2}{m_1}\binom{k_2-2}{m_2}a_{m_1+m_2}(-X_1)^{k_1-2-m_1}(-X_2)^{k_2-2-m_2}.
\end{equation}
\end{prop}

\begin{proof}
Pour obtenir une démonstration formelle, il suffit de calculer $Q(Z,Z)$ et de retrouver $P(Z)$ puis de vérifier les conditions sur les coefficients donnée par la proposition précédente. Pourtant, c'est bien la construction de l'antécédent de $P_f(Z)=\int_{i\infty}^0f(z)(X-z)^{k-2}dz\in W_{k}$ pour $f\in S_k$, donnant la surjectivité de $\psi$, qui nous inspire ce résultat:
\begin{align*}
Q_f(X_1,X_2)&=\int_{i\infty}^0f(z)(X_1-z)^{k_1-2}(X_2-z)^{k_2-2}dz.\\
&=\sum_{m_1=0}^{k_1-2}\sum_{m_2=0}^{k_2-2} \binom{k_1-2}{m_1}\binom{k_2-2}{m_2}\int_{i\infty}^0 f(z)z^{m_1+m_2}(-X_1)^{k_1-2-m_1}(-X_2)^{k_2-2-m_2}.
\end{align*}
Cet élément est bien dans $V_{k_1,k_2}[I_D]$. En développant de la même manière $P_f(Z)$ on peut identifier les coefficients et on obtient un antécédent dans $V_{k_1,k_2}^{\Q}[I_D]$.
\end{proof}

Ceci permet d'obtenir une méthode construction de générateurs de $V_{k_1,k_2}^{\Z}[I_D]$. On va l'appliquer sur $k=12$ où on a:
$$W_{12}=<1-X^{10},X^2-3X^4+3X^6-X^8,4X-25X^3+42X^5-25X^7+4X^9>_{\Q}.$$
Et ainsi on obtient par exemple:
$$V_{4,10}^{\Z}[I_D]=<1-X_1^2X_2^8,(X_1-X_2)(1-X_1X_2^7),(X_1-X_2)(1-X_2^6),P_{4,10}^+(X_1,X_2),P_{4,10}^-(X_1,X_2)>_{\Q}$$
où:
\begin{align*}
P_{4,10}^+(X_1,X_2)&=
\left(\frac{28}{45} X_2^2-\frac{70}{70}X_2^4+\frac{28}{70}X_2^6-\frac{1}{45}X_2^8\right)
+2X_1\left(\frac{8}{45}X_2-\frac{56}{70}X_2^3+\frac{56}{70}X_2^5-\frac{8}{45}X_2^7\right)\\
&+X_1^2\left(\frac{1}{45}-\frac{28}{70}X_2^2+\frac{70}{70}X_2^4-\frac{28}{45}X_2^6\right),\\
P_{4,10}^-(X_1,X_2)&=\left(\frac{16}{5} X_2-\frac{35}{3}X_2^3+\frac{28}{3}X_2^5-\frac{5}{3}X_2^7\right)
+2X_1\left(\frac{2}{5}-\frac{35}{6}X_2^2+\frac{35}{3}X_2^4-\frac{35}{6}X_2^6+\frac{2}{5}X_2^8\right)\\
&+X_1^2\left(-\frac{5}{3}X_2+\frac{28}{3}X_2^3-\frac{35}{3}X_2^5+\frac{16}{5}X_2^7\right).
\end{align*}

\chapter{Relations vérifiées par le polynôme des multipériodes}
\minitoc
\section*{Introduction}

Soit $n\geq 1$ un entier. Soient $k_1,...,k_n\geq 4$ des entiers pairs. Notons à nouveau $\Gamma=P\G$, $S_{k_j}^{\Q}$ l'espace des formes modulaires holomorphes paraboliques pour $\Gamma$ de poids $k_j$dont les coefficients de Fourier sont rationnels $(1\leq j\leq n)$. Pour tout corps de $K$ posons $S_{k_j}^{K}=S_{k_j}^{\Q}\otimes_{\Q} K$ et $S_{k_j}=S_{k_j}^{\C}$.\par
Soit $(f_j)_{1\leq j \leq n}\in\prod_{j=1}^n S_{k_j}$. Notons $k$ le multi-entier $(k_1,...,k_n)$. C'est le \textit{poids} de la famille $(f_j)_{1\leq j \leq n}$.\par
Le \textit{polynôme des multipériodes de longueur} $n$ de la famille $(f_j)_{1\leq j \leq n}$ est défini par:
\begin{align}
P_{f_1,...,f_n}(X_1,...,X_n)&=\int_{0<t_1<...<t_n}f_1(it_1)(X_1-it_1)^{k_1-2}...f_n(it_n)(X_n-it_n)^{k_n-2}\d t_1...\d t_n\\
&=\sum_{\substack{m_j=1\\ 1\leq j\leq n}}^{k_j-1} \prod_{j=1}^n \binom{k_j-2}{m_j-1} \Lambda(f_1,...,f_n;m_1,...,m_n) \frac{X_1^{k_1-m_1-1}}{i^{m_1}}...\frac{X_n^{k_n-m_n-1}}{i^{m_n}}.
\end{align}\par
C'est un polynôme de l'espace $V_k=V_k^{\Z}\otimes\C$ où $V_k^{\Z}=\Z_{k-2}[X_1,...,X_n]$ est le $\Z$-module engendré librement par les $X_1^{i_1}...X_n^{i_n}$ avec $0\leq i_j\leq k_j-2$ pour tout $1\leq j\leq n$. On dispose notamment de l'identification $V_k^{\Z}\cong\otimes_{j=1}^nV^{\Z}_{k_j}$ où $V_{k_j}^{\Z}=\Z_{k_j-2}[X_j]$.\par
Les \textit{multipériodes de longueur} $n$ de la famille sont les valeurs aux multi-entiers critiques, les $(m_1,...,m_n)$ tels que $1\leq m_j \leq k_j-1$ pour tout $1\leq j\leq n$, du prolongement analytique à $\C^n$ de la fonction définie pour $\Re(s_j)>k_j$ par:
\begin{equation}
\Lambda(f_1,...,f_n;s_1,...,s_n)=\int_{0<t_1<...<t_n} f_1(it_1)t_1^{s_1-1}\d t_1...f_n(it_n)t_n^{s_n-1}\d t_n.
\end{equation}\par
Nous cherchons à mettre en équation l'espace des polynômes des multipériodes. Dans le chapitre $2$, on a exhibé des relations vérifiées par ces polynômes. Pour mémoire, on a les relations de Manin généralisées:
\begin{align}
\sum_{\substack{a,b\geq 0\\ a+b=n}}P_{f_1,...,f_a}|_{(S,...,S)}&\otimes P_{f_{a+1},...,f_{a+b}}=0,\\
\sum_{\substack{a,b,c\geq 0\\ a+b+c=n}}P_{f_1,...,f_a}|_{(U^2,...,U^2)}&\otimes P_{f_{a+1},...,f_{a+b}}|_{(U,...,U)}\otimes P_{f_{a+b+1},...,f_{a+b+c}}=0,
\end{align}
et les relations de mélange, pour tout couple $(a,b)$ vérifiant $a+b=n$:
\begin{equation}
\sum_{\sigma\in\mathfrak{S}_{a,b}}P_{f_{\sigma(1)},...,f_{\sigma(n)}}=P_{f_1,...,f_a}\otimes P_{f_{a+1},...,f_{a+b}}.\label{relmelab}
\end{equation}\par
Toutefois, ces relations nous donnent des équations récursives $n$ reliant ces polynômes de différentes longueurs. 
On s'intéresse ici aux relations satisfaites purement par les polynômes des multipériodes de longueur $n$ donnée sans faire intervenir les multipériodes de longueurs inférieures.
Ces relations que nous mettons en évidence sont linéaires et définissent un sous-groupe $W_k^{\Z}\subset V_k^{\Z}$. Nous démontrons ensuite qu'il est le plus petit possible, dans un sens que nous préciserons, tel que son extension au corps des complexes contiennent l'ensemble des polynômes des multipériodes.

\section{Définitions d'analogues multidimensionnelles}

\subsection{La famille des permutés du polynôme des multipériodes}

%

Introduisons la \textit{famille des permutés} $(P_{f_1,...,f_n}^{\sigma})_{\sigma\in\S_n}$, où pour tout $\sigma\in\S_n$ on définit:
\begin{equation}
P^{\sigma}_{f_1,...,f_n}(X_1,...,X_n)=P_{f_{\sigma(1)},...,f_{\sigma(n)}}(X_{\sigma(1)},...,X_{\sigma(n)})\in V_k.
\end{equation}
Définissons alors l'application $\C$-linéaire:
\begin{equation}
R_k:\otimes_{j=1}^n S_{k_j} \to V_k^{\S_n},\quad f_1\otimes...\otimes f_n\mapsto (P^{\sigma}_{f_1,...,f_n})_{\sigma\in\S_n}.
\end{equation}

\begin{prop}
L'application $R_k$ est injective.
\end{prop}

\begin{proof}
Démontrons ceci par récurrence sur $n\geq 1$. L'application $R_{k_1}$ est injective d'après le Théorème d'Eichler-Shimura-Manin. Supposons les applications $R_k$ injectives pour la longueur $n$. Montrons l'injectivité de $R_k$, pour $k=(k_1,...,k_{n+1})$ de longueur $n+1$. Pour cela, écrivons la relation de mélange (\ref{relmelab}) pour $a=n$ et $b=1$:
$$\sum_{\sigma\in\S_{n,1}} P^{\sigma}_{f_1,...,f_{n+1}}(X_1,...,X_{n+1})=P_{f_1,...,f_n}(X_1,...,X_n)P_{f_{n+1}}(X_{n+1}),$$
pour tout $f_j\in S_{k_j} ,1\leq j\leq (n+1)$. Ainsi pour $\rho\in\S_n$, on obtient:
$$\sum_{\sigma\in\S_{n,1}} P^{\sigma}_{f_{\rho(1)},...,f_{\rho(n)},f_{n+1}}(X_{\rho(1)},...,X_{\rho(n)},X_{n+1})
=P^{\rho}_{f_1,...,f_n}(X_1,...,X_n)P_{f_{n+1}}(X_{n+1}).$$
Soit $F\in \Ker(R_{k_1,...,k_{n+1}})$. Posons $F=\sum_{\alpha\in A} F_{\alpha}\otimes f_{\alpha}$ où $F_{\alpha}\in\otimes_{j=1}^n S_{k_j} $ et $f_{\alpha}\in S_{k_{n+1}} $ sont des familles indexées par un ensemble fini $A$. On peut supposer que la famille $(f_{\alpha})_{\alpha\in A}$ est libre. Ainsi le Théorème d'Eichler-Shimura donne l'indépendance des polynômes $(P_{f_{\alpha}})_{\alpha\in A}$ de $V_{k_{n+1}}$.\par
Soit $\rho\in\S_n$. Notons $\tilde{\rho}$ l'unique élément de $\S_{n,1}$ tel que $\tilde{\rho}|_{[1,n]}=\rho$ et $\tilde{\rho}(n+1)=n+1$. Alors:
$$0=\sum_{\alpha\in A} \sum_{\sigma\in\S_{n,1}}R_{k_1,...,k_{n+1}}(F_{\alpha}\otimes f_{\alpha})(\sigma\circ\tilde{\rho})
=\sum_{\alpha\in A}R_{k_1,...,k_n}(F_{\alpha})(\rho)R_{k_{n+1}}(f_{\alpha}).$$
Puisque la famille des $R_{k_{n+1}}(f_{\alpha})=P_{f_{\alpha}}(X_{n+1})$ est libre, alors pour tout $\rho\in \S_n$ et tout $\alpha\in A,$ on a $R_{k_1,...,k_n}(F_{\alpha})(\rho)=0$. C'est à dire:
$$R_{k_1,...,k_n}(F_{\alpha})=0,\text{ pour tout }\alpha\in A,$$
et ainsi $F_{\alpha}=0$ par hypothèse de récurrence. On en déduit bien $F=0$.
\end{proof}

Notons $MP(k)$ l'image de $R_k$. Comme précédemment on peut exploiter la $\Z$-structure de $\otimes_{j=1}^n S_{k_j}$ et considérons pour tout corps de nombre $K$:
\begin{equation}
MP^{K}(k)=MP^{\Q}(k)\otimes K=\{R_k(f_1\otimes...\otimes f_n)\text{ pour }f_1\otimes...\otimes f_n\in\otimes_{j=1}^n S_{k_j}^{K}\}.
\end{equation}

\subsubsection{Action de $P\G^n\rtimes\S_n$ sur $(V_k^{\Z})^{\S_n}$}

Notons $G_n=\Gamma^n\rtimes\S_n$ le produit semi-direct défini par:
\begin{equation}
[\gamma,\rho][\gamma',\rho']=[\gamma^{\rho'}\gamma',\rho\rho'],
\end{equation}
où l'action de $\sigma\in\S_n$ sur $(\gamma_1,...,\gamma_n)\in\Gamma^n$ est donnée par la formule:
\begin{equation}
(\gamma_1,...,\gamma_n)^{\sigma}=(\gamma_{\sigma(1)},...,\gamma_{\sigma(n)}).
\end{equation}
Elle vérifie, pour tout $\gamma,\gamma'\in\Gamma^n$ et $\rho,\rho'\in\S_n$, les propriétés:
\begin{equation}
\left(\gamma^{\rho}\right)^{\rho'}=\gamma^{\rho\rho'}\text{ et }\gamma^{\rho}\gamma'^{\rho}=(\gamma\gamma')^{\rho}.
\end{equation}

\begin{prop}
On a une action à droite de $[\gamma,\rho]\in G_n$ sur $(P^{\sigma})_{\sigma\in\S_n}\in (V_k^{\Z})^{\S_n}$ définie par:
\begin{equation}
\left(\sigma\mapsto P^{\sigma}(X_1,...,X_n)\right)|_{[\gamma,\rho]}=\left(\sigma\mapsto P^{\rho\sigma}(X_1|_{\gamma_{\sigma(1)}},...,X_n|_{\gamma_{\sigma(n)}})\right).\label{defactgn}
\end{equation}
\end{prop}

\begin{proof}
Deux actions à droite sont naturelles sur $(P^{\sigma})_{\sigma\in\S_n}\in (V_k^{\Z})^{\S_n}$:\\
Celle de $\rho\in\S_n$ donnée par:
\begin{equation*}
\left(\sigma\mapsto P^{\sigma}(X_1,...,X_n)\right)|_{\rho}=\left(\sigma\mapsto P^{\rho\sigma}(X_1,...,X_n)\right).
\end{equation*}
Et l'action de $(\gamma_1,...,\gamma_n)\in\Gamma^n$ définie par:
\begin{equation*}
\left(\sigma\mapsto P^{\sigma}(X_1,...,X_n)\right)|_{(\gamma_1,...,\gamma_n)}=\left(\sigma\mapsto P^{\sigma}(X_1|_{\gamma_{\sigma(1)}},...,X_n|_{\gamma_{\sigma(n)}})\right).
\end{equation*}
Ces actions se combinent bien et sont conformes au produit de $G_n$. En effet, on vérifie:
\begin{multline*}
\left(\sigma\mapsto P^{\sigma}(X)\right)|_{[\gamma,\rho][\gamma',\rho']}=\left(\sigma\mapsto P^{\rho\sigma}(X|_{\gamma^{\sigma}})\right)|_{[\gamma',\rho']}\\
=\left(\sigma\mapsto P^{\rho\rho'\sigma}(X|_{\gamma^{\rho'\sigma}\gamma'^{\sigma}})\right)=\left(\sigma\mapsto P^{\sigma}(X)\right)|_{[\gamma^{\rho'}\gamma',\rho\rho']}.
\end{multline*}
\end{proof}

Nous nous proposons de caractériser $MP(k)$ à l'aide de la structure de $G_n$-module de $(V_k^{\Z})^{\S_n}$ ainsi construite.\par
Lorsque $n=2$, on a construit un idéal à gauche $\mathcal{I}_2\subset \Z[\Gamma^2]$ et ainsi un sous-$\Q$-espace vectoriel $W_{k_1,k_2}^{\Q}=V_{k_1,k_2}^{\Q}[\mathcal{I}_2]$ minimal tel que :
$$\Per_{k_1,k_2}\subset W_{k_1,k_2}^{\C}\text{ et }E_{k_1,k_2}^{\Q}\subset W_{k_1,k_2}^{\Q}.$$\par
Voyons comment reformuler ce résultat dans le contexte d'une famille de permutés. Soit $\sigma=(2,1)\in \S_2$ la transposition. Alors la relation $P_{f_1,f_2}(X_1,X_2)|_{(S,S)}=P_{f_2,f_1}(X_2,X_1)=P_{f_1,f_2}^{\sigma}(X_1,X_2)$ permet d'établir une bijection entre les polynômes des bipériodes et la famille des permutés: 
$$\Phi:V_{k_1,k_2}^{\Z}\to (V_{k_1,k_2}^{\Z})^{\S_2},P\mapsto (P,P|_{(S,S)}).$$
Cette application est injective et envoie $\Per_{k_1,k_2}$ sur $MP(k_1,k_2)$. Définissons alors un idéal à gauche de $\Z[G_2]$ par:
$$\mathcal{A}_2=\big([\mathcal{I}_2,id],[(S,S),id]-[(1,1),\sigma]\big).$$
Alors $\left(V_{k_1,k_2}^{\Q}\right)^{\S_2}[\mathcal{A}_2]$ est le plus petit sous-$\Q$-espace vectoriel de $\left(V_{k_1,k_2}^{\Q}\right)^{\S_2}$ vérifiant:
\begin{equation*}
MP(k_1,k_2)\subset \left(V_{k_1,k_2}^{\C}\right)^{\S_2}[\mathcal{A}_2]
\text{ et }\Phi(E_{k_1,k_2}^{\Q})\subset \left(V_{k_1,k_2}^{\Q}\right)^{\S_2}[\mathcal{A}_2].
\end{equation*}

Adaptons ce point de vue lorsque $n\geq 3$. Pour cela, il nous faut déterminer l'idéal à droite de $\Z[G_n]$ annulateur de ces polynômes. Notons ainsi $\mathcal{A}(k)\subset\Z[G_n]$ l'idéal dépendant à priori de $k=(k_1,...,k_n)$ la famille des poids et défini par:
\begin{equation}\label{Akpoly}
\mathcal{A}(k)=\big\{a\in\Z[G_n]\text{ tel que }R_k(f_1\otimes...\otimes f_n)|_a=0,\text{ pour tout }(f_j)_{1\leq j\leq n}\in \prod_{j=1}^n S_{k_j}^{\Q}\}.
\end{equation}\par
Nous allons construire explicitement un idéal $\mathcal{A}_n\subset\Z[G_n]$ tel que pour tout multi-entier $k$ de longueur $n$, 
$\left(V_k^{\Q}\right)^{\S_n}[\mathcal{A}_n]\supset\left(V_k^{\Q}\right)^{\S_n}[\mathcal{A}(k)]$. Puis nous préciserons dans quelle mesure ce $\Q$-espace vectoriel est minimal voir le Corollaire \ref{corofin}. Ainsi relations données par l'idéal $\mathcal{A}_n$ sont en ce sens optimales.

%

\subsection{Homologie singulière de $\H^n$ relative aux pointes}

Reprenons les définitions d'homologie introduites dans la section $3.1.1$.\par 

Pour toute partie $X\subset \H^n$ et $X_0=\pte^n\cap X$ et tout entier $m$, le groupe des \textit{$m$-chaîne de $X$ relative aux pointes} $M_m^{pte}(X,\Z)$ est le $\Z$-module libre engendré par les classes d'homotopie des applications:
$$C:\Delta_m\to X\text{ continues telles que }C(\Delta_m^0)\subset X_0,$$
L'application de bord $\delta_m: M_m^{pte}(X,\Z)\to M_{m-1}^{pte}(X,\Z)$ permet de définir les \textit{groupes d'homologie relative}:
$$H_m^{pte}(X,\Z)=\Ker(\delta_m)/\Im(\delta_{m+1}).$$

\begin{prop}
On peut calculer les groupes d'homologie relative:
$$H_m^{pte}(\H^n,\Z)=\begin{cases}\Z[\pte^n]^0&\text{ si }m=0\\ \Z\text{ si }m=2n\\0&\text{ sinon.}\end{cases}$$
\end{prop}

\begin{proof}
En effet, on dispose des groupes d'homologie singulière suivant :
$$H_m(\H^n,\Z)=\begin{cases}\Z&\text{ si }m=0\text{ ou }2n\\0&\text{ sinon}\end{cases}\text{ et }
H_m(\pte^n,\Z)=\begin{cases}\Z\left[\pte^n\right]&\text{ si }m=0\\0&\text{ sinon}\end{cases},$$\par
Et d'autre part, on peut mettre en évidence la suite exacte longue:
$$\ldots\to H_m(\pte^n,\Z)\to H_m^{pte}(\H^n,\Z)\to H_m(\H^n,\Z)\to\ldots$$ 
Ceci donnant les groupes $H_m^{pte}(\H^n,\Z)$ attendus.
\end{proof}

Le polynôme des multipériodes est lié au $n$-cycle $\tau_n$ suivant défini comme la classe dans $M_n^{pte}(\H^n,\Z)$ de l'application:
\begin{equation}
\Delta_n\to\H^n,(u_0,...,u_n)\mapsto \left(-i\log\left(\sum_{j=0}^{l-1}u_j\right)\right)_{l=1...n}.
\end{equation}\par

Notons $\Omega^n_{par}(\H^n,\C)$ le $\C$-espace vectoriel des $n$-formes différentielles harmoniques sur $\H^n$ et à valeurs dans $\C$ et nulles sur le bord $\partial\H^n=\bigcup_{i=0}^n \H^i\times \pte \times \H^{n-i}$. En particulier, cette dernière condition permet de considérer l'intégrale le long de $\tau_n$. Notons $\Omega^n_{par}(\H^n,V_{k})^{\Gamma^n}=\Omega^n_{par}(\H^n,\C)\otimes_{\C[\Gamma^n]} V_k$ l'ensemble des formes $\omega$ invariante par $\Gamma^n$ au sens où:
$$\omega(\gamma_1z_1,...,\gamma_nz_n;X_1|_{\gamma_1},...,X_n|_{\gamma_n})=\omega(z_1,...,z_n;X_1,...,X_n)\text{ pour tout }(\gamma_1,...,\gamma_n)\in\Gamma^n.$$\par
En particulier, ces formes disposent de coefficients de Fourier en chacune des variables car on dispose de l'homomorphisme injectif $\Z^n\to\Gamma^n$ induit diagonalement par l'injection $\Z\to\Gamma, \alpha\mapsto\mat{1}{\alpha}{0}{1}$.
Définissons ainsi l'espace des formes différentielles holomorphes en chaque variables, nulles en l'infini et invariantes par $\Gamma^n$ et à coefficients de Fourier réels:
\begin{equation}
\Omega_{k}^{\holo}=\otimes_{j=1}^n\Omega_{k_j}^+\subset\Omega^n_{par}(\H^n,V_{k})^{\Gamma^n}.
\end{equation}
C'est le $\R$-espace vectoriel isomorphe à $\otimes_{j=1}^nS_{k_j}^{\R}$ par l'application:
\begin{equation}
\otimes_{j=1}^nS_{k_j}^{\R} \to \Omega^n_{par}(\H^n,V_{k}),\quad\otimes_{j=1}^n f_j\mapsto\wedge_{j=1}^n \omega_{f_j}(z_j,X_j)=\wedge_{j=1}^n f_j(z_j)(X_j-z_j)^{k_j-2}dz_j.
\end{equation}\par
Son conjugué complexe est alors l'ensemble des formes antiholomorphes donné par:
\begin{equation}
\overline{\Omega_{k}^{\holo}}=\otimes_{j=1}^n\Omega_{k_j}^-\subset\Omega^n_{par}(\H^n,V_{k})^{\Gamma^n}.
\end{equation}
Posons ainsi le sous-espace de $\Omega_{par}^n(\H^n,V_k)^{\Gamma^n}$ stable par $\Gamma^n$ et par conjugaison complexe:
\begin{equation}
\Omega_k^{\R}=\Omega_k^+\oplus\Omega_k^-.
\end{equation}\par

D'autre part, introduisons l'accouplement :
\begin{equation}
\Omega^n_{par}(\H^n,\C)\times M_n^{pte}(\H^n,\Z)\to\C,\quad (\omega,C)\mapsto\langle\omega, C\rangle=\int_C \omega.
\end{equation}
Ce formalisme nous permet d'obtenir les polynômes des multipériodes en intégrant une forme différentielle adéquate le long de $\tau_n$:
\begin{equation}
P_{f_1,...,f_n}(X_1,...,X_n)=\left\langle \omega_{f_1}(z_1,X_1)\wedge...\wedge\omega_{f_n}(z_n,X_n),\tau_n\right\rangle,
\end{equation}

On peut ainsi récrire la définition de $\mathcal{A}(k)$ en:
\begin{equation}\label{Akholo}
\mathcal{A}(k)=\{a\in\Z[G_n]\text{ tel que pour tout }\omega\in\Omega_k^{\R},\langle\omega|_{a},\tau_n\rangle=0\}.
\end{equation}\par

Nous allons désormais transporter l'action de $G_n$ sur $M_n^{pte}(\H^n,\Z)^{\S_n}$. Et ainsi démontrer que l'idéal $\mathcal{A}(k)$, donné par les relations vérifiées par le polynôme des multipériodes, reflète les propriétés topologiques de $\tau_n$.

Par exemple, lorsque $n=1$, la $1$-chaîne $\tau_1$ correspond au symbole modulaire $\{i\infty,0\}$ dont le bord $(0)-(i\infty)$ admet $\mathcal{I}_1$ comme annulateur.

\section{Dualité entre formes modulaires et homologie relative}

Rappelons que $G_n=\Gamma^n\rtimes\S_n$.

\subsection{Actions de $\Z[G_n]$ sur $\Omega^m(\H^n,\C)$ et $ M_m^{pte}(\H^n,\Z)$}

Soit $m\geq 1$ un entier.
Rappelons qu'on a une action à droite de $\gamma\in\Gamma^n$ sur les formes différentielles $\omega\in\Omega^m(\H^n,\C)$ donnée par:
 \begin{equation}
\omega|_{(\gamma_1,...,\gamma_n)}\big(z_1,...,z_n\big)=\omega\big(\gamma_1.z_1,...,\gamma_n.z_n\big),
\end{equation}
et une action à gauche sur les $m$-chaînes de $\H^n$ relatives à $\pte^n$ via: 
\begin{equation}
((\gamma_1,...,\gamma_n).C)(t_0,...,t_m)=(\gamma_1.[C_1(t_0,...,t_m)],...,\gamma_n.[C_n(t_0,...,t_m)]).
\end{equation}
Un élément du groupe $\rho\in\S_n$ sur $\H^n$ par permutation des coordonnées. Ceci donne une action à droite sur les formes différentielles $\omega\in\Omega^m(\H^n,\C)$ par:
 \begin{equation}
\omega|_{\rho}\big(z_1,...,z_n\big)=\omega\big(z_{\rho(1)},...,z_{\rho(n)}\big),
\end{equation}
et une action à gauche sur les $m$-chaînes de $\H^n$ à bord dans $\pte^n$ via: 
\begin{equation}
{\rho}.C(t_0,...,t_m)=([C_{\rho^{-1}(1)}(t_0,...,t_m)],...,[C_{\rho^{-1}(n)}(t_0,...,t_m)]).
\end{equation}
On notera l'action à droite opposée par: $(\sigma,C)\mapsto C^{\sigma}=\sigma^{-1}.C$.

\begin{prop}
On a une action à droite de $[\gamma,\rho]\in G_n$ sur $\omega\in\Omega^m(\H^n,\C)$ définie par:
\begin{equation}
\omega|_{[\gamma,\rho]}=(\omega|_{\rho})|_{\gamma},
\end{equation}
et une action à gauche de $[\gamma,\rho]\in G_n$ sur $C\in  M_m^{pte}(\H^n,\Z)$ donnée par:
\begin{equation}
[\gamma,\rho].C={\rho}(\gamma C).
\end{equation}
Par $\Z$-linéarité, on obtient une action de $\Z[G_n]$ à droite sur $\Omega^m(\H^n,\C)$ et une action à gauche sur $M_m^{pte}(\H^n,\Z)$.
\end{prop}

\begin{proof}
On vérifie par le calcul que l'action est bien conforme au produit du groupe. Soient $[\gamma,\sigma],[\gamma',\sigma']\in G_n$. D'une part, on a:
$$\left(\omega|_{[\gamma,\sigma]}\right)|_{[\gamma',\sigma']}(z_j)=\omega|_{[\gamma',\sigma']}(\gamma_{j}z_{\rho(j)})
=\omega(\gamma_{\rho'(j)}\gamma'_{j}z_{\rho(\rho'(j))})
=(\omega|_{\rho\rho'})|_{\gamma^{\rho'}\gamma'}(z_j)=\omega|_{[\gamma^{\rho'}\gamma',\rho\rho']}(z_j).$$
D'autre part, on a:
$$[\gamma,\rho].\left([\gamma',\rho'].(C_j)\right)=[\gamma,\rho].(\gamma'_{\rho'^{-1}(j)} C_{\rho'^{-1}(j)})=(\gamma_{\rho^{-1}(j)}\gamma'_{(\rho\rho')^{-1}(j)}C_{(\rho\rho')^{-1}(j)})=[\gamma^{\rho}\gamma',\rho\rho'](C_j).$$
\end{proof}

Rappelons qu'on a muni $(V_k^{\Z}){\S_n}$ d'une action à droite de $\Z[G_n]$ en (\ref{defactgn}).\\
Ces différentes actions peuvent être reliées par le résultat suivant:

\begin{prop}
Soient $f\in \bigotimes_{j=1}^n S_{k_j}$ et $\sigma\in\S_n$, on a: 
\begin{equation}\label{polytau}
R_k(f)(\sigma)=\left\langle \omega_f(z,X),\tau_n^{\sigma} \right\rangle.
\end{equation}
De plus, on a les propriétés d'invariance suivantes:\\
1) Pour tout $\gamma\in\Gamma^n$ et tout $\rho\in\S_n$, on a:
\begin{align}
\omega_f(\gamma.z,X|_{\gamma})&=\omega_f(z,X),\label{mod}\\
\text{ et }\omega_{\rho(f)}(\rho.z,\rho.X)&=\omega_f(z,X).
\end{align}
2) Pour tout $[\gamma,\rho]\in G_n$, $\omega\in\Omega^m_{par}(\H^n,\C)$ et $C\in  M_m^{pte}(\H^n,\Z)$, on a:
\begin{equation}
\langle \omega|_{[\gamma,\rho]},C\rangle=\langle \omega,[\gamma,\rho].C\rangle.\label{acc}
\end{equation}
\end{prop}

\begin{proof}
Toutes les formules à démontrer proviennent de calculs directs.
Pour démontrer la formule (\ref{mod}), on remarque qu'étant donné que $\omega_f(z,X)=\omega_{f_1}(z_1,X_1)\wedge...\wedge\omega_{f_n}(z_n,X_n)$, il suffit alors de la démontrer dans le cas d'une seule forme:
\begin{multline*}
\omega_f(\gamma.z,X|\gamma)=f(\gamma.z)(\gamma.z-\gamma.X)^{k-2}(cX+d)^{k-2}d(\gamma.z)\\
=f(z)(cz+d)^k\left(\frac{z-X}{(cz+d)(cX+d)}\right)^{k-2}(cX+d)^{k-2}\frac{dz}{(cz+d)^2}\\
=f(z)(z-X)^{k-2}dz=\omega_f(z,X).
\end{multline*}
La propriété d'invariance par $\S_n$ se vérifie par écriture des définitions des actions.\par
Pour démontrer la propriété (\ref{acc}), il suffit de procéder par changements de variables:
$$\left\langle \omega|_{[\gamma,\rho]},C\right\rangle=\int_{z\in C} \omega|_{\rho}(\gamma z)
=\int_{z'\in \gamma.C} \omega(\rho z')
=\int_{z''\in \rho.\gamma .C} \omega(z'')=\left\langle \omega,[\gamma,\rho].C\right\rangle.$$
Enfin ces propriétés permettent de déduire la formule (\ref{polytau}). En effet:
$$R_k(f)(\sigma)=\langle \omega_{\sigma(f)}(z,\sigma X),\tau_n\rangle=\langle \omega_{f}(\sigma^{-1}z,X),\tau_n\rangle
=\langle \omega_{f}(z,X),\sigma^{-1}.\tau_n\rangle.$$
\end{proof}

Ceci permet de déduire le corollaire suivant qui transporte l'action de $\Z[G_n]$ sur $(V_k^{\Z})^{\S_n}$ en une action sur $\left(M_n^{pte}(\H^n,\Z)\right)^{\S_n}$ indépendante des poids.

\begin{coro}\label{coro36}
Soit $[\gamma,\rho]\in G_n$. Pour tout $f\in\bigotimes_{j=1}^n S_{k_j} $, on a:
\begin{equation}
R_k(f)|_{[\gamma,\rho]}=\left(\sigma\mapsto \langle \omega_{f}(z,X), (\gamma^{-1}\tau_n^{\rho})^{\sigma}\rangle\right).
\end{equation}
Cette dualité s'étend par linéarité au groupe $\Z[G_n]$.
\end{coro}

\begin{proof}
En effet, pour tout $\gamma=(\gamma_1,...,\gamma_n)\in\Gamma^n$ et $\rho\in\S_n$, on a:
\begin{align*}
R_k(f)|_{[\gamma,\rho]}&=\left(\sigma\mapsto \langle \wedge_{j=1}^n\omega_{f_j}(z_j,X_j|_{\gamma_{\sigma(j)}}), \tau_n^{\rho\sigma}\rangle\right)&\text{par \ref{polytau},}\\
&=\left(\sigma\mapsto \langle \wedge_{j=1}^n\omega_{f_j}({\gamma^{-1}_{\sigma(j)}}z_j,X_j), (\tau_n)_{(\rho\sigma)(j)}\rangle\right)&\text{par \ref{mod},}\\
&=\left(\sigma\mapsto \langle \wedge_{j=1}^n\omega_{f_j}(z_j,X_j), (\gamma^{-1})_{\sigma(j)}(\tau_n)_{\sigma(\rho(j))}\rangle\right)&\text{par \ref{acc},}\\
&=\left(\sigma\mapsto \langle \omega_{f}(z,X), (\gamma^{-1}.\tau_n^{\rho})^{\sigma}\rangle\right).
\end{align*}
\end{proof}

Ce corollaire nous permet de considérer l'action à droite de $[\gamma,\rho]\in G_n$ sur $\left(  M_n^{pte}(\H^n,\Z)\right)^{\S_n}$ définie par:
\begin{equation}
\left(\sigma\mapsto C(\sigma)\right)|_{[\gamma,\rho]}=\left(\sigma\mapsto (\gamma^{-1})^{\sigma}C(\rho\sigma)\right).
\end{equation}
Définissons:
\begin{equation}
\mathcal{T}_n=\left(\sigma\mapsto\tau_n^{\sigma}\right)\in\left(  M_n^{pte}(\H^n,\Z)\right)^{\S_n}.
\end{equation}
On récrit l'action de $a\in\Z[G_n]$ sur $V_k^{\S_n}$ via une action sur $\left(M_n^{pte}(\H^n,\Z)\right)^{\S_n}$ par la formule:
\begin{equation}
R_k(f)|_{a}=\langle \omega_f,\mathcal{T}_n|_{a}\rangle.
\end{equation}
Ainsi le calcul de $\mathcal{A}(k)$ se réduit à l'étude de l'action de $\Z[G_n]$ sur l'élément $\mathcal{T}_n\in\left(M_n^{pte}(\H^n,\Z)\right)^{\S_n}$. Plus précisément, on a:
\begin{equation}\label{Aktopo}
\mathcal{A}(k)=\{a\in\Z[G_n]\text{ tel que }\langle\omega,\mathcal{T}_n|_a\rangle=0,\text{ pour tout }\omega\in\Omega_k^{\R}\}.
\end{equation}

\subsection{Quelques résultats sur l'homologie de $\H^n$}

\subsubsection{Calcul du bord de $\tau_n$}

Dans le cas $n=1$, l'idéal $\mathcal{I}_1$ est défini comme l'annulateur du bord $(0)-(i\infty)$ de la $1$-chaîne $\tau_1$.
Afin de déterminer $\mathcal{I}_2$, on a vu dans la Proposition \ref{dtau2} qu'on pouvait décomposer le bord de $\tau_2$ en fonction de $\tau_1$. Ce procédé se généralise comme suit.\par
Introduisons la famille d'applications $(\varphi_j^g)_{0\leq j\leq n,g\in\Gamma}:\H^{n-1}\to\H^n$ définie par:
$$\varphi_j^g(z_1,...,z_{n-1})=
\begin{cases}
(g.i\infty,z_1,...,z_{n-1})&\text{ si }j=0\\
(z_1,...,z_{n-1},g.0)&\text{ si }j=n\\
(z_1,...,z_j,g.z_j,z_{j+1}...,z_{n-1})&\text{ sinon}.
\end{cases}$$
Lorsque $g=Id$, on notera simplement $\varphi_j=\varphi_j^{Id}$.

\begin{prop}
On peut déterminer le bord de $\tau_n$ en fonction de $\tau_{n-1}$:
\begin{equation}\label{bordtaun}
\delta_{n}\tau_n=\sum_{j=0}^n (-1)^j\varphi_j(\tau_{n-1}).
\end{equation}
De plus, pour tout $g\in\Gamma$ et tous entiers $m,j\geq 0$ avec $j\leq n$, on a :
\begin{equation}\label{mchtrans}
 M_m^{pte}(\varphi_j^g(\H^{n-1}),\Z)=\varphi_j^g\left( M_m^{pte}(\H^{n-1},\Z)\right).
\end{equation}
\end{prop}

\begin{proof}
Il suffit de calculer pour tout $0\leq j\leq n, \tau_n\circ\delta_{n}^{j}:\Delta_{n-1}\to\H^n$.\\
Ainsi la $l$-ième coordonnée de $\tau_n\circ\delta_{n}^{j}(u_0,...,u_{m-1})$ est donnée par :
$$[\tau_n(...,u_{j-1},0,u_{j}...,u_{m-1})]_l=-i\log\left(\sum_{\substack{\alpha=0\\ \alpha<l}}^{j-1} u_{\alpha}+\sum_{\substack{\alpha=j+1\\ \alpha<l}}^{m} u_{\alpha-1}\right)=[\varphi_j(\tau_{n-1}(u_0,...,u_{m-1}))]_l.$$
Pour démontrer l'égalité des ensembles, il suffit de remarquer que:
$$\pte^n\cap\varphi_j^g\left(\H^{n-1}\right)=\varphi_j^g\left(\pte^{n-1}\right).$$
En effet, on a : $h=(p_1,...,p_n)=(z_1,...,z_j,g.z_j,...,z_{n-1})$\\
$\Leftrightarrow \quad h=(p_1,...,p_j,g.p_j,p_{j+2},...,p_n)=\varphi_j^g(p_1,...,\hat{p_{j+1}},...,p_n).$
\end{proof}

\subsubsection{Isométrie de $\H^n$}

Le demi-plan de Poincaré $\H$ est muni de la mesure donnée par $\Im(z)^{-2}dz\wedge d\bar{z}$. On étend diagonalement cette mesure à $\H^n$ et on s'intéresse aux transformations de $\H^n$ fixant cette mesure. Pour cela, on introduit les applications suivantes :\\
i) Pour toute permutation $\sigma\in \S_n$ de $n$ éléments, on définit l'isométrie :
\begin{equation}
\varphi_{\sigma}:\H^n\to\H^n,(z_1,...,z_n)\mapsto (z_{\sigma(1)},...,z_{\sigma(n)}).
\end{equation}
ii) On introduit l'involution de $\H: c_{-1}(z)=-\bar{z}$ 
et pour tout $\epsilon=(\epsilon_1,...,\epsilon_n)\in\{\pm 1\}^n$ :
\begin{equation}
\varphi_{\epsilon}:\H^n\to \H^n,(z_1,...,z_n)\mapsto (c_{\epsilon_1}(z_1),...,c_{\epsilon_n}(z_n)),\text{ où }c_1=id.
\end{equation}
iii) Pour tout $\gamma=(\gamma_1,...,\gamma_n)\in\Gamma^n$, posons:
\begin{equation}
\varphi_{\gamma}:\H^n\to \H^n,(z_1,...,z_n)\mapsto (\gamma_1.z_1,...,\gamma_n.z_n).
\end{equation}

\begin{proof}
On vérifie simplement que ce sont des isométries en remarquant que les applications construites:
\begin{align*}
\S_n\to Iso(\H^n),\quad &\sigma\mapsto\varphi_{\sigma},\\
 \{\pm 1\}^n\to Iso(\H^n),\quad &\epsilon\mapsto\varphi_{\epsilon},\\
\text{ et }\Gamma^n\to Iso(\H^n),\quad &\gamma\mapsto\varphi_{\gamma}
\end{align*} sont des morphismes de groupes.\par
Il suffit alors de vérifier que la permutation de deux variables est une isométrie de $\H^n$ et que $z\mapsto -\bar{z}$ est une isométrie de $\H$. En effet ce sont toutes les deux des involutions de leurs espaces respectifs.\par
Enfin pour toute matrice $g\in\Gamma$, $z\mapsto g.z$ est une isométrie de $\H$ d'après le calcul bien connu:
\begin{align*}
\Im(g.z)^{-2}d(g.z)\wedge d(g.\bar{z})&=\frac{|cz+d|^4}{\Im(z)^2}\frac{dz}{(cz+d)^2}\wedge \frac{d\bar{z}}{(c\bar{z}+d)^2}\\
&=\Im(z)^{-2}dz\wedge d\bar{z}.
\end{align*}
\end{proof}

\subsubsection{Action de la conjugaison complexe}

\begin{prop}\label{conju}
Soit $f\in S_k^{\R} $ alors on dispose de la formule suivante:
\begin{equation}
\overline{\omega_f(z,X)}=\omega_f(-\bar{z},-X).
\end{equation}
Ceci s'étend naturellement par produit tensoriel à une famille de formes modulaires paraboliques de niveau $1$.
\end{prop}

\begin{proof}
En posant le calcul on remarque que la clé de celui-ci provient de l'égalité $\overline{f(z)}=f(-\bar{z})$ qui se vérifie en regardant par exemple la $q$-série:
$$\overline{f(z)}=\sum_{n>0}\overline{a_n}(f)\overline{\exp(2in\pi z)}=\sum_{n>0}a_n(f)\exp(-2in\pi\bar{z})=f(-\bar{z}).$$
En effet $a_n(f)\in\R$ par hypothèse.
\end{proof}

Pour tout choix de signes $\epsilon\in\{\pm 1\}^n$, on peut définir $\Omega_{k}^{\epsilon}$ comme étant l'image de $\otimes_{j=1}^n S_{k_j}^{\R}$ par l'application :
$$f_1\otimes...\otimes f_n\mapsto c_{\epsilon_1}(\omega_{f_1})\wedge...\wedge c_{\epsilon_n}(\omega_{f_n}),$$
où $c_1(\omega)=\omega$ et $c_{-1}(\omega)=\overline{\omega}$ est le conjugué complexe c'est à dire une forme modulaire antiholomorphe de poids $k$.\\

\begin{coro}
Les applications suivantes sont des isomorphismes de $\R$-espaces vectoriels:
\begin{align}
\text{Pour tout }\gamma\in\Gamma^n,\quad &\varphi_{\gamma}^*:\Omega_k^{\epsilon}\to\Omega_k^{\epsilon},\quad \omega(z)\mapsto \omega(\varphi_{\gamma}z),\\
\text{pour tout }\sigma\in\S_n,\quad &\varphi_{\sigma}^*:\Omega_k^{\epsilon}\to\Omega_{k}^{\sigma(\epsilon)},\quad \omega(z)\mapsto \omega(\varphi_{\sigma}z),\\
\text{et pour tout }\epsilon'\in\{\pm 1\}^n,\quad &\varphi_{\epsilon'}^*:\Omega_k^{\epsilon}\to\Omega_{k}^{\epsilon\epsilon'},\quad \omega(z)\mapsto \omega(\varphi_{\epsilon'}z).
\end{align}
\end{coro}

\begin{proof}
D'après la Proposition \ref{conju}, il est équivalent de définir $\Omega_k^{\epsilon}$ via l'isomorphisme:
\begin{equation}
\varphi_{\epsilon}^*:\Omega_k^{\holo}\to\Omega_k^{\epsilon},\quad\omega(z)\mapsto \omega(\varphi_{\epsilon} z).
\end{equation}
Ainsi ce corollaire repose alors sur le fait que $\varphi_{\gamma}^*$ et $\varphi_{\sigma}^*$ sont des endomorphismes bijectifs de $\Omega_k^{\holo}$ et des calculs explicites de commutateurs:
\begin{align*}
\varphi_{\gamma}\circ\varphi_{\epsilon}&=\varphi_{\epsilon}\circ\varphi_{\varepsilon\gamma\varepsilon},\\
\varphi_{\sigma}\circ\varphi_{\epsilon}&=\varphi_{\sigma(\epsilon)}\circ\varphi_{\sigma},\\
\text{et }\varphi_{\epsilon}\circ\varphi_{\epsilon'}&=\varphi_{\epsilon\epsilon'}.
\end{align*}
\end{proof}

Pour tout $\epsilon\in\{\pm 1\}^n$, définissons l'application $\Z$-linéaire définie par:
\begin{equation}
\phi_{\epsilon}:\Z[G_n]\to \Z[G_n],\quad [(\gamma_1,...,\gamma_n),\sigma]\mapsto [(c_{\epsilon_{\sigma(1)}}(\gamma_1),...,c_{\epsilon_{\sigma(n)}}(\gamma_n)),\sigma],
\end{equation}
avec $c_1(\gamma)=\gamma$ et $c_{-1}(\gamma)=\varepsilon\gamma\varepsilon$.\par
Pour toute partie $I$ de $\Z[G_n]$, posons $I^{\epsilon}=\phi_{\epsilon}(I)$.


\begin{prop}
On a la caractérisation suivante de $\mathcal{A}(k)^{\epsilon}=\phi_{\epsilon}(\mathcal{A}(k))$:
\begin{equation}
a\in\mathcal{A}(k)^{\epsilon}\Leftrightarrow\forall\omega\in\varphi_{\epsilon}^*(\Omega_k^{\R}),\langle\omega,\mathcal{T}_n|_{a}\rangle=0.
\end{equation}
De plus pour tout $\epsilon\in\{\pm 1\}^n,$ on a : $\mathcal{A}(k)^{-\epsilon}=\mathcal{A}(k)^{\epsilon}$.\\
\end{prop}

\begin{proof}
On a : $[\gamma,\sigma]=[1,\sigma][\gamma,1]$ donc: $\varphi_{[\gamma,\sigma]}^*=\varphi_{\sigma}^*\varphi_{\gamma}^*.$\\
On peut aussi remarquer que $\varphi_{\epsilon}(\tau_n^{\sigma})=\tau_n^{\sigma}$ car les imaginaires purs sont stables par $z\mapsto -\bar{z}$. On a alors:
$$\langle \omega,\mathcal{T}_n\rangle=\langle \omega,\varphi_{\epsilon}(\mathcal{T}_n)\rangle=\langle \varphi_{\epsilon}^*\omega,\mathcal{T}_n\rangle.$$
Ceci démontre la première partie car $\varphi_{\epsilon}^*$ échange les espaces $\Omega_k^{\R}$ et $\varphi_{\epsilon}^*(\Omega_k^{\R})$.\par
Ainsi on démontre aisément que pour $\gamma\in\mathcal{A}(k)$, on a pour tout $\omega\in\varphi_{\epsilon}^*(\Omega_k^{\R})$:
$$0=\langle (\varphi_{\epsilon}^*\omega)|_{[\gamma,\sigma]},\varphi_{\sigma(\epsilon)}(\mathcal{T}_n)\rangle=\langle (\varphi_{\sigma(\epsilon)}^*\varphi_{\sigma}^*\varphi_{\gamma}^*\varphi_{\epsilon}^*)\omega,\tau_n\rangle.$$
Il nous reste plus qu'à remarquer d'une part que: 
$\varphi_{\sigma(\epsilon)}^*\varphi_{\sigma}^*=\varphi_{\sigma}^*\varphi_{\epsilon}^*.$
et d'ainsi déduire l'égalité:
$$\varphi_{\sigma(\epsilon)}^*\varphi_{\sigma}^*\varphi_{\gamma}^*\varphi_{\epsilon}^*=\varphi_{\sigma}^*\varphi_{\epsilon}^*\varphi_{\gamma}^*\varphi_{\epsilon}^*=\varphi^*_{\sigma}\varphi^*_{\varepsilon\gamma\varepsilon}=\varphi_{\phi_{\epsilon}([\gamma,\sigma])}^*.$$
Et ainsi $\phi_{\epsilon}([\gamma,\sigma])=[(c_{\epsilon_{\sigma(1)}}(\gamma_1),...,c_{\epsilon_{\sigma(n)}}(\gamma_n)),\sigma]\in\mathcal{A}(k)^{\epsilon}.$\par
Enfin l'égalité des idéaux $\mathcal{A}(k)^{-\epsilon}=\mathcal{A}(k)^{\epsilon}$ correspond à remarquer:
$$\langle \varphi_{\epsilon}^*\omega,C\rangle=0=\overline{\langle \varphi_{\epsilon}^*\omega,C\rangle}=\langle\varphi_{-\epsilon}^*\omega,C\rangle.$$
\end{proof}

Les idéaux $\mathcal{A}(k)^{\epsilon}$ sont ainsi stables par conjugaison par $(\varepsilon,...,\varepsilon)$. De plus, on notera que les définitions:
$$\Omega_k^{\R}=\Omega_k^{\holo}\oplus\overline{\Omega_k^{\holo}}\text{ et }\Omega_k^{\holo}=\Omega_k^{(+,...,+)},$$
permet de donner l'écriture suivante de l'espace de formes:
$$\varphi_{\epsilon}^*(\Omega_k^{\R})=\Omega_k^{\epsilon}\oplus\Omega_k^{-\epsilon}.$$
Il suffit donc de vérifier l'annulation pour les éléments de $\Omega_k^{\epsilon}$.

\subsection{Les $n$-cycles orthogonaux aux formes modulaires}

Au vu de la formule du corollaire \ref{coro36}, pour mémoire:
\begin{equation*}
R_k(f)|_{[\gamma,\rho]}(\sigma)=\langle \omega_{f}(z,X), (\gamma^{-1}.\tau_n^{\rho})^{\sigma} \rangle,
\end{equation*}
définissons l'ensemble des \textsl{chaînes transverses} par:
\begin{equation}
\Omega_k^{\bot}=\{C\in  M_n^{pte}(\H^n,\Z)\text{ tel que }\langle \omega, C\rangle=0,\forall \omega\in\Omega_k^{\R}\}.
\end{equation}
On remarque que cette définition est indépendante du corps de base. On peut remplacer $\Omega_k^{\R}$ par $\Omega_k=\Omega_k^{\R}\otimes_{\R}\C$ ou tout autre extension.\par
On dispose alors de deux annulations essentielles:

\begin{lem}
Soit $\omega\in\Omega_{k}$. On a :
\begin{equation}
\langle \omega, C \rangle=0,\text{ pour tout }C\in \Im(\delta_{n+1}).
\end{equation}
De plus, pour tout entier $0\leq j \leq n$ et toute matrice $g\in\Gamma$, on a :
\begin{equation}
\langle \omega, C \rangle=0,\text{ pour tout }C\in  M_n^{pte}(\varphi_j^g(\H^{n-1}),\Z).
\end{equation}
\end{lem}

\begin{proof}
Soit $C\in \Im(\delta_{n+1})$ alors il existe $P\in   M_{n+1}^{pte}(\H^n,\Z)$ tel que $C=\delta_{n+1}(P)$ donnant lieu au calcul:
$$\langle \omega, C \rangle=\langle \omega, \delta_{n+1} (P)\rangle =\langle d\omega, P\rangle.$$
La dernière égalité provenant du Théorème de Stokes.\\
Puis on a :
$$d(\wedge_{j=1}^n\omega_{f_j})=\sum_{j=1}^n\omega_{f_1}\wedge...\wedge d(\omega_{f_j})\wedge...\wedge\omega_{f_n}=0,$$
car $d\omega_f=d\left(f(z)(X-z)^{k-2}dz\right)=0$ pour tout $f\in S_k $.\par
Pour démontrer la seconde annulation, on dispose de $  M_n^{pte}(\varphi_j^g(\H^{n-1}),\Z)=\varphi_j^g\left(M_n^{pte}(\H^{n-1},\Z)\right)$. 
Ainsi pour tout $C\in  M_n^{pte}(\varphi_j^g(\H^{n-1}),\Z)$, 
il existe une $n$-chaîne $c\in  M_n^{pte}(\H^{n-1},\Z)$ tel que $C=\varphi_j^g(c)$ et on obtient :
$$\langle \omega, C\rangle=\langle \omega, \varphi_j^g(c)\rangle=\langle (\varphi_j^g)^*\omega,c\rangle.$$
où $(\varphi_j^g)^*:\Omega^n(\H^n)\to\Omega^n(\H^{n-1}),\omega\mapsto\left[(z_1,...,z_{n-1})\to\omega(\varphi_j^g(z_1,...,z_{n-1}))\right]$.\par

Or pour tout $\omega\in\Omega_{k}$, $(\varphi_j^g)^*\omega=0$.\\
En effet, dans le cas $j=0$ ou $n$, il n'y a pas de variation suivant $dz_j$ donc $dz_j$ s'envoie sur $0$.\\
Dans les autres cas, $dz_j\wedge dz_{j+1}$ s'envoie sur $dz_j\wedge d(g.z_j)=(c z_j +d)^{-2}dz_j\wedge dz_j=0.$\\
Ainsi on obtient donc bien $\langle \omega, C\rangle=0$.
\end{proof}

\begin{thm}\label{decompomegathm}
On dispose de la caractérisation suivante de $\Omega_{k}$:
\begin{equation}
\Omega_{k}=\{\omega\in\Omega^n_{par}(\H^n,V_k)^{\Gamma^n};d\omega=0
\text{ et }\varphi^*_{j}\varphi_{\sigma}^*\omega=0,\text{ pour } 0\leq j\leq n,\sigma\in\S_n\}.
\end{equation}
On obtient alors:
\begin{equation}\label{thmeq:2}
\delta_{n+1}(M_{n+1}^{pte}(\H^n,\Z))+\sum_{j,g,\sigma}M_n^{pte}(\varphi_{\sigma}\varphi_j^g(\H^{n-1}),\Z)\subset\Omega_k^{\bot}=(\Omega_k^{\holo})^{\bot}.
\end{equation}
\end{thm}

\begin{proof}
L'inclusion $\subset$ est vérifiable par calcul direct et correspond essentiellement au lemme précédent.
Réciproquement, soit $\omega\in\Omega^n_{par}(\H^n,V_k)^{\Gamma^n}$ vérifiant les équations du membre de droite. La forme $\omega$ est une combinaison sur $\C$ de:
$$F(z_1,...,z_n)\d u_1\wedge...\wedge \d u_n,\text{ pour des }u_j\text{ distincts parmi }\{z_1,\bar{z_1},...,z_n,\bar{z_n}\},$$
où $F:\H^n\to V_{k}$ est une fonction harmonique.
On peut supposer les indices croissants pour obtenir des termes linéairement indépendants.\par
Les équations $\varphi_0^*\varphi_{\sigma}^*\omega=0$ pour $\sigma$ parcourant $\S_n$ donne l'annulation des termes ne possédant ni $z_{\sigma(1)}$ ni $\overline{z_{\sigma(1)}}$ parmi l'ensemble des $u_j$. Ainsi il reste uniquement les termes où chacune des variables apparaissent exactement une fois:
$$\omega=\sum_{\epsilon\in\{\pm 1\}^n} F_{\epsilon}(z_1,...,z_n)d(c_{\epsilon_1}(z_1))\wedge...d(c_{\epsilon_n}(z_n)),$$
où $F_{\epsilon}:\H^n\to V_k$ est harmonique et la conjugaison est donnée par $\epsilon\in\{\pm 1\}^n$.\par
La condition $d\omega=0$ et l'invariance par $\Gamma^n$ permet d'obtenir de réécrire ces termes en:
$$F_{\epsilon}(z_1,...,z_n)=f^{\epsilon}_{1}(c_{\epsilon_1}(z_1))(X_1-c_{\epsilon_1}(z_1))^{k_1-2}... f^{\epsilon}_n(c_{\epsilon_n}(z_n))(X_n-c_{\epsilon_n}(z_n))^{k_n-2},$$
où les formes holomorphes $(f^{\epsilon}_1,...,f^{\epsilon}_n)\in\prod_{j=1}^n M_{k_j} $ dépendent de ${\epsilon}\in\{\pm 1\}^n$.\par
La condition $\varphi_1^*\varphi_{\sigma}^*\omega=0$ donne l'annulation de ces termes sauf lorsque $\pm\epsilon=(1,...,1)$. En effet, sinon il existe deux entiers $a$ et $b$ tel que $\epsilon_a=-\epsilon_b$, soit $\sigma\in\S_n$ tel que $\sigma(1)=a$ et $\sigma(2)=b$. L'équation $\varphi_1^*\varphi_{\sigma}^*\omega=0$ donne l'annulation de:
\begin{multline*}
\sum_{\substack{\epsilon_a=1\\ \epsilon_b=-1}} f_a^{\epsilon}(z_1)f_b^{\epsilon}(\bar{z_1})(X_a-z_1)^{k_a-2}(X_b-\bar{z_1})^{k_b-2}G_{\epsilon}^{\sigma}(z_2,...,z_{n-1}) dz_1\wedge d\bar{z_1}\wedge d(c_{\epsilon_{\sigma(3)}}(z_2))\wedge...\wedge d(c_{\epsilon_{\sigma(n)}}(z_{n-1}))\\
+f_b^{-\epsilon}(z_1)f_a^{-\epsilon}(\bar{z_1})(X_a-\bar{z_1})^{k_a-2}(X_b-z_1)^{k_b-2}G_{-\epsilon}^{\sigma}(z_2,...,z_{n-1}) dz_1\wedge d\bar{z_1}\wedge d(c_{-\epsilon_{\sigma(3)}}(z_2))\wedge...\wedge d(c_{-\epsilon_{\sigma(n)}}(z_{n-1})),
\end{multline*}
où $G_{\epsilon}^{\sigma}(z_2,...,z_{n-1})=\prod_{j=3}^n f^{\epsilon}_{\sigma(j)}(c_{\epsilon_{\sigma(j)}}(z_{\sigma(j)}))(X_{\sigma(j)}-c_{\epsilon_{\sigma(j)}}(z_{\sigma(j)}))^{k_{\sigma(j)}-2}.$ La liberté du choix de $\sigma$ offerte par le choix de $\sigma(3),...,\sigma(n)$ réduit ceci à l'annulation de:
$$f_a^{\epsilon}(z_1)f_b^{\epsilon}(\bar{z_1})(X_a-z_1)^{k_a-2}(X_b-\bar{z_1})^{k_b-2}+f_b^{-\epsilon}(z_1)f_a^{-\epsilon}(\bar{z_1})(X_a-\bar{z_1})^{k_a-2}(X_b-z_1)^{k_b-2}=0.$$
Ceci donne l'annulation des deux termes car si $k_a=2$ alors $f_a^{\epsilon}=f_a^{-\epsilon}=0$ car $M_2 =\{0\}$. Sinon on peut identifier les coefficients des monômes en $X_a$ comme nuls et déduire $f_a^{\epsilon}(z_1)f_b^{\epsilon}(\bar{z_1})^=f_b^{-\epsilon}(z_1)f_a^{-\epsilon}(\bar{z_1})=0$.\par
Il ne reste ainsi plus que deux termes non nuls:
\begin{align*}
&f^{(1,..,1)}_{1}(z_1)(X_1-z_1)^{k_1-2}dz_1\wedge...\wedge f^{(1,..,1)}_n(z_n)(X_n-z_n)^{k_n-2}dz_n,\\
\text{et : }&f^{(-1,..,-1)}_{1}(\bar{z_1})(X_1-\bar{z_1})^{k_1-2}d\bar{z_1}\wedge...\wedge f^{(-1,..,-1)}_n(\bar{z_n})(X_n-\bar{z_n})^{k_n-2}d\bar{z_n}.
\end{align*}
La condition d'annulation sur le bord $\partial\H^n$ permet d'obtenir le caractère parabolique des formes modulaires restantes $(f_j^{\epsilon})_{1\leq j\leq n}\in\prod_{j=1}^n S_{k_j}$ pour $\epsilon\in\{(1,...,1),(-1,...,-1)\}$. On obtient ainsi $\omega\in\Omega_k$.\par
Pour démontrer l'inclusion (\ref{thmeq:2}), on remarque à nouveau que pour ensemble de forme $\Omega$ on a: $\overline{\Omega}^{\bot}=\Omega^{\bot}$. En effet, pour $C\in  M_n^{pte}(\H^n,\Z)$ et $\omega\in\Omega$, on a: $\langle \overline{\omega},C\rangle=\overline{\langle\omega, C\rangle}.$ Ainsi:
$$\Omega_k^{\bot}=\left(\Omega_k^{\holo}\oplus\overline{\Omega_k^{\holo}}\right)^{\bot}=(\Omega_k^{\holo})^{\bot}.$$
Puis $\{d\omega=0\}^{\bot}=\Im(\delta_{n+1})$ et pour $0\leq j \leq n$ et $\sigma\in\S_n$ on a:
$$\{\varphi_j^*\varphi_{\sigma}^*\omega=0\}^{\bot}=\varphi_{\sigma}\varphi_j(M_n^{pte}(\H^{n-1},\Z))=M_n^{pte}(\varphi_{\sigma}\varphi_j(\H^{n-1}),\Z).$$
La stabilité par l'action de $\Gamma^n$ de $\Omega_k^{\bot}$ donne alors l'inclusion du théorème.
\end{proof}

\begin{prop}\label{condcycle}
Pour tout $n$-cycle $C\in  M_n^{pte}(\H^n,\Z)$, on a: 
\begin{align}
&\delta_{n}(C)\in \sum_{j,g,\sigma} \delta_{n}\left(M_{n}^{pte}(\varphi_{\sigma}\varphi_j^g(\H^{n-1}),\Z)\right)\\
\Leftrightarrow\quad &C\in \delta_{n+1}(M_{n+1}^{pte}(\H^n,\Z))+\sum_{j,g,\sigma}  M_n^{pte}(\varphi_{\sigma}\varphi_j^g(\H^{n-1}),\Z).
\end{align}
\end{prop}

\begin{proof}
On a déjà montrer que $M_m^{pte}(\H^n,\Z)=0$ sauf si $m\in\{0,2n\}$. Ceci nous donne $\Im(\delta_{n+1})=\Ker(\delta_{n})$ dès que $n\geq 2$ donnant:
$$\delta_{n+1}\left( M_{n+1}^{pte}(\H^n,\Z)\right)=\Ker(\delta_{n}| M_{n}^{pte}(\H^n,\Z)).$$
Ainsi pour tout élément $C$ vérifiant :
$$\delta_{n}(C)\in \sum_{j,g,\sigma} \delta_{n}( M_n^{pte}(\varphi_{\sigma}\varphi_j^g(\H^{n-1}),\Z)),$$
s'écrit  $\delta_{n}(C)=\sum_{j,g,\sigma} \delta_{n}(c_{\sigma,j}^{g}),$ où $c_{\sigma,j}^{g}\in M_n^{pte}(\varphi_{\sigma}\varphi_j^g(\H^{n-1}),\Z)$ et:
$$c_0=C-\sum_{j,g,\sigma} c_{\sigma,j}^{g}\in \Ker (\delta_{n})=Im (\delta_{n+1}).$$
Pour montrer la réciproque, il suffit d'appliquer $\delta_{n}$ à la chaîne et d'utiliser $\delta_{n}\circ\delta_{n+1}=0$.
\end{proof}

\begin{rem}
Cherchons désormais à caractériser l'ensemble des chaînes de $\varphi_{\epsilon}^*(\Omega_k)^{\bot}$.\par
Soit $\varphi_{\epsilon}^*\omega\in\varphi_{\epsilon}^*(\Omega_k)$. Pour obtenir, l'annulation du crochet $\langle \varphi^*_{\epsilon}\omega,C\rangle=\langle\omega,\varphi_{\epsilon}(C)\rangle$ il suffit de vérifier la première assertion pour $\varphi_{\epsilon}(C)$. Ainsi on dispose de l'implication:
\begin{equation}
\varphi_{\epsilon}(C)\in\sum_{j,g,\sigma} \delta_{n}\left( M_{n}^{pte}(\varphi_{\sigma}\varphi_j^g(\H^{n-1}),\Z)\right)
\Rightarrow\forall\omega\in\varphi_{\epsilon}^*(\Omega_k),\langle\omega,C\rangle=0.
\end{equation}
\end{rem}

\section{Calcul de l'idéal des relations}

\subsection{Action de $P\G^n\rtimes\S_n$ sur les chaînes transverses}

Pour déterminer l'idéal $\mathcal{A}(k)$, nous allons étudier les $n$-chaînes $C=\gamma.\tau_n$ où $\gamma\in\Z[\Gamma^n]$ et en particulier, celles qui vérifiant les conditions de la Proposition \ref{condcycle}.\par
L'action de $G_n$ sur l'image des applications $\varphi_j^g$, $\varphi_{\sigma}$ et $\varphi_{\epsilon}$ est donnée par les applications suivantes.\\

Pour tout $0\leq j\leq n$ et $g\in\Gamma$, on définit $\phi_j^g:\Gamma^{n-1}\to\Gamma^n$ par:
$$\phi_j^g(\gamma_1,...,\gamma_{n-1})=
\begin{cases}
(g,\gamma_1,...,\gamma_{n-1})&\text{ si }j=0\\
(\gamma_1,...,\gamma_{n-1},g)&\text{ si }j=n\\
(\gamma_1,...,\gamma_j,g\gamma_j,\gamma_{j+1},...)&\text{ sinon}.
\end{cases}$$
De plus, pour tout $\sigma\in\S_n$ et $\epsilon\in\{\pm 1\}^n$, on rappelle les définitions:
$$\phi_{\sigma}(\gamma_1,...,\gamma_n)=(\gamma_{\sigma(1)},...,\gamma_{\sigma(n)})
\text{ et }\phi_{\epsilon}(\gamma_1,...,\gamma_n)=(c_{\epsilon_1}(\gamma_1),...,c_{\epsilon_n}(\gamma_n)),$$
où $c_1(\gamma)=\gamma$, $c_{-1}(\gamma)=\varepsilon\gamma\varepsilon$ et $\varepsilon=\mat{-1}{0}{\phantom{-}0}{1}$.\par

\begin{prop}
Pour tout $z\in\H^{n-1}$, $\gamma\in\Gamma^{n-1}$, $0\leq j \leq n$ et $g\in\Gamma$, on a :
$$\varphi_j^g(\gamma.z)=\phi_j^g(\gamma) .\varphi_j(z).$$
Pour tout $z\in\H^n$, $\gamma\in\Gamma^n$, $\sigma\in\S_n$ et $\epsilon\in\{\pm 1\}^n$, on a:
$$\varphi_{\sigma}(\gamma.z)=\phi_{\sigma}(\gamma).\varphi_{\sigma}(z)
\text{ et }\varphi_{\epsilon}(\gamma.z)=\phi_{\epsilon}(\gamma).\varphi_{\epsilon}(z).$$
\end{prop}

\begin{proof}
On vérifie les formules par le calcul:
\begin{align*}
\varphi_j^g((\gamma_1,...,\gamma_{n-1}).(z_1,...,z_{n-1}))&=(\gamma_1.z_1,...,\gamma_j.z_j,g\gamma_j.z_j,...,\gamma_{n-1}.z_{n-1})\\
&=(\gamma_1,...,\gamma_j,g\gamma_j,...,\gamma_{n-1}).(z_1,...,z_j,z_j,...,z_{n-1})\\
&=\phi_j^g(\gamma) .\varphi_j(z).
\end{align*}
Et pour le cas plus simple $\varphi_{\sigma}(\gamma.z)=(\gamma_{\sigma(1)}.z_{\sigma(1)},...,\gamma_{\sigma(n)}.z_{\sigma(n)})=\phi_{\sigma}(\gamma).\varphi_{\sigma}(z)$.\par
La dernière provient du calcul $-(\overline{\gamma z})=\varepsilon\gamma.(\bar{z})=(\varepsilon\gamma\varepsilon).(-\bar{z})$.
\end{proof}

\begin{rem}
Les applications ainsi construites sont des morphismes de groupes:
$$\S_n\to Aut(\Gamma^n),\sigma\mapsto\phi_{\sigma}\text{ et }\{\pm 1\}^n\to Aut(\Gamma^n),\epsilon\to\phi_{\epsilon}.$$\par
On dispose ainsi d'un moyen simple d'inversion:
$$\gamma=\phi_{\sigma}(\phi_{\sigma^{-1}}(\gamma))\text{ et }\gamma=\phi_{\epsilon}(\phi_{\epsilon}(\gamma)).$$
\end{rem}

De même pour tout $1\leq j\leq n$, il existe une décomposition similaire en fonction de l'image des $(\phi_j^g)_{g\in\Gamma}$.
\begin{prop}\label{decompgamman}
Pour tout $0\leq j\leq n$, on dispose de la décomposition:
\begin{equation}
\Z[\Gamma^n]=\bigoplus_{g\in\Gamma} \Z[\phi_j^g\left(\Gamma^{n-1}\right)].
\end{equation}
Ainsi tout élément $\gamma\in\Z[\Gamma^n]$, il existe une famille unique de $(\gamma_j^g)_{g\in\Gamma}\in(\Z[\Gamma^{n-1}])^{\Gamma}$ telle que:
$$\gamma=\sum_{g\in\Gamma}\phi_j^g(\gamma_j^g).$$
\end{prop}

\begin{proof}
En effet, pour $(\gamma_1,...,\gamma_n)\in\Gamma^n$, on peut récrire:
$$(\gamma_1,...,\gamma_n)=(...,\gamma_j,\gamma_{j+1}\gamma_j^{-1}\gamma_j,...)
=\phi_j^{\gamma_{j+1}\gamma_j^{-1}}(\gamma_1,...,\gamma_j,\gamma_{j+2},...,\gamma_n).$$
Puis les espaces sont bien en sommes directes car l'application suivante est bien bijective:
$$(\Z[\Gamma^{n-1}])^{\Gamma}\to \Z[\Gamma^n], (\gamma_j^g)_{g\in\Gamma}\mapsto \sum_{g\in\Gamma}\phi_j^g(\gamma_j^g).$$
\end{proof}

\begin{rem}
Un corollaire de ce résultat est que pour tout $0\leq j\leq n$ les espaces $\varphi_j^g(\H^{n-1})$ et $\varphi_j^{g'}(\H^{n-1})$ sont disjoints dès que $g\neq g'$ sauf si:
$$\left(j=0 \text{ et }g'\in g\Gamma_{\infty}\right)\text{ ou }\left(j=n \text{ et }g'\in g\Gamma_{0}\right).$$
\end{rem}

\begin{prop}\label{inver}
Soit $(n,..,1)\in\S_n$ et $(S,...,S)\in\Gamma^n$ alors les images de $\tau_n$ sont liées par:
$$\varphi_{(n,...,1)}(\tau_n)=\varphi_{(S,...,S)}(\tau_n).$$
\end{prop}

\begin{proof}
Le support de $\tau_n$ est donné par $\{(it_1,...,it_n);0<t_1<...<t_n\}$ et donc:
\begin{align*}
\varphi_{(n,...,1)}(Supp(\tau_n))&=\{(it_n,...,it_1);0<t_1<...<t_n\},\\
\text{et }\varphi_{(S,...,S)}(Supp(\tau_n))&=\{(-1/it_1,...,-1/it_n);0<t_1<...<t_n\}.
\end{align*}
\end{proof}

\subsection{L'idéal des relations $\mathcal{A}_n$}

Venons-en à la détermination de $\mathcal{A}(k)\subset\mathbb{Z}[G_n]$, introduit en \ref{Akpoly} puis récrit en \ref{Aktopo} comme:
$$\mathcal{A}(k)=\{a\in\Z[G_n]\text{ tel que }\langle \omega,\left(\tau^{\sigma}_n\right)|_a \rangle=0\text{ pour tout }\omega\in\Omega_{k}\}.$$

Notons $\mathcal{I}_n$ l'idéal à gauche de $\Z[\Gamma^n]$ formé des éléments $\gamma$ vérifiant :
\begin{equation}
\gamma.\delta_{n} (\tau_n)\in \sum_{j=0}^n\sum_{g\in\Gamma} \delta_{n}\left[ M_{n}^{pte}(\varphi_j^g(\H^{n-1}),\Z)\right].
\end{equation}
Introduisons à nouveau l'antiautomorphisme de $\Z[\Gamma^n]$ défini par:
$$\widetilde{(\gamma_1,...,\gamma_n)}=(\gamma_1^{-1},...,\gamma_n^{-1}).$$
Il permet de considérer l'idéal à gauche (resp. à droite) $\widetilde{I}$ associé à tout idéal à droite (resp. à gauche) $I\subset\Z[\Gamma^n]$.

\begin{prop}\label{prop19}
Pour toute famille de poids $k=(k_1,...,k_n)$, on a:
$$[\widetilde{\mathcal{I}_n},1]\subset\mathcal{A}(k).$$
\end{prop}

\begin{proof}
En effet, pour tout $\gamma\in\mathcal{I}_n$, on obtient:
$$\delta_n(\gamma\tau_n)=\gamma\delta_n(\tau_n)\in \sum_{j,g} \delta_{n}\left[ M_{n}^{pte}(\varphi_j^g(\H^{n-1}),\Z)\right].$$
La Proposition \ref{condcycle} permet d'avoir l'appartenance:
$$\gamma\tau_n \in \Im(\delta_{n+1})+\sum_{j,g} M_n^{pte}(\varphi_j^g(\H^{n-1}),\Z).$$
Et le Théorème \ref{decompomegathm} s'applique pour obtenir : $\gamma\tau_n\in\Omega_k^{\bot}$. On obtient alors pour tout $\omega\in\Omega_k$ :
$$\left(\sigma\mapsto\langle\omega,\tau_n^{\sigma}\rangle\right)|_{[\widetilde{\gamma},1]}=\left(\sigma\mapsto\langle\omega,(\gamma\tau_n)^{\sigma}\rangle\right)=0,$$
démontrant que $[\widetilde{\gamma},1]\in\mathcal{A}(k)$.
\end{proof}

\begin{rem}Soit $\epsilon\in\{\pm 1\}^n$. 
On dispose de l'idéal $\mathcal{I}^{\epsilon}_n=\phi_{\epsilon}(\mathcal{I}_n)$ formé par des éléments $\gamma\in\Z[\Gamma^n]$ vérifiant:
$$\varphi_{\epsilon}(\gamma.\delta_{n} (\tau_n))\in \sum_{j=0}^n\sum_{g\in\Gamma} \delta_{n}\left[ M_{n}^{pte}(\varphi_j^g(\H^{n-1}),\Z)\right]$$
$$\Leftrightarrow\gamma.\delta_{n} (\tau_n)\in \sum_{j=0}^n\sum_{g\in\Gamma} \delta_{n}\left[ M_{n}^{pte}(\varphi_{\epsilon}\varphi_j^g(\H^{n-1}),\Z)\right].$$
Ainsi pour toute poids multiple $k$, on a : $[\widetilde{\mathcal{I}_n^{\epsilon}},1]\subset\mathcal{A}(k)^{\epsilon}$.
\end{rem}

On dispose d'une construction par récurrence sur $n$ de $\mathcal{I}_n$.\par
Pour tout entier $j$ vérifiant $0\leq j\leq n$, définissons $\mathcal{I}_{n-1}[j]$ est l'idéal à gauche formé des $\gamma\in\Z[\Gamma^n]$ tels que:
\begin{equation}
\gamma.\varphi_j(\tau_{n-1})\in\sum_{g\in\Gamma} \delta_{n}\left[ M_{n}^{pte}(\varphi_j^g(\H^{n-1}),\Z)\right].
\end{equation}

\begin{prop}
L'idéal à gauche $\mathcal{I}_n$ est déterminé par:
\begin{equation}
\mathcal{I}_n=\bigcap_{j=0}^n\mathcal{I}_{n-1}[j].
\end{equation}
\end{prop}

\begin{proof}
On raisonne par double inclusion en commençant par $\bigcap_{j=0}^n\mathcal{I}_{n-1}[j]\subset\mathcal{I}_n$.\par
Soit $\gamma\in\bigcap_{j=0}^n\mathcal{I}_{n-1}[j]$, c'est-à-dire pour tout $0\leq j\leq n$:
$$\gamma.\varphi_j(\tau_{n-1})\in \sum_{g\in\Gamma} \delta_{n}\left[ M_{n}^{pte}(\varphi_j^g(\H^{n-1}),\Z)\right],$$
et donc le calcul du bord de $\tau_n$ (\ref{bordtaun}) donne: 
$$\gamma.\delta_{n}(\tau_n)=\sum_{j=0}^n (-1)^j \gamma.\varphi_j(\tau_{n-1})\in\sum_{j=0}^n\sum_{g\in\Gamma} \delta_{n}\left[ M_{n}^{pte}(\varphi_j^g(\H^{n-1}),\Z)\right].$$
Et ainsi $\gamma\in\mathcal{I}_n$.\par

Pour démontrer l'inclusion réciproque, on commence par démontrer un résultat annexe:
\textit{
Considérons $g,g'\in\Gamma$, des entiers $0\leq i<j\leq n$ et $C$ une $(n-1)$-chaîne tel que:
$$C\in M_{n-1}^{pte}(\varphi_i^g(\H^{n-1}),\Z)\cap M_{n-1}^{pte}(\varphi_j^{g'}(\H^{n-1}),\Z)\text{ pour }i< j.$$
Alors $C=\delta_{n}(c)$ pour une chaîne $c\in M_{n}^{pte}(\varphi_j^{g'}\varphi_i^g(\H^{n-2}),\Z)$.
}\\
En effet, $C$ est alors un élément de $ M_{n-1}^{pte}(\varphi_i^g(\H^{n-1})\cap\varphi_j^{g'}(\H^{n-1}),\Z)$ car ces espaces topologiques sont contractiles et seules les pointes définissent ainsi les classes d'homotopie. De plus, un calcul immédiat donne : $\varphi_i^g(\H^{n-1})\cap\varphi_j^{g'}(\H^{n-1})=\varphi_j^{g'}\varphi_i^g(\H^{n-2})$.
Ainsi $C\in M_{n-1}^{pte}(\varphi_j^{g'}\varphi_i^g(\H^{n-2}),\Z)=\varphi_j^{g'}\varphi_i^g(M_{n-1}^{pte}(\H^{n-2},\Z))$ d'après (\ref{mchtrans}) et l'homologie de $\H^{n-2}$ donne:
$$M_{n-1}^{pte}(\H^{n-2},\Z)=\delta_n(M_n^{pte}(\H^{n-2},\Z)),$$
démontrant ce lemme intermédiaire.\par
Soient $\gamma\in\mathcal{I}_n$ et $j$ un entier tel que $0\leq j\leq n$. Montrons que $\gamma\in\mathcal{I}_{n-1}[j]$. La formule (\ref{bordtaun}) donne:
$$\varphi_j(\tau_{n-1})=(-1)^j\left(\delta_n\tau_n-\sum_{i\neq j} (-1)^i\varphi_i(\tau_{n-1})\right),$$
montrant l'appartenance suivante:
$$\gamma\varphi_j(\tau_{n-1})\in\left[\sum_{i,g}\delta_{n}\left[ M_n^{pte}(\varphi_i^g(\H^{n-1}),\Z)\right]+\sum_{i\neq j,g} M_{n-1}^{pte}(\varphi_i^g(\H^{n-1}),\Z)\right].$$
Or cet élément est aussi dans $\sum_{g'} M_{n-1}^{pte}(\varphi_j^{g'}(\H^{n-1}),\Z)$. Donc $\gamma\varphi_j(\tau_{n-1})$ est dans la somme des intersections.\par
D'une part, on a simplement:
$$\sum_{i,g}\delta_{n}\left[ M_n^{pte}(\varphi_i^g(\H^{n-1}),\Z)\right]\cap\sum_{g'} M_{n-1}^{pte}(\varphi_j^{g'}(\H^{n-1}),\Z)\subset\sum_{g'}\delta_{n}\left[ M_n^{pte}(\varphi_j^{g'}(\H^{n-1}),\Z)\right].$$
Et le lemme intermédiaire montre d'autre part que:
$$\sum_{i\neq j,g} M_{n-1}^{pte}(\varphi_i^g(\H^{n-1}),\Z)\cap\sum_{g'} M_{n-1}^{pte}(\varphi_j^{g'}(\H^{n-1}),\Z)\subset\sum_{g'}\delta_{n}\left[ M_n^{pte}(\varphi_j^{g'}(\H^{n-1}),\Z)\right].$$
Ainsi on obtient bien $\gamma\in\mathcal{I}_{n-1}[j]$.
\end{proof}

\begin{prop}
Pour tous les entiers $n$ et $j$ tels que $n>1$ et $0\leq j \leq n$, les idéaux $\mathcal{I}_{n-1}[j]$ sont calculables en fonction de $\mathcal{I}_{n-1}$:
\begin{equation}
\mathcal{I}_{n-1}[j]=
\begin{cases}
\phi_0(\mathcal{I}_{n-1})+\mathbb{Z}[\G^n].(1-T,1,...,1)&\text{ si }j=0\\
\phi_n(\mathcal{I}_{n-1})+\mathbb{Z}[\G^n].(1,...,1,1-US)&\text{ si }j=n\\
\phi_j(\mathcal{I}_{n-1})&\text{ sinon.}
\end{cases}
\end{equation}
\end{prop}

\begin{proof}
Ces formules sont dues au résultat de décomposition de $\Z[\Gamma^n]$ suivant les $\phi_j^g$ de la proposition \ref{decompgamman}. Soit $\gamma\in\mathcal{I}_{n-1}[j]$.\par
Lorsque $j\neq 0,n$, on écrit $\gamma=\sum_{g\in\Gamma}\phi_j^g(\gamma_g)$ alors :
$$\sum_{g\in\Gamma}\varphi_j^g(\gamma_g.\tau_{n-1})=\sum_{g\in\Gamma}\phi_j^g(\gamma_g).\varphi_j(\tau_{n-1})\in\sum_{g\in\Gamma}\delta_{n}\left[ M_{n}^{pte}(\varphi_j^g(\H^{n-1}),\Z)\right].$$
Or les espaces $\varphi_j^g(\H^{n-1})$ étant disjoints alors cette dernière somme est directe et l'indépendance des termes donnent, pour tout $g\in\Gamma$:
$$\varphi_j^g(\gamma_g.\tau_{n-1})\in\delta_{n}\left[ M_{n}^{pte}(\varphi_j^g(\H^{n-1}),\Z)\right]=\varphi_j^g\left(\delta_{n}[ M_{n}^{pte}(\H^{n-1},\Z)]\right).$$
Ceci démontrant $\gamma_g\in\mathcal{I}_{n-1}$, pour tout $g\in\Gamma$.\par
La même démonstration s'adapte bien dans les cas où $j=0$ ou $n$ dans ces cas l'indépendance est valide à la condition près:
$$\left(j=0 \text{ et }g'\in g\Gamma_{\infty}\right)\text{ ou }\left(j=n \text{ et }g'\in g\Gamma_{0}\right),$$
expliquant l'apparition de termes supplémentaires.
\end{proof}

\begin{rem}
1) Ceci donne une méthode de construction par récurrence et partant simplement de $\mathcal{I}_1=(1+S,1+U+U^2)\Z[\Gamma]$. Pourtant la recherche de générateurs n'est pas immédiat car on dispose ici d'intersection. Pourtant nous verrons que $\mathcal{I}_n$ est de type fini et nous disposons d'une méthode de construction des générateurs.\par
2) Cette méthode de construction par récurrence stabilise la conjugaison par $(\varepsilon,...,\varepsilon)$. En effet, on démontre simplement que pour $\epsilon\in\{\pm 1\}^n$:
$$\mathcal{I}_n^{\epsilon}=\mathcal{I}_n \Leftrightarrow \epsilon\in\{(1,...,1),(-1,...,-1)\}.$$
Ceci est vrai pour $n=1$ car $\mathcal{I}_1^-=\mathcal{I}_1$. Puis on voit que les formules de $\mathcal{I}_{n-1}[j]$ donne un conjugué lié à celui de $\mathcal{I}_{n-1}$ stable par hypothèse de récurrence.\par
Par conséquence, l'application $P(X_1,...,X_n)\mapsto P(-X_1,...,-X_n)$ est une involution de $V_k[\mathcal{I}_n]$. Ainsi les polynômes pair et impair des multipériodes d'une famille de forme $f_1\otimes...\otimes f_n\in\otimes_{j=1}^n S_{k_j}$, définis par:
\begin{align}
P_{f_1,...,f_n}^+(X_1,..,X_n)&=1/2\left(P_{f_1,...,f_n}(X_1,...,X_n)+P_{f_1,...,f_n}(-X_1,...,-X_n)\right),\\
\text{et }P_{f_1,...,f_n}^-(X_1,..,X_n)&=1/2\left(P_{f_1,...,f_n}(X_1,...,X_n)-P_{f_1,...,f_n}(-X_1,...,-X_n)\right),
\end{align}
sont des éléments de $V_k[\mathcal{I}_n]$. De plus, en considérant la famille des permutés, on obtient un élément de $V_k^{\S_n}[\mathcal{A}(k)]$.
\end{rem}

On va chercher désormais à traduire les relations de mélange. Soient $a,b\geq 0$ vérifiant $n=a+b$. Notons $(I_a,1)$ l'idéal à droite dans $\Z[\Gamma^n]$ engendré par l'image d'un idéal à droite $I_a\subset\Z[\Gamma^a]$ par le morphisme de $\Z[\Gamma^a]$-modules:
$$\Z[\Gamma^a]\to\Z[\Gamma^n],\quad (\gamma_1,...,\gamma_a)\mapsto (\gamma_1,...,\gamma_a,1,...,1).$$
On définit, de manière analogue, pour tout $I_b\subset\Z[\Gamma^b]$ un idéal $(1,I_b)\subset\Z[\Gamma^n]$.
\begin{prop}\label{prop20}
On a les inclusions des idéaux à droite de $\Z[G_n]$:
\begin{equation}
[(\widetilde{\mathcal{I}_a},1),\sum_{\rho\in\S_{a,b}}(\rho)]\subset\mathcal{A}(k)\text{ et }[(1,\widetilde{\mathcal{I}_b}),\sum_{\rho\in\S_{a,b}}(\rho)]\subset\mathcal{A}(k).
\end{equation}
De plus, soit $(n,...,1)\in\S_n$ et $(S,...,S)\in\Gamma^n$ alors:
\begin{equation}
[(S,...,S),id]-[(id,...,id),(n,...,1)]\in\mathcal{A}(k)
\end{equation}
\end{prop}

\begin{proof}
Soit $\gamma\in\Gamma^n$ alors on peut effectuer le calcul:
$$\left(\sigma\mapsto\langle\omega,\tau_n^{\sigma}\rangle\right)|_{[\gamma^{-1},\sum_{\rho\in\S_{a,b}}(\rho)]}=\left(\sigma\mapsto\langle\omega,(\gamma(\sum_{\rho\in\S_{a,b}}\tau_n^{\rho}))^{\sigma}\rangle\right).$$
Or on remarque que $\sum_{\rho\in\S_{a,b}}\tau_n^{\rho}=\tau_a\otimes\tau_b$ ceci permet de séparer les tenseurs purs:
$$\langle \omega_a\otimes\omega_b,\tau_a\otimes\tau_b\rangle=\langle \omega_a,\tau_a\rangle\langle \omega_b,\tau_b\rangle,$$
où $\omega_a\in\Omega^a(\H^a,\C)$ et $\omega_b\in\Omega^b(\H^b,\C)$.\\
Supposons $\gamma=(\widetilde{\gamma_a},\widetilde{\gamma_b})\in(\widetilde{\mathcal{I}_a},1)$, alors :
$$\langle\omega_a\otimes\omega_b,\tau_a\otimes\tau_b\rangle|_{\gamma}=\langle \omega_a,\gamma_a.\tau_a\rangle\langle \omega_b,\gamma_b.\tau_b\rangle=0.$$
Le même calcul reste valide lorsque $\gamma=(\widetilde{\gamma_a},\widetilde{\gamma_b})\in(1,\widetilde{\mathcal{I}_b})$.\par
L'appartenance de l'élément $[(S,...,S),id]-[(id,...,id),(n,...,1)]$ dans $\mathcal{A}(k)$ est une réécriture de la Proposition \ref{inver} après passage à la dualité des actions. On a bien $\widetilde{(S,...,S)}=(S,...,S)$.
\end{proof}

Ceci permet de considérer l'idéal à droite de $\Z[G_n]$ indépendant des poids:
\begin{equation}
\mathcal{A}_n=\sum_{a=1}^n\left([(\widetilde{\mathcal{I}_a},1),\sum_{\sigma\in\S_{a,n-a}} (\sigma)]\right)+\left([(S,...,S),id]-[(1,...,1),(n,...,1)]\right).
\end{equation}

Les Propositions \ref{prop19} et \ref{prop20} démontrent pour toute famille de poids que :
\begin{equation}\label{incl:1}
\mathcal{A}_n\subset\mathcal{A}(k)\text{ et donc }(V_{k}^{\Z})^{\S_n}[\mathcal{A}_n]\supset (V_{k}^{\Z})^{\S_n}[\mathcal{A}(k)].
\end{equation}\par
Les relations du type $[(1,\widetilde{\mathcal{I}_b}),\sum_{\sigma\in \S_{a,b}}(\sigma)]$ sont superflues car elles sont engendrées par :\\
$[(\widetilde{\mathcal{I}_a},1),\sum_{\sigma\in\S_{a,n-a}} (\sigma)]$ et $[(S,...,S),id]-[(1,...,1),(n,...,1)]$.\\
En effet, les formules de constructions récurrentes de $\mathcal{I}_n$ permettent de démontrer que pour tout $1\leq a\leq n$:
\begin{equation}
(S,...,S)(\mathcal{I}_a,1)(S,...,S)=(\mathcal{I}_a^{(a,...,1)},1)=(1,\mathcal{I}_a)^{(a,...,1)}.
\end{equation}

\begin{rem}
Bien que dans la pratique ceci est peu commode, on peut réaliser la construction par récurrence directement sur $\mathcal{A}_n$.
Considérons pour tout $\sigma\in\S_n$ et $0\leq j\leq n$, l'idéal à gauche $\mathcal{A}_{n-1}[j,\sigma]\subset\Z[G_n]$, formé par les $a\in\Z[G_n]$ tel que:
$$\left[\varphi_{\sigma}\varphi_j(\mathcal{T}_{n-1})\right]|_a\in\sum_{g\in\Gamma} \delta_n\left[M_n^{pte}(\varphi_{\sigma}\varphi_j^g(\H^{n-1}),\Z)\right]^{\S_n}.$$
Et ainsi on obtient la formule:
\begin{equation}\label{incl:2}
\mathcal{A}_n=\cap_{j=0}^n\cap_{\sigma\in\S_n}\mathcal{A}_{n-1}[j,\sigma].
\end{equation}
D'autre part, on a une construction par récurrence déterminant $\mathcal{A}_{n-1}[j,\sigma]$ en fonction de $\mathcal{A}_{n-1}$.
Ces idéaux permettent d'introduire les termes de bords:
\begin{equation}
ME_k^{\Q}=\sum_{\sigma\in\S_n}\sum_{j=0}^n (V_{k}^{\Q})^{\S_n}[\mathcal{A}_{n-1}[j,\sigma]].
\end{equation}
Ce sont des permutations de polynômes rationnelles induit par la connaissance de $(V_{l}^{\Q})^{\S_{n-1}}[\mathcal{A}_{n-1}]$ pour les poids $l$ de longueurs $n-1$.
\end{rem}

Dans la suite, on va chercher à préciser les inclusions:
\begin{align*}
MP(k)\subset(V_{k}^{\C})^{\S_n}&[\mathcal{A}(k)]\subset(V_{k}^{\C})^{\S_n}[\mathcal{A}_n],\\
\text{et }ME_k^{\Q}&\subset(V_{k}^{\Q})^{\S_n}[\mathcal{A}_n].
\end{align*}

\subsection{Cohomologie relative de $P\G^n\rtimes \S_n$}

On dispose d'une construction analogue à celle réalisée dans le cas $n=2$ dans un cadre plus large.\par
Soit $n$ un entier. Notons $P_n=\pte^n$ les pointes de $\H^n$. On dispose de l'application:
\begin{equation}
\delta_{m}:\Z[P_n^m]\to\Z[P_n^{m-1}],\quad (a_1,...,a_m)\mapsto\sum_{j=1}^m(-1)^j(a_1,...,\hat{a_j},...,a_m).
\end{equation}
On dispose de la suite exacte longue:
$$...\stackrel{\delta_{m+2}}{\longrightarrow}\Z[P_n^{m+1}]\stackrel{\delta_{m+1}}{\longrightarrow}\Z[P_n^m]\stackrel{\delta_{m}}{\longrightarrow}...
\stackrel{\delta_2}{\longrightarrow}\Z[P_n]\stackrel{\delta_{1}}{\longrightarrow}\Z\stackrel{\delta_{0}}{\longrightarrow} 0.$$
On peut ainsi définir:
\begin{equation}
\Z[P_n^m]^0=\Ker(\delta_{m})=\Im(\delta_{m+1}).
\end{equation}
Définissons, de manière analogue au cas $n=2$, un espace quotient $H_m(P_n)$ en considérant les espaces transverses induits par les dimensions inférieurs et les applications $\varphi_j$:
\begin{equation}
\Z[P_n^m]_{trans}^0=\sum_{j=0}^n\sum_{\sigma\in\S_n}\sum_{g\in\Gamma}\varphi_{\sigma}\varphi_j^g(\Z[P_{n-1}^m]^0),
\end{equation}
qui permet de définir l'espace:
\begin{equation}
H_m(P_n)=\Z[P_n^m]^0/\Z[P_n^m]^0_{trans}.
\end{equation}

\begin{prop}
L'ensemble $P_n$ peut être munie d'une action de $G_n=\Gamma^n\rtimes \S_n$ donnée par:
\begin{equation}
[(\gamma_1,...,\gamma_n),\rho].(p_1,...,p_n)=(\gamma_{\rho^{-1}(1)}.p_{\rho^{-1}(1)},...,\gamma_{\rho^{-1}(n)}.p_{\rho^{-1}(n)}).
\end{equation}
On peut alors étendre cette action diagonalement puis par linéarité à $\Z[P_n^m]$. Les espaces $\Z[P_n^m]^0$ et $\Z[P_n^m]^0_{trans}$ sont stables par cette action, nous donnant ainsi une action à gauche de $G_n$ sur $H_m(P_n)$.
\end{prop}

\begin{proof}
On vérifie simplement que ceci définit bien une action à gauche de $G_n$ sur $P_n$. Puis la stabilité de $\Z[P_n^m]^0$ provient du calcul:
$$\delta_{m+1}([\gamma,\rho].p)=[\gamma,\rho].\delta_{m+1}(p)\text{ pour }p\in\Z[P_n^{m+1}].$$
Et celle de $\Z[P_n^m]^0_{trans}$ de l'échange des $\varphi_{\sigma}\varphi_j^g(\Z[P_{n-1}^m]^0)$ par $\S_n$ et de la stabilité par $\Gamma^{n-1}$.
\end{proof}

\begin{prop}
On peut associer à $\tau_n$ un élément $T_n\in H_n(P_n)$ comme étant la classe de:
\begin{equation}
\delta_{n+1} (U_n)\text{ où }U_n=\left[\tau_n(e_0),...,\tau_n(e_n)\right]\in P_n^{n+1}.
\end{equation}
Ceci nous permet de considérer l'application suivante:
\begin{equation}
\Theta_n:\Z[G_n]\to H_n(P_n),\quad [\gamma,\rho] \mapsto \gamma^{-1}\rho^{-1}T_n.
\end{equation}
Son noyau est $\mathcal{A}_n$. Notons $H_n(P_n)^0$ son image.
\end{prop}

\begin{proof}
Le calcul du noyau peut être fait par récurrence et repose d'une part sur la formule:
$$\mathcal{I}_n=\bigcap_{j=0}^n \mathcal{I}_{n-1}[j].$$
Et en remarquant qu'il existe une relation de récurrence entre $T_n$ et $T_{n-1}$:
$$T_n=\sum_{j=0}^n (-1)^j\varphi_j(U_{n-1})\in H_n(P_n),$$
où $U_{n-1}=\left[\tau_{n-1}(e_0),...,\tau_{n-1}(e_{n-1})\right]$ est un antécédent de $T_{n-1}=\delta_{n}(U_{n-1})$.\\
La liberté de choix de $U_{n-1}$ est bien résolue dans l'espace quotient $H_n(P_n)$.\par
Et d'autre part, sur la formule:
$$\sum_{\rho\in \S_{a,b}} \rho^{-1}T_n=T_a\otimes T_b.$$
\end{proof}

\begin{prop}
Le groupe $H_n(P_n)^0$, l'image de l'application $\Theta_n$, est engendré par les classes dans $H_n(P_n)$ des éléments:
\begin{equation}
\gamma.\varphi_{\sigma}\delta_{n+1}\left[(a,...,a),(a,...,a,b),...(a,b,...,b),(b,...,b)\right],
\end{equation}
où $a,b\in\pte$, $\sigma\in\S_n$ et $\gamma\in\Gamma^n$.
\end{prop}

\begin{proof}
On remarque que le cas où $a=g i\infty$ et $b=g 0$ est immédiat car l'élément s'écrit simplement:
$$\gamma\varphi_{\sigma}\delta_{n+1}\left[(g i\infty,...,g i\infty),(g i\infty,...,g i\infty,g0),...,(g0,...,g0)\right]=\gamma\varphi_{\sigma}(g,...,g)T_n.$$\par
Puis pour $a$ et $b$ quelconque on connait l'existence d'une chaîne de matrices de $PSL_2(\Z)$, $\gamma_1,...,\gamma_N$ telle que:
$$a=\gamma_1 i\infty\to\gamma_10=\gamma_2i\infty\to...\to\gamma_N0=b.$$
Il nous suffit alors de démontrer un résultat de décomposition conforme.
\end{proof}

\begin{lem}
Soient $a,b,c\in\pte$ alors on peut décomposer:
$$T_n(a,b)=\delta_n\left[(a,...,a),(a,...,a,b)...,(b,...,b)\right]$$
en somme de $2^n$ triangles de la forme:
$$\delta_{n+1}\left[(a_1,...,a_n),(a_1,...,a_{n-1},b_n)...,(b_1,...,b_n)\right],$$
où pour tout $1\leq j\leq n, \{a_j,b_j\}$ est soit $\{a,c\}$ soit $\{c,b\}$.
\end{lem}

\begin{proof}
Pour obtenir un candidat à la décomposition, on écrit par abus que le simplexe a pour équation:
$$a<t_1<...<t_n<b.$$
On doit alors intercaler la valeur $c$ pour obtenir les ensembles:
$$\{c<t_1<...<t_n<b\},\quad \{a<t_1<c<t_2<...<t_n<b\},...,\{a<t_1<...<t_n<c\}.$$
Chacun de ces ensembles est de la forme $T_{n_1}(a,c)\otimes T_{n_2}(c,b)$ pour toute pair $n_1+n_2=n$.\\
Il suffit alors d'utiliser la formule $\sum_{\sigma\in\S_{n_1,n_2}} T_n^{\sigma}=T_{n_1}\otimes T_{n_2}$ pour obtenir une décomposition adéquate. Le nombre de triangle obtenu est notamment :
$$\sum_{n_1+n_2=n} \sharp\S_{n_1,n_2}=\sum_{n_1+n_2=n}\frac{n!}{n_1!n_2!}=2^n.$$
Ce lemme permet donc de décomposer de manière récursive sur la distance $d(a,b)$ un triangle $T_n(a,b)$ en triangles $\gamma\varphi_{\sigma}T_n(i\infty,0)$. Ceci donne exactement la proposition.
\end{proof}

Comme dans le cas $n=2$, on remarque que cet espace correspond à l'ensemble des sommes de chaînes transverses globalement fermés quotienté par les sommes des chaînes transverses fermées:

\begin{coro}\label{formH}
On a:
\begin{equation}
H_n(P_n)^0=\left[\left(\sum_{j,g,\sigma}\varphi_{\sigma}\varphi_j^g(\Z[P_{n-1}^{n}])\right)\cap\Z[P_n^n]^0\right]/\sum_{j,g,\sigma}\varphi_{\sigma}\varphi_j^g(\Z[P_{n-1}^{n}]^0).
\end{equation} 
\end{coro}

\begin{proof}
Il suffit de remarquer que $T_n=\delta_{n+1} U_n=\sum_{j=0}^n (-1)^j\varphi_j(U_{n-1})$ donne bien l'inclusion $\subset$. L'inclusion réciproque provient du résultat de décomposition de $H_n(P_n)$ suivant:
\end{proof}

\begin{lem}
Le groupe $H_n(P_n)$ est engendrés par les éléments de la forme:
\begin{equation}
T^{\sigma}_n(a_j,b_j)_{1\leq j\leq n}=\varphi_{\sigma}\delta_{n+1}\left[(a_1,...,a_n),(a_1,...,a_{n-1},b_n),...,(b_1,...,b_n)\right],
\end{equation}
pour toute famille de $2n$ points $a_j,b_j\in\pte$ et $\sigma\in\S_n$.
\end{lem}

\begin{proof}
On démontre ceci par récurrence sur $n$. Il suffit ainsi d'adjoindre un point supplémentaire de $P_n$ et utiliser le fait que $\delta_{n+1}\circ\delta_{n+2}=0$ pour décomposer tout élément de $H_n(P_n)$ en élément de la forme:
$$\varphi_{\sigma}\delta_{n+1}\left[(*,...,*),(*,...,*,b_n),...,(*,...,*,b_n)\right].$$
Ainsi partant d'un élément $\delta_{n+1}\left[(\alpha_1^0,...,\alpha_n^0),....,(\alpha_1^n,...,\alpha_n^n)\right]$, on adjoint les points $(\alpha_1^i,...,\alpha_{n-1}^i,\alpha_n^j)\in P_n$ pour $0\leq i<j\leq n$. Les éléments obtenues sont soit de la forme désirée soit transverse et donc nul dans le quotient.\par
Ceci permet de compléter la démonstration du corolaire. En effet, si une combinaison linéaire de $T_n^{\sigma}(a_j,b_j)$ est transverse alors il s'écrit comme combinaison linéaire d'éléments transverses: $\varphi_{\gamma}T_n^{\sigma}(a,b)$ pour $\gamma\in\Gamma^n,a,b\in\pte$ par projection successive sur les coordonnées.
\end{proof}

\begin{thm}
L'idéal $\mathcal{A}_n$ est de type fini.
\end{thm}

\begin{proof}
L'idée est à nouveau de réduire $\mathcal{A}_n$ à l'annulateur d'une famille finie de simplexe proche du simplexe fondamental $T_n=T_n(i\infty,0)$. On a vu que $\mathcal{A}_n$ est le noyau de l'application $\Z$-linéaire surjective:
$$\Theta_n:\Z[G_n]\to H_n(P_n)^0, [\gamma,\rho] \mapsto \gamma^{-1}T_n^{\rho}.$$
Or étant donnée que nous savons décrire l'image $H_n(P_n)^0$. Il suffit de remarquer que les annulations dans ce groupe proviennent de celle des simplexes à coordonnées des sommets dans $B=\{i\infty,0,1,-1\}$. Pour cela, nous faisons une descente selon chacune des coordonnées sur la hauteur $h$ défini sur $\pte$. Ces simplexes sont alors simplement les images:
$$\Theta_n([(U^{a_1}S^{\delta_1},...,U^{a_n}S^{\delta_n}),\sigma])\text{ pour }(a_j)\in(\Z/3)^n,(\delta_j)\in (\Z/2)^n\text{ et }\sigma\in\S_n.$$
Les combinaisons linéaire de ces chaînes sont pris avec des coefficients parmi $\{-1,0,1\}$ traduisant la présence et l'orientation dans le recouvrement. Ceci démontre bien que $\mathcal{A}_n$ est de type fini et on peut majorer le nombre de générateurs par $3^{n!6^n}$. On peut améliorer cette majoration grâce au travail sur les conjugués, on obtient que $\delta_1=...=\delta_n$ et ainsi le nombre de générateur est au plus $3^{2n!3^n}$.
\end{proof}

\begin{thm}
Pour toute famille de poids $k=(k_1,...,k_n)$, on a l'égalité des $\Q$-espaces vectoriels:
\begin{equation}
\left(V_k^{\Q}\right)^{\S_n}[\mathcal{A}_n]=\left(V_k^{\Q}\right)^{\S_n}[\mathcal{A}(k)]+\sum_{j,\sigma} \left(V_k^{\Q}\right)^{\S_n}[\mathcal{A}_{n-1}[j,\sigma]].
\end{equation}
\end{thm}

\begin{proof}
On a déjà vu que $\mathcal{A}_n\subset\mathcal{A}(k)$ en (\ref{incl:1}) puis que $\mathcal{A}_n\subset\mathcal{A}_{n-1}[j,\sigma]$ pour tout $j,\sigma$ en (\ref{incl:2}). On en déduit l'inclusion :
$$\left(V_k^{\Z}\right)^{\S_n}[\mathcal{A}_n]\supset\left(V_k^{\Z}\right)^{\S_n}[\mathcal{A}(k)]+\sum_{j,\sigma} \left(V_k^{\Z}\right)^{\S_n}[\mathcal{A}_{n-1}[j,\sigma]].$$
Cherchons à démontrer l'égalité des dimensions des $\Q$-espaces vectoriels engendrés par les deux termes de l'inclusion. On commence par utiliser la dualité donnée par le produit bilinéaire symétrique et $G_n$-invariant:
$$\left(V_k^{\Q}\right)^{\S_n}\times\left(V_k^{\Q}\right)^{\S_n}\to \Q, \left((\otimes_{j=1}^n P_j^{\sigma})_{\sigma\in\S_n},(\otimes_{j=1}^n Q_j^{\sigma})_{\sigma\in\S_n}\right)\mapsto \sum_{\sigma\in\S_n}\big|\prod_{j=1}^n [P_j^{\sigma},Q_j^{\sigma}]\big|,$$
induit par celui de la Remarque \ref{prodvk}. Il met en dualité les espaces $\left(V_k^{\Q}\right)^{\S_n}/\left(V_k^{\Q}\right)^{\S_n}|_{\mathcal{A}_n}$ et $\left(V_k^{\Q}\right)^{\S_n}[{\mathcal{A}_n}]$ et  permet de déduire:
$$\dim_{\Q}\left(\left(V_k^{\Q}\right)^{\S_n}|_{\mathcal{A}_n}\right)+\dim_{\Q}\left(\left(V_k^{\Q}\right)^{\S_n}[{\mathcal{A}_n}]\right)=\dim_{\Q}\left(\left(V_k^{\Q}\right)^{\S_n}\right).$$\par
D'autre part, on utilise la suite exacte des $\Q$-espaces vectoriels:
$$0\to \left(V_k^{\Q}\right)^{\S_n}|_{\mathcal{A}_n}\to \left(V_k^{\Q}\right)^{\S_n} \stackrel{P\mapsto T_n\otimes P}{\to} H_n(P_n)^0 \otimes_{\Q[G_n]} \left(V_k^{\Q}\right)^{\S_n}\to 0.$$
Elle est obtenue par tensorisation par le $\Q[G_n]$-module simple $\left(V_k^{\Q}\right)^{\S_n}$ de la suite exacte associée à l'application $\Theta_n$. La simplicité du $\Q[G_n]$-module $\left(V_k^{\Q}\right)^{\S_n}$ est due à l'irréductibilité des représentations $\Gamma\to GL(V_{k_j})$ pour tout $j$. On utilise alors le résultat du Corollaire \ref{formH}, pour obtenir:
\begin{align*}
H_n(P_n)^0&=\left[\left(\sum_{j,g,\sigma}\varphi_{\sigma}\varphi_j^g(\Z[P_{n-1}^{n}])\right)\cap\Z[P_n^n]^0\right]/\sum_{j,g,\sigma}\varphi_{\sigma}\varphi_j^g(\Z[P_{n-1}^{n}]^0)\\
&=\left\{d\omega=(\varphi_{\sigma}\varphi_j)^*\omega=0\text{ pour tout }j,\sigma\right\}^{\bot}/\sum_{j,\sigma}\left\{d\omega=(\varphi_{\sigma}\varphi_j)^*\omega=0\right\}^{\bot}.
\end{align*}
Son produit tensoriel par $\left(V_k^{\Q}\right)^{\S_n}$ est ainsi:
$$H_n(P_n)^0 \otimes_{\Q[G_n]} \left(V_k^{\Q}\right)^{\S_n}=\left(V_k^{\Q}\right)^{\S_n}[{\mathcal{A}(k)}]+\sum_{j,\sigma} \left(V_k^{\Q}\right)^{\S_n}[\mathcal{A}_{n-1}[j,\sigma]].$$
En effet, on a les résultat de dualité donnant d'une part, d'après le Théorème \ref{decompomegathm} :
\begin{align*}
\mathcal{A}(k)&=\{a\in G_n\text{ tel que }\mathcal{T}_n|_a\in\Omega_k^{\bot}\},\\
\text{et }\Omega_k&=\left\{d\omega=(\varphi_{\sigma}\varphi_j)^*\omega=0\text{ pour tout }j,\sigma\right\}\otimes_{\Q[G_n]}\left(V_k^{\Q}\right)^{\S_n}.
\end{align*}
Et d'autre part, pour tout $0\leq j\leq n$ et $\sigma\in\S_n$, on a:
\begin{align*}
\mathcal{A}_{n-1}[j,\sigma]&=\{a\in G_n\text{ tel que }\mathcal{T}_n|_a\in\Omega_k(j,\sigma)^{\bot}\},\\
\text{où }\Omega_k(j,\sigma)&=\left\{d\omega=(\varphi_{\sigma}\varphi_j)^*\omega=0\right\}\otimes_{\Q[G_n]}\left(V_k^{\Q}\right)^{\S_n}.
\end{align*}
Enfin, cette suite exacte donne l'égalité des dimensions:
$$\dim_{\Q}\left(\left(V_k^{\Q}\right)^{\S_n}|_{\mathcal{A}_n}\right)+\dim_{\Q}\left(H_n(P_n)^0 \otimes_{\Q[G_n]} \left(V_k^{\Q}\right)^{\S_n}\right)=\dim_{\Q}\left(\left(V_k^{\Q}\right)^{\S_n}\right).$$
Ceci démontre l'égalité des dimensions recherchée.
\end{proof}

\begin{coro}\label{corofin}
Le $\Q$-espace vectoriel $(V_{k}^{\Q})^{\S_n}[\mathcal{A}_n]$ est le plus petit sous-espace de $(V_{k}^{\Q})^{\S_n}$ contenant $ME_k^{\Q}$ et dont l'extension au corps des complexes contient $MP(k)$.
\end{coro}

\chapter{Généralisations diverses}
\minitoc
\section{Généralisation aux formes de poids $2$}

\subsection{Série génératrice associée aux formes modulaires de poids $2$}

On précise ici l'exemple (\ref{defw2}) pour une forme modulaire parabolique de poids $2$
associée à un groupe de congruence $\Gamma$.  Soit $f\in S_2(\Gamma)$. Posons:
$$V_0=\Gamma\backslash SL_2(\Z).$$
Et on rappelle la définition de la famille de $1$-formes holomorphes de $\H$ indexée par $V_0$:
$$\Omega_f(z)=\sum_{g\in V_0}f|_g(z)dz [g]\in\Omega_{par}^1(\H,\C)^{V_0}.$$

Soit $\gamma\in PSL_2(\Z)$. On définit une action à droite sur $\Ser(V_0)$ par son action sur les indéterminées:
$$[g]\mapsto [g\gamma] \text{ pour toute classe }g\in V_0.$$

\begin{prop}
L'action construite sur $\Ser(V_0)$ est duale à celle construite sur les $J_a^b(\Omega)$, c'est à dire:
\begin{equation}
J_{\gamma}(\Omega_f)=J(\Omega_f|_{\gamma^{-1}}),\text{ pour tout }\gamma\in PSL_2(\Z).
\end{equation}
\end{prop}

\begin{proof}
On effectue le calcul, pour $\gamma\in PSL_2(\Z)$:
\begin{align*}
J_{\gamma}(\Omega_f)&=J_{\gamma.i\infty}^{\gamma.0}\left(\sum_{g\in\Gamma\backslash SL_2(\Z)} f|_g(z)dz[g]\right)\\
&=J_{i\infty}^{0}\left(\sum_{g\in\Gamma\backslash SL_2(\Z)}f|_{g\gamma}[g]\right)\\
&=J_{i\infty}^{0}\left(\sum_{g\in\Gamma\backslash SL_2(\Z)}f|_{g}[{g\gamma^{-1}}]\right)=J(\Omega_f|_{\gamma^{-1}}).
\end{align*}
\end{proof}

Ceci est généralisable pour une famille quelconque de formes de poids $2$.\\
Soient $\Gamma_1,...,\Gamma_n$ une famille de groupes de congruences de $SL_2(\Z)$.
Soient $f_1,...,f_n$ des formes modulaires paraboliques de poids $2$ associées respectivement aux groupes de congruences $\Gamma_1,...,\Gamma_n$. Posons:
$$V_n=\coprod_{i=1}^n \Gamma_i\backslash SL_2(\Z)=\{(g_i,i)\text{ tel que }1\leq i\leq n\text{ et }g_i\in\Gamma_i\backslash\G\}.$$
Définissons une famille de $1$-formes holomorphes indexée par $V_n$:
\begin{align*}
\Omega_{f_1,...,f_n}(z)&=\sum_{i=1}^n \Omega_{f_i}(z) Y_i\\
&=\sum_{i=1}^n \sum_{g_i\in \Gamma_i\backslash SL_2(\Z)} f_i|_{g_i}(z)dz X_{g_i}Y_i.
\end{align*}

On définit ainsi la série en les indéterminés $A_{(g_i,i)}=X_{g_i}Y_i$ pour $(g_i,i)\in V_n$:
\begin{equation}\label{serformp2}
J(f_1,...,f_n)=J(\Omega_{f_1,...,f_n}).
\end{equation}
Les actions sur $Ser\left(\Gamma_i\backslash SL_2(\Z)\right)$ donne une action diagonale sur $Ser\left(\coprod_{i=1}^n\Gamma_i\backslash SL_2(\Z)\right)$:
$$\left(X_{g_i}Y_i\right)|_{\gamma}=X_{g_i\gamma}Y_i,\text{ pour tout }\gamma\in PSL_2(\Z)\text{ et }(g_i,i)\in V.$$

Elle est alors à nouveau compatible à l'action de $PSL_2(\Z)$ sur $\H$:
\begin{equation}
J_{\gamma}(f_1,...,f_n)=J(f_1,...,f_n)|_{\gamma},\text{ pour tout }\gamma\in PSL_2(\Z).
\end{equation}

Et les relations données par la structure se transpose donnant notamment:

\begin{thm}[Relations de Manin non commutatives]
Pour toute famille de formes modulaires paraboliques $f=(f_1,...,f_n)$ de poids $2$, on dispose des relations:
$$J(f)J_S(f)=J_S(f)J(f)=1,$$
$$J_{U^2}(f)J_U(f)J(f)=J_U(f)J(f)J_{U^2}(f)=J(f)J_{U^2}(f)J_U(f)=1.$$
\end{thm}

\subsection{Vecteurs des multipériodes des formes de poids $2$}

Soient $\Gamma\subset SL_2(\Z)$ un groupe de congruence et $f\in S_2(\Gamma)$ une forme modulaire holomorphe pour $\Gamma$ de poids $2$.
Définissons le groupe abélien librement engendré par les classes de $\Gamma\backslash SL_2(\Z)$:
\begin{equation}
V_{\Gamma}^{\Z}=\Z[\Gamma\backslash SL_2(\Z)].
\end{equation}
Nous noterons $\sum \lambda_g [g]$ ses éléments. Pour tout anneau commutatif $A$, nous noterons $V_{\Gamma}^A=V_{\Gamma}^{\Z}\otimes_{\Z} A$ et $V_{\Gamma}=V_{\Gamma}^{\C}$. Le $\C$-espace-vectoriel $V_{\Gamma}$ dispose ainsi d'une structure réelle donnée par $V_{\Gamma}^{\R}$ et d'une conjugaison complexe.\par
Ces groupes abéliens sont munis d'une action à droite de $PSL_2(\Z)$ donnée par $[g]|_{\gamma}=[g\gamma]$.

\begin{defi}
Les nombres complexes suivant sont appelées \textit{périodes} de la forme $f$:
$$\int_{i\infty}^0 f|_{g}(z)dz\text{ pour }g\text{ parcourant }\Gamma\backslash SL_2(\Z).$$
On définit le \textit{vecteur des périodes} de $f$ par:
\begin{equation}
P_f=\sum_{g\in\Gamma\backslash SL_2(\Z)}\int_{i\infty}^0 f|_{g}(z)dz[g]\in V_{\Gamma}.
\end{equation}
\end{defi}

Ce vecteur des périodes vérifie les relations de Manin et caractérise la forme $f$. En effet, l'application $\C$-linéaire suivante est injective:
\begin{equation}
S_2(\Gamma)\to W_{\Gamma}=\{v\in V_{\Gamma};v|_{1+S}=v|_{1+U+U^2}=0\},\quad f\mapsto P_f.
\end{equation}

On généralise ainsi ce formalisme pour une famille de formes de poids $2$.\par
Soit $n\geq 1$. Soient $\Gamma_1,..,\Gamma_n$ des groupes de congruence de $SL_2(\Z)$. Considérons une famille $f_1,...,f_n$ de formes modulaires paraboliques de poids $2$ pour les groupes respectifs $\Gamma_1,...,\Gamma_n$, c'est à dire $f_j\in S_2(\Gamma_j)$ pour $1\leq j\leq n$.\par
Définissons le groupe abélien librement engendré par les $n$-upplets de classes de $\prod_{j=1}^n \Gamma_j\backslash SL_2(\Z)$:
\begin{equation}
V_{\Gamma_1,...,\Gamma_n}^{\Z}=\Z\left[\prod_{j=1}^n \Gamma_j\backslash SL_2(\Z)\right].
\end{equation}
On peut librement l'identifier à $\otimes_{j=1}^n V_{\Gamma_j}^{\Z}$ via $[g_1,...,g_n]\cong [g_1]\otimes...\otimes[g_n]$. 
De plus, nous définissons de manière analogue pour tout anneau commutatif $A$, $V_{\Gamma_1,...,\Gamma_n}^{A}=V_{\Gamma_1,...,\Gamma_n}^{\Z}\otimes_{\Z} A.$
Ceci muni à nouveau l'extension au corps des complexes, noté simplement $V_{\Gamma_1,...,\Gamma_n}$, d'une structure réelle fournie par $V_{\Gamma_1,...,\Gamma_n}^{\R}$.\par

On peut étudier, comme pour les formes de niveau $1$, les termes de différents degrés de la série formelle vue en (\ref{serformp2}):
$$J(f_1,...,f_n)=\sum_{N\geq 0}\int_{0<t_1<...<t_N}\Omega_{f_1,...,f_n}(it_1)\wedge...\wedge\Omega_{f_1,...,f_n}(it_N).$$
Ceci permet de mettre en évidence des vecteurs de $\left(\prod_{j=1}^n V_{\Gamma_j}\right)^{\otimes N}$ que l'on nommera vecteur des multipériodes:

\begin{defi}
On appelle \textit{multipériodes de longueur} $n$ de la famille des formes modulaires paraboliques $(f_j)_{1\leq j\leq n}$ les nombres complexes :
\begin{equation}
\Lambda(f_1|_{g_1},...,f_n|_{g_n};1,...,1)=\int_{0<t_1<...<t_n}f_1|_{g_1}(it_1)...f_n|_{g_n}(it_n)dt_1...dt_n,
\end{equation}
pour toutes les classes $g_j\in \Gamma_j\backslash SL_2(\Z), (1\leq j \leq n)$.\par
On peut les regrouper dans le \textit{vecteur des multipériodes de longueur} $n$ de $(f_1,...,f_n):$
\begin{equation}
P_{f_1,...,f_n}=\sum_{\substack{g_j\in\Gamma_j\backslash SL_2(\Z)\\ 1\leq j\leq n}} \Lambda(f_1|_{g_1},...,f_n|_{g_n};1,...,1)[g_1,...,g_n]\in V_{\Gamma_1,...,\Gamma_n}.
\end{equation}
Pour toute permutation $\sigma\in\S_n$, nous disposons du \textit{permuté du vecteur des multipériodes} :
\begin{equation}
P^{\sigma}_{f_1,...,f_n}=\phi_{\sigma}(P_{f_{\sigma(1),...,\sigma(n)}}),
\end{equation}
où $\phi_{\sigma}:\otimes_{j=1}^n V_{\Gamma_{\sigma(j)}}^{\Z}\to\otimes_{j=1}^n V_{\Gamma_{j}}^{\Z},[g_{\sigma(1)},...,g_{\sigma(n)}]\mapsto [g_1,...,g_n]$.
\end{defi}

On dispose à nouveau d'une écriture unifiée dépendante du même $n$-cycle $\tau_n\in M_n^{pte}(\H^n,\Z)$ et des formes différentielles invariantes.
Pour toute forme modulaire parabolique holomorphe $f\in S_2(\Gamma)$, posons:
$$\omega_f(z)=\sum_{g\in \Gamma\backslash SL_2(\Z)} f|_g(z)dz[g]\in \Omega^1_{par}(\H,\C)\otimes_{\C}V_{\Gamma}=\Omega^1_{par}(\H,V_{\Gamma}).$$

\begin{prop}
Soit $\sigma\in\S_n$. Alors l'accouplement $\Omega^n_{par}(\H^n,\C)\times M_n^{pte}(\H^n,\Z)\to\C$ de l'intégration donne la formule:
\begin{equation}
P^{\sigma}_{f_1,...,f_n}=\langle\omega_{f_1}\wedge...\wedge\omega_{f_n},\tau_n^{\sigma}\rangle.
\end{equation}
\end{prop}

La forme différentielle $\omega_f$ vérifie une propriété liée à la modularité analogue à celle du cas de niveau $1$. Pour tout $f\in S_2(\Gamma)$ et $\gamma\in PSL_2(\Z)$, on a:
\begin{multline}\label{modw2}
\gamma^*\omega_f(z)=\sum_{g\in \Gamma\backslash SL_2(\Z)} f|_g(\gamma z)d(\gamma z)[g]\\
=\sum_{g\in \Gamma\backslash SL_2(\Z)} f|_{g\gamma}(z)dz[g]=\sum_{g\in \Gamma\backslash SL_2(\Z)} f|_g(z)dz[g\gamma^{-1}]=\omega_f(z)|_{\gamma^{-1}}.
\end{multline}
Ceci peut s'écrire de manière formelle:
\begin{equation}
\omega_f\in\Omega_{par}^1(\H,V_{\Gamma})^{P\G}=\Omega_{par}^1(\H,\C)\otimes_{\C[P\G]}V_{\Gamma}.
\end{equation}
On peut en déduire à nouveau la structure de $PSL_2(\Z)$-morphisme et donc les relations de mélange et de Manin sur les vecteurs des multipériodes. En effet, l'application suivante est un morphisme de $\Z[P\G^n\rtimes \S_n]$-modules:
\begin{equation}
Ser\left(\coprod_{1\leq j\leq n} \Gamma_j\backslash \G\right)\to \left(V_{\Gamma_1,...,\Gamma_n}\right)^{\S_n},\quad J(\Omega_{f_1,...,f_n})\mapsto \left(P_{f_1,...,f_n}^{\sigma}\right)_{\sigma\in\S_n}.
\end{equation}

\begin{thm}[Relations de mélange et de Manin généralisées]\ \\
1) Pour toute décomposition entière $a+b=n$ de la longueur, on dispose de la formule de mélange:
\begin{equation}
\sum_{\sigma\in\mathfrak{S}_{a,b}}P_{f_1,...,f_n}^{\sigma}=P_{f_1,...,f_a}\otimes P_{f_{a+1},...,f_{a+b}}.
\end{equation}
2) On dispose de même des relations de Manin généralisées:
\begin{align}
\sum_{\substack{a,b\geq 0\\ a+b=n}}P_{f_1,...,f_a}|_{(S,...,S)}&\otimes P_{f_{a+1},...,f_{a+b}}=0,\\
\sum_{\substack{a,b,c\geq 0\\ a+b+c=n}}P_{f_1,...,f_a}|_{(U^2,...,U^2)}&\otimes P_{f_{a+1},...,f_{a+b}}|_{(U,...,U)}\otimes P_{f_{a+b+1},...,f_{a+b+c}}=0.
\end{align}
où on pose $P_{\emptyset}=1$ pour simplifier les notations.
\end{thm}

\begin{proof}
On dispose de la relation de mélange non commutatives sur la série génératrice, pour mémoire:
$$\Delta J(\Omega_{f_1,...,f_n}) = J(\Omega_{f_1,...,f_n})\otimes J(\Omega_{f_1,...,f_n}).$$
Ainsi que des relations de Manin non commutatives, à savoir:
$$J(\Omega_{f_1,...,f_n})|_S.J(\Omega_{f_1,...,f_n})=1\text{ et }J(\Omega_{f_1,...,f_n})|_{U^2}.J(\Omega_{f_1,...,f_n})|_U.J(\Omega_{f_1,...,f_n})=1.$$
Elle se traduise sur les vecteurs des multipériodes par les formules du théorème après projection par:
$$\Ser(V_n)\to \Ser(V_n)_n \cong V_{\Gamma_1,...,\Gamma_n}^{\S_n},$$
qui est un $\Z[PSL_2(\Z)^n\rtimes \S_n]$-morphisme par construction.
\end{proof}

\subsection{Relations vérifiées par les vecteurs de périodes}

Nous donnons une description rapide des resultats existants sur les vecteurs des périodes.\\
Définissons un sous-groupe de $V_{\Gamma}^{\Z}$, pour tout sous-groupe de congruence $\Gamma\subset SL_2(\Z)$, par:
\begin{equation}
W_{\Gamma}^{\Z}=\{v\in V_{\Gamma}^{\Z}\text{ tel que }v|_{1+S}=v|_{1+U+U^2}=0\}.
\end{equation}
Il contient la droite associée aux formes non paraboliques: 
\begin{equation}
E_{\Gamma}^{\Z}=\langle [\Gamma]-[\Gamma S]\rangle_{\Z}\subset W_{\Gamma}^{\Z}.
\end{equation}
De plus son extension au corps des complexes contient les vecteurs des périodes:
\begin{equation}
\Per_{\Gamma}=\{P_f\text{ pour }f\in S_2(\Gamma)\}\subset W_{\Gamma}=W_{\Gamma}^{\Z}\otimes\C.
\end{equation}\par

D'autre part, lorsque $\varepsilon\Gamma\varepsilon=\Gamma$, on peut décomposer $V_{\Gamma}=V_{\Gamma}^+\oplus V_{\Gamma}^-$ où:
\begin{equation}
V_{\Gamma}^{\pm}=\{v\in V_{\Gamma}\text{ tel que }c_{\varepsilon}(v)=\pm v\},
\end{equation}
et $c_{\varepsilon}([g])=[\varepsilon g\varepsilon]$ pour tout $g\in\Gamma\backslash SL_2(\Z)$.\\
On peut alors scinder $W_{\Gamma}^{\Z}$ suivant cette décomposition car $\varepsilon\mathcal{I}_1\varepsilon=\mathcal{I}_1$:
$$W_{\Gamma}^{\Z}=W_{\Gamma}^{+,\Z}\oplus W_{\Gamma}^{-,\Z},\text{ où }W_{\Gamma}^{\pm,\Z}=W_{\Gamma}^{\Z}\cap V_{\Gamma}^{\pm}.$$
 Et en particulier, on a $E_{\Gamma}^{\Z}\subset W_{\Gamma}^{+,\Z}$.

\begin{prop}
On dispose de la décomposition:
\begin{equation}
W_{\Gamma}=\Per_{\Gamma}\oplus \overline{\Per_{\Gamma}}\oplus E_{\Gamma}
\end{equation}
De plus, si $\varepsilon\Gamma\varepsilon=\Gamma$ alors:
\begin{equation}
W_{\Gamma}^+=\Per_{\Gamma}^+\oplus E_{\Gamma}\text{ et }W_{\Gamma}^-=\Per_{\Gamma}^-.
\end{equation}
\end{prop}

Et ceci permet de déduire le caractère minimal suivant des relations de Manin:

\begin{coro}
Le $\Q$-espace vectoriel $W_{\Gamma}^{\Q}$ est le plus petit sous-espace-vectoriel de $V_{\Gamma}^{\Q}$ contenant $E_{\Gamma}^{\Q}$ tel que son extension au corps des complexes contient $\Per_{\Gamma}$.
\end{coro}

\subsection{Relations vérifiées par les vecteurs des bipériodes}

Soient $\Gamma_1$ et $\Gamma_2$ des groupes de congruence de $SL_2(\Z)$.
Soient $f_1$ et $f_2$ des formes modulaires paraboliques de poids $2$ pour les groupes respectifs $\Gamma_1$ et $\Gamma_2$.
Le \textit{vecteur des bipériodes} de $f_1$ et $f_2$ est défini par:
\begin{equation}
P_{f_1,f_2}=\sum_{g_1\in\Gamma_1\backslash SL_2(\Z)}\sum_{g_2\in\Gamma_2\backslash SL_2(\Z)}\int_{0}^{\infty}\int_{0}^{t_1}f_1|_{g_1}(it_1)f_2|_{g_2}(it_2)dt_1dt_2[g_1,g_2]\in V_{\Gamma_1,\Gamma_2}.
\end{equation}
Nous noterons $\Per_{\Gamma_1,\Gamma_2}$ le $\C$-espace vectoriel formé par ces vecteurs.

Notons:
\begin{equation}
W_{\Gamma_1,\Gamma_2}^{\Z}=\{v\in V_{\Gamma_1,\Gamma_2}^{\Z}\text{ tel que }v|_{\gamma}=0\text{ pour tout }\gamma\in \widetilde{\mathcal{I}_2}\}
\end{equation}
Définissons le sous-groupe des vecteurs vérifiant les relations de Manin diagonales:
\begin{equation}
V_{\Gamma_1,\Gamma_2}^{\Z}[I_D]=\{v\in V_{\Gamma_1,\Gamma_2}^{\Z}\text{ tel que }v|_{(1,1)+(S,S)}=v|_{(1,1)+(U,U)+(U^2,U^2)}=0\}.
\end{equation}
Puis posons:
\begin{equation}
E_{\Gamma_1,\Gamma_2}^{\Z}=\left(W_{\Gamma_1}^{\Z}\otimes (V_{\Gamma_2}^{\Z})^{T}\right)+ V_{\Gamma_1,\Gamma_2}^{\Z}[I_D]+ \left((V_{\Gamma_1}^{\Z})^{US}\otimes W_{\Gamma_2}^{\Z}\right).
\end{equation}
On démontre le théorème suivant précisant l'inclusion:
\begin{equation}
\Per_{\Gamma_1,\Gamma_2}\subset W_{\Gamma_1,\Gamma_2}=W_{\Gamma_1,\Gamma_2}^{\Z}\otimes_{\Z}\C.
\end{equation}

\begin{thm}\label{thmp2}
Le $\Q$-espace vectoriel $W_{\Gamma_1,\Gamma_2}^{\Q}=W_{\Gamma_1,\Gamma_2}^{\Z}\otimes_{\Z}\Q$ est le plus petit sous-espace-vectoriel $W$ de $V_{\Gamma_1,\Gamma_2}^{\Q}$ tel que:
\begin{equation}
\Per_{\Gamma_1,\Gamma_2}\subset W\otimes \C \text{ et } E_{\Gamma_1,\Gamma_2}^{\Q}\subset W.
\end{equation}
\end{thm}

Pour cela, on recherche à calculer l'idéal à droite de $\Z[PSL_2(\Z)^2]$ défini par:
\begin{equation}
\mathcal{J}_{\Gamma_1,\Gamma_2}=\{\gamma\in\Z[PSL_2(\Z)^2]\text{ tel que }P_{f_1,f_2}|_{\gamma}=0\text{ pour tout }(f_1,f_2)\in S_2(\Gamma_1)\times S_2(\Gamma_2)\}.
\end{equation}

Posons pour tout couple $(f_1,f_2)\in S_2(\Gamma_1)\times S_2(\Gamma_2)$:
$$\omega_{f_1,f_2}=\sum_{g_1\in\Gamma_1\backslash SL_2(\Z)}\sum_{g_2\in\Gamma_2\backslash SL_2(\Z)}\left(f_1|_{g_1}(z_1)dz_1\right)\wedge\left(f_2|_{g_2}(z_2)dz_2\right)[g_1,g_2]
\cong \omega_{f_1}\wedge\omega_{f_2}.$$
La formule liant $P_{f_1,f_2}$ à $\tau_2$ permet de ramener les calculs à ceux déjà réalisés:
\begin{equation}
P_{f_1,f_2}=\langle\omega_{f_1,f_2},\tau_2\rangle.
\end{equation}

\begin{prop}
Soit $(\gamma_1,\gamma_2)\in PSL_2(\Z)^2$.\\
La famille de $2$-formes holomorphes $\omega_{f_1,f_2}$ vérifie la propriété d'invariance:
\begin{equation}
\omega_{f_1,f_2}(\gamma_1.z_1,\gamma_2.z_2)|_{(\gamma_1,\gamma_2)}=\omega_{f_1,f_2}(z_1,z_2)
\end{equation}
L'action sur le vecteur des bipériodes se transpose sur l'homologie de $\H^2$:
\begin{equation}
P_{f_1,f_2}|_{(\gamma_1,\gamma_2)}=\langle\omega_{f_1,f_2},(\gamma_1^{-1},\gamma_2^{-1})\tau_2\rangle.
\end{equation}
\end{prop}

\begin{proof}
On a : $\omega_{f_1,f_2}(z_1,z_2)=\omega_{f_1}(z_1)\wedge\omega_{f_2}(z_2)$ donc il suffit d'obtenir l'invariance de l'action de $\gamma\in PSL_2(\Z)$ sur une famille de $1$-formes holomorphes. Or on a déjà:
$$\omega_f(\gamma.z)|_{\gamma}=\sum_{g\in\Gamma\backslash PSL_2(\Z)}f|_g(\gamma.z)d(\gamma.z)[g\gamma]
=\sum_{g\in\Gamma\backslash PSL_2(\Z)}f|_{g\gamma}(z)dz[g\gamma]=\omega_f(z).$$
Une démonstration similaire à celle du cas du niveau $1$ donne:
$$P_{f_1,f_2}|_{(\gamma_1,\gamma_2)}=\langle\omega_{f_1,f_2}(z_1,z_2)|_{(\gamma_1,\gamma_2)},\tau_2\rangle
=\langle\omega_{f_1,f_2}(\gamma_1^{-1}.z_1,\gamma_2^{-1}.z_2),\tau_2\rangle=\langle\omega_{f_1,f_2},(\gamma_1^{-1},\gamma_2^{-1})\tau_2\rangle.$$
\end{proof}

\begin{coro}
On a l'inclusion des idéaux:
$\mathcal{I}_2\subset \widetilde{\mathcal{J}_{\Gamma_1,\Gamma_2}}.$
\end{coro}

\begin{proof}
La proposition précédente démontre que $\omega_{f_1,f_2}$ est un élément de l'espace des $2$-formes différentielles harmoniques sur $\H^2$, à valeurs dans $V_{\Gamma_1,\Gamma_2}$, nulles en l'infini et invariantes par $PSL_2(\Z)^2$:  
$$\Omega_{par}^2(\H^2,V_{\Gamma_1,\Gamma_2})^{P\G^2}=\Omega_{par}^2(\H^2,\C)\otimes_{\C[PSL_2(\Z)^2]}V_{\Gamma_1,\Gamma_2}.$$
Définissons les sous-espaces vectoriels suivant de ce dernier:
\begin{align*}
\Omega^{\holo}_{\Gamma_1,\Gamma_2}&=\{\omega_{f_1,f_2}\text{ pour }f_1\otimes f_2\in S_2(\Gamma_1)\otimes S_2(\Gamma_2)\},\\
\text{et }\overline{\Omega^{\holo}_{\Gamma_1,\Gamma_2}}&=\{\overline{\omega_{f_1,f_2}}\text{ pour }f_1\otimes f_2\in S_2(\Gamma_1)\otimes S_2(\Gamma_2)\}.
\end{align*}
On dispose ainsi de l'espace : $\Omega_{\Gamma_1,\Gamma_2}=\Omega_{\Gamma_1,\Gamma_2}^{\holo}\oplus \overline{\Omega_{\Gamma_1,\Gamma_2}^{\holo}}\subset \Omega_{par}^2(\H^2,V_{\Gamma_1,\Gamma_2})^{P\G^2}$.\par

Pour démontrer le corollaire, il suffit de déterminer:
\begin{equation}
\left(\Omega_{\Gamma_1,\Gamma_2}\right)^{\bot}=\{c\in M_2^{pte}(\H^2,\Z);\langle \omega,c\rangle=0\text{ pour tout }\omega\in\Omega_{\Gamma_1,\Gamma_2}\}.
\end{equation}
En effet la proposition précédente permet d'écrire :
\begin{equation}
\widetilde{\mathcal{J}_{\Gamma_1,\Gamma_2}}=\{\gamma\in\Z[PSL_2(\Z)^2]\text{ tel que }\gamma.\tau_2\in\left(\Omega_{\Gamma_1,\Gamma_2}\right)^{\bot}\}.
\end{equation}
Et la définition de $\mathcal{I}_2$ est:
$$\mathcal{I}_2=\{\gamma\in\Z[PSL_2(\Z)^2]\text{ tel que }\delta_2(\gamma.\tau_2)\in H^0+D^0+V^0 \}.$$
Or comme dans le cas du niveau $1$, on donne une nouvelle description de $\Omega_{\Gamma_1,\Gamma_2}$:
$$\Omega_{\Gamma_1,\Gamma_2}=\{\omega\in\Omega_{par}^2(\H^2,V_{\Gamma_1,\Gamma_2})^{P\G^2}\text{ tel que }d\omega=\varphi_0^*\omega=\varphi_1^*\omega=\varphi_2^*\omega=0\}.$$
Ceci permet d'écrire:
$$\{d\omega=\varphi_0^*\omega=0\}^{\bot}+\{d\omega=\varphi_1^*\omega=0\}^{\bot}+\{d\omega=\varphi_2^*\omega=0\}^{\bot}\subset\Omega_{\Gamma_1,\Gamma_2}^{\bot},$$
équivalent pour tout élément $c\in M_2^{pte}(\H^2,\Z)$ à:
$$\delta_2 c\in H^0+D^0+V^0\Rightarrow c\in \Omega_{\Gamma_1,\Gamma_2}^{\bot}.$$
Ceci démontre bien l'inclusion des idéaux $\mathcal{I}_2\subset \widetilde{\mathcal{J}_{\Gamma_1,\Gamma_2}}$.
\end{proof}

Nous avons désormais tous les éléments pour la démonstration du Théorème \ref{thmp2}:

\begin{proof}
En effet, la représentation de $PSL_2(\Z)^2$ fournie par $V_{\Gamma_1,\Gamma_2}$ est irréductible. Ainsi le morphisme de $\Q[P\G^2]$-modules est surjective par le Lemme de Schur:
$$\Q[P\G^2]\to End_{\Q}(V_{\Gamma_1,\Gamma_2}^{\Q}),\quad \gamma\mapsto\left(v\mapsto v|_{\gamma}\right).$$
Ainsi tous les sous-$\Q$-espaces vectoriels de $V_{\Gamma_1,\Gamma_2}^{\Q}$ s'écrivent comme l'annulateur d'un élément de $\Q[P\G^2]$. 
Il suffit ainsi d'obtenir $W_{\Gamma_1,\Gamma_2}^{\Q}=V_{\Gamma_1,\Gamma_2}^{\Q}[\mathcal{J}_{\Gamma_1,\Gamma_2}]+E_{\Gamma_1,\Gamma_2}^{\Q}$ pour démontrer la minimalité comme $\Q$-espace vectoriel de $W_{\Gamma_1,\Gamma_2}^{\Q}$.\par  
Pour cela, on observe que l'application surjective: $\Theta_2:\Z[PSL_2(\Z)^2]\to H_2(P_2)^0$ dont le noyau défini $\mathcal{I}_2$ donne l'exactitude de la suite:
$$0\to V_{\Gamma_1,\Gamma_2}^{\Q}|_{\mathcal{I}_2}\to V_{\Gamma_1,\Gamma_2}^{\Q} \to \left[H_2(P_2)^0\otimes_{\Q[P\G^2]}V_{\Gamma_1,\Gamma_2}^{\Q}\right]^{\Gamma^2}\to 0.$$
Et ce denier peut être identifié à:
$$\left[V_{\Gamma_1,\Gamma_2}^{\Q}\otimes H_2(P_2)^0\right]^{\Gamma^2}
\cong V_{\Gamma_1,\Gamma_2}^{\Q}[\mathcal{J}_{\Gamma_1,\Gamma_2}]+E_{\Gamma_1,\Gamma_2}^{\Q}.$$
Car d'une part :
$$H_2(P_2)^0=\{d\omega=\varphi_0^*\omega=\varphi_1^*\omega=\varphi_2^*\omega=0\}^{\bot}/\left(\{d\omega=\varphi_0^*\omega=0\}^{\bot}+\{d\omega=\varphi_1^*\omega=0\}^{\bot}+\{d\omega=\varphi_2^*\omega=0\}^{\bot}\right).$$
Et d'autre part:
\begin{align*}
\Omega_{\Gamma_1,\Gamma_2}&=\{d\omega=\varphi_0^*\omega=\varphi_1^*\omega=\varphi_2^*\omega=0\}\otimes_{\Q[P\G^2]}V_{\Gamma_1,\Gamma_2}^{\Q},\\
\text{et }E_{\Gamma_1,\Gamma_2}^{\Q}&=V_{\Gamma_1,\Gamma_2}^{\Q}[I_H]+V_{\Gamma_1,\Gamma_2}^{\Q}[I_D]+V_{\Gamma_1,\Gamma_2}^{\Q}[I_V].
\end{align*}
Ceci permet d'obtenir l'égalité: $\dim_{\Q}\left(W_{\Gamma_1,\Gamma_2}^{\Q}\right)=\dim_{\Q}\left(V_{\Gamma_1,\Gamma_2}^{\Q}[\mathcal{J}_{\Gamma_1,\Gamma_2}]+E_{\Gamma_1,\Gamma_2}^{\Q}\right)$. Puis on obtient l'égalité de ces espaces car l'idéal $\mathcal{I}_2$ est inclus dans les idéaux $\widetilde{\mathcal{J}_{\Gamma_1,\Gamma_2}}$, $I_H$ , $I_D$ et $I_V$ donc induit l'inclusion inverse de leurs annulateurs.
\end{proof}

\begin{prop}
On dispose de la décomposition du $\C$-espace vectoriel:
\begin{equation}
W_{\Gamma_1,\Gamma_2}=\Per_{\Gamma_1,\Gamma_2}\oplus \overline{\Per_{\Gamma_1,\Gamma_2}}\oplus E_{\Gamma_1,\Gamma_2}.
\end{equation}
Si l'on suppose de plus $\varepsilon\Gamma_1\varepsilon=\Gamma_1$ et $\varepsilon\Gamma_2\varepsilon=\Gamma_2$ alors on peut considérer les parties paire et impaire et obtenir:
\begin{equation}
W_{\Gamma_1,\Gamma_2}^+=\Per_{\Gamma_1,\Gamma_2}^+\oplus E_{\Gamma_1,\Gamma_2}^+ \text{ et }
W_{\Gamma_1,\Gamma_2}^-=\Per_{\Gamma_1,\Gamma_2}^-\oplus E_{\Gamma_1,\Gamma_2}^-.
\end{equation}
\end{prop}

\subsection{Calcul de $V_{\Gamma_1,\Gamma_2}^{\Z}[I_D]$}

Soit $(\Gamma_1,\Gamma_2)$ le plus petit sous-groupe de $PSL_2(\Z)$ contenant $\Gamma_1$ et $\Gamma_2$.\\
Soit $\Gamma_1\cap\Gamma_2$ est le plus grand sous-groupe de $PSL_2(\Z)$ contenu dans $\Gamma_1$ et $\Gamma_2$. On a:
$$\Gamma_1\cap\Gamma_2\subset \Gamma_1,\Gamma_2 \subset (\Gamma_1,\Gamma_2).$$
On peut construire les applications $\Z$-linéaires:
\begin{align}
\Psi:V_{\Gamma_1\cap\Gamma_2}^{\Z}\to V_{\Gamma_1,\Gamma_2}^{\Z},\quad&[\Gamma_1\cap\Gamma_2 g]\mapsto [\Gamma_1 g,\Gamma_2 g],\\
\Phi:V_{\Gamma_1,\Gamma_2}^{\Z}\to V_{(\Gamma_1,\Gamma_2)}^{\Z},\quad&[\Gamma_1 g_1,\Gamma_2 g_2]\mapsto [(\Gamma_1,\Gamma_2) g_1]-[(\Gamma_1,\Gamma_2) g_2].
\end{align}
Ces applications sont des homomorphismes de $PSL_2(\Z)$-module, où l'on muni $V_{\Gamma_1,\Gamma_2}^{\Z}$ de la structure diagonale. En effet, pour tout $\gamma\in PSL_2(Z)$:
\begin{align*}
\Phi([g_1,g_2])|_{\gamma}&=[(\Gamma_1,\Gamma_2) g_1]|_{\gamma}-[(\Gamma_1,\Gamma_2) g_2]|_{\gamma}\\
&=[(\Gamma_1,\Gamma_2) g_1\gamma]-[(\Gamma_1,\Gamma_2) g_2\gamma]=\Phi([g_1\gamma,g_2\gamma]),\\
\text{et }\Psi([g])|_{\gamma}&=[\Gamma_1g,\Gamma_2g]|_{\gamma}=[\Gamma_1g\gamma,\Gamma_2g\gamma]=\Psi([g\gamma]).
\end{align*}

\begin{prop}
La suite suivante est exacte:
\begin{equation}
0\to W_{\Gamma_1\cap\Gamma_2}^{\Q}\stackrel{\Phi}{\to}V_{\Gamma_1,\Gamma_2}^{\Q}[I_D]\stackrel{\Psi}{\to}W_{(\Gamma_1,\Gamma_2)}^{\Q}\to 0.
\end{equation}
De plus, elle est scindée et on peut écrire:
$$V_{\Gamma_1,\Gamma_2}^{\Q}[I_D]=W_{\Gamma_1\cap\Gamma_2}^{\Q}\oplus W_{(\Gamma_1,\Gamma_2)}^{\Q}.$$
\end{prop}

\begin{proof}
Les applications étant des homomorphismes de groupes, la restriction prise aux espaces des périodes à bien un sens.\\
Puis l'application $\Phi$ est injective car $(\Gamma_1\cap\Gamma_2)g=(\Gamma_1 g)\cap(\Gamma_2 g)$. Ainsi sa restriction, $W_{\Gamma_1\cap\Gamma_2}\to V_{\Gamma_1,\Gamma_2}[I_D]$ aussi.\\
On a $\Ker \Psi=Im \Phi$ qui est valide dans $V_{\Gamma_1,\Gamma_2}$ donc aussi dans $V_{\Gamma_1,\Gamma_2}[I_D]$.\\
Enfin, on montre explicitement la surjectivité en construisant un antécédent à tout élément de $W_{(\Gamma_1,\Gamma_2)}$. On commence par l'élément non parabolique on a: 
\begin{align*}
\Psi([\Gamma_1,\Gamma_2 S]-[\Gamma_1 S,\Gamma_2 ])&=2[(\Gamma_1,\Gamma_2)]-2[(\Gamma_1,\Gamma_2)S]\in E_{(\Gamma_1,\Gamma_2)}^{\Q}\\
([\Gamma_1,\Gamma_2 S]-[\Gamma_1 S,\Gamma_2])|_{(1,1)+(S,S)}&=0\\
([\Gamma_1,\Gamma_2 S]-[\Gamma_1 S,\Gamma_2])|_{(1,1)+(U,U)+(U^2,U^2)}&=0
\end{align*}
Car $SU=T^{-1}\in\Gamma_{\infty}\subset\Gamma_j$ et donc $\Gamma_j SU=\Gamma_j$ pour $j=1,2$.\par
Pour $v\in \Per_{(\Gamma_1,\Gamma_2)}$, il existe une forme $f\in S_2((\Gamma_1,\Gamma_2))$ tel que $v=P_f$. Posons:
$$u=\langle \sum_{g\in (\Gamma_1,\Gamma_2)}f|_g(z)dz[g,gS],\tau_1\rangle,\text{ alors }\Psi(u)=P_f-P_f|_S=2P_f.$$
Lorsque $v\in \overline{\Per_{(\Gamma_1,\Gamma_2)}}$, il suffit de prendre le conjugué complexe. L'égalité des dimensions obtenue sur $\C$ sont encore valide sur $\Q$ et la suite reste exacte sur les corps des rationnelles.
\end{proof}

Comme dans le cas du niveau $1$, la démonstration nous donne une méthode de construction valide sur $\Q$ après avoir identifiée les périodes comme étant les coordonnées d'un vecteur de $W_{(\Gamma_1,\Gamma_2)}^{\C}$. La connaissance des générateurs des espaces de longueur $1$ permet donc d'avoir une construction explicite des espaces rationnelles de longueur $2$.

\section{Généralisation aux représentations irréductibles de $PSL_2(\Z)$} 

\subsection{Vecteurs des périodes d'une représentation}

Soit $V^{\Z}$ un $PSL_2(\Z)$-module. Pour tout anneau commutatif $A$, notons $V^A=V^{\Z}\otimes_{\Z} A$ le $A$-module associé. Nous noterons simplement $V=V^{\Z}\otimes\C$ le $\C$-espace vectoriel associé, on le suppose de dimension finie. De plus, munissons le groupe $V^{\Z}$ d'une conjugaison compatible à l'action. C'est l'endomorphisme $\Z$-linéaire vérifiant:
\begin{equation}
c: V^{\Z}\to V^{\Z}\text{ tel que } c(v|_{\gamma})=c(v)|_{\epsilon\gamma\epsilon}\text{ pour tout }\gamma\in PSL_2(\Z)\text{ et }v\in V^{\Z},
\end{equation}
où $\varepsilon=\mat{-1}{0}{\phantom{-}0}{1}$.\par
Définissons alors le $\R$-espace-vectoriel des \textit{formes modulaires paraboliques harmoniques réelles}:
\begin{equation}
\Omega_{V}=\{\omega\in\Omega_{par}^1(\H,V)\text{ tel que }\omega(\gamma.z)|_{\gamma}=\omega(z)\text{ et }\bar{\omega}(z)=c(\omega)(-\bar{z})\}.
\end{equation}
On décompose celui-ci en $\Omega_V=\Omega_V^{\holo}\oplus c(\Omega_V^{\holo})$ où:
$$\Omega_V^{\holo}=\{\omega\in\Omega_V\text{ holomorphe}\}.$$
Considérons l'espace des \textit{vecteurs des périodes}:
\begin{equation}
\Per_{V}=\{\langle \omega,\tau_1\rangle\text{ pour tout }\omega\in\Omega_{V}^{\holo}\},
\end{equation}
et celui déterminé par les relations de Manin:
\begin{equation}
W_{V}^{\Z}=\{v\in V^{\Z}\text{ tel que }v|_{1+S}=v|_{1+U+U^2}=0\}.
\end{equation}
Il contient l'espace $E_{V}^{\Z}=(V^{\Z})^T|_{(1-S)}$.
\begin{prop}
On dispose de la décomposition:
\begin{equation}
W_{V}=\Per_{V}\oplus\overline{\Per_{V}}\oplus E_V^{\C}.
\end{equation}
\end{prop}

\begin{proof}
Définissons l'application:
\begin{align*}
\Omega^{\holo}_{V}\oplus c(\Omega_{V}^{\holo})&\to Z_{par}^1(PSL_2(\Z),V)=\{\varphi:PSL_2(\Z)\to V;\varphi(T)=0\text{ et }\varphi(\gamma_1\gamma_2)=\varphi(\gamma_1)|_{\gamma_2}+\varphi(\gamma_2)\}\\
\omega&\mapsto \varphi_{\omega}:\left(\gamma\mapsto \langle\omega,\{i\infty,\gamma^{-1}i\infty\}\rangle\right).
\end{align*}
Cette application est bien définie car pour tout $\omega$, $\varphi_{\omega}(T)=0$ et:
\begin{multline*}
\varphi_{\omega}(\gamma_1\gamma_2)=\langle\omega,\{i\infty,(\gamma_1\gamma_2)^{-1}i\infty\}\rangle\\
=\langle\omega,\{\gamma_2^{-1}i\infty,\gamma_2^{-1}\gamma_1^{-1}i\infty\}\rangle+\langle\omega,\{i\infty,\gamma_2^{-1}i\infty\}\rangle
=\varphi_{\omega}(\gamma_1)|_{\gamma_2}+\varphi_{\omega}(\gamma_2).
\end{multline*}
C'est en faite un isomorphisme car on peut construire une application réciproque. En effet, le Théorème d'Eichler-Shimura où dans notre contexte la surjectivité de l'application $\Theta_1:\Z[PSL_2(\Z)]\to H_1(P_1)^0$ permet de reconstruire à partir de $\varphi_{\omega}$ les périodes
$\langle \omega,\{a,b\}\rangle$ pour tout $a,b\in\pte$. Ces valeurs permette d'obtenir:
$$Div^0(\pte)\to V,\quad (b)-(a)\mapsto \langle \omega,\{a,b\}\rangle.$$
et les conditions d'invariances déterminent entièrement la forme par la correspondance de Riemann-Roch.\par
Pour utiliser un résultat de cohomologie générale de $PSL_2(\Z)$, il nous reste à déterminer les cobords:
$$B_{par}^1(PSL_2(\Z),V)=\{\gamma\mapsto v|_{1-\gamma}\text{ pour }v\in V^T\}.$$
L'isomorphisme $\varphi\mapsto\varphi(S)$ donne donc bien $W_V=\Per_{V}\oplus\overline{\Per_V}\oplus V^T|_{(1-S)}$.\\
Car on a $\overline{\langle\omega,\tau_1\rangle}=\langle c(\omega),\tau_1\rangle$ par invariance de $\omega\in\Omega_V^{\holo}$.
\end{proof}

\subsection{Vecteurs des bipériodes d'un couple de représentations}

Soient $V_1^{\Z}$ et $V_2^{\Z}$ des $P\G^2$-modules fournissant une représentation sur $\Q$ irréductible et de dimension finie. On suppose de plus qu'il existe des conjugaisons:
\begin{equation*}
c_j: V_j^{\Z}\to V_j^{\Z}\text{ tel que } c_j(v|_{\gamma})=c_j(v)|_{\epsilon\gamma\epsilon}\text{ pour tout }\gamma\in PSL_2(\Z)\text{ et }v\in V_j^{\Z},\text{ pour }j=1,2.
\end{equation*}
Définissons:
\begin{equation}
\Omega_{V_1,V_2}^{\holo}=\Omega_{V_1}^{\holo}\wedge \Omega_{V_2}^{\holo}\subset \Omega^2_{par}(\H^2,V_1\otimes V_2).
\end{equation}
Ceci permet de considérer l'espace des bipériodes:
\begin{equation}
\Per_{V_1,V_2}=\{\langle \omega,\tau_2\rangle\text{ pour }\omega\in\Omega_{V_1,V_2}^{\holo}\}.
\end{equation}
Ainsi que l'espace vérifiant les relations de Manin de longueur $2$:
\begin{equation}
W_{V_1,V_2}^{\Z}=\{v\in V_1^{\Z}\otimes V_2^{\Z}\text{ tel que }v|_g=0\text{ pour tout }g\in\widetilde{\mathcal{I}_2}\}.
\end{equation}
Puis définissons le sous-espace particulier de $W_{V_1,V_2}^{\Z}$ suivant:
\begin{equation}
E_{V_1,V_2}^{\Z}=\left(W_{1}^{\Z}\otimes (V_2^{\Z})^T\right)+\left((V_1^{\Z})^{US}\otimes W_{2}^{\Z}\right)+ (V_1^{\Z}\otimes V_2^{\Z})[I_D].
\end{equation}

\begin{thm}
On dispose de la décomposition:
\begin{equation}
W_{V_1,V_2}=\Per_{V_1,V_2}\oplus \overline{\Per_{V_1,V_2}}\oplus E_{V_1,V_2}.
\end{equation}
\end{thm}

\begin{proof} Pour tout idéal à gauche $I$ tel que $\widetilde{I}$ agisse sur un module $M$ à droite, notons:
$$M[I]=\{m\in M\text{ tel que }m|_{\gamma}=0\text{ pour tout }\gamma\in \widetilde{I}\}.$$
Cette notation donne $W_{j}^{\Z}=V_j^{\Z}[\mathcal{I}_1]$ et $W_{V_1,V_2}^{\Z}=(V_1^{\Z}\otimes V_2^{\Z})[\mathcal{I}_2]$.\\
On va à nouveau introduite l'application de multiplication par $(1,1)+(S,S)$ pour cela nous élargissons l'espace de départ pour bien obtenir une application surjective. Définissons:
$$Per^{+,+}_{V_1,V_2}=\Per_{V_1,V_2}\text{ et }Per^{-,-}_{V_1,V_2}=(c_1\otimes c_2)(\Per_{V_1,V_2}),$$
ce sont tous les deux des sous-espaces de $W_{V_1,V_2}$ car on a vu que $\mathcal{I}_2$ est stable par conjugaison diagonale par $\varepsilon$. Puis posons:
$$Per^{-,+}_{V_1,V_2}=(c_1\otimes id)(\Per_{V_1,V_2})\text{ et }Per^{+,-}_{V_1,V_2}=(id\otimes c_2)(\Per_{V_1,V_2}),$$
ils sont eux dans $(V_1\otimes V_2)[\mathcal{I}_2^{-,+}]$. C'est quatre espaces sont clairement deux à deux disjoints. Posons:
$$\delta:V_1\otimes V_2\to V_1\otimes V_2,\quad v_1\otimes v_2\mapsto c_1(v_1)\otimes v_2.$$
Cette involution est bien $\C$-linéaire et échange les espaces $W_{V_1,V_2}$ et $\delta(W_{V_1,V_2})=(V_1\otimes V_2)[\mathcal{I}_2^{-,+}]$. Posons $E_M=M^T|_{(1-S)}$ et définissons:
$$\phi_S:W_{V_1,V_2}+\delta(W_{V_1,V_2})\to \left(W_{V_1}/E_{V_1}\right)\otimes\left(W_{V_2}/E_{V_2}\right),$$
la projection dans l'espace quotient de l'image par $(1,1)+(S,S)$.\par
L'étude des relations $\mathcal{I}_2$ montre que cette application est bien définie. Elle est surjective.\par
En effet, un tenseur pûr de l'image $P_1\otimes P_2$ est associé à $\omega_1\in\Omega_{V_1}^{\holo}\oplus c_1(\Omega_{V_1}^{\holo})$ et $\omega_2\in\Omega_{V_2}^{\holo}\oplus c_2(\Omega_{V_2}^{\holo})$ vérifiant $P_j=\langle \omega_j,\tau_1\rangle$ pour $j=1,2$ d'après le théorème précédent. Ainsi l'élément:
$$P=\langle \omega_1\wedge\omega_2,\tau_2\rangle\in \bigoplus_{\epsilon_1,\epsilon_2} Per^{\epsilon_1,\epsilon_2}_{V_1,V_2}\subset W_{V_1,V_2}+\delta(W_{V_1,V_2}),$$
admet pour image par $\varphi_S$ l'élément $P_1\otimes P_2$ car $[(1,1)+(S,S)]\tau_2=\tau_1\times\tau_1$.\par
Calculons désormais le noyau de $\varphi_S$, on a:
$$\Ker(\varphi_S)=(V_1\otimes V_2)[I_H]+ (V_1\otimes V_2)[I_V]+ (V_1\otimes V_2)[I_D] + (V_1\otimes V_2)[I_D^{-,+}].$$
La démonstration pour le cas du niveau $1$ reposait entièrement sur l'étude des relations et donc s'adapte sans difficulté. Ceci permet d'obtenir:
\begin{align*}
\Ker(\varphi_S)\cap W_{V_1,V_2}^{\Z}&=(V_1^{\Z}\otimes V_2^{\Z})[I_H]+ (V_1^{\Z}\otimes V_2^{\Z})[I_V]+ (V_1^{\Z}\otimes V_2^{\Z})[I_D],\\
\text{et }\Ker(\varphi_S)\cap \delta(W_{V_1,V_2}^{\Z})&=(V_1^{\Z}\otimes V_2^{\Z})[I_H]+ (V_1^{\Z}\otimes V_2^{\Z})[I_V]+ (V_1^{\Z}\otimes V_2^{\Z})[I_D^{-,+}].
\end{align*}
Ceci nous donnant:
$$W_{V_1,V_2}=\Per_{V_1,V_2}\oplus Per^{-,-}_{V_1,V_2}\oplus \left(E_{V_1,V_2}^{\Z}\otimes \C\right).$$
Il reste à spécifier certain de ces espaces.\par
Pour tout $\omega\in\Omega_{V_1,V_2}^{\holo}$, on a: 
$(c_1\otimes c_2)\langle \omega,\tau_2\rangle=\overline{\langle \omega,\tau_2\rangle},$
par invariance de la conjugaison sur $\omega$ et car $\tau_2$ est stable par $z\mapsto(-\bar{z})$. Ceci donne:
$$\Per_{V_1,V_2}^{-,-}=(c_1\otimes c_2)\Per_{V_1,V_2}=\overline{\Per_{V_1,V_2}}.$$\par
D'autre part, on a:
$$(V_1^{\Z}\otimes V_2^{\Z})[I_H]=W^{\Z}_{1}\otimes (V_2^{\Z})^T\text{ et }(V_1^{\Z}\otimes V_2^{\Z})[I_D]=(V_1^{\Z})^{US}\otimes W_{2}^{\Z}.$$
Car les relations déterminant les idéaux $I_H$ et $I_V$ se scinde respectivement selon deux idéaux $(I_1,1)$ et $(1,I_2)$ qui commutent. 
Et on applique alors la formule:
$$(V_1^{\Z}\otimes V_2^{\Z})[(I_1,1)+(1,I_2)]=V_1^{\Z}[I_1]\otimes V_2^{\Z}[I_2].$$
\end{proof}

\begin{rem}
Ceci permet de caractériser les bipériodes de manière précise. Seul le calcul de $V_1^{\Z}\otimes V_2^{\Z}[I_D]$ semble dépendre du contexte bien qu'il soit un espace vérifiant les relations Manin de longueur $1$.
\end{rem}

\subsection{Vecteurs des multipériodes d'une famille de représentations}

On remarque que l'action de $PSL_2(\Z)$ sur $\H$ est souvent complétée par celle de $\varepsilon$ donnée par $\varepsilon.z=-\bar{z}$ pour tout $z\in\H$. Ainsi plutôt de que de choisir des actions de $PSL_2(\Z)$ compatible à une conjugaison, on prendra des actions de $PGL_2(\Z)=\{\mat{a}{b}{c}{d};ad-bc=\pm 1\}/\{\pm id\}$ et l'action de $\varepsilon$ correspondra à la conjugaison.\par

Soit $\rho:PGL_2(\Z)\to GL(V_{\rho}^{\Z})$ une représentation irréductible sur $\Z$ de $PGL_2(\Z)$. On suppose que $V_{\rho}^{\Z}$ est un $\Z$-module libre de rang fini sans torsion. Ainsi $\rho$ s'entend au corps des complexes en posant $V_{\rho}=V_{\rho}^{\Z}\otimes\C$. Alors on peut construire l'espace des \textit{formes modulaires paraboliques et harmoniques} pour cette représentation:
\begin{equation}
\Omega_{\rho}=\{\omega\in\Omega^1_{par}(\H,V_{\rho})\text{ tel que }\omega(\gamma.z)=\rho(\gamma)\left[\omega(z)\right]\text{ pour }\gamma\in PGL_2(\Z)\text{ et }d\omega=0\}.
\end{equation}
Notons $\Omega_{\rho}^{\holo}$ le sous-espace des formes holomorphes. Alors $\rho(\varepsilon)\Omega_{\rho}^{\holo}$ est celui des formes antiholomorphes.

\begin{prop}
L'application linéaire suivante est injective:
\begin{equation}
\Omega_{\rho}^{\holo}\to V_{\rho},\quad \omega\mapsto \langle \omega,\tau_1\rangle.
\end{equation}
Nous noterons $Per(\rho)$ son image.\\
Posons : $W_{\rho}=\{v\in V_{\rho}\text{ tel que }\rho(1+S)(v)=\rho(1+U+U^2)(v)=0\}$. Alors on a : 
\begin{equation}
W_{\rho}=Per(\rho)\oplus\overline{Per(\rho)}\oplus \rho(1+S)(V_{\rho}^T).
\end{equation}
\end{prop}

\begin{proof}
On peut construire l'application:
$$\Omega_{\rho}\to H^1_{par}(PGL_2(\Z),V_{\rho}),\omega\mapsto(\varphi_{\omega}:\gamma\mapsto \langle \omega,\{\gamma i\infty,i\infty\}).$$
Soit $\omega\in\Omega_{\rho}$. On a bien $\varphi_{\omega}(T)=0$, $\varphi_{\omega}(\varepsilon)=0$ et $\varphi_{\omega}(\gamma_1\gamma_2)=\rho(\gamma_1)\varphi_{\omega}(\gamma_2)+\varphi_{\omega}(\gamma_1)$. Puis la théorie des symboles modulaires, c'est à dire la surjectivité de l'application $\Theta_1$ permet d'obtenir l'application injective:
$$H^1_{par}(PGL_2(\Z),V_{\rho})\to V_{\rho}^{Div^0(\pte)},$$
tel que l'image de $\varphi_{\omega}$ soit l'application: $(b)-(a)\mapsto \langle \omega,\{a,b\}\rangle.$ La composée des deux applications est injective:
$$\Omega_{\rho}\to H^1_{par}(PGL_2(\Z),V_{\rho})\to V_{\rho}^{Div^0(\pte)}.$$
Et $H^1_{par}(PGL_2(\Z),V_{\rho})\to W_{\rho},\quad \varphi\mapsto \varphi(S)$ est bijective où:
$$W_{\rho}=\{v\in V_{\rho}\text{ tel que }\rho(1-\varepsilon)(v)=\rho(1+S)(v)=\rho(1+U+U^2)(v)=0\}.$$
Ainsi $\Omega_{\rho}\to V_{\rho},\quad \omega\mapsto \varphi_{\omega}(S)$ est bien injective. De plus, on dispose de l'inclusion $Per(\rho)\subset W_{\rho}$.
\end{proof}

\begin{thm}
Soient $\rho_1,...,\rho_n$ une famille de représentations complexes de $PGL_2(\Z)$. Alors l'application linéaire suivante est injective:
\begin{equation}
\bigotimes_{j=1}^n\Omega_{\rho_j}^{\holo}\to \left(\bigotimes_{j=1}^n V_{\rho_j}\right)^{\S_n},\quad \otimes_{j=1}^n\omega_j\mapsto \left(\sigma\mapsto\langle\wedge_{j=1}^n\omega_j,\tau_n^{\sigma}\rangle\right).
\end{equation}
De plus, si l'on note $MP(\otimes_{j=1}^n\rho_j)$ son image. Alors:
$$MP(\otimes_{j=1}^n\rho_j)\subset W_{\otimes_{j=1}^n\rho_j}=\left(\bigotimes_{j=1}^n V_{\rho_j}\right)^{\S_n}[\mathcal{A}_n].$$
\end{thm}

\begin{proof}
On peut démontrer l'injectivité par récurrence sur $n$, en adaptant une démonstration précédente. Pour cela, on utilise:
$$\sum_{\sigma\in\S_{n,1}} \tau_{n+1}^{\sigma}=\tau_n\otimes\tau_1.$$
Puis la liberté du cas $n=1$, correspondant à la proposition précédente, permet bien de déduire l'injectivité par récurrence.\par
Pour démontrer l'inclusion, il suffit de s'intéresser à l'action de $PGL_2(\Z)\rtimes \S_n$. Elle est donnée par:
\begin{align*}
[(\gamma_1,...,\gamma_n),\sigma].\left(\sigma'\mapsto \langle \wedge_{j=1}^n\omega_j,\tau_n^{\sigma'}\rangle\right)
&=\left(\sigma'\mapsto \otimes_{j=1}^n\rho_j(\gamma_{\sigma'(j)})(\langle \wedge_{j=1}^n\omega_j,\tau_n^{\sigma\sigma'}\rangle)\right)\\
&=\left(\sigma'\mapsto \langle \wedge_{j=1}^n\rho_j(\gamma_{\sigma'(j)})[\omega_j(z)],\tau_n^{\sigma\sigma'}\rangle\right)\\
&=\left(\sigma'\mapsto \langle \wedge_{j=1}^n\omega_j(\gamma_{\sigma'(j)} z),\tau_n^{\sigma\sigma'}\rangle\right)\\
&=\left(\sigma'\mapsto \langle \wedge_{j=1}^n\omega_j(z),(\gamma_1,...,\gamma_n)^{\sigma'}\tau_n^{\sigma\sigma'}\rangle\right)\\
&=\left(\sigma'\mapsto \langle \wedge_{j=1}^n\omega_j(z),((\gamma_1,...,\gamma_n)\tau_n^{\sigma})^{\sigma'}\rangle\right).
\end{align*}
Elle se transporte donc bien sur l'homologie et l'inclusion est donc due à la suite exacte:
$$0\to \mathcal{A}_n\to \Z[PSL_2(\Z)^n\rtimes \S_n] \to H_n(P_n)^0\to 0.$$
L'action de $\varepsilon$ étant dans ce contexte triviale: $\varepsilon=1$.
\end{proof}

On cherche désormais à résoudre les différentes actions de $\varepsilon$. Pour cela, posons:
$$\Omega_{\rho_j}^+=\Omega_{\rho_j}^{\holo}\text{ et }\Omega_{\rho_j}^-=\rho_j(\varepsilon)\Omega_{\rho_j}^{\holo}\text{ pour }1\leq j\leq n.$$
Ainsi pour toute famille de signes $(\epsilon_j)_{1\leq j\leq n}$, posons:
$$\Omega_{\rho_1\otimes...\otimes\rho_n}^{\epsilon_1,...,\epsilon_n}=\bigotimes_{j=1}^n\Omega_{\rho_j}^{\epsilon_j}.$$
On dispose du résultat suivant:

\begin{prop}
Pour toute famille de signe $\epsilon\in\{\pm 1\}^n$, l'application linéaire suivante est bien définie et est injective:
\begin{equation}
\Omega_{\rho}^{\epsilon}\to \left(V_{\rho}\right)^{\S_n},\quad \otimes_{j=1}^n\omega_j\mapsto \left(\sigma\mapsto\langle\wedge_{j=1}^n\omega_j,\tau_n^{\sigma}\rangle\right).
\end{equation}
L'image de cette application que nous noterons $MP^{\epsilon}(\rho)$ est contrôlée par:
\begin{equation}
MP^{\epsilon}(\rho)\subset \left(V_{\rho}\right)^{\S_n}[\mathcal{A}_n^{\epsilon}].
\end{equation}
\end{prop}

\begin{proof}
Ce résultat est une simple généralisation du théorème précédent lorsque que $\varepsilon$ est centrale. On peut ainsi adapter directement la démonstration précédente.
\end{proof}

Nous définissons à présent l'espace rationnel:
\begin{equation}
ME_{\rho}^{\epsilon,\Z}=\sum_{j=0}^n\sum_{\sigma\in\S_n} \left(V_{\rho}^{\Z}\right)^{\S_n}[\mathcal{A}_{n-1}[j,\sigma]^{\epsilon}].
\end{equation}
On dispose du théorème suivant:

\begin{thm}
Soit $\epsilon\in\{\pm 1\}^n$. Alors $\left(V_{\rho}^{\Q}\right)^{\S_n}[\mathcal{A}_n^{\epsilon}]$ est le plus petit sous-espace de $\left(V_{\rho}^{\Q}\right)^{\S_n}$ contenant $ME_{\rho}^{\epsilon,\Q}$ tel que son extension au corps des complexes contiennent  $MP^{\epsilon}(\rho)$.
\end{thm}

\begin{proof}
Pour cela, il suffit d'étudier le cas de l'idéal maximal vérifiant l'hypothèse:
$$\mathcal{A}(\rho)=\{a\in\Z[G_n]\text{ tel que }\langle \omega, \tau_n^{\sigma}\rangle|_a=0\text{ pour tout }\omega\in \Omega_{\rho}\}.$$
Le théorème précédent démontre que $\mathcal{A}_n\subset\mathcal{A}(\rho)$ et ainsi :
$$\left(V_{\rho}^{\Q}\right)^{\S_n}[\mathcal{A}(\rho)]\subset\left(V_{\rho}^{\Q}\right)^{\S_n}[\mathcal{A}_n].$$
D'autre part la construction de $\mathcal{A}_n$ démontre que $ME_{\rho}^{\Q}\subset \left(V_{\rho}^{\Q}\right)^{\S_n}[\mathcal{A}_n]$.\par
Il suffit donc d'obtenir l'égalité des dimensions pour démontrer le théorème. Or on dispose de la suite exacte:
$$0\to \left(V_{\rho}^{\epsilon,\Q}\right)^{\S_n}|_{\mathcal{A}_n} \to \left(V_{\rho}^{\epsilon,\Q}\right)^{\S_n} 
\to \left[ \left(V_{\rho}^{\epsilon,\Q}\right)^{\S_n} \otimes H_n(P_n)^0 \right]^{G_n} \to 0.$$
Et ce dernier est isomorphe à $\left(V_{\rho}^{\Q}\right)^{\S_n}[\mathcal{A}(\rho)]+ME_{\rho}^{\Q}$ d'après la dualité mis en place et les résultats obtenues sur l'homologie. 
Le résultat est alors valide pour un $\epsilon$ quelconque en passant à la conjugaison par $\rho_j(\varepsilon)$.
\end{proof}

\nocite{*}
\bibliography{bibmod}
\addcontentsline{toc}{chapter}{Bibliographie}

\cleardoublepage 
\renewcommand{\nomname}{Index des notations} 
\addcontentsline{toc}{chapter}{Index des notations}
\printnomenclature[2cm]
\end{document}